\documentclass{amsbook}

\makeatletter
\@addtoreset{equation}{section}

\makeatother

\usepackage{amsfonts, amsmath, amssymb, amsthm}
\usepackage{url}
\usepackage[dvipdfmx]{graphicx}
\usepackage{bmpsize}
\usepackage[dvipdfmx]{color}
\usepackage{amscd}
\usepackage[right]{lineno}
\usepackage{cleveref}
\usepackage{MnSymbol}
\usepackage{mathrsfs}
\usepackage{cancel}
\usepackage{tikz-cd}
\usepackage{enumerate}
\usepackage[all]{xy}

\theoremstyle{definition}
\newtheorem{definition}{Definition}
\crefname{definition}{Definition}{Definitions}
\newtheorem{theorem}{Theorem}[section]
\crefname{theorem}{Theorem}{Theorems}
\newtheorem{lemma}{Lemma}[section]
\newtheorem{proposition}{Proposition}[section]
\newtheorem{corollary}{Corollary}[section]
\theoremstyle{remark}
\newtheorem{example}{Example}[section]
\newtheorem{remark}{Remark}[section]

\newtheorem{conjecture}{Conjecture}[section]
\newtheorem{convention}{Convention}
\newtheorem{problem}{Problem}
\newtheorem{claim}{Claim}
\newtheorem{notation}{Notation}

\newtheorem{thmx}{Theorem}

\newcommand{\teich}{\mathcal{T}}

\newcommand{\pairingsymb}{\mathscr{P}}
\newcommand{\paircot}{\mathcal{P}^{\dagger}}
\newcommand{\teichmullernorm}{\boldsymbol{\tau}}
\newcommand{\LOneNorm}{\boldsymbol{n}}

\newcommand{\GuaidF}{\mathfrak{gf}}
\newcommand{\GoodS}{\mathfrak{L}}

\newcommand{\switch}{{\bf Swh}}

\newcommand{\Bers}[1]{\mathcal{T}^B_{#1}}

\newcommand{\proje}{\textbf{pr}}

\newcommand{\tv}[2]{\left\lsem#1\right\rsem_{{\scriptscriptstyle #2}}}
\newcommand{\tvD}[2]{\left\langlebar#1\right\ranglebar^{\scriptscriptstyle cycl}_{{\scriptscriptstyle #2}}}
\newcommand{\tvDBel}[2]{\left\langlebar#1\right\ranglebar_{{\scriptscriptstyle #2}}}
\newcommand{\tvdag}[2]{\left\lsem#1\right\rsem^\dagger_{{\scriptscriptstyle #2}}}

\newcommand{\ahlforsW}[1]{\mathscr{A}_{#1}}
\newcommand{\ahlforsWhat}[1]{\widehat{\mathscr{A}}_{#1}}
\newcommand{\Bersemb}{\mathscr{B}}

\newcommand{\isomorphism}{\mathscr{K}}

\newcommand{\verticalincop}{\mathbf{vi}}
\newcommand{\verticalprojop}{\mathbf{vp}}

\newcommand{\verticalinc}[2]{{\verticalincop}_{\scriptscriptstyle  #1:#2}}
\newcommand{\verticalincmodel}[1]{\verticalincop^{\mathbb{T}}_{\scriptscriptstyle #1}}
\newcommand{\verticalincmodelDol}[1]{\verticalincop^{\mathbb{T}:Dol}_{\scriptscriptstyle #1}}

\newcommand{\verticalincdag}[2]{\verticalincop^\dag_{\scriptscriptstyle  #1:#2}}

\newcommand{\verticalproj}[2]{{\verticalprojop}_{\scriptscriptstyle  #1:#2}}
\newcommand{\verticalprojdag}[2]{{\verticalprojop^\dag}_{\scriptscriptstyle  #1:#2}}

\newcommand{\Bersproj}{\mathscr{P}^B}
\newcommand{\trivialization}{\mathscr{I}}
\newcommand{\trivializationD}{\mathscr{I}^D}
\newcommand{\connectinghomo}{\mathscr{D}_0}

\newcommand{\canoflip}{{\boldsymbol \alpha}}
\newcommand{\canoflipdagger}{{\boldsymbol{dual}}^\dagger}

\newcommand{\tb}{\mathscr{T}\!\!\mathscr{B}}
\newcommand{\tbb}{\mathbf{TB}}

\newcommand{\sob}{\mathscr{C}}

\makeindex

\begin{document}

\title[The $L^1$-$L^\infty$-geometry]{The $L^1$-$L^\infty$-geometry of Teichm\"uller space \\ -- Second order infinitesimal structures --}
\author{Hideki Miyachi}

\date{\today}
\address{School of Mathematics and Physics,
College of Science and Engineering,
Kanazawa University,
Kakuma-machi, Kanazawa,
Ishikawa, 920-1192, Japan
}
\email{miyachi@se.kanazawa-u.ac.jp}
\thanks{This work is partially supported by JSPS KAKENHI Grant Numbers
20H01800,
20K20519,
22H01125}

\maketitle

\tableofcontents

\chapter*{Abstract}
The $L^1$-$L^\infty$ geometry is the Finsler geometry of the Teichm\"uller space by the Teichm\"uller metric and the $L^1$-norm function of holomorphic quadratic differentials.
By definition, the Teichm\"uller metric is the dual metric to the $L^1$-norm function, and the $L^1$-norm function is thought of as the co-norm of the Teichm\"uller metric.
In this paper, aiming to develop the $L^1$-$L^\infty$-geometry and the differential geometry on the Teichm\"uller space, we formulate the second order infinitesimal structures (the infinitesimal structures on the (co)tangent bundles) over the Teichm\"uller space. Indeed, we will give model spaces of the second order infinitesimal spaces, and the dual pairing between the (second order) tangent and cotangent spaces.

By applying our formulation, we give affirmative answers to two folklore. We first show that the map from the space of holomorphic quadratic differentials to the tangent bundle defined by assigning Teichm\"uller Beltrami differentials is a real-analytic diffeomorphism on every stratum of the space of holomorphic quadratic differentials defined according to the structures of zeros. Second, we show that the Teichm\"uller metric is real-analytic on the image of each stratum. We also develop the $L^1$-$L^\infty$-geometry by giving a new relation, called the infinitesimal duality, between the $\partial$-derivatives of the Teichm\"uller metric and the $L^1$-norm function.

\chapter{Introduction}
\label{chap:introduction}
The \emph{Teichm\"uller space} $\teich_g$ of Riemann surfaces of genus $g$ ($\ge 2$) is a deformation space of marked Riemann surfaces of genus $g$. The Teichm\"uller space $\teich_g$ is an infinite branched covering space of the Riemann moduli space of Riemann surfaces of genus $g$, and the orbifold covering group is essentially identified with the mapping class group of an orientable closed surface of genus $g$.
It was shown by Teichm\"uller \cite{MR0003242} that the Teichm\"uller space is homeomorphic to the Euclidean space $\mathbb{R}^{6g-6}$ of dimension $6g-6$.
The Teichm\"uller space $\teich_g$ admits a natural complex structure, and is known to be biholomorphically equivalent to a bounded domain in $\mathbb{C}^{3g-3}$.

\section{The $L^1$-$L^\infty$-geometry}
\subsection{Background}
In the Teichm\"uller theory, the infinitesimal deformation of a Riemann surface $M_0$ is formulated by differentiating the family of quasiconformal mappings. Hence, the holomorphic tangent space at $x_0=(M_0,f_0)\in \teich_g$ is described as the quotient space of the space $L^\infty_{(-1,1)}(M_0)$ of $L^\infty$-measurable $(-1,1)$ forms on $M_0$ with $L^\infty$-norm (for the notation, see \Cref{chap:Teichmuller_theory}). The holomorphic cotangent space at $x_0=(M_0,f_0)\in \teich_g$ is canonically identified with the space $\mathcal{Q}_{x_0}=\mathcal{Q}_{M_0}$ of holomorphic quadratic differentials on $M_0$ with the $L^1$-norm function\index{Quadratic differential!$L^1$ norm function}:
$$
\LOneNorm(q)=\|q\|=\dfrac{1}{2i}\iint_{M_0}|q(z)|d\overline{z}\wedge dz
$$
for $q=q(z)dz^2\in \mathcal{Q}_{x_0}$.
The identification comes from the natural pairing
\begin{equation}
\label{eq:pairing_intro}
\llangle \mu,q\rrangle =\dfrac{1}{2i}\iint_{M_0}\mu(z)q(z)d\overline{z}\wedge dz
\end{equation}
for $\mu=\mu(z)d\overline{z}/dz\in L^\infty_{(-1,1)}(M_0)$ and $q=q(z)dz^2\in \mathcal{Q}_{x_0}$. In fact, the equivalence relation for defining the the holomorphic tangent space is given by the pairing, which is known as Teichm\"uller's theorem. 

The \emph{Teichm\"uller metric} $\teichmullernorm$\index{Teichm\"uller metric} on $\teich_g$ is a Finsler metric on the (holomorphic) tangent bundle $T\teich_g$ which is defined as the dual metric of the $L^1$-norm function on the space $\mathcal{Q}_g$ of holomorphic quadratic differentials (the holomorphic cotangent bundle) via the above pairing:
\begin{equation}
\label{eq:teichmuller_intro}
\teichmullernorm(x_0,v)=
\sup\{{\rm Re}\llangle \mu, q\rrangle\mid q\in \mathcal{Q}_{x_0}, \LOneNorm(q)=1\},
\end{equation}
where $x_0=(M_0,f_0)\in \teich_g$, $v\in T_{x_0}\teich_g$ and $\mu\in L^\infty_{(-1,1)}(M_0)$ presents the tangent vector $v$. The $L^1$-norm is \emph{reflexively dual} in the sense that
\begin{equation}
\label{eq:L1_intro}
\LOneNorm(q)=
\sup\{{\rm Re}\llangle \mu, q\rrangle\mid v=[\mu]\in T_{x_0}\teich_g, \teichmullernorm(x_0,v)=1\}
\end{equation}
for $q\in \mathcal{Q}_{x_0}$
(see \S\ref{subsec:reflexive_dual}).

The phrase ``the $L^1$-$L^\infty$-geometry" in the title is derived from the situation. Namely, it  means the geometry of Teichm\"uller space in terms of the $L^1$-norm on the holomorphic quadiratic differenitals (cotangent bundles), and the Teichm\"uller metric ($L^\infty$-norm) of the tangent bundles via the duality by the pairing, which is given from the duality between the $L^1$-space and the $L^\infty$-space. Incidentally, the Weil-Petersson geometry stands for the $L^2$-geometry on the Teichm\"uller space in this framework.

Royden gives important contributions to the $L^1$-$L^\infty$-geometry that the $L^1$-norm function and the Teichm\"uller metrics are of class $C^1$ on the complement of the zero sections, and that the Teichm\"uller metric coincides with the Kobayashi metric under the canonical complex structure (cf. \cite{MR0288254}. See also \cite{MR2194466}).
The length metric of the Teichm\"uller metric is called the \emph{Teichm\"uller distance}.
The characterization of the Teichm\"uller distance by the Kobayashi distance provides interactions between Complex geometry and the Teichm\"uller geometry (Extremal length geometry) (e.g. \cite{MR1635773}, \cite{MR0624820},  \cite{MR0955842}, \cite{MR1142683}, \cite{MR1031909}, \cite{MR2525030}, \cite{MR4028456}, and  \cite{MR4633651}).
There are enormous researches on the geometry of the Teichm\"uller distance from several points of view. See \cite{MR2655331} for instance.

The space $\mathcal{Q}_g=\cup_{x\in \teich_g}\mathcal{Q}_x$ of holomorphic quadratic differentials has interesting properties in other aspects as well. The space $\mathcal{Q}_g$ admits a natural stratification defined from the data of zeros of holomorphic quadratic differentials (cf. \S\ref{sec:statification_space_QD}). The Teichm\"uller Beltrami differentials \eqref{eq:TB-diff_formal} are the directions of geodesics and induce a natural flow action on $\mathcal{Q}_g$ preserving the stratified structure, which called the \emph{Teichm\"uller geodesic flow}. The dynamics of the Teichm\"uller geodesic flow presents the affine deformations of (singular) flat structures on a surface, and it is recently applied for several fields and for attacking and solving important problems (e.g. \cite{MR2913101},
\cite{MR3071503}, 
\cite{MR2264836}, 
\cite{MR2316268},
\cite{MR2010740}, 
\cite{MR2000471}, 
\cite{MR644018}, 
\cite{MR3822740},
\cite{MR866707},
and 
\cite{MR1094714}).

\subsection{Teichm\"uller Beltrami differenitals}
In the $L^1$-$L^\infty$-geometry, the following measurable $(-1,1)$-form
\begin{equation}
\label{eq:TB-diff_formal}
k\dfrac{\overline{q}}{|q|}=k\dfrac{\overline{q(z)}}{|q(z)|}\dfrac{d\overline{z}}{dz}\quad (k\ge 0, q\in \mathcal{Q}_{M_0}-\{0\})
\end{equation}
plays an important and fundamental role. Such $(-1,1)$-forms are called the \emph{(formal) Teich\"uller Beltrami differentials} (\cite{MR0003242}). The Teichm\"uller Beltrami differential is real-analytic expect at the zeros of the quadratic differential, and has no continuous-limit at each zero.

The Teichm\"uller Beltrami differentials are extremal in the two situations. When $0<k<1$, the quasiconformal mapping $g$ whose Beltrami differential is equal to \eqref{eq:TB-diff_formal} has a unique extremal in terms of the maximal dilatation in its homotopy class. This fact is recently known as Teichm\"uller's uniquness theorem (e.g. \cite[Theorem 5.9]{MR1215481}). The unique extremality of the differentials \eqref{eq:TB-diff_formal} shows that $\teich_g$ is homeomorphic to $\mathbb{R}^{6g-6}$, and makes the Teichm\"uller space $\teich_g$ to be a unique geodesic metric in terms of the Teichm\"uller distance.

When $v\in T_{x_0}\teich_g$ is presented by the Teichm\"uller Beltrami differenital \eqref{eq:TB-diff_formal} for $k>0$, it is unique \emph{infinitesimally extremal} in the sense that
$$
\teichmullernorm(x_0,v)=k=\left\|k\dfrac{\overline{q}}{|q|}\right\|_\infty={\rm Re}\left\llangle k\dfrac{\overline{q}}{|q|},\dfrac{q}{\|q\|}\right\rrangle,
$$
and, in addition, when a $(-1,1)$-form $\mu$ on $M_0$ presents $v$, $\|\mu\|_\infty\ge k$, and the equality holds only if $\mu=k\overline{q}/|q|$ almost everywhere on $M_0$. This is also known as Teichm\"uller's theorem (e.g. \cite
[Corollary 1]{MR0245787}, \cite{MR0393477} and \cite{MR0505549}).
The unique infinitesimally extremality of the differential \eqref{eq:TB-diff_formal} gives a bijection
$$
T^*_{x_0}\teich_g=\mathcal{Q}_{M_0}\ni q\mapsto \left[\|q\|\dfrac{\overline{q}}{|q|}\right]\in T_{x_0}\teich_g
$$
between the (holomorphic) cotangent space and the (holomorphic) tangent space (preserving the zero sections),
where the bracket $[\cdot]$ means the infinitesimal equivalence class (cf. \S\ref{sec:infinitesimal_deformation_RS}). Actually, the correspondence extends as a fiber bundle (not vector-bundle)-isomorphism between the (holomorphic) cotangent bundle to the (holomorphic) tangent bundle over the Teichm\"uller space.

 
It is known as that the importance of the differential \eqref{eq:TB-diff_formal} goes back to Herbert Gr{\"o}tzsch's works on extremal problems of conformal mappings from 1928 in which contain a powerful technique ``the length-area method" and an important observation on the (unique) extremality of affine mappings.
See Alberge-Papadopoulos \cite{MR4321185} and fruitful commentaries on Gr{\"o}tzsch's works and the history of the theory of quasiconformal mappings and the Teichm\"uller theory  in the same volume.
%
%
%
%
%
%

\section{Results}
The purpose of this paper is to develop the $L^1$-$L^\infty$-geometry on the Teichm\"uller space.
We will face two kinds of results in this paper. 
First, we formulate and develop the theory on the second order infinitesimal structures on the Teichm\"uller space. We mean by the \emph{second order infinitesimal structures} on the Teichm\"uller space the (co)tangent structures on the (co)tangent bundle over $\teich_g$. These are stages for discussing the differential geometry (Connections and Finsler geometry etc.). See \S\ref{sec:intro_future} below.

Next, applying our formulations of the second order infinitesimal structures, we will give affirmative answers to  (or concrete proofs of) two folklore.

\subsection{Two folklore}
The first folklore which we deal with is as follows:

\begin{thmx}[Teichmuller Beltrami map is diffeomorphic]
\label{thm:TB_map_is_diffeomorphism}
The Teichm\"uller Beltrami map
$$
\tb\colon \mathcal{Q}^\times_g\ni q\mapsto \left[\|q\|\dfrac{\overline{q}}{|q|}\right]\in T^\times \teich_g
$$
 is a real-analytic diffeomorphism onto its image on any stratum in $\mathcal{Q}_g$, where the symbol $\times$ means the complement of the zero section from each space.
\end{thmx}

The proof appears in \Cref{chap:TB_map}.
To this end, we first check that the Teichm\"uller Beltrami maps are real-analytic on each stratum (\S\ref{sec:smoothness_TBmap}). However, the author confesses that the real-analyticity itself is seemed to be well-known to experts. Indeed, our strategy is to show the following commutative diagram:
\begin{equation}
\label{eq:com_di_TB}
\minCDarrowwidth90pt
\xymatrix@C=50pt@R=40pt{
T\mathcal{Q}_g \ar[rd]^{D\Pi^\dagger_{\teich_g}}
\\
\mathcal{Q}^\times_g \ar[u]^{-\frac{1}{2}\mathscr{X}_{\LOneNorm^2}}\ar[r]_{\tb} 
& T\teich_g
%
%
}
\end{equation}
where $D\Pi^\dagger_{\teich_g}$ is the differential of the projection $\Pi^\dagger_{\teich_g}\colon \mathcal{Q}_g\to \teich_g$ and $\mathscr{X}_{\LOneNorm^2}$ is the Hamiltonian vector field of the square $\LOneNorm^2$ of the $L^1$-norm function on $\mathcal{Q}_g$ in terms of a canonical holomorphic symplectic structure on $\mathcal{Q}_g$ (cf. \S\ref{subsec:Hamiltonian_vector_field} and \S\ref{sec:smoothness_TBmap}).
Though the commutative diagram may also be known to experts (cf. \cite{MR1283559}), we confirm the diagram from the variational formula of the $L^1$-norm function (\Cref{coro:canonical_section}). The constant $-1/2$ at the head of the Hamiltonian vector field in the diagram may depend on the convention. 
The real-analyticity discussed here cannot be straightforwardly improved in the sense that there is $q_0\in \mathcal{Q}_g$ such that the Teichm\"uller Beltrami map is not differentiable in some direction at $q_0$ (cf. \cite{MR0288254} and \cite[Proposition 7.5.5]{MR2245223}).

After verifying the real-analyticity, we give a variational formula of the Teichm\"uller Beltrami maps (in our setting) to prove the non-degeneracy of the differential of the Teichm\"uller Beltrami map $\tb$ (cf. \Cref{prop:derivative_pairing_Teichmuller_Beltrami} and \S\ref{subsec:proof_TB_map_is_diffeomorphism}).

Following \Cref{thm:TB_map_is_diffeomorphism},
we define a subset $T\!\teich_g(k_1,\cdots,k_n;\epsilon)$ in $T\teich_g$ by
\begin{align*}
T\!\teich_g(k_1,\cdots,k_n;\epsilon)
&=\left\{v=\left[\|q\|\dfrac{\overline{q}}{|q|}\right]\in T\teich_g\mid q\in \mathcal{Q}_g(k_1,\cdots,k_n;\epsilon)
\right\}
\end{align*}
for $k_1$, $\cdots$, $k_n\in \mathbb{N}$ with $k_1+\cdots+k_n=4g-4$ and $\epsilon=\pm 1$,
where $\mathcal{Q}_g(k_1,\cdots,k_n;\epsilon)$ is the stratum with data $(k_1,\cdots,k_n)$ and $\epsilon$. From \Cref{thm:TB_map_is_diffeomorphism}, the Teichm\"uller Beltrami map
$$
\tb\colon \mathcal{Q}_g(k_1,\cdots,k_n;\epsilon)\to T\!\teich_g(k_1,\cdots,k_n;\epsilon)
$$
is a real-analytic diffeomorphism, and hence, the collection $\{T\!\teich_g(k_1,\cdots,k_n;\epsilon)\}$ defines a stratification in the (holomorphic) tangent bundle $T\teich_g$ by real-analytic submanifolds.
The $L^1$-norm function $\LOneNorm$ is real-analytic on each stratum in $\mathcal{Q}_g$ (e.g. \S\ref{sec:statification_space_QD}). The (unique) infinitesimal extremality deduces the \emph{duality} between the Teichm\"uller metric and the $L^1$-norm function:
\begin{equation}
\label{eq:T-N-dual}
\teichmullernorm(\Pi_{\teich_g}\circ \tb(q),\tb(q))=\LOneNorm(q)=\|q\|
\end{equation}
for $q\in \mathcal{Q}_g^\times$ (where $\Pi_{\teich_g}\colon T\teich_g\to \teich_g$ is the projection).
Thus, \Cref{thm:TB_map_is_diffeomorphism} gives an affirmative answer to the following second folklore.

\begin{thmx}[Teichm\"uller metric is real-analytic]
\label{thm:teichmuller_metric_is_real_analytic_strata}
The Teichm\"uller metric is real-analytic on each stratum in $T\teich_g$.
\end{thmx}
%

\subsection{Infinitesimal Duality}
The duality \eqref{eq:T-N-dual} between the $L^1$-norm function and the Teichm\"uller metric described  by the following commutative diagram:
$$
\minCDarrowwidth90pt
\xymatrix@C=50pt@R=40pt@r{
\mathcal{Q}_g=T^*\!\teich_g
\ar[rd]^{\tb}
\ar[r]^{\LOneNorm}
\ar[d]_{\Pi^\dagger_{\teich_g}}
& 
\mathbb{R} \\
\teich_g
&
T\teich_g
\ar[u]^{\teichmullernorm}
\ar[l]_{\Pi_{\teich_g}}
}
$$
We will also obtain variational formulae of the $L^1$-norm function of holomorphic quadratic  differentials and the Teichm\"uller metric in our setting (cf. \Cref{thm:variation_L1-norm} and \Cref{thm:derivative_T-metric}). 
From our vatiational formulae, we obtain another duality property in the $L^1$-$L^\infty$-geometry, called the \emph{infiniteimal duality}
by checking the the following commutative diagram:
\begin{equation}
\label{eq:infinitesimal_duality_diagram}
\minCDarrowwidth90pt
\xymatrix@C=50pt@R=40pt{
T^*\mathcal{Q}_g=T^*\!T^*\!\teich_g 
\ar@/_20pt/[rd]_{\Pi^\dagger_{\mathcal{Q}_g}}
&
T\mathcal{Q}_g=TT^*\!\teich_g 
\ar[r]^{\switch_{\teich_g}}
\ar[l]_{\canoflipdagger_{\teich_g}}
\ar@/^20pt/[d]^{\Pi_{\mathcal{Q}_g}}
&T^*T\teich_g 
\ar@/^20pt/[d]^{\Pi^\dagger_{T\teich_g}}
\\
&\mathcal{Q}^\times _g=(T^*)^\times \teich_g 
\ar[ul]_{-\partial \LOneNorm^2}
\ar[u]^{-\frac{1}{2}\mathscr{X}_{\LOneNorm^2}}
\ar[r]^{\tb}
&T^\times \teich_g
\ar[u]^{\partial\teichmullernorm^2}
}
\end{equation}
(cf. \S\ref{sec:infiniteismal_duality}).
The infinitesimal duality implies that the Teichm\"uller metric and the $L^1$-norm function are also in a dual relationship in the infinitesimal sense.
The maps $\switch_{\teich_g}\colon T\mathcal{Q}_g\to T^*T\teich_g$ and $\canoflipdagger_{\teich_g}\colon T\mathcal{Q}_g\to T^*\mathcal{Q}_g$ appearing in the diagram are biholomorphisms, which called the \emph{switch} and the \emph{dualization}, defined on the tangent bundle $T\mathcal{Q}_g$ over $\mathcal{Q}_g$ (cf. \S\ref{sec:infiniteismal_duality}).  As we discuss later, such maps are defined canonically in a general situation (cf. \Cref{chap:double_tangent_space}).

It is possible that the infinitesimal duality holds for dual Finsler metrics under a general setting. In fact, a reasonable conjecture is that the reflexive duality like as \eqref{eq:teichmuller_intro} and \eqref{eq:L1_intro} is sufficient (cf. \S\ref{subsec:reflexive_dual}). 

\subsection{Horizontal lifts}
We will define two kinds of horizontal lifts of the vector bundle $\mathcal{Q}_g\to \teich_g$ in 
\S\ref{subsec:tangent_space_to_the_unit_sphere_bundle} and \S\ref{subsec:conjectual_picture}
at generic differentials in $\mathcal{Q}_g$, where for a complex vector bundle $E\to \teich_g$, a $\mathbb{C}$-subspace of the tangent space $T_vE$ ($v\in E$) is called a \emph{horizontal lift} if it is $\mathbb{C}$-isomorphic to the tangent space to $\teich_g$ via the differential $TE\to T\teich_g$ of the projection.
Our horizontal lift given here are subspaces of the real tangent space $T^{\mathbb{R}}_q\mathcal{SQ}_g$ to the unit sphere bundle $\mathcal{SQ}_g\to \teich_g$ at a generic differential $q\in \mathcal{Q}_g$.
The horozontal lift in \S\ref{subsec:tangent_space_to_the_unit_sphere_bundle} is given from a general point of view. The horizontal lift is \S\ref{subsec:conjectual_picture} is defined from a complex analytic property of the Teichm\"uller-Beltrami map. Indeed, the latter is characterized as the maximal $\mathbb{C}$-subspace among $\mathbb{C}$-subspaces of $T^{\mathbb{R}}_q\mathcal{SQ}_g$ in which the $(1,0)$-part of the differential of the Teichm\"uller-Beltrami map is $\mathbb{C}$-linear. Hence, the images of the differential of the Teichm\"uller-Beltrami map also define horizontal lifts of the tangent bundle $T\teich_g\to \teich_g$.

\subsection{Cases of Riemann surfaces of analytically finite type}
For the simplicity of the notation, in this paper, we discuss mainly the Teichm\"uller space of closed Riemann surface of genus $g$ ($\ge 2$). 

Our second order infinitesimal structures are formulated from both the Teichm\"uller theory (quasiconformal deformations), and the Kodaira-Spencer theory (sheaf cohomologies).
After setting an appropriate situation (for instance, we think $\Theta_M(-D)$ and $\Omega_{M}^{\otimes 2}(D)$ ($D$ is the divisor of the marked points) instead of $\Theta_M$ and $\Omega_M^{\otimes 2}$), we can check that our results here also naturally hold in the case of the Teichm\"uller space of Riemann surfaces of analytically finite type with negative Euler characteristic, where $\Theta_M(-D)$ and $\Omega_{M}^{\otimes 2}(D)$ are the sheaf of holomorphic tangent vectors with zeros at the support of $D$, and the sheaf of holomorphic quadratic differentials with at most simple poles at the support of $D$. When $D=0$ as a divisor we write $\Theta_{M}(0)=\Theta_M$ and $\Omega_{M}^{\otimes 2}(D)=\Omega_{M}^{\otimes 2}$ here (e.g. \cite{MR499279}).

\section{Future}
\label{sec:intro_future}
Our formulations of the second order infinitesimal spaces over the Teichm\"uller space, and our results, \Cref{thm:TB_map_is_diffeomorphism}, \Cref{thm:teichmuller_metric_is_real_analytic_strata}, and the infinitesimal duality \eqref{eq:infinitesimal_duality_diagram}, give an interactive communication between the $L^1$-geometry on $\mathcal{Q}_g$ and the $L^\infty$-geometry on $T\teich_g$ in the infinitesimal level. 


\subsection{}
The $L^1$-$L^\infty$-geomety is thought of as a (complex) Finsler geometry with the Teichm\"uller metric (the $L^\infty$-norm), and also recognized as the (complex) Hamiltonian-Lagrange geometry with the $L^1$-norm function. These geometries are the infinitesimal geometries on the (holomorphic) tangent bundle $T\teich_g$ and the holomorphic cotangent bundles $T^*\teich_g=\mathcal{Q}_g$.
Our formulation of second order infinitesimal spaces will help the study of the complex Finsler geometry on the Teichm\"uller space $\teich_g$ from a higher perspective.
In this paper, a \emph{complex Finsler metric} $F=F(x,\xi)$ is by definition, a continuous function on the complex (holomorphic) tangent bundle $TM$ over a complex manifold $M$ with the following three properties:
\begin{itemize}
\item[(1)] (Complex homogeneity) $F(x,\alpha \xi)=|\alpha| F(x,\xi)$ for $\alpha\in \mathbb{C}$, $x\in M$, and $\xi\in T_xM$; 
\item[(2)]
(Non negativity)
$F(x,\xi)\ge 0$ for $x\in M$ and $\xi\in T_xM$;
and
\item[(3)] (Triangle inequality) $F(x,\xi_1+\xi_2)\le F(x,\xi_1)+F(x,\xi_2)$ for $x\in M$, $\xi_1$ and $\xi_2\in T_xM$
\end{itemize}
(cf. \cite{MR0833808}). Several versions of the definition of complex Finsler metrics are known. For instance, in \cite{MR0377126}, complex Finsler metrics (structures) are defined with the condition (1), the positivity and the smoothness except for the zero-section (see also \cite{MR0211370} and \cite{MR0216446}). Sometimes, the following pseudoconvexity condition (3') is adopted instead of the condition (3):
\begin{itemize}
\item[(3')]
(Pseudoconvexity)
the complex Hessian $\begin{bmatrix} \dfrac{\partial^2 F^2}{\partial \xi_i\partial\overline{\xi_j}}\end{bmatrix}$ of the square of $F$ is positive definite on $TM-\{0\}$.
\end{itemize}
(cf. \cite{MR1323428} and \cite{MR2102340}).
The regularity of $F$ or $F^2$ is dependent on the situation (cf. \cite{MR1323428}, \cite{MR2102340}, \cite{MR3033515} and \cite{MR0833808}).

A strategy in the complex Finsler geometry is to study the Hermitian metric defined by the complex Hessian (the Levi form)
of the square $F^2$ or $F$ itself on $TM$ or the tautological line bundle of $TM$.
Indeed, it is known
that for a complex vector bundle $E$ over $M$, there is one to one correspondence between the set of Hermitian structures on the tautological line bundle of $E$ and the Finsler structures on $E$ (cf. \cite{MR1172107} and \cite[\S4]{MR0377126}). 

In any case, the complex Finsler geometry is a complex differential geometry on the holomorphic tangent bundle $TM$ and the holomorphic cotangent space $T^*M$ with a Finsler metric $F$. 
The infinitesimal calculations (linear and non-linear connections, curvatures, sprays, etc.) of Finsler metrics are carried out on the second order (or more higher order) infinitesimal spaces (\cite{MR1323428}, \cite{MR1172107},  \cite{MR1031389}, \cite{MR0377126}, \cite{MR1403586}, and \cite{MR2102340}).
Further, Finsler manifolds are also studied with linear and homogeneous nonlinear connections and nonhomogeneous nonlinear ones with a canonical vector field (the complex Liouville vector field) in $TTM$, and it also constructed a conservative connection for a mechanical system (complex Hamiltonian-Lagrange geometry) (\cite{MR0281134}, \cite{MR0336636}, \cite{MR0341361} and \cite{MR0247599}).

\subsection{}
After formulating the second order infinitesimal structures on the Teichm\"uller space, it is natural to ask the existence of natural Finsler structures on the second order infinitesimal spaces to discuss the $L^1$-$L^\infty$-geometry. The infinitesimal duality given here is a gift from the (original) duality between the $L^1$-norm function and the Teichm\"uller metric. If natural $L^1$-$L^\infty$ structures on the second order (or more higher order) infinitesimal spaces exist, it may yield more detailed dualities between the $L^1$-norm function and the Teichm\"uller metric in the second order (or more higher order) level.

\subsection{}
Our models of second order infinitesimal spaces will be also naturally used for understanding and formulating \emph{(linear or non-linear) connections} on $T\teich_g$ and $T^*\teich_g=\mathcal{Q}_g$ from a bird's-eye view (cf. \cite{MR0350650}). 
In fact, it is known that in the case of real differentiable manifold $M$, any (linear) connection $\nabla$ on $TM$ is presented as $\nabla_XY=K(DY(X))$ for $X\in T_pM$ and a $C^1$-vector field $Y$ around $p\in M$ by a $C^\infty$ map $K\colon TTM\to TM$ with suitable properties
(cf. \cite[Proposition 4.1]{MR1390760}. See also \S\ref{subsec:vertical_space_manifold}). The same kind of formulation can be considered similary for the connections of the (holomorphic) tangent bundles over complex manifolds. Recently, there are already many calculations (descriptions) and important estimates of the Levi-Civita connection and the curvature tensors of the Weil-Petersson metrics (e.g. \cite{MR0204641}, \cite{MR0136730},  \cite{MR1164870},\cite{MR0982185}, \cite{MR0842050}, \cite{MR2641916}). We hope our formulation helps not only the study on the Weil-Petersson geometry (the $L^2$-geometry) but on the studies of general (Riemannian, Hermitian, K\"ahler) metrics.

\subsection{}
The second order infinitesimal (cone) structures are also naturally discussed for the Deligne-Mumford compactification of the moduli space of Riemann surfaces of genus $g$ (Riemann surfaces of analytically finite type). There are many researches of the asymptotic behavior of tensors defined from the Weil-Petersson metric (e.g. \cite{MR0417456} and \cite{MR2641916}). The formulation of the second order infinitesimal spaces or cones on the Deligne-Mumford compactification will provide a systematic formulation for studying the asymptotic behaviors of the tensors defined not only from the Weil-Petersson metric but from general (Riemannian, Hermitian, K\"ahler) metrics.

Further, as we see in \S\ref{subsec:geometry_double_tangent_vectors}, a second order tangent vectors on the Teichm\"uller space presents an infinitesimal deformation defined from the holomorphic family of Riemann surfaces over a $2$-dimensional polydisk. The degenerating families over a $2$-dimensional polydisk are studied deeply by Takamura (cf. \cite{MR2023456}, \cite{MR2254876}), and he discusses splitting phenomena of the central fiber and barking deformations, and classifies the atomic fibrations. A second order tangent vector at the boundary of the moduli space is (possibly) thought of as an infinitesimal version of the splitting deformation.


%
%
%
%

\section{Organization of the paper}
The Teichm\"uller metric, the $L^1$-norm function, and the Teichm\"uller Beltrami map are functions defined on the holomorphic tangent bunldle $T\teich_g$ and holomorphic cotangent bundles $T^*\teich_g=\mathcal{Q}_g$ over the Teichm\"uller space. Hence, to study such functions in the infinitesimal level, we need to formulate the second order infinitesimal spaces over the Teichm\"uller space.

To this end, in \Cref{chap:double_tangent_space}, we recall and give fundamental properties of the second order infinitesimal structures on complex manifolds. Indeed, in the chapter we will introduce (recall) three canonical biholomorphic mappings
\begin{align*}
\canoflip_M &\colon TTM\to TTM \\
\switch_M &\colon TT^*\!M\to T^*\!TM \\
\canoflipdagger_M &\colon TT^*\!M\to T^*\!T^*\!M
\end{align*}
which are called the \emph{flip}, the \emph{switch}, and \emph{dualization}, respectively, where $TTM=T(TM)$, $T^*TM=T^*(TM)$, $TT^*\!M=T(T^*\!M)$ and $T^*\!T^*\!M=T^*\!(T^*\!M)$. These mappings satisfy several relations. See \Cref{fig:double_infinitesimal_space} in \S\ref{sec:tangent_cotangent_space}. 
In this paper, we sometime denote by $\mathcal{P}_M$ the natural pairing between holomorphic tangent and cotangent bundles $TM$ and $T^*\!M$ in thinking of the pairing function as a holomorphic function on the Whitney sum $TM\oplus T^*\!M$. 
We call $TTM$ the \emph{double tangent bundle} in this paper. The reason for ``double" is that ``$TT$"(double $T$) appears in the symbol.
The space $TTM$ is also called the \emph{second order tangent bundle} (\cite{MR1724021}) and the \emph{second tangent bundle} (\cite{MR2428390}) for example.

In \Cref{chap:Teichmuller_theory}, we shall recall basic notion and results in the Teichm\"uller theory, and the Kodaira-Spencer theory. We also recall the stratified structure of the space of holomorphic quadratic differentials.

In \Cref{chap:Teichmuller_space_of_tori}, we discuss our results in the most simplest case.
Indeed, we deal with the Teichm\"uller $\teich_1$ of tori. Under a canonical complex structure on $\teich_1$, the Teichm\"uller metric is the Poincar\'e metric of curvature $-4$. We will see the infinitesimal duality holds in this case.

In \Cref{Chap:model_pairing}, we will introduce the model spaces of the second order infinitesimal spaces over the holomorphic tangent and cotangent spaces over $\teich_g$. These spaces consist of $TT\teich_g$, $T^*\!T\teich_g$, $TT^*\!\teich_g=T\mathcal{Q}_g$ and $T^*\!T^*\!\teich_g=T^*\!\mathcal{Q}_g$. The model space of the holomorphic tangent space $TT^*\!\teich_g=T\mathcal{Q}_g$ over $\mathcal{Q}_g$ is already given by Hubbard and Masur \cite{MR523212}, and we will recall their constructuon. We also give the models $\mathcal{P}_{\mathbb{TT}}$ and $\paircot$ of the pairings $\mathcal{P}_{T\teich_g}$ and $\mathcal{P}_{\mathcal{Q}_g}$ between the holomorphic tangent and holomorphic cotangent spaces over $T\teich_g$ and $\mathcal{Q}_g$. We also discuss the holomorphic symplectic structure on $\mathcal{Q}_g$ following Kawai's formula \cite{MR1386110}. 

In \Cref{chap:trivialization_model} and \Cref{chap:direct_limits}, we discuss the trivializations and the direct limits of the model spaces. In \Cref{chap:double_tangent_space_model}, we will confirm that our model space of the double tangent space $TT\teich_g$ is an actual model. Namely, we will check that when we identify the Teichm\"uller space $\teich_g$ as a bounded domain (the Bers slice) via the Bers embedding, our model of the double tangent space $TT\teich_g$ actually coincides with the double tangent space over the bounded domain. 

From \Cref{chap:variation_formula_pairing} to \Cref{chap:TB_map}, we will apply our model spaces for studying the structures of $T\teich_g$ and $\mathcal{Q}_g$. We will give a variational formula of the pairing function on the Whitney sum $T\teich_g\oplus \mathcal{Q}_g$ in \Cref{chap:variation_formula_pairing}. By applying the variational formula, we check that our models of $T^*\!T\teich_g$, $TT^*\!\teich_g$ , and $T^*\!T^*\!\teich_g$ are naturally recognized as actual models.  In \Cref{chap:flip_switch_dual_lie}, we describe the filp, switch, dualization and Lie bracket under our models. 

In \Cref{chap:variational_formula_L1-teich}, we give variational formulae of the $L^1$-norm function and the Teichm\"uller metric, and prove the infinitesimal duality theorem. In the chapter, we also give a proof of Royden's results on the regularity of the $L^1$-norm function under our setting, recall Royden's regularity criterion of the dual Finsler metric to check the $C^1$-regularity of the Teichm\"uller metric. For reader's convenience, we will give a proof of the criterion in \S\ref{sec:Royden_s_criterion}.
We also discuss the infinitesimal structure of the unit sphere (ball) bundle in $\mathcal{Q}_g$. After confirming that the unit ball bundle in $\mathcal{Q}_g$ is strictly $(3g-2)$-convex at every generic boundary point, we pose conjectures on the negative directions of the Levi form of the $L^1$-norm function and on the plurisubharmonicity of the Teichm\"uller metric.
In \Cref{chap:TB_map}, we will prove \Cref{thm:TB_map_is_diffeomorphism}, and discuss a conjectural picture on the geometry of the unit sphere bundles.

In the last chapter \Cref{chap:appendices}, as appendices, we discuss three subjects. First, we confirm the correspondence between the (infinitesimal) Teichm\"uller theory and the Kodaira-Spencer theory. Second, we confirm the regularity of the integral operators which we use in this paper. The proof of the regularity given here is due to 
Professor Hiroshi Yanagihara. Third, we will discuss Royden's criterion on the regularity of the dual Finsler metric. The proof is mostly same as that given in Gardiner's book \cite{MR903027}. However, we give a complete proof of Royden's criterion here because of the difficulty to access the source at present. 

\section*{Acknowledgements}
The author thanks Professor Athanase Papadopoulos and Professor Ken'ichi Ohshika for warm encouragement and discussions. Actually, casual discussions with them provide motivations of author's works.
The author thanks Professor Norbert A'Campo for fascinating discussions.
The author also thanks Professor Hiroshi Yanagihara for useful comments to the author's question and allowing him to give his proof in this paper, and Professor Shigeharu Takayama and Professor Tadashi Ashikaga for helpful discussions.
 
 A part of this work is announced in the annual report of a scientific program ``Teichm\"uller Theory: Classical, Higher, Super and Quantum" in 2023 at the Mathematisches Forschungsinstitut Oberwolfach (MFO) organized by Professor Ken'ichi Ohshika, Professor Athanase Papadopoulos, Professor Robert C. Penner, and Professor Anna Wienhard, and in the proceedings of the 14th MSJ- Seasonal Institute ``New Aspects of Teichm\"uller theory" in 2022 and 2023 at Gakusyuin university and the university of Tokyo organized by Professor Ken'ichi Ohshika, 
Professor Sumio Yamada, Professor Nariya Kawazumi, Professor Takuya Sakasai, and Professor Tsukasa Ishibashi.
The author thanks the organizers, and the staffs of the venues, Gakusyuin university, the university of Tokyo, and MFO, for their warm hospitalities.

\chapter{Second order infinitesimal spaces on complex manifolds}
\label{chap:double_tangent_space}
In this chapter, we will discuss the tangent and cotangent bundles over the tangent and cotangent bundles over a complex manifold $M$. 
We call such spaces the \emph{second order infinitesimal spaces}. We have four second order infinitesimal spaces $TTM$, $T^*\!TM$ (over $TM$) and $TT^*\!M$, $T^*\!T^*\!M$ (over $T^*\!M$). We recall and give elementary properties these four spaces and basic maps between two of them as described in \Cref{fig:double_infinitesimal_space}. We use the dagger ``$\dagger$" to describe the notion in the cotangent bundle.
For a general properties of complex manifolds, see \cite{MR1393941}. Especially, for double (holomorphic) tangent spaces, see Abate-Patrizio \cite{MR1323428}, Aikou \cite{MR1172107}, Fisher-Turner \cite{MR1724021}, Fukui \cite{MR1031389}, Konieczna-Urba{\'n}ski \cite{MR1684522}, Michor \cite{MR2428390}, Munteanu \cite{MR2102340}, Pradines \cite{MR0388432}, or Sakai \cite{MR1390760} for instance.

We only focus on complex manifolds in this paper, but almost results given here on second order infinitesimal spaces hold for differentiable manifolds. 

\section{Summary : Second order infinitesimal spaces and Relations}
\label{sec:double_infinitesimal_spaces_and_commutative_conditions}
Before discussing the details on second order infinitesimal spaces, we first summarize the relations of the spaces.

There are four second order infinitesimal spaces for complex manifolds, $TTM$, $T^*\!TM$, $TT^*\!M$ and $T^*\!T^*\!M$. We will define a self map called \emph{flip} $\canoflip_M$ on $TTM$, the map called the \emph{switch} $\switch_M\colon TT^*\!M\to T^*\!TM$ and the map called the \emph{dualization} $\canoflipdagger_M\colon TT^*\!M\to T^*\!T^*\!M$, which satisfy the commutative diagrams in \Cref{fig:double_infinitesimal_space}.
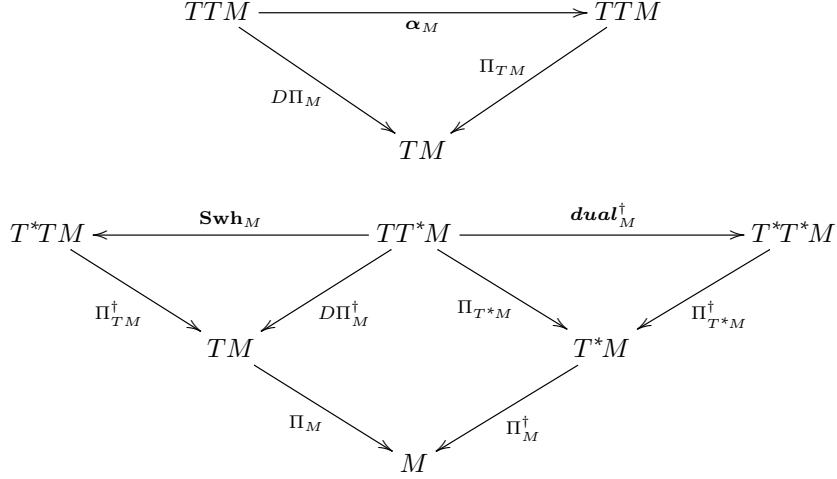
\begin{figure}[t]
$$
\xymatrix@C=50pt@R=40pt{
TTM
\ar[rd]_{D\Pi_{M}}
\ar[rr]_{\canoflip_M}
&
&
TTM
\ar[ld]_{\Pi_{TM}}
\\
&
TM
}
$$

$$
\xymatrix@C=40pt@R=30pt{
T^*\!TM
\ar[rd]_{\Pi^\dagger_{TM}}
&
&
TT^*\!M
\ar[ld]^{D\Pi^\dagger_M}
\ar[ll]_{\switch_M}
\ar[rd]_{\Pi_{T^*\!M}}
\ar[rr]^{\canoflipdagger_M}
&
&
T^*\!T^*\!M 
\ar[ld]^{\Pi^\dagger_{T^*\!M}}
\\
& 
TM
\ar[rd]_{\Pi_M}
&
&
T^*\!M 
\ar[ld]^{\Pi^\dagger_M}
&\\
&
& 
M
&
&
}
$$
\caption{Commutative diagrams : Second order infinitesimal spaces}
\label{fig:double_infinitesimal_space}
\end{figure}

\begin{notation}
For a subset $E\subset M$,  we define the fiber over $E$ by
\begin{align*}
(TT)_E(M)&=(\Pi_M\circ \Pi_{TM})^{-1}(E), \quad
(T^*\! T)_E(M)=(\Pi_M\circ \Pi^\dagger_{TM})^{-1}(E) \\
(TT^*\!)_E(M)&=(\Pi^\dagger_M\circ \Pi_{T^*\!M})^{-1}(E), \quad
(T^*\! T^*\!)_E(M)=(\Pi^\dagger_M\circ \Pi^\dagger_{T^*\!M})^{-1}(E).
\end{align*}
\end{notation}
For $p\in M$, we denote by $(TT)_p(M)$, $(T^*\! T)_p(M)$, $(TT^*\!)_E(M)$
and $(T^*\! T^*\!)_p(M)$ when $E=\{p\}$.
Each maps mentioned above preserve for the fibers over a point on $M$.

\section{Tangent and cotangent bundles}
\label{sec:tangent_cotangent_space}
Let $M$ be a complex manifold of dimension $n$. $M$ is a real dimensional real analytic manifold of (real) dimension $2n$. The complexification of the real tangent space $\tilde{T}M\otimes \mathbb{C}$ of $M$ is a complex vector bundle of (complex) rank $2n$ and is canonically splitted as the Whitney sum (direct sum) $T^{1,0}M\oplus T^{0,1}M$ of the holomorphic vector bundles of rank $n$. The spaces $T^{1,0}M$ and $T^{0,1}M$ are eigen spaces of the action of the complex structure on $M$ for the eigen values $i$ and $-i$ respectively (cf. \cite[Chapter IX]{MR1393941}). We call $T^{1,0}M$ and $T^{0,1}M$ the \emph{holomorphic and anti-holomorphic tangent bundle} of $M$. By definition, any complexified tangent vector $v$ is uniquely represented as the sum $v=v^{10}+v^{01}$ of vectors $v^{10}\in T^{1,0}M$ and $v^{01}\in T^{0,1}M$. We call $v^{pq}$ the \emph{$(p,q)$-part} of $v$ for $(p,q)=(1,0)$ and $(0,1)$. We denote by $\proje^{pq}\colon \tilde{T}M\otimes \mathbb{C}\to T^{p,q}M$ the projection.

For instance, suppose $M=\mathbb{C}^n$ and let $z_k=x_k+iy_k$ for $1\le k\le n$. Then, the underlying real manifold of $\mathbb{C}^n$ is the Euclidean space $\mathbb{R}^{2n}$ with coordinates $x_k$, $y_k$ ($1\le k\le n$). Any complexified tangent vector in $\tilde{T}\mathbb{C}\otimes \mathbb{C}$ satisfies
$$
\sum_{k=1}^n\left(\alpha_k\partial_{x_k}+\beta_k \partial_{y_k}\right)=
\sum_{k=1}^n(\alpha_k+i\beta_k)\partial_{z_k}
+
\sum_{k=1}^n(\alpha_k-i\beta_k)\partial_{\overline{z}_k}
$$
with
$$
\sum_{k=1}^n(\alpha_k+i\beta_k)\partial_{z_k}
\in T^{1,0}\mathbb{C}^n,\quad
\sum_{k=1}^n(\alpha_k-i\beta_k)\partial_{\overline{z}_k}
\in T^{0,1}\mathbb{C}^n,
$$
where $\alpha_k$, $\beta_k\in \mathbb{C}$, and
$\partial_{z_k}=\dfrac{1}{2}\left(\partial_{x_k}-i\partial_{y_k}\right)$ and 
$\partial_{\overline{z}_k}=\dfrac{1}{2}\left(\partial_{x_k}+i\partial_{y_k}\right)$.
Then,
$$
\partial_{z_1},\cdots, \partial_{z_n}
$$
consists of the (standard) basis of on the holomorphic tangent space around $p\in M$, and
\begin{equation}
\label{eq:tangent_space_basis}
\begin{CD}
z(U)\times \mathbb{C}^n@>>> (\Pi_M)^{-1}(U)\subset T^{1,0}M \\
(z,\eta) @>>> {\displaystyle \eta=\sum_{j=1}^n \eta_j\partial_{z_j}}
\end{CD}
\end{equation}
is a local chart of $T^{1,0}M$.

\begin{convention}
We adopt the following notation:
\begin{itemize}
\item
We denote by $TM$ and $T_pM$ the holomorphic tangent bundle of $M$, and the holomorphic tangent space at $p\in M$, for the simplicity.
\item As we mentioned above, for the basis of the tangent space on $TM$ in the coordinates, we abbreviate $\partial_{z_i}$,  $\partial_{\eta_i}$ to $\partial_{z_i}$, $\partial_{\eta_i}$, etc.
\end{itemize}
\end{convention}

We can easily check that for any $p\in M$ and $v\in T_pM=T^{1,0}_pM$, there is a holomorphic map $f\colon \{t\in \mathbb{C}\mid |t|<\epsilon\}\to M$ (for sufficiently small $\epsilon$) such that $f(0)=p$ and $Df|_0[(\partial/\partial t)|_{t=0}]=v$, and vice versa.

%

The (holomorphic) \emph{cotangent bundle} $T^*\!M=(T^*)^{1,0}M$ is the dual bundle of the (holomorphic) tangent bundle. As in the above case, it is formally defined as the $(1,0)$-part of the complexification $\tilde{T}^*M\otimes \mathbb{C}$ of the (real) cotangent bundle $\tilde{T}^*M$ of $M$. Let $\Pi^\dagger_M\colon T^*\!M\to M$ be the projection. Let $(U,z=(z_1,\cdots,z_n))$ be a local chart of $M$ at $p\in U\subset M$. Then,
$$
dz_1=dx_1+idy_1,\cdots, dz_n=dx_n+idy_n
$$
consists of the (standard) basis of on the cotangent space around $p\in M$, where $z_j=x_j+iy_j$($j=1,\cdots,n$), and
\begin{equation}
\label{eq:cotangent_space_basis}
\begin{CD}
z(U)\times \mathbb{C}^n @>>> (\Pi^\dagger_M)^{-1}(U)\subset T^*\!M
\\
 (z,\omega) @>>> {\displaystyle\omega= \sum_{j=1}^n \omega_jdz_j}
\end{CD}
\end{equation}
is a local chart of $T^*\!M$. The duality between $TM$ and $T^*\!M$ is given by the \emph{pairing}
$$
\mathcal{P}_M\colon TM\oplus T^*\!M\to \mathbb{C} 
$$
defined by
\begin{equation}
\label{eq:pairing_on_manifold}
\mathcal{P}_M\left(\sum_{j=1}^n\eta_j\partial_{z_j},\sum_{j=1}^n\omega_jdz_j\right)
=\sum_{j=1}^n \eta_j\omega_j
\end{equation}
for $p\in M$.
The pairing is a holomorphic function on the Whitney sum $TM\oplus T^*\!M$.

\section{Double tangent space}
\label{sec:double_tangent_bundle}
%
%

\subsection{Double tangent space}
\label{subsec:double_tangent_space}
The \emph{(holomorphic) double tangent bundle} $\pi_{TM}\colon TTM\to TM$ is the (holomorphic) tangent bundle $TTM=T(TM)$ of $TM$.
The double tangent space $TTM$ is a complex manifold of dimension $4n$.
With the local chart \eqref{eq:tangent_space_basis},
the double tangent space is trivialized as
\begin{equation}
\label{eq:double_tangent_vector}
\begin{CD}
\Pi_{TM}^{-1}(\Pi_M^{-1}(U))
@>>> (z(U)\times \mathbb{C}^n)\times (\mathbb{C}^n\times \mathbb{C}^n)
\\
{\displaystyle
\sum_{j=1}^n\xi_j\left(\partial_{z_i}\right)_\eta
+
\sum_{j=1}^n\zeta_j\left(\partial_{\eta_i}\right)_\eta
}
@>>>
(z,\eta,\xi,\zeta).
\end{CD}
\end{equation}
Under the local coordinates, 
the projection $\Pi_{TM}$ is described as
$$
\begin{CD}
TTM @>{\Pi_{TM}}>> TM \\
(z,\eta,\xi,\zeta) @>>> (z,\eta).
\end{CD}
$$

\subsection{Transition functions}
Let $(w_1,\cdots, w_n,H_1,\cdots,H_n)$ be another local chart around $\eta\in TM$. When
\begin{align*}
\sum_{j=1}^n\eta_j\left(\partial_{z_j}\right)_p
&=\sum_{j=1}^nH_j\left(\partial_{w_j}\right)_p
\\
\sum_{j=1}^n\xi_j\left(\partial_{z_j}\right)_\eta
+
\sum_{j=1}^n\zeta_j\left(\partial_{\eta_j}\right)_\eta
&=
\sum_{j=1}^n\Xi_j\left(\partial_{w_j}\right)_\eta
+
\sum_{j=1}^nZ_j\left(\partial_{H_j}\right)_\eta
\end{align*}
in $TTM$, the relation of the coefficients is given by
\begin{equation}
\label{eq:relation_local_charts}
\begin{bmatrix}
H_1 \\ \vdots \\ H_n \\
\Xi_1 \\ \vdots \\ \Xi_n \\ Z_1 \\ \vdots \\ Z_n
\end{bmatrix}
=
\begin{bmatrix}
(w_1)_{z_1} & \cdots & (w_1)_{z_n} & 0 & \cdots & 0 & 0 & \cdots & 0 \\
\vdots & & \vdots & \vdots & & \vdots & \vdots & & \vdots \\
(w_n)_{z_1} & \cdots & (w_n)_{z_n} & 0 & \cdots & 0 & 0 & \cdots & 0\\
0 & \cdots & 0 & (w_1)_{z_1} & \cdots & (w_1)_{z_n} & 0 & \cdots & 0 \\
\vdots & & \vdots & \vdots & & \vdots & \vdots & & \vdots \\
0 & \cdots & 0 & (w_n)_{z_1} & \cdots & (w_n)_{z_n} & 0 & \cdots & 0 \\
\vdots & & \vdots & \Gamma_{11} & \cdots &\Gamma_{1n} & (w_1)_{z_1} & \cdots & (w_1)_{z_n} \\
\vdots & & \vdots & \vdots & & \vdots \\
0 & \cdots & 0 & \Gamma_{n1} & \cdots & \Gamma_{nn} & (w_n)_{z_1} & \cdots & (w_n)_{z_n} 
\end{bmatrix}
\begin{bmatrix}
\eta_1 \\ \vdots \\ \eta_n \\ \xi_1 \\ \vdots \\ \xi_n \\ \zeta_1 \\ \vdots \\ \zeta_n
\end{bmatrix},
\end{equation}
where
$(w_i)_{z_j}=\dfrac{\partial w_i}{\partial z_j}$ and 
$$
\Gamma_{ij}(\eta)=\sum_{k=1}^n\eta_k\left(\dfrac{\partial w_i}{\partial z_j}\right)_{z_k}
=\sum_{k=1}^n\eta_k\dfrac{\partial^2 w_i}{\partial z_jz_k}.
$$

\begin{remark}[Convention]
\label{remark:cooridnate_change}
Henceforth, we abbreviate \eqref{eq:relation_local_charts} to
$$
\begin{bmatrix}
H \\
\Xi \\
Z
\end{bmatrix}
=
\begin{bmatrix}
J & 0 & 0 \\
0 & J & 0 \\
0 & \Gamma(\eta) & J
\end{bmatrix}
\begin{bmatrix}
\eta \\
\xi \\
\zeta
\end{bmatrix}
$$
for short, where $J=\begin{bmatrix}(w_i)_{z_j}\end{bmatrix}$ and $\Gamma(\eta)=\begin{bmatrix} \Gamma_{ij}(\eta)\end{bmatrix}$ (this means that the $(i,j)$-component of $J$ and $\Gamma(\eta)$ are $(w_i)_{z_j}$ and $\Gamma_{ij}(\eta)$, respectively).
In the following argument, we will discuss the change of the coefficients of the vectors (covectors) with respect to the change of local charts, and use this notation in the discussion.
\end{remark}

\subsection{Geometry of second order tangent vectors}
\label{subsec:geometry_double_tangent_vectors}
The second derivative of holomorphic maps, which is often called a \emph{$2$-jet} from a small disk to $M$, can be described in the (holomorphic) double tangent space.
Let $f\colon \{|t|<\epsilon\}\to M$ be a holomorphic map with $f(0)=p$ and $Df(\partial/\partial t|_{t=0})=u\in T_pM$. Set $F(t)=(F_1(t),\cdots,F_n(z))=z\circ f(t)$. Then,
$$
V_f=\sum_{j=1}^n(F_j)'(0)\left(\partial_{z_i}\right)_u
+
\sum_{j=1}^n(F_j)''(0)\left(\partial_{\eta_i}\right)_u
=u+\sum_{j=1}^n(F_j)''(0)\left(\partial_{\eta_i}\right)_u
$$
is a well-defined tangent vector in $T_uTM$. 
Since $u=D\Pi_M(V_f)=\Pi_{TM}(V_f)$, in general, some vector in $TTM$ cannot be represented as the second derivative of any holomorphic map from the unit disk. In fact, one can easily see that for any $u\in T_pM$ and $V\in T_uTM$, there is a holomorphic map $f\colon \{|t|<\epsilon\}\times \{|s|<\epsilon\}\to M$ such that
\begin{align}
u&=Df_{(0,0)}\left(\left.\partial_{ s}\right|_{s=0}\right)
=\sum_{i=1}^n\dfrac{\partial f_j}{\partial s}(0,0)\left(\partial_{z_i}\right)_p
\nonumber
\\
V&=\sum_{i=1}^n\dfrac{\partial f_j}{\partial t}(0,0)\left(\partial_{z_i}\right)_u
+
\sum_{i=1}^n\dfrac{\partial^2 f_j}{\partial t\partial s}(0,0)\left(\partial_{\eta_i}\right)_u
\label{eq:vectors_DTS2}
\end{align}
in the coordinates \eqref{eq:double_tangent_vector} around $u=\Pi_{TM}(V)\in TM$, where $z\circ f(t,s)=(f_1(t,s)$, $\cdots$, $f_n(t,s))$ (see Figure \ref{fig:DTS_image}. See also \cite[\S2]{MR1816050}). Thus, the $2$-jet space (i.e. the space of germs of holomorphic maps from a small disk to $M$) is a subspace of $TTM$ described by
$$
\{V\in TTM\mid \Pi_{TM}(V)=D\Pi_M(V)\}.
$$
In the case of real differentiable manifolds, Yano and Ishihara call this subbundle the \emph{tangent bundle of order $2$} in \cite[Chapter X]{MR0350650}.

\begin{figure}
\includegraphics[height = 5cm, bb = 6 0 553 341]{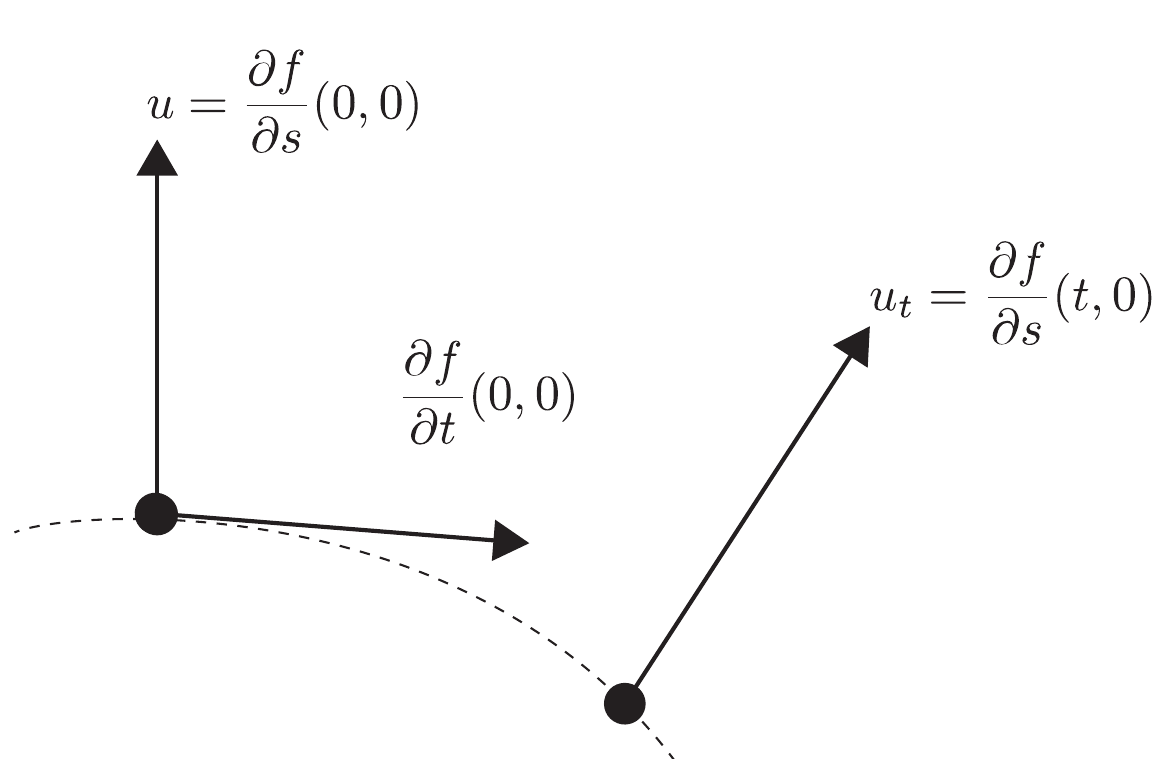}
\label{fig:DTS_image}
\caption{Vectors in the (holomorphic) double tangent space are presented by the second derivatives of holomorphic maps from a $2$-dimensional polydisk. To be more precise, assume $M=\mathbb{C}^n$. Let $u\in TM$ and $\{|t|,|s|<\epsilon\}\to M$ be a holomorphic map with $u=\dfrac{\partial f}{\partial s}(0,0)$. 
Then, $\{u_t\}_{|t|<\epsilon}=\left\{\dfrac{\partial f}{\partial s}(t,0)\right\}_{|t|<\epsilon}$ is a path through $u_0=u$ in $TM$. 
Vectors \eqref{eq:vectors_DTS2} in $T_uTM$ consists of two components.
One which measures the difference between $u=u_0$ and $u_t$ appears in the second term of \eqref{eq:vectors_DTS2}. 
The other term $\dfrac{\partial f}{\partial t}(0,0)$ records how the base points of the vectors $u_t$ approach to that of $u$ at $t=0$.
}
\end{figure}

\subsection{Horizontal projection and Vertical space}
\label{subsec:vertical_space_manifold}
Let $u\in TM$ and $p=\Pi_M(\eta)$. 
The \emph{horizontal projection} is the differential $(D\Pi_M)_{\eta}$ of the projection
$\Pi_M\colon TM\to M$ at $\eta$ which is a surjective linear map from $T_\eta TM$ to $T_pM$:
$$
\begin{CD}
T_\eta
TM\ni {\displaystyle \sum_{j=1}^n\xi_j\left(\partial_{z_i}\right)_\eta
+
\sum_{j=1}^n\zeta_j\left(\partial_{\eta_i}\right)_\eta}
@>{(D\Pi_M)_{\eta}}>>
{\displaystyle\sum_{j=1}^n\xi_j\left(\partial_{z_i}\right)_p}\in T_pM.
\end{CD}
$$
%
The \emph{vertical space}
$T_\eta^VTM=T_\eta T_pM$ 
is the kernel $\ker(D\Pi_M)|_{\eta}$ of the differential $(D\Pi_M)|_\eta\colon T_\eta TM\to T_pM$.
The vertical space is presented as
$$
T_\eta^VTM=T_\eta T_pM=
\left\{\sum_{j=1}^n\zeta_j\left(\partial_{\eta_i}\right)_\eta\mid \zeta\in \mathbb{C}^n\right\}.
$$
Since $T_pM$ is a vector space, $T_\eta T_pM$ is naturally identified with $T_pM$ by the \emph{vertical inclusion}:
\begin{equation}
\label{eq:vertical_projection}
\begin{CD}
T_pM
\ni
{\displaystyle\sum_{j=1}^n\zeta_j\left(\partial_{z_i}\right)_p}
@>{\verticalinc{M}{\eta}}>>
{\displaystyle\sum_{j=1}^n\zeta_j\left(\partial_{\eta_i}\right)_\eta}
\in \ker(D\pi_M)|_{\eta}\subset T_\eta TM.
\end{CD}
\end{equation}
Under the local coordinates \eqref{eq:double_tangent_vector},
the vertical inclusion and the horizontal projection are presented by 
$$
\begin{CD}
T_pM @>{\verticalinc{M}{\eta}}>> T_{\eta}TM @>{D\Pi_M}>> T_pM\\
(z,\zeta) @>>> (z,\eta,0,\zeta) \\
@. (z,\eta,\xi,\zeta) @>>> (z,\xi).
\end{CD}
$$

\section{Cotangent spaces to the tangent bundle}
\label{sec:cotangent_to_tangent_bundle}
The (holomorphic) cotangent space is the dual bundle to the tangent bundle.
With the local chart \eqref{eq:tangent_space_basis},
the cotangent space is trivialized as
\begin{equation}
\label{eq:double_cotangent_vector}
\begin{CD}
(\pi^\dagger_{TM})^{-1}(\pi_M^{-1}(U))
@>>> (z(U)\times \mathbb{C}^n)\times (\mathbb{C}^n\times \mathbb{C}^n)
\\
{\displaystyle
\sum_{j=1}^n\lambda_j\left(dz_i\right)_\eta
+
\sum_{j=1}^n\mu_j\left(d\eta_i\right)_\eta
}
@>>>
(z,\eta,\lambda,\mu).
\end{CD}
\end{equation}
A cotangent vector in $T^*_\eta TM$ is presented as
$$
\sum_{j=1}^n\lambda_jdz_j+\sum_{j=1}^n\mu_jd\eta_j
=
\sum_{j=1}^n\Lambda_jdw_j+\sum_{j=1}^nM_jdH_j
$$
in two different charts
if and only if
$$
\begin{bmatrix}
H \\
\Lambda \\
M
\end{bmatrix}
=
\begin{bmatrix}
J & 0 & 0 \\
0 & J^\dagger & \Gamma'(H) \\
0 & 0 & J^\dagger
\end{bmatrix}
\begin{bmatrix}
\eta \\
\lambda \\
\mu
\end{bmatrix},
$$
where
$J^\dagger =\begin{bmatrix} (z_j)_{w_i}\end{bmatrix}$
and
$\Gamma'(H)=\begin{bmatrix}\Gamma'_{ij}(H)\end{bmatrix}$ with
$$
\displaystyle
\Gamma'_{ij}(H)=\sum_{l=1}^nH_l\left(\dfrac{\partial z_j}{\partial w_i}\right)_{w_l}
=\sum_{l=1}^nH_l\dfrac{\partial^2 z_j}{\partial w_i\partial w_l}.
$$
The pull-back map $(\Pi_{M})^*\colon T^*_{p}M\to T_\eta^*TM$ defined from the projection $\Pi_M\colon TM\to M$ is the inclusion
$$
\begin{CD}
T^*_{p}M
\ni 
{\displaystyle \sum_{j=1}^n\lambda_j(dz_j)_p}
@>{(\Pi^\dagger_{M})^*}>>
{\displaystyle \sum_{j=1}^n\lambda_j(dz_j)_\eta}
\in T^*_\eta TM.
\end{CD}
$$
The projection
$$
\begin{CD}
T^*_\eta TM\ni 
{\displaystyle \sum_{j=1}^n\lambda_j(dz_j)_\eta+\sum_{j=1}^n\mu_j(d\eta_j)_\eta}
@>{\verticalproj{M}{\eta}}>>
{\displaystyle \sum_{j=1}^n\mu_j(dz_j)_p}
\in
T^*_pM
\end{CD}
$$
 is canonically defined. We call the projection the \emph{vertical projection}.
Under the chart defined above,
the projection for the vertical space is presented as
$$
\begin{CD}
T_p^*M @>{(\Pi^\dagger_{M})^*}>> T_\eta^*TM @>{\verticalproj{M}{\eta}}>> T_p^*M \\
(z,\lambda) @>>> (z,\eta, \lambda, 0) \\
@.(z,\eta,\lambda,\mu) @>>> (z,\mu).
\end{CD}
$$

\section{Tangent spaces and cotangent spaces to the cotangent bundle}
\label{sec:tangent_on_cotangent_bundles}

\subsection{Tangent space}
With the local chart \eqref{eq:cotangent_space_basis},
the tangent space $TT^*\!M$ is trivialized as
\begin{equation}
\label{eq:co_double_tangent_vector}
\begin{CD}
\Pi_{T^*\!M}^{-1}((\Pi^\dagger_M)^{-1}(U))
@>>> (z(U)\times \mathbb{C}^n)\times (\mathbb{C}^n\times \mathbb{C}^n)
\\
{\displaystyle
\sum_{j=1}^n\alpha_j\left(\partial_{z_i}\right)_\omega
+
\sum_{j=1}^n\beta_j\left(\partial_{\omega_i}\right)_\omega
}
@>>>
(z,\omega,\alpha,\beta).
\end{CD}
\end{equation}
For two presentations of tangent vectors to the cotangent bundle under the above two local charts,
$$
\sum_{i=1}^n\alpha_i\partial_{z_i}+\sum_{i=1}^n\beta_i\partial_{\omega_i}
=\sum_{i=1}^nA_i\partial_{w_i}+\sum_{i=1}^nB_i\partial_{\Omega_i}
$$
if and only if
\begin{equation}
\label{eq:transition_tan_cotan}
\begin{bmatrix}
\Omega \\
A \\
B 
\end{bmatrix}
=
\begin{bmatrix}
J^\dagger & 0 & 0 \\
0 &  J & 0 \\
0 & \Gamma''(\omega) & J^\dagger
\end{bmatrix}
\begin{bmatrix}
\omega \\
\alpha \\
\beta
\end{bmatrix}
\end{equation}
where
$\Gamma''(\omega)=\begin{bmatrix} \Gamma''_{ij}(\omega)\end{bmatrix}$ and
\begin{equation}
\label{eq:gamma_pra_pra}
\Gamma''_{ij}(\omega)
=\sum_{k=1}^n\omega_k\left(\dfrac{\partial z_k}{\partial w_i}\right)_{z_j}
.
\end{equation}
Then, the differential $D\Pi_M^\dagger$ of the projection $\Pi^\dagger_M\colon T^*\!M\to M$ is the \emph{horizontal projection}
$$
D\Pi_M^\dagger\colon
TT^*\!M\ni \sum_{i=1}^n\alpha_i\partial_{z_i}+\sum_{i=1}^n\beta_i\partial_{\omega_i}
\to\sum_{i=1}^n\alpha_i\partial_{z_i}\in TM.
$$
The kernel of the differential of the projection $T^*\!M\to M$, called the \emph{vertical space}, is naturally 
identified with the cotangent bundle via the inclusion
$$
\begin{CD}
T_p^*M\ni 
{\displaystyle \sum_{i=1}^n\beta_idz_i}
@>{\verticalincdag{M}{u}}>>
{\displaystyle
\sum_{i=1}^n\beta_i\partial_{\omega_i}\in T_\omega T^*\!M}.
\end{CD}
$$
Under the coordinate defined above,
the vertical inclusion is presented by
$$
\begin{CD}
T_p^*M @>{\verticalincdag{M}{\omega}}>> T_\omega T^*\!M @>{D\Pi^\dagger_M}>> T_pM \\
(z,\beta) @>>> (z,\omega,0,\beta) \\
@. (z,\omega,\alpha,\beta) @>>> (z,\alpha).
\end{CD}
$$
The image of $\verticalincdag{M}{\omega}$ is called the \emph{vertical space} in $T_\omega T^*\!M$.

\subsection{Cotangent space}
With the local chart \eqref{eq:cotangent_space_basis},
the tangent space $TT^*\!M$ is trivialized as
\begin{equation}
\label{eq:co_double_cotangent_vector}
\begin{CD}
(\Pi^\dagger_{T^*\!M})^{-1}((\Pi^\dagger_M)^{-1}(U))
@>>> (z(U)\times \mathbb{C}^n)\times (\mathbb{C}^n\times \mathbb{C}^n)
\\
{\displaystyle
\sum_{j=1}^n\nu_j\left(dz_i\right)_\omega
+
\sum_{j=1}^n\tau_j\left(d\omega_i\right)_\omega
}
@>>>
(z,\omega,\nu,\tau).
\end{CD}
\end{equation}
For two presentations of cotangent vectors to the cotangent bundle under the above two local charts, we can see that
$$
\sum_{i=1}^n\nu_i dz_i+\sum_{i=1}^n\tau_id\omega_i
=\sum_{i=1}^nN_idw_i+\sum_{i=1}^nT_id\Omega_i
$$
if and only if
\begin{equation}
\label{eq:transition_cotan_cotan}
\begin{bmatrix}
\Omega \\
N \\
T 
\end{bmatrix}
=
\begin{bmatrix}
J^\dagger & 0 & 0 \\
0 &  J^\dagger & -\Gamma''(\omega) \\
0 & 0 & J
\end{bmatrix}
\begin{bmatrix}
\omega \\
\nu \\
\tau
\end{bmatrix}
\end{equation}
where $\Gamma''(\omega)$ is defined in \eqref{eq:gamma_pra_pra}. 
The pull-back map $(\Pi^\dagger_{M})^*\colon T^*_{p}M\to T_\omega^*T^*\!M$ defined from the projection $\Pi^\dagger_M\colon T^*\!M\to M$ is the inclusion
$$
\begin{CD}
T^*_{p}M
\ni 
{\displaystyle \sum_{j=1}^n\nu_j(dz_j)_p}
@>{(\Pi^\dagger_{M})^*}>>
{\displaystyle \sum_{j=1}^n\nu_j(dz_j)_\omega}
\in T^*_\omega T^*\!M.
\end{CD}
$$
The projection
$$
\begin{CD}
T^*_\eta T^*\!M\ni 
{\displaystyle \sum_{j=1}^n\nu_j(dz_j)_\omega+\sum_{j=1}^n\tau_j(d\omega_j)_\omega}
@>{\verticalprojdag{M}{\omega}}>>
{\displaystyle \sum_{j=1}^n\tau_j(\partial_{z_j})_p}
\in
T_pM
\end{CD}
$$
 is canonically defined. We call the projection the \emph{vertical projection}.
Under the coordinate defined above,
the vertical projection is presented by
$$
\begin{CD}
T_p^*M @>{(\Pi^\dagger_M)^*}>> T^*_\omega T^*\!M @>{\verticalprojdag{M}{\omega}}>> T_pM \\
(z,\nu) @>>> (z,\omega,\nu,0) \\
@. (z,\omega,\lambda,\tau) @>>> (z,\tau).
\end{CD}
$$

\section{Flip, Switch, and Dualization}

\subsection{Flip}
\label{subsubsec:canonical_flip}
On the double tangent space $TTM$, there is a natural holomorphic involution
$$
\canoflip_M\colon TTM\to TTM
$$
called the \emph{flip} defined by
$$
\canoflip_M\colon (TT)_UM\ni (z,\eta,\xi,\zeta)\mapsto (z,\xi,\eta,\zeta)\in (TT)_UM
$$
under the local coordinates defined in \eqref{eq:double_tangent_vector}.
When a second order tangent vector is prescribed by a holomorphic map from a $2$-dimensional polydisk as \S\ref{subsec:geometry_double_tangent_vectors}, it is merely a changing the orders of variables. 
The flip is not a fiber-bundle isomorphism over $TM$ but
a bundle-isomorphism of the fibration $\Pi_M\circ \Pi_{TM}\colon TTM\to M$ via the double projection.

By definition,
\begin{align}
\Pi_{TM}\circ \canoflip_M &=D\Pi_{M}
\label{eq:filp_1}
\\
\canoflip_M \circ \canoflip_M &=id_{TTM}
\label{eq:filp_2}
\\
\canoflip_M(V+\verticalinc{M}{\eta}(\beta))&=\canoflip_M(V)+\verticalinc{M}{D\Pi_{M}(V)}(\beta)
\label{eq:filp_3}
\end{align}
for $\eta$, $\beta\in T_pM$ and $V\in T_\eta TM$.
See \cite[\S4 of Chapter I]{MR96276}. 

For readers, we check the well-definedness. Let $V=(z,\eta,\xi,\zeta)\in (TT)_UM$. Let $(w,H,\Xi,Z)$ be anothor presentation of $V$ with the local chart $(w,H)$.
Let $J=((w_i)_{z_j})$ and $\Gamma(\eta)=(\Gamma_{ij}(\eta))$ as \Cref{remark:cooridnate_change}.
From the above discussion, $H=J\eta$, $\Xi=J\xi$ and $Z=\Gamma(\eta) \xi+J\zeta$. The $i$-th coordinate of $\Gamma(\eta) \xi$ satisfies
\begin{align*}
\sum_{l=1}^n\Gamma_{il}(\eta)\xi_l&=
\sum_{l=1}^n\xi_l\left(\sum_{k=1}^n\eta_k\left(\dfrac{\partial w_i}{\partial z_l}\right)_{z_k}\right)
=\sum_{k=1}^n\eta_k\left(\sum_{l=1}^n\xi_l\left(\dfrac{\partial w_i}{\partial z_k}\right)_{z_l}\right)
=\sum_{k=1}^n\Gamma_{ik}(\xi)\eta_k
\end{align*}
which coincides with the $i$-th coordinate of $\Gamma(\xi) \eta$. Therefore,
the change of the coordinates presentations of tangent vectors is obtained as
$$
\begin{bmatrix}
\Xi \\
H \\
Z
\end{bmatrix}
=
\begin{bmatrix}
J\xi\\
J\eta \\
\Gamma(\xi) \eta+J\zeta
\end{bmatrix}
=
\begin{bmatrix}
J & 0 & 0 \\
0 & J & 0 \\
0 &\Gamma(\xi) & J
\end{bmatrix}
\begin{bmatrix}
\xi \\
\eta \\
\zeta
\end{bmatrix},
$$
which means $(z,\xi,\eta,\zeta)$ is thought of as an element of $TTM$.

\begin{remark}
In \cite{MR96276}, Kobayashi calls $\canoflip_M$ the \emph{involutive automorphism}. Here, it is named according to Sakai \cite{MR1390760}.
\end{remark}

\subsubsection*{{\bf Lie bracket}}
The \emph{Lie bracket} is an anti-symmetric product on the Lie algebra of the space of differentiable sections of the complexified tangent bundle $\tilde{T}M\otimes \mathbb{C}$. The subbundles $T^{1,0}M$ and $T^{0,1}M$ are closed under the operations of the Lie bracket (cf. \cite[Theorem 2.8 in Chapter IX]{MR1393941}). We denote by $[X,Y]$ the Lie bracket between sections $X$ and $Y$ of $\tilde{T}M\otimes \mathbb{C}$.

The following formula follows from the straight-forward calculation with the notation \eqref{eq:double_tangent_vector}. We give a proof for completeness (cf. (4.3) in \cite[II \S4]{MR1390760}).

\begin{proposition}[Lie bracket]
\label{prop:Lie_bracket}
Let $p\in M$. For $C^1$-vector fields $X$ and $Y$ of type $(1,0)$ around $p$, 
we have
$$
(\verticalinc{M}{X_p})^{-1}\Big(
\left(
\canoflip_M
\left(
\proje^{10}
\left(DY(X)_{Y_p}
\right)
\right)
\right)
-\proje^{10}\left(DX(Y)_{X_p}\right)
\Big)=[X,Y]_p.
$$
\end{proposition}

\begin{proof}
We may assume that $M$ is a neighborhood $U$ of $p\in \mathbb{C}^n$. Let $\displaystyle X=\sum_{i=1}^nX_i(z)\partial_{z_i}$ and $\displaystyle Y=\sum_{i=1}^nY_i(z)\partial_{z_i}$.
Then,
for $z\in U$,
\begin{align*}
\canoflip_M\left(
(DY(X)_{Y_{z}})^{10}
\right)
&=Y_{z}+
\sum_{j=1}^n
\left(\sum_{i=1}^nX_i(z)\dfrac{\partial Y_j}{\partial z_i}(z)\right)
\left.\partial_{\eta_j}\right|_{X_z}
\\
(DX(Y)_{X_{z}})^{10}
&=Y_z+
\sum_{j=1}^n
\left(\sum_{i=1}^nY_i(z)\dfrac{\partial X_j}{\partial z_i}(z)\right)
\left.\partial_{\eta_j}\right|_{X_z}.
\end{align*}
Therefore,
the difference
$$
\canoflip_M\left(
(DY(X)_{Y_{z}})^{10}
\right)
-(DX(Y)_{X_{z}})^{10}
$$
is in the vertical space
and
\begin{align*}
&(\verticalinc{M}{X_{z}})^{-1}\left(
\canoflip_M\left(
(DY(X)_{Y_{z}})^{10}
\right)-
(DX(Y)_{X_{z}})^{10}
\right) \\
&=
(\verticalinc{M}{X_{z}})^{-1}\left(
\sum_{j=1}^n
\left(\sum_{i=1}^nX_i(z)\dfrac{\partial Y_j}{\partial z_i}(z)
-
\sum_{i=1}^nY_i(z)\dfrac{\partial X_j}{\partial z_i}(z)\right)
\left.\partial_{\eta_j}\right|_{X_z}
\right)
\\
&=[X,Y]_z,
\end{align*}
and we have done.
\end{proof}

\subsection{Switch from $T^*\!TM$ to $TT^*\!M$}
\label{sec:flip_TstarTMand TTstarM}
We claim

\begin{proposition}[Switch]
\label{prop:switch}
With the local charts  \eqref{eq:double_cotangent_vector} and \eqref{eq:co_double_tangent_vector},
the correspondence
$$
\switch_M\colon TT^*\!M\to T^*\!TM
$$
defined by
$$
\begin{CD}
(TT^*\!)_UM @>{\switch_M}>> (T^*\!T)_UM \\
(z,\omega,\alpha,\beta) @>>> (z,\alpha,\beta,\omega)
\end{CD}
$$
is a well-defined biholomorphism, which satisfies
$$
\Pi^\dagger_{TM}\circ \switch_M=D\Pi^\dagger_M.
$$
\end{proposition}

We call the biholomorphism $\switch_M$ the \emph{switch} from $T^*\!TM$ to $TT^*\!M$.


\begin{proof}
We only check the well-definedness. We use the notation in \S\S\ref{sec:cotangent_to_tangent_bundle},\ref{sec:tangent_on_cotangent_bundles} frequently.
Let $V=(z,\omega,\alpha,\beta)\in TT^*\!M$ and $(w,\Omega,A,B)$ the presentation of $V$ in the coordinates $(w,b)$. Then, $b=J^\dagger a$, $A=J\alpha$ and $B=\Gamma''(\omega)\alpha+J^\dagger\beta$.
Then, the $i$-th coordinate of $\Gamma''(\omega)\alpha$ is
\begin{align*}
\sum_{l=1}^n\Gamma''_{il}(\omega)\alpha_l
&=\sum_{l=1}^n\sum_{k=1}^n\omega_k\alpha_l\left(\dfrac{\partial z_k}{\partial w_i}\right)_{z_l}
=\sum_{l=1}^n\sum_{k=1}^n\omega_k
\left(\sum_{s=1}^nA_s\dfrac{\partial z_l}{\partial w_s}\right)
\left(\dfrac{\partial z_k}{\partial w_i}\right)_{z_l}
\\
&=
\sum_{k=1}^n\sum_{s=1}^na_kA_s\sum_{l=1}^n\left(\dfrac{\partial z_k}{\partial w_i}\right)_{z_l}\dfrac{\partial z_l}{\partial w_s}
=\sum_{k=1}^n\omega_k\sum_{s=1}^nA_s\left(\dfrac{\partial z_k}{\partial w_i}\right)_{w_s} \\
&=\sum_{l=1}^n\Gamma'_{il}(A)\omega_l,
\end{align*}
and it is nothing but the $i$-th coordinate of $\Gamma'(A)a$.
Therefore, we obtain
$$ 
\begin{bmatrix}
A \\
B \\
\Omega
\end{bmatrix}
=
\begin{bmatrix}
J \alpha\\
J^\dagger\beta+\Gamma'(A)\omega \\
J^\dagger \omega
\end{bmatrix}
=
\begin{bmatrix}
J & 0 & 0 \\
0 & J^\dagger & \Gamma'(A) \\
0 & 0 & J^\dagger
\end{bmatrix}
\begin{bmatrix}
\alpha \\
\beta \\
\omega
\end{bmatrix},
$$
which implies that $(z,\alpha,\beta,\omega)$ is thought of an element in $(T^*\!T)_UM$.
\end{proof}

\subsection{Dualization}
From the transitions \eqref{eq:transition_tan_cotan} and \eqref{eq:transition_cotan_cotan},
under the local charts \eqref{eq:co_double_tangent_vector} and \eqref{eq:co_double_cotangent_vector}, we have a bundle isomorphism
$$
\canoflipdagger_M\colon TT^*\!M\to T^*\!T^*\!M
$$
defined by
$$
\begin{CD}
(TT^*\!)_UM @>{\canoflipdagger_M}>> (T^*\!T^*\!)_UM \\
(z,\omega,\alpha,\beta) @>>> (z,\omega,\beta,-\alpha),
\end{CD}
$$
which satisfies
$$
\verticalprojdag{M}{\omega}\circ \canoflipdagger_M=-D\Pi^\dagger_M.
$$
We call the isomorphism $\canoflipdagger_M$ the \emph{dualization}.
The isomorphism $\canoflipdagger_M$ is derived from the canonical holomorphic symplectic structure on $T^*\!M$ (e.g. \cite[Chapter 2]{MR1853077}):
Indeed, the cotangent bundle $T^*\!M$ has a natural holomorphic symplectic form defined by
$$
\omega_M=\sum_{j=1}^nd\omega_j\wedge dz_j=d\left(\sum_{j=1}^n\omega_jdz_j\right)
$$
under the coordinates discussed above. The holomorphic symplectic structure leads a canonical isomorphism
\begin{equation}
\label{eq:dualization_symplectic}
\begin{CD}
TT^*\!M @>>> T^*\!T^*\!M \\
X @>>> [Y\mapsto 2\omega_M(X,Y)]
\end{CD}
\end{equation}
which is nothing but the dualization discussed above, where ``$2$" at the head of $\omega_M$ in the linear map \eqref{eq:dualization_symplectic} comes from the convention of the exterior derivative:
$$
d\omega_j\wedge dz_j(X,Y)=\dfrac{1}{2}
(d\omega_j(X)dz_j(Y)- d\omega_j(Y)dz_j(X))
$$
(cf. \cite[p.35]{MR1393940}).
For the adjustment to our description of the holomorphic symplectic form on $\mathcal{Q}_g$ given in \Cref{prop:holomorphic_symplectic_form} later, we notice that this convention is also used in Kawai's calculation (cf. \cite[(3.1)]{MR1386110}. See also \S\ref{subsec:Kawai}).
\eqref{eq:dualization_symplectic} follows from the following proposition.

\begin{proposition}
\label{prop:dualization_symplectic_form_M}
The dualization $\canoflipdagger_M$ satisfies the following condition:
$$
\mathcal{P}_{T^*\!M}(V,\canoflipdagger_M(W))=2\omega_M(W,V)
$$
for all $V$, $W\in T_\omega T^*\!M$ and $\omega\in T^*\!M$.
\end{proposition}

Indeed, let $V=\sum_{i=1}^n(a_i\partial_{z_i}+b_i\partial_{\omega_i})$ and $W=\sum_{i=1}^n(\alpha_i\partial_{z_i}+\beta_i\partial_{\omega_i})$.
Then
$$
\canoflipdagger_M(W)=\sum_{i=1}^n(\beta_idz_i-\alpha_id\omega_i)
$$
and
\begin{align*}
\mathcal{P}_{T^*\!M}(V,\canoflipdagger_M(W))
&=
\sum_{i=1}^n(a_i\beta_i-b_i\alpha_i)=\sum_{i=1}^n(dz_i(V)d\omega_i(W)-dz_i(W)d\omega_i(V)) \\
&=2\omega_M(W,V).
\end{align*}
\subsection{Hamiltonian vector field}
\label{subsec:Hamiltonian_vector_field}
By mimicking the real case (cf. \cite[\S18]{MR1853077}), we define the \emph{Hamiltonian vector field} $\mathscr{X}_H$ for a $C^1$-function $H$ on $T^*\!M$ to be a $(1,0)$-vector-field $\mathscr{X}_H$ on $T^*\!M$ by
$$
\iota_{\mathscr{X}_H}\omega_M=\partial H.
$$
Namely, $\mathscr{X}_H$ is defined by the following equation:
$$
\omega_M((\mathscr{X}_H)_\omega,V)=\partial H|_\omega(V)
=\mathcal{P}_{T^*\!M}(V,\partial H|_v)
$$
for all $V\in T^{1,0}_\omega T^*\!M$ and $\omega\in T^*\!M$. If $W\in T_\omega T^*\!M$ satisfies $\canoflipdagger_M(W)=\partial H|_\omega$, 
\begin{align*}
\omega_M((\mathscr{X}_H)_\omega,V)
&=\mathcal{P}_{T^*\!M}(V,\partial H|_\omega)=\mathcal{P}_{T^*\!M}(V,\canoflipdagger_M(W))\\
&=2\omega_M(W,V).
\end{align*}
For the record, we summarize as follows.
\begin{proposition}
\label{prop:H-v}
Under the above notation, $\mathscr{X}_H=2(\canoflipdagger_M)^{-1}(\partial H)$.
\end{proposition}

\section{Derivative of the Pairing function}

%

As noticed in \S\ref{sec:tangent_cotangent_space}, the pairing between $TM$ and $T^*\!M$ is a holomorphic function on the Whitney sum. By definition,
$$
TM\oplus T^*M=\{(v,\omega)\in TM\times T^*\!M\mid \Pi_M(v)=\Pi^\dagger_M(\omega)\}.
$$
Hence,
$$
T_{(v,\omega)}(TM\oplus T^*\!M)=\{(V_1,V_2)\in T_vTM\times T_\omega T^*\!M\mid
D\Pi_M|_v[V_1]=D\Pi^\dagger_M|_\omega[V_2]\}.
$$
We claim:

\begin{proposition}[Derivative of the pairing, Flip and Switch]
\label{prop:Derivative_of_pairing_and_flip}
Let $v\in TM$ and $\omega\in T^*\!M$.
For $V_1\in T_vTM$ and $V_2\in T_\omega T^*\!M$ with $D\Pi_{M}|_v[V_1]=D\Pi_{M}^\dagger|_{\omega}[V_2]$, 
$$
D\mathcal{P}_M|_{(v,\omega)}[V_1,V_2]=\mathcal{P}_{TM}|_{D\Pi_M|_v[V_1]}(\canoflip_M(V_1),\switch_M(V_2)).
$$
\end{proposition}

\begin{proof}
We may assume that $M$ is a domain in $\mathbb{C}^n$, $TM=\{(z,\eta)\in U\times \mathbb{C}^n\}$ and $T^*\!M=\{(z,a)\in U\times \mathbb{C}^n\}$.
Suppose $\displaystyle v=((z_i),(\eta_i))$, $\omega=((z_i),(a_i))$, $\displaystyle V_1=((z_i),(\eta_i),(\alpha_i),(\beta_i))$
and $\displaystyle V_1=((z_i),(a_i),(A_i),(B_i))$ for $1\le i\le n$.
Notice from the assumption that
\begin{align*}
(z,\alpha)&=\Pi_{TM}(z,\alpha,\eta,\beta)) \\
&=\Pi_{TM}(\canoflip_M(V_1))=D\Pi_{TM}|_v[V_1]\\
&=D\Pi^\dagger|_{\omega}[V_2]=\Pi^\dagger_{TM}(\switch_M(V_2))\\
&=D\Pi^\dagger|_{\omega}(z,A,B,a) =(z,A)
\end{align*}
and $\alpha=A$.
%
Since the pairing is presented as
$\mathcal{P}_M((z,\eta),(z,a))=\sum_{j=1}^n\eta_ja_j$,
we obtain
$$
D\mathcal{P}_M=\sum_{j=1}^n a_jd\eta_j+\sum_{j=1}^n\eta_jda_j
$$
and
\begin{align*}
D\mathcal{P}_M|_{(v,\omega)}[V_1,V_2]=\sum_{j=1}^n\eta_jB_j+\sum_{j=1}^n \beta_ja_j
=\mathcal{P}_{TM}|_{D\Pi_M|_v[V_1]}(\canoflip_M(V_1), \switch_M(V_2)),
\end{align*}
which implies what we wanted.
\end{proof}

\section{Intrinsic characterization of Pairing on $TM$}
Let $p\in M$.
In the following statement, we identify $T^*\!_pM$ as a subspace in $T^*\!_uTM\subset (T^*\!T)_pM$ via the horizontal inclusions and $T_pM$ as the vertical space in $T_uTM\subset (TT)_pM$ for each $u\in T_pM$.
We claim
\begin{proposition}[Intrinsic characterization of the pairing $\mathcal{P}_{TM}$]
\label{prop:characterization_pairing}
For any $u\in T_pM$ and $\zeta\in T^*\!_pM$,
the pairing $\mathcal{P}_{TM}|_u$ on $T_uTM\oplus T_u^*\!TM$ satisfies
\begin{align}
\mathcal{P}_{TM}|_u(V,\omega)&=\mathcal{P}_M|_p(D\Pi_M(V),\omega) 
\quad(V\in T_uTM, \ \omega\in T^*\!_pM)
\label{eq:prop_characterization_pairing1}
\\
\mathcal{P}_{TM}|_u(v,\Omega)&=\mathcal{P}_M|_p(v, \Omega')
\quad(v\in T_pM, \ \Omega\in T^*\!_uTM)
\label{eq:prop_characterization_pairing2}
\\
D\mathcal{P}_M|_{(u,\zeta)}[V,W]
&=\mathcal{P}_{TM}|_{D\Pi_M(V)}(\canoflip_M(V),\switch_M(W))
\label{eq:prop_characterization_pairing3} \\
&\qquad\quad(V\in T_uTM, \ W\in T_\zeta T^*\!M)
\nonumber
\end{align}
with $D\Pi^\dagger_{TM}[W]=D\Pi_{TM}[V]\in T_pM$, where
$\Omega'\in T^*\!_pM$ is the image of vertical projection of $\Omega$.

Conversely, if a function $\mathcal{P}'$ on $(TT)_pM\oplus (T^*\!T)_pM$ which defines a $\mathbb{C}$-bilinear map $\mathcal{P}'\colon T_uTM\oplus V\in T^*\!_uTM\to \mathbb{C}$ on each $u\in T_pM$
satisfies the above three equations for all $u\in T_pM$ and $\zeta\in T^*\!_pM$, then $\mathcal{P}'=\mathcal{P}_{TM}$ on $\mathcal{P}'$ on $(TT)_pM\oplus (T^*\!T)_pM$.
\end{proposition}

\begin{proof}
Since the statement is a local property,
we may assume that $M$ is a domain in $\mathbb{C}^n$, $TM=M\times \mathbb{C}^n$ and $T^*\!M=M\times \mathbb{C}^n$ are domains in $\mathbb{C}^{2n}$ with the coordinate functions $((z_i),(\eta_i))\in  TM$ and $((z_i),(a_i))\in  T^*\!M$.
Let
\begin{align*}
V&=\sum_{j=1}^n\alpha_j
\left(\partial_{z_j}\right)_{((z_i),(\eta_i))}+
\sum_{j=1}^n\beta_j
\left(\partial_{\eta_j}\right)_{((z_i),(\eta_i))}\\
\Omega
&=\sum_{j=1}^nA_j(dz_j)_{((z_i),(\eta_i))}+
\sum_{j=1}^nB_j(d\eta_j)_{((z_i),(\eta_i))}
\\
W&=\sum_{j=1}^nX_j
\left(\partial_{z_j}\right)_{((z_i),(a_i))}+
\sum_{j=1}^nY_j
\left(\partial_{a_j}\right)_{((z_i),(a_i))}.
\end{align*}
Since $D\Pi^\dagger_{TM}[W]=D\Pi_{TM}[V]$,
$$
((z_i),(X_i))=D\Pi^\dagger_{TM}[W]=D\Pi_{TM}[V]=((z_i),(\alpha_i)).
$$
Under the notation, the pairing $\mathcal{P}_{TM}$ is described as 
$$
\mathcal{P}_{TM}\left(
V
,
\Omega
\right)
=
\sum_{j=1}^n(\alpha_jA_j+\beta_jB_j).
$$
From the assumption, $D\Pi_M(V)\in T_pM$, $v\in T_pM\subset T_uTM$, $\omega\in T^*\!_pM\subset T^*\!_uTM$ and $\Omega'\in T^*\!_pM$ are presented as
\begin{align*}
&D\Pi_M(V)=\sum_{j=1}^n\alpha_j
\left(\partial_{z_j}\right)_{(z_i)},\quad
v=\sum_{j=1}^n\beta_j\partial_{\eta_j},\\
&\omega=\sum_{j=1}^nA_j(dz_j)_{(z_i)},
\quad \Omega'=\sum_{j=1}^nB_j(dz_j)_{(z_i)}.\\
&\canoflip_M(V)=\sum_{j=1}^n\eta_j
\left(\partial_{z_j}\right)_{((z_i),(\alpha_i))}+
\sum_{j=1}^n\beta_j
\left(\partial_{\eta_j}\right)_{((z_i),(\alpha_i))}
\\
&\switch_M(W)=
\sum_{j=1}^nY_j
\left(dz_j\right)_{((z_i),(\alpha_i))}+
\sum_{j=1}^na_j
\left(d \eta_j\right)_{((z_i),(\alpha_i))}.
\end{align*}
Hence, we can easily see that $\mathcal{P}_{TM}$ satisfies \eqref{eq:prop_characterization_pairing1} and \eqref{eq:prop_characterization_pairing2}.
From \Cref{prop:Derivative_of_pairing_and_flip}, $\mathcal{P}_{TM}$ satisfies \eqref{eq:prop_characterization_pairing3}.

Suppose that for any $u\in T_pM$, a $\mathbb{C}$-bilinear map $\mathcal{P}'|_u\colon T_uTM\oplus V\in T^*\!_uTM\to \mathbb{C}$ satisfies \eqref{eq:prop_characterization_pairing1}, \eqref{eq:prop_characterization_pairing2} and \eqref{eq:prop_characterization_pairing3}.
Define $K_{ij}^k$ ($1\le i,j\le n$, $1\le k\le 4$) by
\begin{align*}
\mathcal{P}'|_u\left(
V
,
\Omega
\right)
&=
\sum_{i,j=1}^n(K_{ij}^1(u)\alpha_iA_j+K_{ij}^2(u)\alpha_iB_j+K_{ij}^3(u)\beta_iA_j+K_{ij}^4(u)\beta_iB_j).
\end{align*}
Notice for the above expression that each $K^\cdot_{ij}$ may depend on $u\in T_pM$ in general (but not from the calculation below by the three equations in the end).

From \eqref{eq:prop_characterization_pairing1}, 
$$
\sum_{i,j=1}^n(K_{ij}^1(u)\alpha_iA_j+K_{ij}^3(u)\beta_iA_j)
=\sum_{j=1}^n\alpha_jA_j.
$$
Since the right-hand side is independent of the choice of $u$, $\alpha$, $\beta$ and $A$, we have $K_{ij}^1(u)=\delta_{ij}$ (Kroneker's delta), and $K_{ij}^3(u)=0$ for $1\le i,j\le n$ and $u\in T_pM$.
From \eqref{eq:prop_characterization_pairing2}
$$
\sum_{i,j=1}^n(K_{ij}^3(u)\beta_iA_j+K_{ij}^4(u)\beta_iB_j)
=\sum_{i,j=1}^nK_{ij}^4(u)\beta_iB_j
=\sum_{j=1}^n\beta_jB_j.
$$
which means $K_{ij}^4(u)=\delta_{ij}$ for $1\le i,j\le n$ and $u\in T_pM$.  From \eqref{eq:prop_characterization_pairing3}, 
and the proof of \Cref{prop:Derivative_of_pairing_and_flip},
\begin{align*}
D\mathcal{P}_M\mid_{(u,\zeta)}[V,W]
&=\sum_{j=1}^n(a_j\beta_j+\eta_jY_{j}) \\
\mathcal{P}'|_{D\Pi_M(V)}(\canoflip_M(V),\switch_M(W))
&=\sum_{j=1}^n(a_j\beta_j+\eta_jY_j)+\sum_{i,j=1}^nK_{ij}^2(D\Pi_M(V))\eta_ia_j.
\end{align*}
From the assumption, $\displaystyle \Pi_{TM}(V)=\Pi^\dagger_{TM}(\Omega)=u=((z_i),(\eta_i))$ and $\Pi_{T^*\!M}(W)=\zeta=((z_i),(a_i))$
can be taken arbitrary in $T_pM$ and $T^*\!_pM$ respectively with fixing $((z_i),(\alpha_i))=D\Pi_M[V]\in T_pM$,
we have $K_{ij}^2=0$  on $T_p(M)$ for $1\le i,j\le n$.
Therefore, we conclude that $\mathcal{P}'=\mathcal{P}_{TM}$.
\end{proof}

\chapter{Teichm\"uller theory}
\label{chap:Teichmuller_theory}

\section{Teichm\"uller space and Bers embedding}
\label{sec:Teichmuller_space}
For the contents of this section, readers can refer to
the books \cite{MR903027} ,\cite{MR1215481}, \cite{MR815922}, and \cite{MR927291}
for instance.

\subsection{Teichm\"uller space}
Let $\Sigma_g$ be an orientable closed surface of genus $g \ge 2$.
A \emph{marked Riemann surface $(X,f)$ of genus $g$} is a pair of a closed Riemann surface $X$ of genus $g$ and an orientation preserving homeomorphism $f\colon \Sigma_g\to X$. Two marked Riemann surfaces $(X_1,f_1)$ and $(X_2,f_2)$ are said to be \emph{Teichm\"uller equivalent} if there is a biholomorphism $h\colon X_1\to X_2$ such that $h\circ f_1$ is homotopic to $f_2$. The set $\teich_g$ of equivalence classes is called the \emph{Teichm\"uller space of closed Riemann surfaces of genus $g$}.

\subsection{Teichm\"uller space as a bounded domain}

\subsubsection{Automorphic forms on Riemann surfaces}
Let $\{(U_i,z_i)\}_{i\in I}$ be an analytic coordinate chart of a Riemann surface $M_0$.
For $p,q\in \mathbb{Z}$,
a \emph{$(p,q)$-form} $\varphi=\varphi(z)dz^pd\overline{z}^q$ is an assignment of a function $\varphi_i$ on $z_i(U_i)$ for each $i\in I$ with
$$
\varphi_j(z_{ij}(z))\left(\dfrac{dz_{ij}}{dz}(z)\right)^p\left(\overline{\dfrac{dz_{ij}}{dz}(z)}\right)^p=\varphi_i(z)
$$
for $z\in z_i(U_i\cap U_j)$,
where $z_{ij}=z_j\circ z_i^{-1}\colon z_i(U_i\cap U_j)\to z_j(U_i\cap U_j)$.
Notice that for a $(p,q)$-form $\varphi_1$ and an $(r,s)$-form $\varphi_2$, the product
$$
\varphi_1\varphi_2=\varphi_1(z)\varphi_2(z)dz^{p+r}d\overline{z}^{q+s}
$$
is a $(p+r,q+s)$-form.

We call a $(p,q)$-form $\varphi$ to be \emph{measurable}, \emph{bounded (measurable)}, \emph{smooth}, \emph{holomorphic}$, $\emph{meromorphic} if so is each $\varphi_i$ on $z_i(U_i)$ for $i\in I$. For simplicity, $(p,0)$-forms is called \emph{$p$-forms}. From historical reasons, $2$-forms, $(0,1)$-forms, and $(1,1)$-forms are called \emph{quadratic differentials}, \emph{vector fields}, and \emph{area forms}, respectively. For instance for a quadratic differential $Q$ and a vector field $\xi$, 
the product $\xi Q=Q\xi$ is a $1$-form.

\subsubsection{Automorphic forms on domains}
Let $D$ be a domain in $\hat{\mathbb{C}}$ such that $\hat{\mathbb{C}}\setminus D$ contains at least three points, and let $\Gamma$ be a subgroup of the group of biholomorphisms on $D$. Let $L^\infty(D,\Gamma)$ is the set of bounded measurable functions $\mu$ on $D$ satisfying $\mu(\gamma(z))(\overline{\gamma'(z)}/\gamma'(z))=\mu(z)$ for $z\in D$ and $\gamma\in \Gamma$. $L^\infty(D,\Gamma)$ is a complex Banach space with the essential supremum norm $\|\cdot\|_\infty$. Let $B_2(D,\Gamma)$ be the set of holomorphic functions $\varphi$ on $D$ with $\varphi(\gamma(z))\gamma'(z)^2=\varphi(z)$ for $z\in D$ and $\gamma\in \Gamma$. $B_2(D,\Gamma)$ is also a complex Banach space with the sup norm
$$
\|\varphi\|_\infty=\sup_{z\in D}\lambda_D(z)^{-2}|\varphi(z)|,
$$
where $\lambda_D=\lambda_D(z)|dz|$ is the Poincare metric on $D$ of curvature $-4$.

Suppose that $M=D/\Gamma$ is a closed surface of genus $g$. Then, $L^\infty(D,\Gamma)$ and $B_2(D,\Gamma)$ are canonically identified with the complex Banach spaces $L^\infty_{(-1,1)}(M)$ of bounded measurable $(-1,1)$-forms and $\mathcal{Q}_M$ of holomorphic quadratic differentials on $M$, respectively. There is a canonical pairing
\begin{align}
\label{eq:pairing}
L^\infty_{(-1,1)}(M)\times \mathcal{Q}_M\ni (\mu,q)\mapsto 
\llangle\mu,q\rrangle
&
:=\iint_{M}\mu(z)q(z)dxdy \\
&=\dfrac{1}{2i}
\iint_M\mu(z)q(z)d\overline{z}\wedge dz
\nonumber
\end{align}

\subsection{Bers projection and Bers embedding}
\label{subsec:Bers_embedding}
We fix $x_0=(M_0,f_0)\in \teich_g$ and the Fuchsian group $\Gamma_0$ acting on the unit disk $\mathbb{D}=\{|z|<1\}$ with $\mathbb{D}/\Gamma_0=M_0$.

Let $\mathbb{D}^*=\hat{\mathbb{C}}\setminus \overline{\mathbb{D}}$. 
For simplicity, we write $L^\infty=L^\infty(\mathbb{D},\Gamma_0)$ and $B_2=B_2(\mathbb{D}^*,\Gamma_0)$. Let $(L^\infty)_1$ be the unit ball of $L^\infty$ with center at the origin. Elements in $(L^\infty)_1$ are called \emph{Beltrami differentials}. For any $\mu\in (L^\infty)_1$, we consider a unique quasiconformal mapping $W_\mu$ on $\hat{\mathbb{C}}$ satisfying
$$
(W_\mu)_{\overline{z}}=\begin{cases}
\mu(W_\mu)_z & (z\in \mathbb{D}) \\
0 & (z\in \hat{\mathbb{C}}\setminus \mathbb{D}),
\end{cases}
$$
and $W_\mu(z)=z+o(1)$ as $z\to \infty$. Then $W_\mu$ descends to a quasiconformal mapping $w_\mu$ from $M_0$ to $M_\mu:=W_\mu(\mathbb{D})/W_\mu \Gamma (W_\mu)^{-1}$ and the map
\begin{equation}
\label{eq:bers_projection}
\Bersproj_{x_0}\colon (L^\infty)_1\ni \mu\mapsto (M_\mu,w_\mu)\in \teich_g
\end{equation}
is well-defined and surjective. We call the map $\Bersproj_{x_0}$ the \emph{Bers projection}. 

We also define a holomorphic map
$$
(L^\infty)_1\ni \mu\mapsto \left.\left(\dfrac{W_\mu'''}{W_\mu'}-\dfrac{3}{2}\left(\dfrac{W_\mu''}{W_\mu'}\right)^2\right)\right|_{\mathbb{D}^*}\in B_2
$$
which descends to an injective map $\Bersemb_{x_0}\colon \teich_g\to B_2$ via the Bers projection. The image $\Bers{x_0}$ of $\Bersemb_{x_0}$ is a bounded domain in $B_2$ containing the origin, and the origin of $B_2$ corresponds to $x_0\in \teich_g$. We call $\Bersemb_{x_0}$ the \emph{Bers embedding} of $\teich_g$ with base point $x_0$. The Teichm\"uller space $\teich_g$ admits a (unique) complex structure which makes $\Bersproj_{x_0}$ and $\Bersemb_{x_0}$ holomorphic (cf. \Cref{fig_BBA}). 
\begin{figure}[t]
$$
\xymatrix@C=50pt@R=40pt{
(L^\infty)_1
\ar[d]_{\Bersproj_{x_0}}
\\
\teich_g
\ar[r]_{\Bersemb_{x_0}}
&
B_2
\ar@{.>}[ul]_{\ahlforsW{x_0}}
}
$$
\caption{Bers projection, Bers embedding, and the Ahlfors-Weill section. The dotted arrow means the Ahlfors-Weill section is defined locally.}
\label{fig_BBA}
\end{figure}
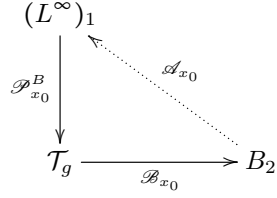

We define $\ahlforsW{x_0}\colon B_2\to L^\infty$ by
\begin{equation}
\label{eq:Ahlfors-Weill-map}
\ahlforsW{x_0}(\varphi)(z)=-\dfrac{1}{2}\lambda_{\mathbb{D}}(z)\varphi(1/\overline{z})(1/\overline{z})^4=
-\dfrac{1}{2\overline{z}^4}(|z|^2-1)^2\varphi(1/\overline{z})\quad
(z\in \mathbb{D}).
\end{equation}
The restriction of $\ahlforsW{x_0}$ to the open ball in $B_2$ of radius $2$ with center the origin is a holomorphic local right-inverse of $\Bersemb_{x_0}\circ \Bersproj_{x_0}$. Namely
\begin{equation}
\label{eq:right_inv}
\Bersemb_{x_0}\circ (\Bersproj_{x_0}\circ\ahlforsW{x_0})=id
\end{equation}
on the open ball. We call $\ahlforsW{x_0}$ the \emph{Ahlfors-Weill} section for $\Bersproj_{x_0}$.
In particular, the Bers projection is a holomorphic submersion.

\section{Infinitesimal theory}
\label{sec:infinitesimal_deformation_RS}

\subsection{Tangent spaces and Cotangent spaces}
\label{subsec:tangent_spaces_and_cotangent_spaces}
Let $x_0=(M_0,f_0)\in \teich_g$. As above-mentioned, via the universal covering $\mathbb{H}\to M_0$, there is a natural isometric isomorphism $L^\infty_{(-1,1)}(M_0)\to L^\infty(\mathbb{D},\Gamma_0)$.

Let $L^\infty_{(-1,1)}(M_0)_1$ be the open unit ball in $L^\infty_{(-1,1)}(M_0)$. For $\mu\in L^\infty_{(-1,1)}(M_0)_1$, there is a quasiconformal mapping $g^\mu\colon M_0\to M_\mu$ whose complex dilatation is equal to $\mu$. Then, the Bers projection \eqref{eq:bers_projection} is presented as
$$
L^\infty_{(-1,1)}(M_0)_1\ni \mu\to (M_\mu,g^\mu\circ f_0)\in \teich_g
$$
under the situation. Hence the (holomorphic) tangent space at $x_0$ is presented as
the quotient space of $L^\infty_{(-1,1)}(M_0)$ by the kernel of the differential of the Bers projection.
Indeed, the Teichm\"uller lemma asserts that the tangent space $T_{x_0}\teich_g$ is represented as the quotient space 
\begin{equation}
\label{eq:Teichmuller_lemma}
T_{x_0}\teich_g\cong L^\infty_{(-1,1)}(M_0)/\{\mu\in L^\infty_{(-1,1)}(M_0)\mid \llangle \mu,q\rrangle=0, q\in \mathcal{Q}_{M_0}\}.
\end{equation}
Hence, the pairing \eqref{eq:pairing} descends to the non-degenerate pairing
\begin{equation}
\label{eq:pairing_Teichmuller}
T_{x_0}\teich_g\times \mathcal{Q}_{M_0}\ni ([\mu],q)\mapsto 
\langle [\mu],q\rangle:=\llangle \mu,q\rrangle.
\end{equation}
From \eqref{eq:pairing_Teichmuller}, the space $\mathcal{Q}_{x_0}=\mathcal{Q}_{M_0}$ is canonically recognized as the (holomorphic) cotangent space $T^*_{x_0}\teich_g$ of $\teich_g$ at $x_0=(M_0,f_0)$.

\begin{remark}
After identifying $\mathcal{Q}_g=\cup_{x\in \teich_g}\mathcal{Q}_x$ with the cotangent bundle of $\teich_g$, when we use the notation defined at \eqref{eq:pairing_on_manifold}, 
\begin{equation}
\label{eq:pairing_Teichmuller_2}
\langle v,q\rangle=\mathcal{P}_{\teich_g}(v,q)
\end{equation}
holds for $v\in T\teich_g$ and $q\in \mathcal{Q}_g=T^*\teich_g$.
We will use the both notations appropriately, since the right-hand side of \eqref{eq:pairing_Teichmuller} seems not to be familiar with readers (or Teichm\"uller theorists).
\end{remark}

\subsection*{Strict convexity of $\mathcal{Q}_{M_0}$}
The $\mathbb{C}$-vector space $\mathcal{Q}_{M_0}$ is a complex Banach space with the $L^1$-norm
$$
\LOneNorm(q)=\|q\|=\iint_{M_0}|q(z)|dxdy=\dfrac{1}{2i}\iint_{M_0}|q(z)|d\overline{z}\wedge dz.
$$
The $L^1$-norm function $\LOneNorm$ on $\mathcal{Q}_g$ is of class $C^1$ except for the zero section (cf. \cite[Lemma 3]{MR0288254}). We will discuss the regurality of the $L^1$-norm function in \S\ref{sec:Royden_theorem_revisited}.

The following is well-known. However, we give a brief proof for the sake of readers.

\begin{proposition}[Strictly convexitiy]
\label{prop:strictly_convex-L1-norm}
The $L^1$-norm on $\mathcal{Q}_{M_0}$ is \emph{strictly convex} in the sense that for $q_1$, $q_2\in\mathcal{Q}_{M_0}$ with $q_1\ne q_2$ and $\|q_1\|=\|q_2\|=1$, $\|q_1+q_2\|<2$.
\end{proposition}

\begin{proof}
To this end, we first claim that for $\alpha$, $\beta\in \mathcal{Q}_{M_0}$ with $\alpha\ne 0$, $\|\beta\|=1$,
$$
{\rm Re}\iint_{M_0}\dfrac{\overline{\alpha}}{|\alpha|}\beta=1
$$
if and only if $\alpha=t\beta$ for $t\in \mathbb{R}$ with $t\ne 0$. We only check the ``only if"-part. Since ${\rm Re}(\overline{\alpha}\beta/|\alpha\beta|)\le 1$ on $M_0$ except for finite points (zeros of $\alpha$ and $\beta$), we have
$$
0\le \iint_{M_0}\left(1-{\rm Re}\dfrac{\overline{\alpha(z)}\beta(z)}{|\alpha(z)\beta(z)|}\right)|\beta(z)|dxdy
=1-{\rm Re}\iint_{M_0}\dfrac{\overline{\alpha}}{|\alpha|}\beta=0
$$
from the assumption. Since $|\beta|>0$ and $\overline{\alpha}\beta/|\alpha\beta|=|\alpha|\beta/|\beta|\alpha$ almost everywhere on $M_0$, the imaginary part of a meromorphic function $\beta/\alpha$ vanishes almost everywhere on $M_0$. From the identity theorem, we obtain $\alpha=t\beta$ for some $t\in \mathbb{R}$ with $t\ne 0$.

Let us return to the proof of the proposition.
Let $q_1$, $q_2\in \mathcal{Q}_{M_0}$ with $\|q_1\|=\|q_2\|=1$. 
When $\|q_1+q_2\|=2$, $q_1+q_2\ne 0$ and
\begin{align*}
2=\|q_1+q_2\|
&={\rm Re}\iint_{M_0}\dfrac{\overline{q_1+q_2}}{|q_1+q_2|}(q_1+q_2)
\\
&={\rm Re}\iint_{M_0}\dfrac{\overline{q_1+q_2}}{|q_1+q_2|}q_1
+{\rm Re}\iint_{M_0}\dfrac{\overline{q_1+q_2}}{|q_1+q_2|}q_2\le \|q_1\|+\|q_2\|=2.
\end{align*}
From the above argument, we deduce that $t_1q_1=q_1+q_2=t_2q_2$ for some $t_1$, $t_2\in \mathbb{R}$ with $t_1\ne 0$ and $t_2\ne 0$. Since $\|q_1\|=\|q_2\|=1$ and $q_1+q_2\ne 0$, we obtain $q_1=q_2$.
\end{proof}

\subsection{Trivializations of tangent bundle and cotangent bundle}
\label{subsec:trivialization_tangent_bundle_cotangent_bundle}
Since the Teichm\"uller space is biholomorphic to a bounded domain in $B_2=B_2(\mathbb{D}^*,\Gamma_0)$ (with the notation in \S\ref{subsec:Bers_embedding}), the (holomorphic) tangent bundle $T\teich_g$ is naturally identified with the trivial bundle $\teich_g\times B_2\to \teich_g$.

A basis $\varphi_1$, $\cdots$, $\varphi_{3g-3}\in B_2$ of $B_2$
gives a global (holomorphic) trivialization of $\mathcal{Q}_g=T^*\teich_g$ defined by
$$
\mathcal{Q}_g\ni q\mapsto (x,
(\langle D(\Bersemb_{x_0}^{-1})|_{x}[\varphi_1],q\rangle, \cdots,
\langle D(\Bersemb_{x_0}^{-1})|_{x}[\varphi_{3g-3}],q\rangle))\in \teich_g\times \mathbb{C}^{3g-3},
$$
where $x\in \teich_g$ with $q\in \mathcal{Q}_x$, and $D(\Bersemb_{x_0}^{-1})$ is the differential of the inverse of the Bers embedding. Since $\mathcal{Q}_{x_0}$ is the dual of $B_2$, for any $q\in \mathcal{Q}_g$, there is a unique $Q\in \mathcal{Q}_{x_0}$ such that
\begin{equation}
\label{eq:trivialization_pairing_Bers}
\langle (D(\Bersemb_{x_0}|_{x})^{-1}[\varphi_k],q\rangle=
\langle (D(\Bersemb_{x_0}|_{x_0})^{-1}[\varphi_k],Q\rangle=
\langle [\ahlforsW{x_0}(\varphi_k)],Q\rangle
\end{equation}
for all $k=1$, $\cdots$, $3g-3$. Threfore, we obtain a natural holomorphic identification
\begin{equation}
\label{eq:trivialization_cotangent_bundle}
T^*\teich_g\cong \mathcal{Q}_g\ni q \to (x,Q)\in \teich_g\times \mathcal{Q}_{x_0}.
\end{equation}
We can easily check that  the identification \eqref{eq:trivialization_cotangent_bundle} is independent of the choice of the frame $\{\varphi_1$, $\cdots$, $\varphi_{3g-3}\}$ of $B_2$.


\subsection{Teichm\"uller Beltrami differentials}
\label{subsubsec:TB_differenital}
The \emph{Teichm\"uller Beltrami differential} on $M_0$ is a $(-1,1)$-bounded measurable form of the form
$$
k\dfrac{\overline{q}}{|q|}=k\dfrac{\overline{q(z)}}{|q(z)|}\dfrac{d\overline{z}}{dz}.
$$
for $q\in \mathcal{Q}_{M_0}$ and $k>0$
(e.g. \cite[p.29]{MR590044}).
The Teichm\"uller Beltrami differential have an extremal property, called the infinitesimal unique extremal property. Namely, when $\mu\in L^\infty_{(-1,1)}(M_0)$ satisfies
$$
\llangle\mu,\varphi\rrangle=\llangle\overline{q}/|q|,\varphi \rrangle
$$
for all $\varphi\in \mathcal{Q}_{M_0}$, then $\|\mu\|_\infty\ge 1$, and the equality holds only if $\mu=\overline{q}/|q|$ almost everywhere on $M_0$  (cf. \cite[\S4.10, Corollary 2]{MR1730906} and the discussion below). Furthremore, each tangent vector $v\in T_{x_0}\teich_g$ admits a unique Teichm\"uller Beltrami differntial as a representative.

The quasiconformal mapping defined from the Teichm\"uller Beltrami differntial $k\overline{q}/|q|$ with $0<k<1$ is called the \emph{formal Teichm\"uller map} (cf. \cite[\S5.2.2]{MR1215481}).
A formal Teichm\"uller map defined from a pair $(k,q)$ with $k=\|q\|$ is called the \emph{Teichm\"uller map}.
Formal Teichm\"uller maps have an extremal property.
Namely, any formal Teichm\"uller map $f$ on $M_0$ is uniquely extremal in the sense that the maximal dilatation of $f$ attains minimal only at $f$ among all quasiconformal mappings on $M_0$ homotopic to $f$ (e.g. \cite[\S4.7, Corollary 1]{MR1730906} and \cite[Theorem 5.9]{MR1215481}).


\subsection{Teichm\"uller metric}
\label{subsec:Teichmuller_metric}
For $x=(M,f)\in \teich_g$, the \emph{Teichm\"uller metric} $\teichmullernorm$ is a Finsler metric on $\teich_g$ defined by
\begin{equation}
\label{eq:Teichmuller-norm}
\teichmullernorm(x,v)=\sup\left\{{\rm Re}\,\langle v,q\rangle\mid \|q\|=1, q\in \mathcal{Q}_{M}\right\}
\end{equation}
for $v\in T_x\teich_g$. By definition, the Teichm\"uller metric on $T\teich_g$ is thought of as the dual (Finsler) metric to the $L^1$-norm on $\mathcal{Q}_g\cong T^*\teich_g$ (cf. \cite[\S2]{MR0288254}). 

Royden shows that the Teichm\"uller metric coincides with the Kobayashi metric on $\teich_g$, and that the Teichm\"uller metric is of class $C^1$ on the tangent bundle except for the zero section (cf. \cite[Lemma 3, Theorem 3]{MR0288254}). We will deal with the regulality of the Teichm\"uller metric in \S\ref{sec:Royden_theorem_revisited}.

The length distance with respect to the Teichm\"uller metric on $\teich_g$ is called the \emph{Teichm\"uller distance} on $\teich_g$. The Teichm\"uller space with the Teichm\"uller distance is a unique geodesic and complete metric space which is deduced from the Teichm\"uller uniqueness theorem and the compactness of the family of quasiconformal mappings with uniformly bounded maximal dilatation.

\section{Kodaira-Spencer theory}
\label{subsubsec:1st_cohomology_tangent_space}
\subsection{Notation on the sheaf cohomology}
We first fix notations. For a sheaf $\mathscr{S}$ over a topological space $M$ and an open set $U\subset M$, we denote by $\Gamma(U,\mathscr{S})$ the set of sections of $\mathscr{S}$ over $U$. Let $\mathcal{U}=\{U_i\}_{i\in I}$ is a locally finite covering of $M$. For a $k$-chain $\sigma\in C^k(\mathcal{U},\mathscr{S})$, we denote by $\sigma_{i_0\cdots i_k}\in \Gamma(I_{i_0}\cap \cdots \cap U_{i_k},\mathscr{S})$ the $(i_0\cdots i_k)$-component of $\sigma$. 

Let $\delta\colon C^k(\mathcal{U},\mathscr{S})\to C^{k+1}(\mathcal{U},\mathscr{S})$ denote the coboundary operator. Let $Z^k(\mathcal{U},\mathscr{S})$ and $B^k(\mathcal{U},\mathscr{S})$ be the spaces of cocycles and coboundaries, and $H^k(\mathcal{U},\mathscr{S})$ the $k$-th cohomology group with coefficient $\mathscr{S}$. For instance, $H^0(M,\mathscr{S})=\Gamma(M,\mathscr{S})$. 
It is known that when a covering $\mathcal{U}=\{U_i\}_{i\in I}$ satisfies $H^1(U_i,\mathscr{S})=0$ for $i\in I$, the natural inclusion 
$$
H^1(\mathcal{U},\mathscr{S})\to H^1(M,\mathscr{S})
$$
defined from the direct limit is isomorphic and 
$$
H^2(\mathcal{U},\mathscr{S})\to H^2(M,\mathscr{S})
$$
is injective (cf. \cite[Theorems 3.4 and 3.5]{MR815922}).

Let $M$ be a Riemann surface.
Let $\Theta_M$ and $\Omega_M$ be the sheaves of germs of holomorphic vector fields and holomorphic $1$-forms on $M$, respectively. For $p,q,r\in \mathbb{Z}$ and $\mathscr{S}=\Theta_{M}$ or $\Omega_M$, we denote by $\mathcal{A}^{p,q}(\mathscr{S}^{\otimes r})$ the \emph{sheaf of germs of $C^\infty$-$(p,q)$-forms with coefficients in $\mathscr{S}^{\otimes r}$}. For instance, a section of $\mathcal{A}^{p,q}(\Omega_M^{\otimes r})$ (resp. $\mathcal{A}^{p,q}(\Theta_M^{\otimes r})$) on a domain $V\subset M$ is a smooth $(p+r,q)$-form on $V$ (resp. smooth $(p-r,q)$-forms), respectively. Therefore, there are natural identifications $\mathcal{A}^{p,q}(\Omega_M^{\otimes r})=\mathcal{A}^{0,q}(\Omega_M^{\otimes (p+r)})$ and $\mathcal{A}^{p,q}(\Theta_M^{\otimes r})=\mathcal{A}^{0,q}(\Omega_M^{\otimes (p-r)})$. For instance, $\mathcal{A}^{0,1}(\Omega_{M_0})$ is the sheaf of germs of smooth area forms on $M_0$.

Let $\mathcal{U}=\{(U_i,z_i)\}_{i\in I}$ be a locally finite analytic chart of $M$.
For $\mathscr{S}=\Theta_M$ or $\Omega_M$, the anti-holomorphic derivative $\overline{\partial}$ defines a sheaf homomorphism
$$
\begin{CD}
C^k(\mathcal{U},\mathcal{A}^{p,0}(\mathscr{S})) @>{\overline{\partial}}>> C^k(\mathcal{U},\mathcal{A}^{p,1}(\mathscr{S}))  \\
\sigma=\{\sigma_{i_0\cdots i_k}\}_{i_0,\cdots, i_k\in I} @>>>\overline{\partial}\sigma=\{(\sigma_{i_0\cdots i_k})_{\overline{z}}\}_{i_0,\cdots, i_k\in I}
\end{CD}
$$
%
%

\subsection{The first cohomology group as the tangent space}
\label{subsec:the_first_cohomology}
The following exact sequence
\begin{equation}
\label{eq:Dolbeault-sequence}
\begin{CD}
0 @>>> \Theta_{M_0} @>{inc}>> \mathcal{A}^{0,0}(\Theta_{M_0}) @>{-\overline{\partial}}>> \mathcal{A}^{0,1}(\Theta_{M_0}) @>>> 0,
\end{CD}
\end{equation}
where ``inc" means the inclusion,
leads the Dolbeault theorem
\begin{equation}
\label{eq:Dou}
H^1(M_0,\Theta_{M_0}) \cong \Gamma(M_0,\mathcal{A}^{0,1}(\Theta_{M_0}))/
\overline{\partial}\Gamma(M_0,\mathcal{A}^{0,0}(\Theta_{M_0}))
\end{equation}
(cf. \cite[Theorem 3.13]{MR815922}).
The canonical isomorphism $H^1(M_0,\Theta_{M_0})\cong T_{x_0}\teich_g$ follows from the coincidence of the right-hand sides of \eqref{eq:Teichmuller_lemma} and \eqref{eq:Dou}, which is verified from the Serre duality theorem (or the non-degeneracy of the Serre pairing defined by the residue map). See \cite[Theorem 7.9]{MR1215481} and \cite[\S17]{MR648106}. See also \S\ref{subsec:appendix_1} for the discussion on the sign of the $\overline{\partial}$-operator in \eqref{eq:Dolbeault-sequence}.

For reader's convenience, we shall give an explicit correspondence.
We fix a covering $\mathcal{U}=\{U_i\}_{i\in I}$ of $M_0$ with $H^1(\mathcal{U},\Theta_{M_0})\cong H^1(M_0,\Theta_{M_0})$.
For $X\in Z^1(\mathcal{U},\Theta_{M_0})$,
take $\xi\in C^0(\mathcal{U},\mathcal{A}^{0,0}(\Theta_0))$ such that $\delta\xi=X$. Then $\mu=-\overline{\partial }\xi_i$ on $U_i$ is naturally thought of as an element in $\Gamma(M_0,\mathcal{A}^{0,1}(\Theta_{M_0}))\subset L^\infty_{(-1,1)}(M_0)$, and
\begin{equation}
\label{eq:identify_H1_T_g}
\begin{CD}
H^1(M_0,\Theta_{M_0})\cong H^1(\mathcal{U},
\Theta_{M_0})\ni [X]@>{\mathscr{T}_{x_0}}>> [\mu]=[-\overline{\partial}\xi_i]\in T_{x_0}\teich_g
\end{CD}
\end{equation}
is a $\mathbb{C}$-linear isomorphism. See the discussion in \cite[Theorem 5.4]{MR815922}.

%
\subsection{Presentation of the pairing by the residue}
\label{subsec:residue}
The exact sequence of sheaves
\begin{equation}
\label{eq:exact_sequence_sheaves_differential_forms}
\begin{CD}
0@>>> \Omega_{M_0}@>>> \mathcal{A}^{0,0}(\Omega_{M_0})@>{d}>>  \mathcal{A}^{0,1}(\Omega_{M_0})@>>>0
\end{CD}
\end{equation}
leads the isomorphism
\begin{equation}
\label{eq:exact_sequence_sheaves_differential_forms_1}
H^1(M_0,\Omega_{M_0})\cong H^0(M_0, \mathcal{A}^{0,1}(\Omega_{M_0}))/dH^0(M_0,\mathcal{A}^{0,0}(\Omega_{M_0})).
\end{equation}
This yields the residue map
$$
{\rm Res}\colon H^1(M_0,\Omega_{M_0})\to \mathbb{C}
$$
defined by
\begin{equation}
\label{eq:definition_residue}
{\rm Res}([\omega])=\dfrac{1}{2\pi i}\iint_{M_0}\Omega(z)d\overline{z}\wedge dz
\end{equation}
for $[\omega]\in H^1(M_0,\Omega_{M_0})$,
where $\Omega\in H^0(M_0, \mathcal{A}^{0,1}(\Omega_{M_0}))$ is the corresponding area form via the isomorphism \eqref{eq:exact_sequence_sheaves_differential_forms_1}.
Namelly, there is $\Omega'\in C^0(M_0, \mathcal{A}^{0,0}(\Omega_{M_0}))$ such that $d\Omega'=\overline{\partial}\Omega'=\Omega$ and $\delta\Omega'=\omega$
(cf. \cite[\S17.1]{MR648106}).
\begin{convention}
\label{convention:1}
As \eqref{eq:definition_residue}, in this paper, we adopt $d\overline{z}\wedge dz$ as the basis of $2$-forms. 
The reason why we use $d\overline{z}\wedge dz$ as a basis of $2$-forms comes from the exact sequence \eqref{eq:exact_sequence_sheaves_differential_forms}. Namely, for $\omega=\omega(z)dz$, $d\omega=\omega_{\overline{z}}(z)d\overline{z}\wedge dz$. See \cite[(c) in \S15.9]{MR648106}.
\end{convention}

The pairing \eqref{eq:pairing_Teichmuller} is presented with the residue as
\begin{align}
\label{eq:pairing_2}
\langle v,q\rangle
&=\dfrac{1}{2i}\iint_{M_0}(-(\xi_i)_{\overline{z}})qd\overline{z}\wedge dz =-\dfrac{1}{2i}\iint_{M_0}d(\xi_i qdz) \\
&=-\pi \,{\rm Res}([Xq])
\nonumber
\end{align}
since 
$\delta (\{\xi_iq\}_{i\in I})=Xq=\{X_{ij}q\}_{i,j\in I}\in Z^1(\mathcal{U},\Omega_{M_0})$,
where $v=[X]=[\{X_{ij}\}_{i,j\in I}]\in H^1(\mathcal{U},\Theta_{M_0})\cong T_{x_0}\teich_g$, and $\xi=\{\xi_i\}_{i\in I}\in C^0(\mathcal{U},\Theta_{M_0})$ with $\delta\xi=X$. In particular,
when $H^1(U_i,\Theta_{M_0})=0$ for $i\in I$, the residue pairing
$$
\begin{CD}
H^1(\mathcal{U},\Theta_{M_0})\times \mathcal{Q}_{x_0}@>>> \mathbb{C} \\
([X],q) @>>> -\pi \,{\rm Res}([Xq])
\end{CD}
$$
is non-degenerate.

\section{Stratification of the space of quadratic differentials}
\label{sec:statification_space_QD}
Let $Z$ be a manifold. A \emph{stratification} of $Z$ is a locally finite collection of locally closed submanifolds $\{Z_j\}_{j\in J}$ of $Z$, the \emph{strata}, indexed by a set $J$ such that
\begin{enumerate}
\item
$Z=\cup_{j\in J}Z_j$; and
\item
$Z_j\cap \overline{Z_k}\ne \emptyset$ if and only if $Z_j\subset \overline{Z_k}$.
\end{enumerate}
(cf. \cite{MR3413977}). The conditions induce a partial order on $J$, where $j\le k$ if $Z_j\subset \overline{Z_k}$.

Each $q\in \mathcal{Q}_g$ determines a topological data $(k_1,\cdots,k_n;\epsilon)$ where
$k_1$, $\cdots$, $k_n$ are the orders of zeros: $\epsilon=+1$ if $q$ is the square of an abelian differentiall $\epsilon=-1$ otherwise. A stratum $\mathcal{Q}_g(k_1,\cdots,k_n;\epsilon)$ consists all quadratic differentials determining the data $(k_1,\cdots,k_n;\epsilon)$. The principal stratum $\mathcal{Q}_1$ corresponds to $(1,\cdots,1;-1)$. The principal stratum $\mathcal{Q}_1$ is open and generic, and each stratum is a complex submanifold of $\mathcal{Q}_g$ (cf. \cite{MR1094714}. See also \cite{MR3413977}). The index of the stratum satisfies $\sum_{i=1}^n k_i=4g-4$. In general, some data $(k_1,\cdots,k_n;\epsilon)$ cannot be realized by a holomorphic quadratic differential. See Masur-Smillie \cite{MR1214233}.

Any point in each stratum admits a local complex chart defined by the relative period of the abelian differentials defined from the square roots of the differential. Hence, the $L^1$-norm function is real analytic on each stratum because the $L^1$-norm is written as the quadratic form of the period.

For the record, we summery as follows.
\begin{proposition}
\label{prop:L1-norm-realanalytic-strata}
The $L^1$-norm function on $\mathcal{Q}_g$ is real-analytic on each stratum 
\end{proposition}

\chapter{Teichm\"uller space of tori}
\label{chap:Teichmuller_space_of_tori}

In this chapter, we devote to discuss the model case (the most simplest case) of our results.
We treat the Teichm\"uller space $\teich_1$ of tori.
The case of once punctured tori is treated in the same way as that in the case of flat tori. Indeed, the canonical completion at the puncture gives an isomorphism between the Teichm\"uller space of once punctured tori and that of flat tori. The case of four puncture sphere is treated in a similar way.

\section{Teichm\"uller space of tori}
The Teichm\"uller space $\teich_1$ of tori is the space of the Teichm\"uller equivalence classes of marked complex tori of dimension $1$, where a \emph{complex torus} of dimension $1$ is a Riemann surface homeomorphic to a torus. For $\tau\in \mathbb{H}=\{\tau\in \mathbb{C}\mid {\rm Im}(\tau)>0\}$, let $\Gamma_\tau=\langle 1,\tau\rangle$ be a $\mathbb{Z}$ lattice in $\mathbb{C}$ generated by $1$ and $\tau$, and set $M_\tau=\mathbb{C}/\Gamma_\tau$. Let $\Sigma_1=M_i$ ($i$ is the imaginary unit). Let $f_\tau\colon \Sigma_1\to M_\tau$ be a homeomorphism with $(f_\tau)_*(1)=1$ and $(f_\tau)_*(i)=\tau$. Then, $(M_\tau,f_\tau)$ is a marked complex torus, and 
$$
\mathbb{H}\ni \tau\mapsto (M_\tau,f_\tau)\in \teich_1
$$
is bijective. It is well-known that this identification become a biholomorphic under a canonical complex structure on $\teich_1$.

\section{Tangent and cotangent bundles}
For $(M_\tau,f_\tau)\in \teich_1$, the space $\mathcal{Q}_\tau$ of holomorphic quadratic differentials on $M_\tau$ is described by 
$$
\mathcal{Q}_\tau=\{q=q\,dz^2\mid q\in \mathbb{C}\}
$$
where the coordinate $z$ is induced from $\mathbb{C}$ via the projection $\mathbb{C}\to M_\tau=\mathbb{C}/\Gamma_\tau$ since any $\Gamma_\tau$-periodic holomorphic function on $\mathbb{C}$ is constant. We identify the space $\mathcal{Q}_1$ of holomorphic quadratic differentials with $\mathbb{H}\times \mathbb{C}$ by
\begin{equation}
\label{eq:Beltrami_torus1}
\mathcal{Q}_1\ni ((M_\tau,f_\tau),-2i\,q\,dz^2)\mapsto \left(\tau,q\right)\in \mathbb{H}\times \mathbb{C}.
\end{equation}
The constant ``$-2i$" in the trivialization \eqref{eq:Beltrami_torus1} is a regulation which makes the pairing between $T\teich_1$ and $\mathcal{Q}_1$ to be the ordinary complex Euclidean pairing as we see at
\eqref{eq:Beltrami_torus2}.
The $L^1$-norm function $\LOneNorm$ on $\mathcal{Q}_1$ is given by
$$
\LOneNorm(\tau,q)=
\dfrac{1}{2i}\iint_{M_\tau}|-2iq|d\overline{z}\wedge dz=2|q|{\rm Im}(\tau). 
$$
The $L^\infty_{(-1,1)}(M_\tau)$ consists of bounded measurable $(-1,1)$-form $\mu'=\mu'(z)d\overline{z}/dz$ such that $\mu'(z)$ is periodic in terms of the action of the lattice $\Gamma_\tau$. Then, $\mu'$ is infinitesimally equivalent to a unique constant $(-1,1)$-form $(\mu/{\rm Im}(\tau))d\overline{z}/dz$ with
$$
\displaystyle \mu=\dfrac{1}{2i}\iint_{M_\tau}\mu' (z)d\overline{z}\wedge dz.
$$
We notice that the Beltrami differential $\mu(t)$ of the quasiconformal (affine) deformation from $M_\tau$ to $M_{\tau+t \mu}$ ($\mu\in \mathbb{C}$) behaves
\begin{equation}
\label{eq:Beltrami_torus}
\mu(t)=t\dfrac{i\mu}{2{\rm Im}(\tau)}\dfrac{d\overline{z}}{dz}+o(t)
\end{equation}
as $t\to 0$. Hence, we define a trivialization of $T\teich_1$ by
\begin{equation}
\label{eq:trivialization_TT1}
T\teich_1 \ni \left((M_\tau,f_\tau),\left[\dfrac{i\mu}{2{\rm Im}(\tau)}\dfrac{d\overline{z}}{dz}\right]\right)
\mapsto \left(\tau,\mu\right)\in  \mathbb{H}\times \mathbb{C}
\end{equation}
Under the coordinates of $T\teich_g$ and $\mathcal{Q}_1$, the pairing function is
\begin{equation}
\label{eq:Beltrami_torus2}
\mathcal{P}_{\teich_1}((\tau,\mu),(\tau,q))=
\dfrac{1}{2i}\iint_{M_\tau}\dfrac{i\mu}{2{\rm Im}(\tau)}(-2iq)d\overline{z}\wedge dz
=\mu q.
\end{equation}
Therefore, the Teichm\"uller metric on $\teich_1$, which is the dual Finsler metric with respect to the $L^1$-norm function, is obtained by
$$
\teichmullernorm(\tau, \mu)=\sup
\left\{{\rm Re}\left(\mu q\right)\mid 
\LOneNorm(\tau,q)=1
\right\}
=\dfrac{|\mu|}{2{\rm Im}(\tau)},
$$
which implies that the Teichm\"uller metric on $\teich_1$ is nothing but the Poincar\'e metric on the upper-half plane $\mathbb{H}$ with curvature $-4$
(cf. \cite[Theorem 6.6.5]{MR2245223} and \cite[Chapter IV, Proposition 1.3]{MR2194466}).

The Teichm\"uller Beltrami map $\tb\colon \mathcal{Q}_1\to T\teich_1$ is presented by
$$
\tb(\tau, q)=\left((M_\tau,f_\tau),
\left[2i\,{\rm Im}(\tau)\overline{q}\dfrac{d\overline{z}}{dz}
\right]
\right)=(\tau,4\overline{q}{\rm Im}(\tau)^2).
$$
In this case, it is immediately checked that the Teichm\"uller Beltrami map is a real-analytic diffeomorphism from $\mathcal{Q}_1$ to $T\teich_1$.

\section{Second order infinitesimal spaces and Infinitesimal Duality}
All the second order infinitesimal spaces of $\teich_1$ are holomorphically trivial to $\mathbb{H}\times \mathbb{C}\times \mathbb{C}\times \mathbb{C}$. We identify
\begin{align*}
TT\teich_1\ni \xi\partial_\tau|_{(\tau,\mu)}+\zeta\partial_\mu|_{(\tau,\mu)}\mapsto (\tau,\mu,\xi,\zeta)\in \mathbb{H}\times \mathbb{C}\times \mathbb{C}\times \mathbb{C}\\
T^*\!T\teich_1\ni \kappa d\tau|_{(\tau,\mu)}+\lambda d\mu|_{(\tau,\mu)}\mapsto (\tau,\mu,\kappa, \lambda)\in \mathbb{H}\times \mathbb{C}\times \mathbb{C}\times \mathbb{C}\\
T\mathcal{Q}_1\ni \alpha\partial_\tau|_{(\tau,q)}+\beta\partial_q|_{(\tau,q)}\mapsto (\tau,q,\alpha,\beta)\in \mathbb{H}\times \mathbb{C}\times \mathbb{C}\times \mathbb{C}\\
T^*\!\mathcal{Q}_1\ni \gamma d\tau|_{(\tau,q)}+\eta dq|_{(\tau,q)}\mapsto (\tau,q,\gamma,\eta)\in \mathbb{H}\times \mathbb{C}\times \mathbb{C}\times \mathbb{C}.
\end{align*}
 
The flip, the switch, and the dualization are 
\begin{align*}
\canoflip_{\teich_1}(\tau,\mu,\xi,\zeta)&=(\tau,\xi,\mu,\zeta)
\\
 \switch_{\teich_1}(\tau,q,\alpha,\beta)&=(\tau,\alpha,\beta,q)
 \\
 \canoflipdagger_{\teich_1}
(\tau,q,\alpha,\beta)&=(\tau,q,\beta,-\alpha).
\end{align*}
From \eqref{eq:Beltrami_torus2}, \Cref{prop:Derivative_of_pairing_and_flip} and \Cref{prop:characterization_pairing}, the pairing functions and the holomorphic symplectic form are
\begin{align*}
\mathcal{P}_{T\teich_1}((\tau,\mu,\xi,\zeta),(\tau,\mu,\kappa,\lambda)
&=\xi\kappa+\zeta\lambda \\
\mathcal{P}_{\mathcal{Q}_1}((\tau,q,\alpha,\beta),(\tau,q,\gamma,\eta)
&=\alpha\gamma+\beta\eta \\
\omega_{\teich_1}((\tau,q,\alpha_1,\beta_1),(\tau,q,\alpha_2,\beta_2)
&=-\dfrac{1}{2}(\alpha_1\beta_2-\beta_1\alpha_2).
\end{align*}
Thus, the pairings are presented as those in the case of trivial bundles. This is caused from the regulation of the trivializations of $T\teich_1$ and $\mathcal{Q}_1$ so that the pairing between $T\teich_1$ and $\mathcal{Q}_1$ is presented as the ordinary Euclidean case as \eqref{eq:Beltrami_torus2}.

The differentials of the $L^1$-norm function $\LOneNorm$ on $\mathcal{Q}_1$ and the Teichm\"uller metric $\teichmullernorm$ on $T\teich_1$ are obtained by
\begin{align*}
\partial \LOneNorm|_{(\tau,q)}&=
-i|q|d\tau+{\rm Im}(\tau)\dfrac{\overline{q}}{|q|}dq
=
\left(\tau,q,
-|q|,{\rm Im}(\tau)\dfrac{\overline{q}}{|q|}
\right) \\
\partial \teichmullernorm|_{(\tau,\mu)}&=
-\dfrac{|\mu|}{4i{\rm Im}(\tau)^2}d\tau+
\dfrac{\overline{\mu}}{4|\mu|{\rm Im}(\tau)}d\mu
=
\left(\tau,\mu,
-\dfrac{|\mu|}{4i{\rm Im}(\tau)^2},
\dfrac{\overline{\mu}}{4|\mu|{\rm Im}(\tau)}
\right).
\end{align*}
Hence, we see that the $L^1$-norm function and the Teichm\"uller metric are of class $C^1$ except for the zero sections. In particular,
\begin{align*}
\partial \LOneNorm^2|_{(\tau,q)}&=
-4i|q|^2{\rm Im}(\tau)d\tau+4\overline{q}{\rm Im}(\tau)^2dq
=
\left(\tau,q,
-4i|q|^2{\rm Im}(\tau),
4\overline{q}{\rm Im}(\tau)^2
\right) \\
\partial \teichmullernorm^2|_{(\tau,\mu)}&=
-\dfrac{|\mu|^2}{4i{\rm Im}(\tau)^3}d\tau+
\dfrac{\overline{\mu}}{4{\rm Im}(\tau)^2}d\mu
=
\left(\tau,\mu,
-\dfrac{|\mu|^2}{4i{\rm Im}(\tau)^3},
\dfrac{\overline{\mu}}{4{\rm Im}(\tau)^2}
\right).
\end{align*}
Therefore, at $(\tau,q)=-2iq\,dz^2\in \mathcal{Q}_\tau$,
we obtain
\begin{align*}
\switch_{\teich_1}\circ (\canoflipdagger_{\teich_1})^{-1}(-\partial \LOneNorm^2|_{(\tau,q)})
&=
\switch_{\teich_1}\circ (\canoflipdagger_{\teich_1})^{-1}
\left(\tau,q,
4i|q|^2{\rm Im}(\tau),
-4\overline{q}{\rm Im}(\tau)^2
\right)
\\
&=
\switch_{\teich_1}
\left(\tau,q,
4\overline{q}{\rm Im}(\tau)^2,
4i|q|^2{\rm Im}(\tau)
\right)
\\
&=
\left(\tau,
4\overline{q}{\rm Im}(\tau)^2,
4i|q|^2{\rm Im}(\tau),q
\right)
\\
\partial \teichmullernorm^2|_{\tb(\tau,q)}
&=
\partial \teichmullernorm^2|_{(\tau,4\overline{q}{\rm Im}(\tau)^2)}
\\
&=
\left(
\tau,
4\overline{q}{\rm Im}(\tau)^2,
4i|q|^2{\rm Im}(\tau),
q
\right).
 \end{align*}
This coincidence implies that the infinitesimal duality holds in this model case. We will discuss in \S\ref{sec:infiniteismal_duality} for general cases.

\section{Levi convexities and CR structures of the unit sphere bundles}
\label{sec:convexity_unit_sphere_bundle}
We will discuss in \S\ref{subsec:tangent_space_to_the_unit_sphere_bundle} the structure of the (real) tangent space of the unit sphere bundle
$$
\mathcal{S}\mathcal{Q}_g=\{q\in \mathcal{Q}_g\mid \LOneNorm(q)=1\}
$$
in $\mathcal{Q}_g$. Here we deal with the simplest case $g=1$. 

The real tangent space $T^{\mathbb{R}}_q\mathcal{S}\mathcal{Q}_1$
at $q\in \mathcal{S}\mathcal{Q}_1$ is presented as
$$
T^{\mathbb{R}}_q\mathcal{S}\mathcal{Q}_1=\{V\in T_q\mathcal{Q}_1\mid
{\rm Re}(\partial \LOneNorm|_q(V))=0\}.
$$
It is easy see that $i\partial_q=(0,i)\in T_q\mathcal{Q}_1$ is in the real tangent space $T^{\mathbb{R}}_q\mathcal{S}\mathcal{Q}_1$.
In this case, the horizontal subspace $T_{q}^H\mathcal{S}\mathcal{Q}_1$ coincides with the maximal complex subspace of the (real) tangent space $T^{\mathbb{R}}_q\mathcal{S}\mathcal{Q}_1$ and
$$
T_{q}^H\mathcal{S}\mathcal{Q}_1=
\left\{
a({\rm Im}(\tau),iq)\mid a\in \mathbb{C}
\right\}
$$
and
$$
T^{\mathbb{R}}_q\mathcal{S}\mathcal{Q}_1=T_{q}^H\mathcal{S}\mathcal{Q}_1
\oplus \{t(0,i)\mid t\in \mathbb{R}\}.
$$

\subsection{Sign of the Levi form}
\label{subsec:Sign_of_the_Levi_form}
For making a conjectural picture of the structure of the sphere bundle, we continue to discuss with this simplest case (cf. \Cref{conj:Levi-L1}).
The complex Hessian (Levi matrix) of the $L^1$-norm function
is 
$$
{\rm Levi}(\LOneNorm)|_q=
\begin{bmatrix}
0 & -i\dfrac{q}{2|q|} \\
i\dfrac{\overline{q}}{2|q|} & \dfrac{{\rm Im}(\tau)}{2|q|}
\end{bmatrix}
=\dfrac{1}{2|q|}\begin{bmatrix}
0 & -iq \\
i \overline{q} & {\rm Im}(\tau)
\end{bmatrix}.
$$
This defines the Hermitian form on $T_q\mathcal{Q}_1$. Notice that
\begin{align*}
(0,i){\rm Levi}(\LOneNorm)|_q\begin{pmatrix} 0 \\ -i\end{pmatrix}
&=\dfrac{{\rm Im}(\tau)}{2|q|} >0 \\
({\rm Im}(\tau),iq){\rm Levi}(\LOneNorm)|_q\begin{pmatrix} {\rm Im}(\tau) \\ -i\overline{q}\end{pmatrix}
&=-\dfrac{|q|}{2}{\rm Im}(\tau) <0.
\end{align*}
Thus, the Levi form of $\LOneNorm$ is negative on the horizontal subspace $T^H_q\mathcal{S}\mathcal{Q}_1$ and positive on the vertical space.
We notice that the eigenvalues of the Levi form is $\lambda_{\pm}=({\rm Im}(\tau)\pm \sqrt{{\rm Im}(\tau)^2+4|q|^2})/2$, and the corresponding eigenspaces are spanned by
$$
V_\pm =\begin{pmatrix}
-2iq \\
{\rm Im}(\tau)^2\pm 4|q|^2
\end{pmatrix}.
$$ 
In particular, the eigenspace for the negative eigenvalue is not a subspace in the real tangent space to the unit sphere bundle.

On the other hand,
the complex Hessian (Levi matrix) of the Teichm\"uller metric $\teichmullernorm$ is 
$$
{\rm Levi}(\teichmullernorm)|_{(\tau,\mu)}=
\begin{bmatrix}
\dfrac{|\mu|}{4{\rm Im}(\tau)^3}
&
\dfrac{i\mu}{8|\mu|{\rm Im}(\tau)^2}
\\
-\dfrac{i\overline{\mu}}{8|\mu|{\rm Im}(\tau)^2}
&
\dfrac{1}{8|\mu|{\rm Im}(\tau)}
\end{bmatrix}
$$
which is positive definite. Hence, the Teichm\"uller metric is strictly plurisubharmonic on $T\teich_1^\times =T\teich_1-\{0\}$, and the unit sphere bundle $\mathcal{S}T\teich_1$ is pseudoconvex at every point.

\subsection{CR structures on the unit tangent bundles}
\label{subsec:CR-tori}
If we restrict the differential of the Teichm\"uller Beltrami map $\tb_0$ to the horizontal subspace $T^H_q\mathcal{S}\mathcal{Q}_1$,
$$
D\tb_0|_{q_0}\left({\rm Im}(\tau)\partial_\tau+iq\partial q\right)
={\rm Im}(\tau)\partial_\tau+(-4i\overline{q}{\rm Im}(\tau)^2)\partial_\mu.
$$
This means that the restriction of $D\tb_0|_{q_0}$ is $\mathbb{C}$-linear (compare with \Cref{prop:C-linear-tb}), and the image of $T^H_q\mathcal{S}\mathcal{Q}_1$ coincides with the maximal complex subspace in the real tangent space $T^\mathbb{R}_{\tb_0(q)}\mathcal{S}T\teich_1$ of the unit sphere bundle with respect to the Teichm\"uller metric.
Thus, the subbundle
$$
\cup_{q\in \mathcal{S}\mathcal{Q}_1}T^H_q\mathcal{S}\mathcal{Q}_1
$$
defines an (abstract) CR structure on $\mathcal{S}\mathcal{Q}_1$ and the Teichm\"uller Beltrami map $\tb_0$ is a CR-isomorphism between unit sphere bundles (cf. \cite[\S9]{MR1211412}). See \Cref{prob:2} in \S\ref{sec:onjectural_picture2}.

%
%


\chapter{Models of Spaces and Pairings}
\label{Chap:model_pairing}

\section{Model of second order infinitesimal spaces}
One of the purpose of this paper is to give models of second order infinitesimal spaces of the Teichm\"uller space $\teich_g$ in view of the theory of moduli of Riemanns surfaces.
To this end, we first formulate the model of the second order infinitesimal spaces.

Let $M$ be a complex manifold of dimension $n$.
Let $p\in M$. 
In this paper,
a \emph{model of the second order infinitesimal spaces of $M$ at $p\in M$} consisting of
\begin{itemize}
\item
collections of complex vector spaces of dimension $2n$:
$$
\{V_v\}_{v\in T_pM},
\{V^*_v\}_{v\in T_pM},
\{W_\omega\}_{\omega\in T^*_pM},
\{W^*_\omega\}_{\omega\in T^*_pM},
$$
where $V_v^*$ and $W_\omega^*$ are dual spaces of $V_v$ and $W_\omega$, respectively;
\item
cor $\omega\in T^*_pM$, each $W_\omega$ admits a complex symplectic form 
$\mathbf{\Omega}_\omega$;
\item
maps : the filp $A_p\colon V_p\to  V_p$, the switch $B_p\colon W_p\to V^*_p$, and the dualization $C_\omega\colon W_\omega\to W^*_\omega$, 
where 
\begin{align*}
V_p
&=\{(v,X)\mid v\in T_pM, X\in V_v\}, \\
V^*_p
&=\{(v,Z)\mid v\in T_pM, Z\in V^*_v\},\\
W_p
&=\{(\omega,\Omega)\mid \omega\in T^*_pM, \Omega\in W_\omega\};
\end{align*}
\item
Commutative diagrams. In each diagram, horizontal lines are all exact. :
$$
\begin{CD}
0@>>> T_pM @>{\verticalinc{M}{v}}>> T_vTM @>{D\Pi_M}>> T_pM @>>> 0  \\
@. @| @V{\Phi_v}VV @| \\
0@>>> T_pM @>{vi_v}>> V_v @>{Dpr_v}>> T_pM @>>> 0 
\end{CD}
$$
$$
\begin{CD}
0@>>> T^*_pM @>{(\Pi^\dagger_{M})^*}>> T^*_v\!TM @>{\verticalproj{M}{v}}>> T^*_pM @>>> 0 \\
@. @| @V{\Phi^\dagger_v}VV @| \\
0@>>> T^*_pM @>{pr^*_v}>> V^*_v @>{vp_v}>> T^*_pM @>>> 0
\end{CD}
$$
$$
\begin{CD}
0 @>>> T_p^*M @>{\verticalincdag{M}{\omega}}>> T_\omega T^*\!M @>{D\Pi^\dagger_M}>> T_pM @>>> 0 \\
@. @| @V{\Psi_\omega}VV @| \\
0@>>> T^*_pM @>{vi^\dagger_\omega}>> W_\omega @>{Dpr^\dagger_\omega}>> T_pM @>>> 0
\end{CD}
$$
$$
\begin{CD}
0@>>> T_p^*M @>{(\Pi^\dagger_M)^*}>> T^*_\omega T^*\!M @>{\verticalprojdag{M}{\omega}}>> T_pM @>>> 0 \\
@. @| @V{\Psi^\dagger_\omega}VV @| \\ 
0@>>> T^*_pM @>{pr^{\dagger,*}_\omega}>> W^*_\omega @>{vp^\dagger_\omega}>> T_pM @>>> 0
\end{CD}
$$
(Notice that in each commutative diagram, the vertical map in the middle is (automatically) an isomorphism).
%
\end{itemize}
with the properties that
\begin{enumerate}
\item
$\Phi^\dagger_v$ and $\Psi^\dagger_\omega$ are duals of $\Phi_v$ and $\Psi_\omega$, respectively. Namely,
\begin{align*}
\mathcal{P}_{TM}(V_1,V_2) &=\mathcal{P}_{V_v}(\Phi_v(V_1), \Phi^\dagger_v(V_2)) \\
\mathcal{P}_{T^*\!M}(\Omega_1,\Omega_2) &=\mathcal{P}_{W_\omega}(\Psi_\omega(\Omega_1), \Psi^\dagger_\omega(\Omega_2))
\end{align*}
for $V_1\in T_vTM$, $V_2\in T^*_v\!TM$, $\Omega_1\in T_\omega T^*\!M$ and $\Omega_2\in T^*_\omega\!T^*\!M$,
where $\mathcal{P}_{V_v}$ (resp. $\mathcal{P}_{W_\omega}$) is the dual pairing between $V_v$ and $V^*_v$ (resp. $W_\omega$ and $W^*_\omega$);
\item
$A_p\circ A_p=id_{V_p}$;
\item
for $V\in V_v$. if $u=Dpr_v(V)$, $A_p(V)\in V_u$ and $Dpr_u\circ A_p(V)=v$;
\item
let $v\in T_pM$ and $V\in V_v$. Set $u=Dpr_v(V)$. For any surjective $\mathbb{C}$-linear mappings $L_u\colon V_u\to T_pM$ and $L_v\colon V_v\to T_pM$ with $L_u\circ vi_u=L_v\circ vi_v=id_{T_pM}$, $L_u\circ A_p(V)=L_v(V)$;
\item
when $p$ varies, the derivative of the pairing function $\mathcal{P}_M$ on $TM\oplus T^*\!M$ in the direction $(V_1,V_2)\in T_vTTM\oplus T_\omega T^*\!M$ at $(v,\omega)\in TM\oplus T^*\!M$ satisfies
$$
D\mathcal{P}_M|_{(v,\omega)}[V_1,V_2]=\mathcal{P}_{V_u}(A_p(V_1),B_p(V_2))
$$
where $u=Dpr_v(V_1)=Dpr^\dagger_\omega(V_2)$; and
\item
for $\Omega_1$, $\Omega_2\in T_\omega T^*\!M$,
$$
\mathcal{P}_{W_\omega}(\Psi_\omega(\Omega_1),C_\omega(\Psi_\omega(\Omega_2)))
=2\mathbf{\Omega}_\omega(\Psi_\omega(\Omega_2),\Psi_\omega(\Omega_1)).
$$
\end{enumerate}
In our case,
$M=\teich_g$ and
we will denote the model spaces by
\begin{align*}
V_v&=\mathbb{T}_{[Y]}[\mathcal{U}]=\mathbb{T}^{Dol}_{[Y]}[\mathcal{U}]
=\mathbb{T}^{Bel}_{[\nu]}[\mathcal{U}]
\\
V^*_v&=\mathbb{T}^\dagger_{[Y]}[\mathcal{U}]
\\
W_\omega&={\bf H}^1(\mathcal{U},\mathbb{L}_{q_0}) \\
W^*_\omega&={\bf H}^{1,\dagger}(\mathcal{U},\mathbb{L}_{q_0})
\end{align*}
for $p=x_0=(M_0,f_0)\in\teich_g$, $v=[Y]=[\nu]\in H^1(\mathcal{U},\Theta_{M_0})=T_{x_0}\teich_g$, $\omega=q_0\in \mathcal{Q}_{x_0}$ and an appropriate covering $\mathcal{U}$ of $M_0$,
defined at \Cref{def:model_space}, 
\Cref{def:model_space_dol_independent},
\Cref{def:dol_pre_Bel}, 
\S\ref{sec:Model_Tstar}, 
\S\ref{sec:tangent-Qg},
and
\S\ref{sec:ModelTTster_TstarTstar}
later.
Models of pairings
\begin{align*}
\mathcal{P}_{\mathbb{TT}}
&\colon \mathbb{T}_{[Y]}[\mathcal{U}]\times \mathbb{T}^\dagger_{[Y]}[\mathcal{U}]\to \mathbb{C} \\
\paircot
&\colon {\bf H}^1(\mathcal{U},\mathbb{L}_{q_0})\times {\bf H}^{1,\dagger}(\mathcal{U},\mathbb{L}_{q_0})\to \mathbb{C}
\end{align*}
are defined in \S\ref{subsec:model_pairing_TT_TstarT} and \S\ref{subsec:model_pairing_cotangent}. The model of the holomorphic symplectic form
$$
\omega_{\mathcal{Q}_g}
\colon {\bf H}^1(\mathcal{U},\mathbb{L}_{q_0})\times {\bf H}^1(\mathcal{U},\mathbb{L}_{q_0})\to \mathbb{C}
$$
is given in \S\ref{subsec:Kawai} by reformulating Kawai's symplectic form (\cite{MR1386110}) in our setting.

\section{Lie bracket and Lie derivative}
\label{sec:Lie_derivative}
Let $(U,z)$ be an analytic coordinate chart of $M_0$ and $p,q\in \mathbb{N}\cup\{0\}$. 
For $\xi\in H^0(U,\mathcal{A}^{0,p}(\Theta_{M_0}))$, $\eta\in H^0(U,\mathcal{A}^{0,q}(\Theta_{M_0}))$, we define the \emph{Lie bracket} of $\xi_1$ and $\xi_2$ by
\begin{equation}
\label{eq:Lie-derivative}
[\xi,\eta]=\left(\xi(z)\eta_z(z)-\xi_z(z)\eta(z)\right)d\overline{z}^{p+q}\otimes \partial_{z}\in 
H^0(U,\mathcal{A}^{0,p+q}(\Theta_{M_0}))
\end{equation}
(cf. \cite[p.265]{MR815922}).
For $\xi$, $\eta\in \Gamma(U,\mathcal{A}^{0,0}(\Theta_{M_0}))$, we can see that
$$
\overline{\partial}[\xi,\eta]=[\overline{\partial}\xi,\eta]+[\xi,\overline{\partial}\eta]
\in \Gamma(U,\mathcal{A}^{0,1}(\Theta_{M_0})).
$$
For $p\ge 0$ and $r,s\in \mathbb{Z}$,
we define the \emph{Lie derivative} 
$$
L_{\cdot}(\cdot)\colon \Gamma(U,\mathcal{A}^{0,r}(\Omega_{M_0}^{\otimes p}))\oplus
\Gamma(U,\mathcal{A}^{0,s}(\Theta_{M_0}))
\to
\Gamma(U,\mathcal{A}^{0,r+s}(\Omega_{M_0}^{\otimes p}))
$$
by
\begin{equation}
\label{eq:lie_derivative_definition}
L_{\xi}(\omega)(z)dz^{p}d\overline{z}^{r+s}
=
(\xi(z)\omega_z(z)+p\,\xi_z(z)\omega(z))dz^pd\overline{z}^{r+s},
\end{equation}
where $\omega=\omega(z)dz^pd\overline{z}^r$ and  $\xi=\xi(z)\dfrac{d\overline{z}^s}{dz}$.
We often use the following formula in the following argument:
For $Q\in \Gamma(U,\mathcal{A}^{0,0}(\Omega_{M_0}^{\otimes 2}))$
and $\xi$, $\eta \in \Gamma(U,\mathcal{A}^{0,0}(\Omega_{M_0}))$,
\begin{align}
&\overline{\partial}L_\xi(Q)=L_{\overline{\partial}\xi}(Q)+L_{\xi}(\overline{\partial}Q)
\label{eq:Lie-derivative_0}
\\
&L_\xi(\eta Q) =L_\eta(\xi Q)
=\partial(\xi \eta Q)
\label{eq:Lie-derivative_01}
\\
&L_\xi(\eta Q)-\xi L_\eta(Q)+[\xi,\eta]Q=0
\label{eq:Lie-derivative_1}
\\
&L_{\xi}(L_\eta(Q))-L_\eta(L_\xi(Q))=L_{[\xi,\eta]}(Q).
\label{eq:Lie-derivative_2}
\end{align}

\begin{convention}
\label{convention:2}
As discussed in \Cref{convention:1},
for $\mu\in \Gamma(U,\mathcal{A}^{0,1}(\Theta_M))$ and $Q\in \Gamma(U,\mathcal{A}^{0,0}(\Omega_M^{\otimes 2}))$, we set
$$
\mu Q=\mu(z)Q(z)d\overline{z}\wedge dz.
$$
\end{convention}

For $\xi,\eta\in \Gamma(U,\mathcal{A}^{0,0}(\Theta_M))$, $\mu\in \Gamma(U,\mathcal{A}^{0,1}(\Theta_M))$ and $Q\in \Gamma(U,\mathcal{A}^{0,0}(\Omega_M^{\otimes 2}))$, 
$\xi Q=\xi(z)Q(z)dz\in \Gamma(U,\mathcal{A}^{0,0}(\Omega_M))$ and 
$$
\xi \mu Q=\mu \xi Q=\xi(z)\mu(z)Q(z)d\overline{z}\in \Gamma(U,\mathcal{A}^{0,1}).
$$
Hence
\begin{align}
&d(\xi Q)=\overline{\partial}(\xi Q)
=(\xi_{\overline{z}}(z)Q(z)+\xi Q_{\overline{z}}(z))d\overline{z}\wedge dz
=(\overline{\partial}\xi)Q+\xi (\overline{\partial}Q)
\label{eq:Lie-derivative_02}\\
&d(\xi \mu Q)=\partial (\xi \mu Q)=(\xi_z\mu Q+\xi \mu_z Q+\xi \mu Q_z)dz\wedge d\overline{z}
=-L_\xi(\mu Q)
\label{eq:Lie-derivative_03}
\\
&L_\xi(\mu Q)-\mu L_\xi(Q)+[\mu,\xi]Q=0
\label{eq:Lie-derivative_04}
\end{align}

\section{Model spaces of $T_{[Y]}T\teich_g$}
\label{sec:Model_space_definition}
Let $M_0$ be a closed Riemann surface of genus $g$.
In this section, we fix a locally finite covering $\mathcal{U}=\{U_i\}_{i\in I}$ of $M_0$.

\subsection{Operators}
For $X=\{X_{ij}\}_{i,j\in I}$, $Y=\{Y_{ij}\}_{i,j\in I}\in Z^1(\mathcal{U},\Theta_{M_0})$,
we define (the bilinarization of) the \emph{primary obstruction} $\zeta(X,Y)\in Z^2(\mathcal{U},\Theta_{M_0})$ for $X$ and $Y$ by
$$
\zeta(X,Y)_{ijk}=\dfrac{1}{2}([X_{ij},Y_{jk}]+[Y_{ij},X_{jk}])
$$
on $U_{i}\cap U_{j}\cap U_{k}$ (cf. \cite[\S5.1]{MR815922}). 
By definition,
$$
\zeta(X,Y)=\zeta(Y,X)
$$
for $X$, $Y\in Z^1(\mathcal{U},\Theta_{M_0})$.
We also define
\begin{align*}
&S\colon  C^0(\mathcal{U}, \mathcal{A}^{0,0}(\Theta_{M_0}))
\oplus Z^1(\mathcal{U}, \Theta_{M_0})\to 
C^1(\mathcal{U}, \Theta_{M_0}) \\
&K,K'\colon  C^0(\mathcal{U}, \mathcal{A}^{0,0}(\Theta_{M_0}))\oplus 
Z^1(\mathcal{U}, \Theta_{M_0})\to 
C^1(\mathcal{U}, \Theta_{M_0})
\end{align*}
by 
\begin{align*}
S(\xi,Y)_{ij} & =\dfrac{1}{2}[\xi_i+\xi_j,Y_{ij}] \\
K(\alpha,Y)_{ij} & = [\alpha_i,Y_{ij}] \\
K'(\alpha,Y)_{ij} & = [\alpha_j,Y_{ij}].
\end{align*}
Then, for $\alpha\in C^0(\mathcal{U},\Theta_{M_0})$, 
\begin{align}
\delta(S(\alpha,Y))
&=\zeta(\delta \alpha,Y)=\zeta(Y,\delta\alpha)
\label{eq:S-L-zeta1}
\\
K(\alpha,Y)
& =-\dfrac{1}{2}[\partial \alpha,Y]+S(\alpha,Y).
\label{eq:K-bra-S}
\\
K'(\alpha,Y)
& =\dfrac{1}{2}[\partial \alpha,Y]+S(\alpha,Y).
\label{eq:K_prime-bra-S}
\end{align}
(cf. \cite[\S5.1]{MR815922}). For $X=\{X_{ij}\}_{i,j\in I}\in C^1(\mathcal{U},\Theta_{M_0})$,
we set
\begin{equation}
\label{eq:filp}
X^*=\{X_{ji}\}_{i,j\in I}.
\end{equation}

\subsection{Model spaces $\mathbb{T}_Y[\mathcal{U}] $ for $Y\in Z^1(\mathcal{U},\Theta_{M_0})$}
\label{subsec:model_space_definition}
Henceforth, we fix a $1$-cocycle $Y=\{Y_{ij}\}\in Z^1(\mathcal{U},\Theta_{M_0})$.
We consider the following $\mathbb{C}$-linear maps
\begin{align}
&D_0^{Y}\colon  C^0(\mathcal{U},\Theta_{M_0})^{\oplus 2}
\to Z^1(\mathcal{U},\Theta_{M_0})\oplus C^1(\mathcal{U},\Theta_{M_0}) 
\label{eq:D_0_pre}
\\
&D_1^Y\colon  Z^1(\mathcal{U},\Theta_{M_0})\oplus C^1(\mathcal{U},\Theta_{M_0})\to C^1(\mathcal{U},\Theta_{M_0})\oplus C^2(\mathcal{U},\Theta_{M_0})
\label{eq:D_1_pre}
\end{align}
defined by
\begin{align}
D^Y_0(\alpha,\beta)&=(\delta \alpha, \delta \beta+K(\alpha,Y)) 
\label{eq:D_0}
\\
D^Y_1(X,\dot{Y})&=
\left(
\dot{Y}+\dot{Y}^*+[X,Y],
\delta \left(\dot{Y}+\dfrac{1}{2}[X,Y]\right)-\zeta(X,Y)
\right),
\label{eq:D_1}
\end{align}
where $\alpha=\{\alpha_i\}_i$, $\beta=\{\beta_i\}_i\in C^0(\mathcal{U},\Theta_{M_0})$, $X=\{X_{ij}\}$, $Y=\{Y_{ij}\}\in Z^1(\mathcal{U},\Theta_{M_0})$, $\dot{Y}=\{\dot{Y}_{ij}\}_{i,j}\in C^1(\mathcal{U},\Theta_{M_0})$ and 
$$
[X,Y]=\{[X_{ij},Y_{ij}]\}_{i,j}\in C^1(\mathcal{U},\Theta_{M_0}).
$$
We claim

\begin{lemma}
\label{lem:D_0-D_1-exact}
$D^Y_1\circ D^Y_0=0$.
\end{lemma}

\begin{proof}
Let $\alpha=\{\alpha_i\}_i$, $\beta=\{\beta_i\}_i\in C^0(\mathcal{U},\Theta_{M_0})$.
The $(i,j)$-component of the first coordinate of $D^Y_1\circ D^Y_0(\alpha,\beta)$
satisifes
$$
(\beta_j-\beta_i)+[\alpha_i,Y_{ij}]+(\beta_i-\beta_j)+[\alpha_j,Y_{ji}]
+[\alpha_j-\alpha_i,Y_{ij}]=0.
$$
From (b) of \Cref{prop:linear-map-L} and \eqref{eq:S-L-zeta1},
the $(i,j,k)$-component of the second coordinate of $D^Y_1\circ D^Y_0(\alpha,\beta)$ satisfies
\begin{align*}
&[\alpha_i,Y_{ij}]+[\alpha_j,Y_{jk}]+[\alpha_k,Y_{ki}] \\
&+\dfrac{1}{2}[\alpha_j-\alpha_i,Y_{ij}]+\dfrac{1}{2}[\alpha_k-\alpha_j,Y_{jk}]+\dfrac{1}{2}[\alpha_i-\alpha_k,Y_{ki}] -\zeta(\delta \alpha,Y)\\
&=S(\alpha,Y)_{ij}+S(\alpha,Y)_{jk}+S(\alpha,Y)_{ki} - \delta(S(\alpha,Y))_{ijk} =0
\end{align*}
which implies what we wanted.
\end{proof}

\begin{definition}[Model spaces for cocycles]
\label{def:model_space_for_cocycle}
The \emph{model space of the double tangent space for a cocycle $Y\in Z^1(\mathcal{U},\Theta_{M_0})$} by
$$
\mathbb{T}_Y[\mathcal{U}] = \ker(D^Y_1)/{\rm Im}(D^Y_0).
$$
We denote by $\tv{X,\dot{Y}}{Y}$ the equivalence class of $(X,\dot{Y})\in \ker(D^Y_1)$ in $\mathbb{T}_Y[\mathcal{U}]$. 
\end{definition}

\subsection{Vertical spaces in the model space}
\label{subsec:vertical_space}
We discuss the vertical spaces in the model space.

\begin{definition}[Vertical spaces]
\label{def:vertical_space_TY}
The \emph{vertical spaces} $\mathbb{T}^{V}_Y[\mathcal{U}]$
 in $\mathbb{T}_Y[\mathcal{U}]$
is defined by
\begin{align}
\mathbb{T}^{V}_Y[\mathcal{U}]=\left\{
\tv{X,\dot{Y}}{Y}\in \mathbb{T}_{Y}\mid \mbox{$[X]=0$ in $H^1(\mathcal{U},\Theta_{M_0})$}
\right\}
\label{eq:vertical_space}
\end{align}
\end{definition}

\begin{proposition}[Vertical space]
\label{prop:vertical_space}
The verical space $\mathbb{T}^{V}_Y[\mathcal{U}]$ is canonically isomorphic to $H^1(\mathcal{U},\Theta_{M_0})$.
Indeed, 
\begin{equation}
\label{eq:vertical_isomorphism}
\verticalincmodel{Y}\colon
\mathbb{T}^{V}_Y[\mathcal{U}]\ni \tv{\delta \alpha,\dot{Y}}{Y}\mapsto [\dot{Y}-K(\alpha,Y)]\in H^1(\mathcal{U},\Theta_{M_0})
\end{equation}
is a well-defined $\mathbb{C}$-linear isomorphism, where $\alpha\in C^0(\mathcal{U},\Theta_{M_0})$. 
\end{proposition}

\begin{proof}
We will show that
Let $\tv{X,\dot{Y}}{Y}\in \mathbb{T}^{V}_Y[\mathcal{U}]$. Let $(X,\dot{Y})\in \ker(D^Y_1)$ be a representative of $\tv{X,\dot{Y}}{Y}$. By definition, $X=\delta \alpha$ for some $\alpha\in C^0(\mathcal{U},\Theta_{M_0})$. Hence,
$$
0=\delta(\dot{Y}+\dfrac{1}{2}[X,Y])-\zeta(X,Y)=\delta(\dot{Y}-K(\alpha,Y))
$$
and
\begin{align*}
(\dot{Y}-K(\alpha,Y))_{ji}
&=\dot{Y}_{ji}-[\alpha_j,Y_{ji}]=-\dot{Y}_{ij}-[\delta \alpha,Y_{ij}]+[\alpha_j,Y_{ij}]  \\
&=-(\dot{Y}_{ij}-[\alpha_i,Y_{ij}])
=-(\dot{Y}-K(\alpha,Y))_{ij}.
\end{align*}
Hence, $\dot{Y}-K(\alpha,Y)\in Z^1(\mathcal{U},\Theta_{M_0})$.

Let $(\delta \alpha',\dot{Y}') \in \ker(D^Y_1)$ be another representative of $\tv{X,\dot{Y}}{Y}\in \mathbb{T}_{Y}[\mathcal{U}]$. 
Since $(\delta\alpha-\delta\alpha', \dot{Y}-\dot{Y}')$ is trivial in $\mathbb{T}_Y[\mathcal{U}]$, from the definition of the model space,
there are $\beta$, $\gamma\in C^0(\mathcal{U},\Theta_{M_0})$ such that 
$\delta\alpha-\delta\alpha'=\delta \gamma$ and
$\dot{Y}-\dot{Y}'=\delta\beta+K(\gamma,Y)$.
Since $M_0$ admits no non-trivial holomorphic vector field, $\gamma=\alpha-\alpha'$. Hence $\dot{Y}-\dot{Y}'=\delta\beta+K(\alpha-\alpha',Y)$, which is equivalent to
$$
(\dot{Y}-K(\alpha,Y))-(\dot{Y}'-K(\alpha',Y))=\delta\beta.
$$
This means that the map \eqref{eq:vertical_isomorphism} is well-defined. 

Suppose $\tv{\delta \alpha,\dot{Y}}{Y}$ is in the kernel of $\verticalincmodel{Y}$. Then,  $\dot{Y}-K(\alpha,Y)=\delta\beta$ for some $\beta$, and
$$
(\delta \alpha,\dot{Y})=(\delta\alpha,\delta\beta+K(\alpha,Y))=D_0^Y(\alpha,\beta).
$$
Therefore $\tv{\delta \alpha,\dot{Y}}{Y}=0$ in $\mathbb{T}_Y[\mathcal{U}]$. Hence $\verticalincmodel{Y}$ is injective.

For $[W]\in H^1(\mathcal{U},\Theta_{M_0})$, $(0,W)\in \ker(D^Y_1)$ and $
\verticalincmodel{Y}(\tv{0,W}{Y})=[W]$. 
Therefore, $\verticalincmodel{Y}$ is a $\mathbb{C}$-linear isomorphism.
\end{proof}

%

\subsection{Model space $\mathbb{T}_{[Y]}[\mathcal{U}] $ for $[Y]\in H^1(\mathcal{U},\Theta_{M_0})$}
\label{subsec:model_space_definition_2}
We shall check the following.
\begin{proposition}
\label{prop:isomorphism_tile_L}
For $\beta\in C^0(\mathcal{U},\Theta_{M_0})$,
a linear isomorphism on $Z^1(\mathcal{U},\Theta_{M_0})\oplus C^1(\mathcal{U},\Theta_{M_0})$ defined by
\begin{equation}
\label{eq:cohomologus}
\tilde{\mathcal{L}}_{\beta;Y}(X,\dot{Y})=
(X,\dot{Y}+K'(\beta,X))
\end{equation}
descends to an isomorphism $\mathcal{L}_{\beta;Y}$ from $\mathbb{T}_Y[\mathcal{U}]$ and $\mathbb{T}_{Y+\delta\beta}[\mathcal{U}]$. Furthermore,
the isomorphism $\mathcal{L}_{\beta;Y}$ maps $\mathbb{T}^V_{Y}[\mathcal{U}]$ onto $\mathbb{T}^V_{Y+\delta \beta}[\mathcal{U}]$. 
\end{proposition}

\begin{proof}
From the definition, if $\tilde{\mathcal{L}}_{\beta;Y}$ descends to the linear map from $\mathbb{T}_Y[\mathcal{U}]$ and $\mathbb{T}_{Y+\delta\beta}[\mathcal{U}]$, the linear map sends the vertical space in $\mathbb{T}_Y[\mathcal{U}]$ to that in $\mathbb{T}_{Y+\delta}[\mathcal{U}]$. Hence, we only check the descendant of $\tilde{L}_{\beta;Y}$.

%
Let $(X,\dot{Y})\in \ker(D_1^Y)$.
Then,
\begin{align*}
&(\dot{Y}_{ij}+[\beta_j,X_{ij}])+(\dot{Y}_{ji}+[\beta_i,X_{ji}])+[X_{ij},Y_{ij}+\delta \beta] \\
&=\dot{Y}_{ij}+[\beta_j,X_{ij}]+\dot{Y}_{ji}-[\beta_i,X_{ij}]+[X_{ij},Y_{ij}]+[X_{ij},\beta_j-\beta_i] \\
&=\dot{Y}_{ij}+\dot{Y}_{ji}+[X_{ij},Y_{ij}]=0
\end{align*}
for $i$, $j\in I$.
From \eqref{eq:S-L-zeta1}, we have
\begin{align*}
\delta\left(\dfrac{1}{2}[X,\delta\beta]\right)-\zeta(X,\delta\beta)
&=
\delta\left(\dfrac{1}{2}[X,\delta\beta]-S(\beta,X)\right)
=-\delta\left(\dfrac{1}{2}[\delta\beta,X]+S(\beta,X)\right)
\\
&=-\delta K'(\beta,Y).
\end{align*}
Therefore, we obtain
\begin{align*}
&\delta\left(\dot{Y}+K'(\beta,Y)+\dfrac{1}{2}[X,Y+\delta \beta]\right)-\zeta(X,Y+\delta\beta) \\
&=\delta\left(\dot{Y}+\dfrac{1}{2}[X,Y]\right)-\zeta(X,Y)=0.
\end{align*}
Therefore, $\tilde{\mathcal{L}}_{\beta;Y}(X,\dot{Y})\in \ker(D_1^{Y+\delta\beta})$.

Assume $(X,\dot{Y})\in {\rm Im}(D_0^Y)$.
Then, $X=\delta\alpha$ and $\dot{Y}=\delta\gamma+K(\alpha,Y)$ for some $\alpha$, $\gamma\in C^0(\mathcal{U},\Theta_{M_0})$.
Since \begin{align*}
(\dot{Y}+K'(\beta,X))_{ij}
&=\gamma_j-\gamma_i+
[\alpha_i,Y_{ij}]+[\beta_j,\alpha_j-\alpha_i] \\
&=\gamma_j-\gamma_i+
[\alpha_i,Y_{ij}+\beta_j-\beta_i]+[\beta_j,\alpha_j]-[\beta_i,\alpha_i] \\
&=(\gamma_j+[\beta_j,\alpha_j])-(\gamma_i+[\beta_i,\alpha_i])+
[\alpha_i,Y_{ij}+\beta_j-\beta_i].
\end{align*}
we have
$$
\dot{Y}+K'(\beta,Y)=\delta(\gamma+[\beta,\alpha])+K(\alpha,Y+\delta\beta),
$$
which means $\tilde{\mathcal{L}}_{\beta;Y}(X,\dot{Y})\in {\rm Im}(D_0^{Y+\delta\beta})$.
Therefore, the map \eqref{eq:cohomologus} descends to a $\mathbb{C}$-linear map $\mathcal{L}_{\beta;Y}$ from $\mathbb{T}_Y[\mathcal{U}]\to \mathbb{T}_{Y+\delta\beta}[\mathcal{U}]$. Since $\tilde{\mathcal{L}}_{-\beta;Y}\circ\tilde{\mathcal{L}}_{\beta;Y}=id$, $\mathcal{L}_{\beta;Y}$ is isomorphic.
\end{proof}

\begin{definition}[Model space and Vertical space for cohomology classes]
\label{def:model_space}
For $\beta\in C^0(\mathcal{U},\Theta_{M_0})$, we identify $\mathbb{T}_Y[\mathcal{U}]$ and $\mathbb{T}_{Y+\delta\beta}[\mathcal{U}]$ by $\mathcal{L}_{\beta;Y}$. This identification is well-defined since $\mathcal{L}_{\beta';Y+\delta \beta}\circ  \mathcal{L}_{\beta;Y}=\mathcal{L}_{\beta'+\beta;Y}$ for $\beta$, $\beta'\in C^0(\mathcal{U},\Theta_{M_0})$.
This procedure makes a $\mathbb{C}$-vector space $\mathbb{T}_{[Y]}[\mathcal{U}]$. 
We call the space $\mathbb{T}_{[Y]}[\mathcal{U}]$ the \emph{model space of the double tangent space at $[Y]\in H^1(\mathcal{U},\Theta_{M_0})$}. We also define the \emph{vertical space} $\mathbb{T}_{[Y]}^V[\mathcal{U}]$ in the same manner.
\end{definition}

We denote by $\tv{X,\dot{Y}}{[Y]}\in \mathbb{T}_{[Y]}[\mathcal{U}]$ the equivalence class of $\tv{X,\dot{Y}}{Y}\in \mathbb{T}_Y[\mathcal{U}]$.
In \S\ref{sec:double_tangent_main},
we will show that $\mathbb{T}_Y[\mathcal{U}]$ and $\mathbb{T}_{[Y]}[\mathcal{U}]$ are naturally identified with the tangent space to the tangent bundle over $\teich_g$ at the tangent vector $[Y]\in H^1(\mathcal{U},\Theta_{M_0})$.

\begin{remark}[Sum on the model]
\label{remark:sum_double_tangent_space}
From the definition, for the additive operation on $\mathbb{T}_{[Y]}[\mathcal{U}]$,
$$
a \tv{X,\dot{Y}}{[Y]}+b\tv{X',\dot{Y}'}{[Y]}=\tv{Z,\dot{W}}{[Y]}
$$
if $Z=a X+b X'$ and $\dot{W}=a \dot{Y}+b\dot{Y}'$, where $a$, $b\in \mathbb{C}$, $(X,\dot{Y})$, $(X',\dot{Y}')\in \ker(D^Y_1)$ 
are representatives of $\tv{X,\dot{Y}}{Y}$ and $\tv{X',\dot{Y}'}{Y}$ for fixed $Y\in Z^1(\mathcal{U},\Theta_{M_0})$, respectively.
\end{remark}

\subsection{Vertical space and Cohomology group}
Since
\begin{align*}
\left((\dot{Y}+K'(\beta,\delta\alpha)-K(\alpha,Y+\delta\beta)\right)_{ij}
&=\dot{Y}+[\beta_j,\alpha_j-\alpha_i]-[\alpha_i,Y]-[\alpha_i,\beta_j-\beta_i] \\
&=(\dot{Y}-K(\alpha,Y))+\delta(\{[\beta_i,\alpha_i]\}_{i\in I})
\end{align*}
for $\alpha$, $\beta\in C^0(\mathcal{U},\Theta_{M_0})$,
from \Cref{prop:vertical_space,prop:isomorphism_tile_L}, we obtain a $\mathbb{C}$-linear isomorphism
\begin{equation}
\label{eq:vertical_isomorphism_cohomology_class}
\verticalincmodel{[Y]}\colon
\mathbb{T}^{V}_{[Y]}[\mathcal{U}]\ni \tv{\delta \alpha,\dot{Y}}{[Y]}\mapsto [\dot{Y}-K(\alpha,Y)]\in H^1(\mathcal{U},\Theta_{M_0})
\end{equation}
which satisfies the following commutative diagram:
\begin{equation}
\label{eq:commute_vertical_TYU}
\begin{CD}
\mathbb{T}^{V}_{Y}[\mathcal{U}]
@>{\verticalincmodel{Y}}>>
H^1(\mathcal{U},\Theta_{M_0})
\\
@V{\cong}VV @|
\\
\mathbb{T}^{V}_{[Y]}[\mathcal{U}]
@>>{\verticalincmodel{[Y]}}>
H^1(\mathcal{U},\Theta_{M_0}),
\end{CD}
\end{equation}
where the left vertical arrow is the natural identification in \Cref{def:model_space}.
The inverse of the map $\verticalincmodel{[Y]}$ is
\begin{equation}
\label{eq:commute_vertical_TYU2}
\begin{CD}
H^1(\mathcal{U},\Theta_{M_0}) @>>> \mathbb{T}^{V}_{[Y]}[\mathcal{U}] \\
[\dot{Y}] @>>> \tv{0,\dot{Y}}{[Y]}.
\end{CD}
\end{equation}

\section{Dolbeaut type presentation of Double tangent spaces}
\label{sec:Dolbeaut_type_presentation}
Throughout this section, we assume that a locally finite covering $\mathcal{U}=\{U_i\}_{i\in I}$ satisfies $H^1(U_i,\Theta_{M_0})=0$ for all $i\in I$. In this case
$$
\begin{CD}
C^0(\mathcal{U},\mathcal{A}^{0,0}(\Theta_{M_0}))@>{-\overline{\partial}}>> C^0(\mathcal{U},\mathcal{A}^{0,1}(\Theta_{M_0}))
\end{CD}
$$
is surjective (cf. \cite[Theorem 3.13]{MR815922}).
This section deals with two presentations of the double tangent spaces by complex $(-1,1)$-forms.
The first presentation is defined with respecting to the cohomology presentations of the basepoint. The second is defined with the presentation by the Beltrami differential of the basepoint. The first presentation is a mediator for connecting between the space $\mathbb{T}_{[Y]}[\mathcal{U}]$ and the second presentation.

\subsection{Dolbeault type presentations for cocycles}
\label{subsec:Dol_type_representation}
In the following diagram, we abbreviate $C^{k}(\mathcal{U},\mathscr{S})$, 
 $Z^k(\mathcal{U},\mathscr{S})$,
 $\Theta_{M_0}$ and $\mathcal{A}^{p,q}(\Theta_{M_0})$
 to 
 $C^{k}(\mathscr{S})$,
 $Z^k(\mathscr{S})$,
  $\Theta$ and $\mathcal{A}^{p,q}_\Theta$
  for simplicity.
The exact sequence \eqref{eq:Dolbeault-sequence} 
leads the following commutative diagram:
$$
\minCDarrowwidth10pt
\begin{CD}
C^0(\Theta)^{\oplus 2}
@>{\iota\oplus \iota}>>
C^0(\mathcal{A}^{0,0}_\Theta)^{\oplus 2}
@>{(-\overline{\partial})\oplus(-\overline{\partial})}>>
C^0(\mathcal{A}^{0,1}_\Theta)^{\oplus 2}
\\
@V{D^Y_0}VV @V{D^{Y;0}_0}VV @V{D^{Y;1}_0}VV \\
Z^1(\Theta)\oplus C^1(\Theta)
@>{\iota\oplus \iota}>>
Z^1(\mathcal{A}^{0,0}_\Theta)\oplus C^1(\mathcal{A}^{0,0}_\Theta)
@>{(-\overline{\partial})\oplus(-\overline{\partial})}>>
Z^1(\mathcal{A}^{0,1}_\Theta)\oplus C^1(\mathcal{A}^{0,1}_\Theta)
\\
@V{D^Y_1}VV @V{D^{Y;0}_1}VV \\
C^1(\Theta)\oplus C^2(\Theta)
@>{\iota\oplus \iota}>>
C^1(\mathcal{A}^{0,0}_\Theta)\oplus C^2(\mathcal{A}^{0,0}_\Theta),
\end{CD}
$$
where the $\mathbb{C}$-linear maps $D^{Y;0}_0$, $D^{Y;0}_1$, and $D^{Y;1}_0$ are natural extensions of $D^Y_0$ and $D^Y_1$ (defined in \eqref{eq:D_0} and \eqref{eq:D_1}) with notation in \S\ref{sec:Lie_derivative}:
\begin{align}
D^{Y;0}_0(\alpha,\beta)&=(\delta \alpha, \delta \beta+K(\alpha,Y)) 
\label{eq:D_0_Dol}
\\
D^{Y;0}_1(\xi,\dot{\eta})&=
\left(
\dot{\eta}+\dot{\eta}^*+[\xi,Y],
\delta \left(\dot{\eta}+\dfrac{1}{2}[\xi,Y]\right)-\zeta(\xi,Y)
\right),
\label{eq:D_1_Dol}
\end{align}
where $\alpha$, $\beta\in C^0(\mathcal{U},\mathcal{A}^{0,0}(\Theta_{M_0}))$, $\xi\in Z^1(\mathcal{U},\mathcal{A}^{0,0}(\Theta_{M_0}))$ and $\dot{\eta}\in C^1(\mathcal{U},\mathcal{A}^{0,0}(\Theta_{M_0}))$.
Since the map $(-\overline{\partial})\oplus(-\overline{\partial})$ at the right in the first line of the diagram is surjective, we have an exact sequence
\begin{equation}
\label{eq:exact_T_Y}
\minCDarrowwidth10pt
\begin{CD}
 @. @. \ker(D^{Y;0}_0)@>{(-\overline{\partial})\oplus(-\overline{\partial})}>> \ker(D^{Y;1}_0) \\ @>{\connectinghomo}>> 
\mathbb{T}_Y[\mathcal{U}] @>{\iota\oplus \iota}>> \ker(D^{Y;0}_1)/{\rm Im}(D^{Y;0}_0),
\end{CD}
\end{equation}
where $\connectinghomo$ is the connecting homomorphism.
We claim

\begin{lemma}
\label{lem:D-prime}
$\ker(D^{Y;0}_1)/{\rm Im}(D^{Y;0}_0)=\{0\}$.
\end{lemma}

\begin{proof}
Let $(\Xi,H)\in \ker(D^{Y;0}_1)$. By definition,
\begin{align*}
& H_{ij}+H_{ji}+[\Xi_{ij},Y_{ij}]=0 \\
&\delta \left(H+\dfrac{1}{2}[\Xi,Y]\right)-\zeta(\Xi,Y)=0
\end{align*}
for $i$, $j\in I$.
Let $\xi\in C^0(\mathcal{A}^{0,0}_\Theta)$ with $\delta\xi=\Xi$.
Since $\zeta(\Xi,Y)=\delta(S(\xi,Y))$,
from \eqref{eq:S-L-zeta1} and \eqref{eq:K-bra-S},
the above two equations imply that $H-K(\xi,Y)\in Z^1(\mathcal{A}^{0,0}_\Theta)$.
Take $\xi'\in C^0(\mathcal{A}^{0,0}_\Theta)$ with $\delta\xi'=H-K(\xi,Y)$. Then,
$$
D^{Y;0}_0(\xi,\xi')=
(\delta \xi, \delta \xi'+K(\xi,Y))=
(\Xi,H),
$$
which means $(\Xi,H)\in {\rm Im}(D^{Y;0}_0)$.
\end{proof}

\begin{definition}[Dolbeault type presentation for cocycle]
\label{def_mode_Dolbeaut}
We call the quotient space
\begin{equation}
\label{eq:definition_dol_rep}
\mathbb{T}^{Dol}_Y[\mathcal{U}]=\ker(D^{Y;1}_0)/
(-\overline{\partial})\oplus (-\overline{\partial})
(\ker(D^{Y;0}_0)).
\end{equation}
the \emph{Dolbeault type presentation of the model space for cocycle $Y$}.
\end{definition}

\begin{remark}
\label{remark:dol1}
Though the connecting homomorphism $\connectinghomo$ might not be defined for arbitrary coverings, we also define
$$
\mathbb{T}^{Dol}_Y[\mathcal{U}]=\ker(D^{Y;1}_0)/
(-\overline{\partial})\oplus (-\overline{\partial})
(\ker(D^{Y;0}_0))
$$
for an (arbitrary) locally finite covering on $M_0$.
\end{remark}

From \eqref{eq:exact_T_Y} and \Cref{lem:D-prime},
we obtain the following
(cf. \eqref{eq:Dou}).

\begin{theorem}[Dolbeault type presentation of model space]
\label{thm:Dolbeault_pre}
Under the above notation,
the connecting homomorphism $\connectinghomo$ 
leads the isomorphism
from $\mathbb{T}^{Dol}_Y[\mathcal{U}]$ onto $\mathbb{T}_Y[\mathcal{U}]$.
\end{theorem}
For $(\mu,\dot{\nu})\in \ker(D^{Y;1}_0)$, we denote by $\tvD{\mu,\dot{\nu}}{Y}$ the equivalence class of $(\mu,\dot{\nu})$ in $\mathbb{T}^{Dol}_Y[\mathcal{U}]$.

For the record, we notice the following:
\begin{itemize}
\item
The kernels $\ker(D''_1)$ and $\ker(D'_0)$ are
\begin{align*}
\ker(D^{Y;1}_0)&=
\{(\mu,\dot{\nu})
\in 
H^0(\mathcal{U},\mathcal{A}^{0,1}(\Theta_{M_0}))
\oplus C^0(\mathcal{U},\mathcal{A}^{0,1}(\Theta_{M_0}))
\mid
\delta\dot{\nu}+[\mu,Y]=0\}
\\
\ker(D^{Y;0}_0)&=\{(\xi,\dot{\eta})
\in 
H^0(\mathcal{U},\mathcal{A}^{0,0}(\Theta_{M_0}))
\oplus C^0(\mathcal{U},\mathcal{A}^{0,0}(\Theta_{M_0}))
\mid
\delta\dot{\eta}+[\xi,Y]=0\},
\end{align*}
where $[\mu,Y]=\{[\mu,Y_{ij}]\}_{i,j\in I}$ and $[\xi,Y]=\{[\xi_i,Y_{ij}]\}_{i,j\in I}$.
\item
$\connectinghomo(\tvD{\mu,\dot{\nu}}{Y})=\tv{X,\dot{Y}}{Y}$ if and only if there are $\xi$, $\dot{\eta}\in C^0(\mathcal{U},\mathcal{A}^{0,0}(\Theta_{M_0}))$ such that
\begin{align*}
(\mu,\dot{\nu})&=(-\overline{\partial}\xi,-\overline{\partial}\dot{\eta})
 \\
 (X,\dot{Y})&=
 (\delta \xi,\delta \dot{\eta}+K(\xi,Y)).
\end{align*}
\end{itemize}

We notice the following.

\begin{proposition}
\label{prop:equivalence_class}
Fix $\eta\in C^0(\mathcal{U},\mathcal{A}^{0,0}(\Theta_{M_0}))$ with $\delta \eta = Y$.
For $(\mu_1,\dot{\nu}_1)$, $(\mu_2,\dot{\nu}_2)\in \ker(D^{Y;1}_0)$, $(\mu_1,\dot{\nu}_1)$ and $(\mu_2,\dot{\nu}_2)$ are equivalent in the quotient space $\mathbb{T}^{Dol}_Y[\mathcal{U}]$ if and only if there is $\alpha$, $\beta\in \Gamma(M_0,\mathcal{A}^{0,0}(\Theta_{M_0}))$ such that
\begin{itemize}
\item[{\rm (1)}]
$\mu_1=\mu_2-\overline{\partial}\alpha$; and
\item[{\rm (2)}]
$(\dot{\nu}_1)_i=(\dot{\nu}_2)_i+\overline{\partial}[\alpha,\eta_i]+\overline{\partial} \beta$ on $U_i$ for $i\in I$.
\end{itemize}
\end{proposition}

\begin{proof}
Suppose (1) and (2) hold. Since
$$
([\alpha,\eta_j]+\beta)-
([\alpha,\eta_i]+\beta)+[-\alpha,Y_{ij}]=0
$$
on $U_i\cap U_j$,
$(-\alpha, \{[\alpha,\eta_i]+\beta\}_{i\in I})\in \ker(D^{Y;0}_0)$. Therefore, $(\mu_1,\dot{\nu}_1)$ and $(\mu_2,\dot{\nu}_2)$ are equivalent in $\mathbb{T}^{Dol}_Y[\mathcal{U}]$.

Suppose $(\mu_1,\dot{\nu}_1)$ and $(\mu_2,\dot{\nu}_2)$ are equivalent in $\mathbb{T}^{Dol}_Y[\mathcal{U}]$. Then, there are $\alpha\in \Gamma(M_0,\mathcal{A}^{0,0}(\Theta_{M_0}))$ and $\gamma\in C^0(\mathcal{U},\mathcal{A}^{0,0}(\Theta_{M_0}))$  such that $\mu_1=\mu_2-\overline{\partial}\alpha$, $\dot{\nu}_1=\dot{\nu}_2-\overline{\partial}\gamma$ and
\begin{equation}
\label{eq:gamma_alpha_Y}
\delta \gamma+K(\alpha,Y)=0.
\end{equation}
Since $\delta \eta=Y$,
$$
0=\gamma_j-\gamma_i+
[\alpha,\eta_j]-[\alpha,\eta_i]
$$
on $U_i\cap U_j$.
This means that
$$
\beta =-\gamma_i-[\alpha,\eta_i]
$$
on $U_i$ becomes a global smooth vector field on $M_0$ and satisfies
\begin{align*}
\overline{\partial}\beta
&=
(\dot{\nu}_1)_i-(\dot{\nu}_2)_i-\overline{\partial}[\alpha,\eta_i]
\end{align*}
on $U_i$.
\end{proof}

\subsection{Vertical spaces in the Dolbeaut type presentation for cocycle}
\label{subsec:vertical_space_Dol}
We discuss the Dolbeaut type presentation of the vertical space of the model space.

\begin{definition}[Vertical spaces in the Dolbeaut type presentation]
\label{def:vertical_space_Dol}
The \emph{vertical spaces}  $\mathbb{T}^{Dol,V}_Y[\mathcal{U}]$ in $\mathbb{T}_Y^{Dol}[\mathcal{U}]$ is defined by
\begin{align}
\mathbb{T}^{Dol,V}_Y[\mathcal{U}]=\{
\tvD{\mu,\dot{\nu}}{Y}\in \mathbb{T}^{Dol}_Y\mid 
\mbox{$[\mu]=0$ in $T_{x_0}\teich_g$}
\}.
\label{eq:Dol_vertical_space}
\end{align}
\end{definition}

\begin{proposition}[Connecting isomorphism and vertical spaces]
\label{prop:Dol_vertical_space_isomorphism}
Let $\mathcal{U}=\{U_i\}_{i\in I}$ be a locally finite covering of $M_0$ with $H^1(U_i,\Theta_{M_0})=0$ for $i\in I$. 
Then, the connecting isomorphism
$$
\connectinghomo\colon \mathbb{T}^{Dol}_Y[\mathcal{U}]\to \mathbb{T}_Y[\mathcal{U}]
$$
induces an isomorphism between the vertical spaces.
\end{proposition}

\begin{proof}
Let $\tvD{\mu,\dot{\nu}}{Y}\in \mathbb{T}^{Dol}_Y[\mathcal{U}]$. Suppose $\connectinghomo(\tvD{\mu,\dot{\nu}}{Y})=\tv{X,\dot{Y}}{Y}\in \mathbb{T}_Y$.
By the definition of the connecting homomorphism, there are $\alpha$, $\beta\in C^0(M_0,\mathcal{A}^{0,0}(\Theta_{M_0}))$ such that $\partial \alpha_i=-\mu$, $\overline{\partial}\beta_i=-\dot{\nu}_i$ for $i\in I$, $\delta \alpha=X$ and $\delta\beta +K(\alpha,Y)=\dot{Y}$. 

Assume first that $\tvD{\mu,\dot{\nu}}{Y}\in \mathbb{T}^{Dol,V}_Y[\mathcal{U}]$.
Since $[\mu]=0$, there is a section $\gamma\in \Gamma(M_0,\mathcal{A}^{0,0}(\Theta_{M_0}))$ such that $\overline{\partial}\gamma=-\mu$. We define $\xi\in C^0(\mathcal{U},\Theta_{M_0})$ by $\xi_i=\alpha_i-\gamma$ on $U_i$ for $i\in I$. Then, $\delta\xi=X$, and hence$[X]=0$. This means that $\tv{X,\dot{Y}}{Y}\in \mathbb{T}^{V}_Y[\mathcal{U}]$.

Conversely, suppose $\tv{X,\dot{Y}}{Y}\in \mathbb{T}^{V}_Y$. Then, there is $\xi\in C^0(\mathcal{U},\Theta_{M_0})$ such that $\delta\xi=X$.Hence, $\gamma=\alpha-\xi$ defines a global smooth vector field on $M_0$ with $\overline{\partial}\gamma=-\mu$. This means that $[\mu]=0$ and $\tvD{\mu,\dot{\nu}}{Y}\in \mathbb{T}^{Dol,V}_Y[\mathcal{U}]$.

Thus we conclude that $\connectinghomo(\mathbb{T}^{Dol,V}_Y[\mathcal{U}])=\mathbb{T}^{V}_Y[\mathcal{U}]$, and the vertical spaces are isomorphic via the connecting isomorphism $\connectinghomo$.
\end{proof}
%
%
%
%
%
%

\subsection{Dolbeault type presentations for cohomology classes}
\label{subsec:Dol_type_representation_class}
We start with the following lemma.

\begin{lemma}
\label{lem:Dol_base_change}
For $\beta\in C^0(\mathcal{U},\Theta_{M_0})$,
a linear isomorphism on $C^0(\mathcal{U},\mathcal{A}^{0,1}_{\Theta})^{\oplus 2}$ defined by
\begin{equation}
\label{eq:cohomologus_Dol}
\tilde{\mathcal{L}}^{Dol}_{\beta;Y}(\mu,\dot{\nu})=
(\mu,\dot{\nu}-[\mu,\beta])
\end{equation}
descends to an isomorphism $\mathcal{L}^{Dol}_{\beta;Y}$ from $\mathbb{T}^{Dol}_Y[\mathcal{U}]$ to $\mathbb{T}^{Dol}_{Y+\delta\beta}[\mathcal{U}]$.
Furthermore,
if a locally finite covering $\mathcal{U}=\{U_i\}_{i\in I}$ of $M_0$ satisfies $H^1(U_i,\Theta_{M_0})=0$ for $i\in I$. Then,
$\mathcal{L}^{Dol}_{\beta;Y}$ preserves the vertical spaces, and satisfies the following commutative diagram:
\begin{equation}
\label{eq:commulte_Dol_double_beta}
\begin{CD}
\mathbb{T}^{Dol}_Y[\mathcal{U}] @>{\connectinghomo}>> \mathbb{T}_Y[\mathcal{U}] 
\\
@V{\mathcal{L}^{Dol}_{\beta;Y}}VV @VV{\mathcal{L}_{\beta;Y}}V
\\
\mathbb{T}^{Dol}_{Y+\delta \beta}[\mathcal{U}] @>{\connectinghomo}>> \mathbb{T}_{Y+\delta \beta}[\mathcal{U}]  \\
\end{CD}
\end{equation}
\end{lemma}

\begin{proof}
First we check $\tilde{\mathcal{L}}^{Dol}_{\beta;Y}$ descends to a map from $\mathbb{T}^{Dol}_Y[\mathcal{U}]$ to $\mathbb{T}^{Dol}_{Y+\delta \beta}[\mathcal{U}]$.
Let $(\mu,\dot{\nu})\in \ker(D^{Y;1}_0)$ and $(\tilde{\mu},\tilde{\dot{\nu}})=\tilde{\mathcal{L}}^{Dol}_{\beta;Y}(\mu,\dot{\nu})$. Then,
\begin{align*}
\tilde{\dot{\nu}}_j-\tilde{\dot{\nu}}_i+[\tilde{\mu},Y_{ij}+\beta_j-\beta_i]
&=(\dot{\nu}_j+\partial [\xi_j,\beta_j])-(\dot{\nu}_i+\partial  [\xi_i,\beta_i])
+[\mu,Y_{ij}]+[\mu,\beta_j]-[\mu,\beta_i] \\
&=-[\mu,\beta_j]+[\mu,\beta_i]
+[\mu,\beta_j]-[\mu,\beta_i]=0.
\end{align*}
Hence, $(\tilde{\mu},\tilde{\dot{\nu}})\in \ker(D^{Y+\delta\beta;1}_0)$.

Suppose that $(\mu,\dot{\nu})\in {\rm Im}((-\overline{\partial})\oplus (-\overline{\partial}))$. There are $\xi\in H^0(\mathcal{U},\mathcal{A}^{0,0}(\Theta_{M_0}))$ and $\eta\in C^0(\mathcal{U},\mathcal{A}^{0,0}(\Theta_{M_0}))$ such that $\mu_i=-\overline{\partial}\xi_i$, $\dot{\nu}_i=-\overline{\partial}\dot{\eta}_i$ and
$$
\dot{\eta}_j-\dot{\eta}_i+[\xi,Y_{ij}]=0
$$
for $i$, $j\in I$.
Let $\tilde{\xi}=\xi$ and $\tilde{\eta}=\dot{\eta}-[\xi,\beta]$. Then
\begin{align*}
\tilde{\eta}_j-\tilde{\eta}_i+[\tilde{\xi},Y_{ij}+\beta_j-\beta_i]
&=(\dot{\eta}_j-[\xi_j,\beta_j])-(\dot{\eta}_i-[\xi_i,\beta_i])+[\xi,Y_{ij}+\beta_j-\beta_i]=0
\\
(-\overline{\partial})\oplus (-\overline{\partial})(\tilde{\xi},\tilde{\eta})_i
&=(-\overline{\partial}\tilde{\xi}_i,-\overline{\partial} \tilde{\eta}_i)=(-\overline{\partial}\xi_i, -\overline{\partial} \dot{\eta}_i-[-\overline{\partial}\xi_i,\beta_i] )\\
&=(\mu_i, \dot{\nu}_i-[\mu,\beta_i])=\tilde{\mathcal{L}}^{Dol}_{\beta;Y}(\mu,\dot{\nu})_i
\end{align*}
for $i\in I$.
Hence 
$\tilde{\mathcal{L}}^{Dol}_{\beta;Y}(\mu,\dot{\nu})\in {\rm Im}((-\overline{\partial})\oplus (-\overline{\partial}))$. Since $\tilde{\mathcal{L}}^{Dol}_{-\beta;Y+\delta\beta}\circ \tilde{\mathcal{L}}^{Dol}_{\beta;Y}$ is the identity map, $\tilde{\mathcal{L}}^{Dol}_{\beta;Y}$ descends a $\mathbb{C}$-linear isomorphism
$\mathcal{L}^{Dol}_{\beta;Y}\colon \mathbb{T}^{Dol}_Y[\mathcal{U}]\to \mathbb{T}^{Dol}_{Y+\delta\beta}[\mathcal{U}]$.

Next, we check the diagram \eqref{eq:commulte_Dol_double_beta} is commute in the case where $H^1(U_i,\Theta_{M_0})=0$ for $i\in I$.
Let $(\mu,\dot{\nu})\in \ker(D^{Y;1}_0)$, and $\tv{X,\dot{Y}}{Y}=\connectinghomo(\tvD{\mu,\dot{\nu}}{Y})$. Then, there are $\xi$, $\dot{\eta}\in C^0(\mathcal{U},\mathcal{A}^{0,0}(\Theta_{M_0})$ such that
\begin{align*}
(\mu_i,\dot{\nu}_i)&=(-\overline{\partial})\oplus (-\overline{\partial})(\xi,\dot{\eta})_i=(-\overline{\partial}\xi_i,-\overline{\partial}\dot{\eta}_i) \\
X_{ij}&=\xi_j-\xi_i \\
\dot{Y}_{ij}&=\dot{\eta}_j-\dot{\eta}_i+[\xi_i,Y_{ij}]
\end{align*}
for $i$, $j\in I$.
Let $\tilde{\xi}=\xi$ and $\tilde{\eta}=\dot{\eta}-[\xi,\beta]$. Then,
for $\beta\in C^0(\mathcal{U},\Theta_{M_0})$,
\begin{align*}
(-\overline{\partial})\oplus (-\overline{\partial})(\tilde{\xi},\tilde{\eta})_i
&=(-\overline{\partial}\tilde{\xi}_i,-\overline{\partial}\tilde{\eta}_i)
=(\mu_i, \dot{\nu}_i-[\mu_i,\beta_i])=\tilde{\mathcal{L}}^{Dol}_{\beta;Y}(\mu,\dot{\nu})_i \\
\tilde{\xi}_j-\tilde{\xi}_i
&=\xi_j-\xi_i=X_{ij}\\
\tilde{\eta}_j-\tilde{\eta}_i+[\tilde{\xi}_i,Y_{ij}+\beta_j-\beta_i]
&=(\dot{\eta}_j-[\xi_j,\beta_j])-(\dot{\eta}_i-[\xi_i,\beta_i])+
[\xi_i,Y_{ij}+\beta_j-\beta_i]\\
&=\dot{\eta}_j-\dot{\eta}_i+[\xi_i,Y_{ij}]-[\xi_j,\beta_j]+[\xi_i,\beta_j] \\
&=\dot{Y}_{ij}+K'(\beta,X)_{ij}.
\end{align*}
This implies that the connecting homomorphism $\connectinghomo$ sends
the equivalence class of $\tilde{\mathcal{L}}^{Dol}_{\beta;Y}(\mu,\dot{\nu})$ in 
$\mathbb{T}^{Dol}_{Y+\delta \beta}[\mathcal{U}]$
to that of $\tilde{\mathcal{L}}_{\beta;Y}(X,\dot{Y})$ in $\mathbb{T}_{Y+\delta \beta}[\mathcal{U}] $. This means that the diagram \eqref{eq:commulte_Dol_double_beta} is commute.
From \Cref{prop:Dol_vertical_space_isomorphism}, the isomorphism $\mathcal{L}^{Dol}_{\beta;Y}$ preserves the vertical spaces.
\end{proof}

\begin{definition}[Model space and Vertical space for cohomology classes]
\label{def:model_space_dol_independent}
For $\beta\in C^0(\mathcal{U},\Theta_{M_0})$, we identify $\mathbb{T}^{Dol}_Y[\mathcal{U}]$ and $\mathbb{T}^{Dol}_{Y+\delta\beta}[\mathcal{U}]$ by $\mathcal{L}^{Dol}_{\beta;Y}$. 
The identified $\mathbb{C}$-vector space $\mathbb{T}^{Dol}_{[Y]}[\mathcal{U}]$ is called the \emph{Dolbeault type presentation of the model space of the double tangent space at $[Y]\in H^1(M_0,\Theta_{M_0})\cong \teich_{x_0}\teich_g$}. We also define the \emph{vertical space} $\mathbb{T}^{Dol,V}_{[Y]}[\mathcal{U}]$ in the same manner.
\end{definition}

From \Cref{lem:Dol_base_change}, we have
\begin{corollary}
\label{coro:Dol_tangent_connecing_homo}
The connecting homomorphism $\connectinghomo$ induces an isomorphism
$$
\connectinghomo\colon \mathbb{T}^{Dol}_{[Y]}[\mathcal{U}]\to \mathbb{T}_{[Y]}[\mathcal{U}]
$$
which respects the vertical spaces.
\end{corollary}

\subsection{Vertical spaces and Tangent spaces}
Fix $\eta\in C^0(\mathcal{U},\mathcal{A}^{0,0}(\Theta_{M_0}))$ with $\delta\eta=Y$.
From \Cref{prop:equivalence_class} and \Cref{lem:Dol_base_change}, 
\begin{equation}
\label{eq:vertical_isomorphism_cohomology_class_dol}
\verticalincmodel{[Y]}\colon
\mathbb{T}^{Dol, V}_{[Y]}[\mathcal{U}]\ni \tvD{-\overline{\partial}\alpha,\dot{\nu}}{[Y]}\mapsto [\{\dot{\nu}_i-\overline{\partial}[\alpha,\eta_i]\}_{i\in I}]\in T_{x_0}\teich_g
\end{equation}
is a well-defined $\mathbb{C}$-linear mapping which satisfies the following diagram:
\begin{equation}
\label{eq:commute_vertical_TYU_Dol}
\begin{CD}
\mathbb{T}^{Dol, V}_{Y}[\mathcal{U}]
@>{\verticalincmodel{Y}}>>
T_{x_0}\teich_g
\\
@V{\cong}VV @|
\\
\mathbb{T}^{Dol, V}_{[Y]}[\mathcal{U}]
@>>{\verticalincmodel{[Y]}}>
T_{x_0}\teich_g,
\end{CD}
\end{equation}
since $\delta(\eta+\beta)=Y+\delta\beta$ for $\beta \in C^0(\mathcal{U},\Theta_{M_0})$.
The inverse of the map $\verticalincmodel{[Y]}$ is
\begin{equation}
\label{eq:commute_vertical_TYU_Dol2}
\begin{CD}
T_{x_0}\teich_g@>>> \mathbb{T}^{V}_{[Y]}[\mathcal{U}] \\
[\dot{\nu}] @>>> \tvD{0,\dot{\nu}}{[Y]}.
\end{CD}
\end{equation}

\subsection{Dolbeaut type presentation in terms of Beltrami differentials}
\label{subsec:Dolbeaut_type_presentation_Beltrami_differentials}
In this section, we discuss the presentation of the double tangent space in terms of Beltrami differentials.
The construction is almost same as that of the Dolbeaut presentation for cohomology classes.

Let $[Y]\in H^1(\mathcal{U},\Theta_{M_0})$.
For $\eta\in C^0(\mathcal{U},\mathcal{A}^{0,0}(\Theta_{M_0}))$ with $\delta\eta=Y$, we set
\begin{align*}
\mathbb{Z}^{Bel}_\eta[\mathcal{U}]
&=
\{(\mu,\dot{\nu})
\in 
H^0(\mathcal{U},\mathcal{A}^{0,1}(\Theta_{M_0}))
\oplus C^0(\mathcal{U},\mathcal{A}^{0,1}(\Theta_{M_0}))
\mid
\delta(\dot{\nu}+[\mu,\eta])=0\}
\\
\mathbb{B}^{Bel}_\eta[\mathcal{U}]
&=
\{(\xi,\zeta)
\in 
H^0(\mathcal{U},\mathcal{A}^{0,0}(\Theta_{M_0}))
\oplus C^0(\mathcal{U},\mathcal{A}^{0,0}(\Theta_{M_0}))
\mid
\delta(\zeta+[\xi,\eta])=0\},
\end{align*}
where $[\mu,\eta]=\{[\mu,\eta_i]\}_{i\in I}$ and $[\xi,\eta]=\{[\xi,\eta_i]\}_{i\in I}$.
We notice that as subspaces of $H^0(\mathcal{U},\mathcal{A}^{0,1}(\Theta_{M_0}))
\oplus C^0(\mathcal{U},\mathcal{A}^{0,1}(\Theta_{M_0}))$
 and $H^0(\mathcal{U},\mathcal{A}^{0,0}(\Theta_{M_0}))
\oplus C^0(\mathcal{U},\mathcal{A}^{0,0}(\Theta_{M_0}))$,
the above spaces $\mathbb{Z}^{Bel}_{\eta}[\mathcal{U}]$ and $\mathbb{B}^{Bel}_{\eta}[\mathcal{U}]$ coincide with the spaces $\ker(D^{Y;1}_0)$ and $\ker(D^{Y;0}_0)$ defined in \S\ref{subsec:Dol_type_representation}, respectively.

The linear map
$$
(-\overline{\partial})\oplus (-\overline{\partial})
\colon \mathbb{B}^{Bel}_\eta[\mathcal{U}]\to \mathbb{Z}^{Bel}_\eta[\mathcal{U}]
$$
is well-defined.
Indeed,
for $(\xi,\zeta)\in \mathbb{B}^{Bel}_\eta[\mathcal{U}]$,
let
$(\mu,\dot{\nu})=(-\overline{\partial})\oplus (-\overline{\partial})(\xi,\zeta)=(-\overline{\partial}\xi, -\overline{\partial}\zeta)$.
Then,
\begin{align*}
\delta(\dot{\nu}+[\mu,\eta])_{ij}&=(\dot{\nu}_j+[\mu,\eta_j])-(\dot{\nu}_i+[\mu,\eta_i]) \\
&=-\overline{\partial}(\xi_j-\xi_i)-[\overline{\partial}\xi,\eta_j-\eta_i])
=-\overline{\partial}\left((\xi_j-\xi_i)+[\xi,Y_{ij}]\right)
\\
&=-\overline{\partial}\left((\xi_j+[\xi,\eta_j]-(\xi_i+[\xi,\eta_i])\right)=0
\end{align*}
for $i$, $j\in I$.
We define 
\begin{equation}
\label{eq:dolbeaut_presentation_Beltrami}
\mathbb{T}^{Bel}_\eta[\mathcal{U}]
=
\mathbb{Z}^{Bel}_\eta[\mathcal{U}]
/
(-\overline{\partial})\oplus (-\overline{\partial})
(\mathbb{B}^{Bel}_\eta[\mathcal{U}]).
\end{equation}
We denote by $\tvDBel{\mu,\dot{\nu}}{\eta}\in \mathbb{T}^{Bel}_\eta[\mathcal{U}]$ the equivalence class of $(\mu,\dot{\nu})\in \mathbb{Z}^{Bel}_\eta[\mathcal{U}]$.
By definition,
there is a canonical identification
\begin{equation}
\label{eq:canonical_identification_Dol}
\mathbb{T}^{Dol}_Y[\mathcal{U}]
\ni 
\tvD{\mu,\dot{\nu}}{Y}
\mapsto \tvDBel{\mu,\dot{\nu}}{\eta}\in
\mathbb{T}^{Bel}_\eta[\mathcal{U}]
\end{equation}
since $\delta\eta=Y$.


The presentation $\mathbb{T}^{Bel}_\eta[\mathcal{U}]$ depends on the choice of $\eta$ with $\delta\eta=Y$.
Indeed, let $\eta'\in C^0(\mathcal{U},\mathcal{A}^{0,0}(\Theta_{M_0}))$
with $\delta \eta'=Y$. Then, the following diagram is commutative:
$$
\begin{CD}
\mathbb{B}^{Bel}_\eta[\mathcal{U}]
@>{(\xi,\zeta)\mapsto (\xi,\zeta+[\xi,\eta-\eta'])}>>
\mathbb{B}^{Bel}_{\eta'}[\mathcal{U}] \\
@V{(-\overline{\partial})\oplus (-\overline{\partial})}VV @VV{(-\overline{\partial})\oplus (-\overline{\partial})}V
\\
\mathbb{Z}^{Bel}_{\eta}[\mathcal{U}]
@>{(\xi,\zeta)\mapsto (\mu,\dot{\nu}+[\mu,\eta-\eta'])}>>
\mathbb{Z}^{Bel}_{\eta'}[\mathcal{U}].
\end{CD}
$$
In the diagram, we notice that $\eta-\eta'$ defines a global vector field on $M_0$
and $\overline{\partial}(\delta\eta)=\overline{\partial}(\delta\eta')=0$. Therefore,
we obtain a canonical isomorphism
\begin{equation}
\label{eq:canonical_isomorphism_Bel_Dol}
\mathbb{T}^{Bel}_{\eta}[\mathcal{U}]
\ni\tvDBel{\mu,\dot{\nu}}{\eta}
\mapsto
\tvDBel{\mu,\dot{\nu}+[\mu,\eta-\eta']}{\eta'}
\in\mathbb{T}^{Bel}_{\eta'}[\mathcal{U}].
\end{equation}

By applying the above discussion, for $\beta\in C^0(\mathcal{U},\Theta_{M_0})$, we can check that the linear map
$\tilde{\mathcal{L}}^{Dol}_{\beta;Y}$ on $C^0(\mathcal{U},\mathcal{A}^{0,1}_{\Theta})^{\oplus 2}$ defined in \eqref{eq:cohomologus_Dol} sends $\mathbb{Z}^{Bel}_{\eta}[\mathcal{U}]$
onto $\mathbb{Z}^{Bel}_{\eta+\beta}[\mathcal{U}]$, and descend to an isomorphism
\begin{equation}
\label{eq:canonical_isomorphism_Bel_Dol2}
\mathbb{T}^{Bel}_{\eta}[\mathcal{U}]\ni
\tvDBel{\mu,\dot{\nu}}{\eta}
\mapsto \tvDBel{\mu,\dot{\nu}-[\mu,\beta]}{\eta+\beta}
\in \mathbb{T}^{Bel}_{\eta+\beta}[\mathcal{U}].
\end{equation}
\begin{definition}[Dolbeaut presentation in terms of Beltrami differentials]
\label{def:dol_pre_Bel}
Let $[\nu]\in T_{x_0}\teich_g$. We define the \emph{Dolbeaut presentation of  the model space in terms of Beltrami differentials} $\mathbb{T}^{Dol}_{[\nu]}[\mathcal{U}]$ as the isomorphism class of the vector spaces $\{\mathbb{T}_{\eta}[\mathcal[\mathcal{U}]\}_\eta$ 
via \eqref{eq:canonical_isomorphism_Bel_Dol} and \eqref{eq:canonical_isomorphism_Bel_Dol2}, where $\eta$ runs all cochain $\eta\in C^0(\mathcal{U},\mathcal{A}^{0,0}(\Theta_{M_0}))$ with $[-\overline{\partial}\eta]=[\nu]$.
For $\tvDBel{\mu,\dot{\nu}}{\eta}\in \mathbb{T}^{Bel}_{\eta}[\mathcal{U}]$,
we denote by $\tvDBel{\mu,\dot{\nu}}{[\nu]}\in \mathbb{T}^{Bel}_{[\nu]}[\mathcal{U}]$
the equivalence class of $\tvDBel{\mu,\dot{\nu}}{\eta}$.
\end{definition}

Let $[Y]\in H^0(\mathcal{U},\Theta_{M_0})$ with $\mathscr{T}_{x_0}([Y])=[\nu]$.
We notice the following.
\begin{itemize}
\item
Two spaces $\mathbb{T}^{Dol}_{[Y]}[\mathcal{U}]$ and $\mathbb{T}^{Dol}_{[\nu]}[\mathcal{U}]$ are naturally isomorphic by  a canonical identification \eqref{eq:canonical_identification_Dol}.
\item
The connecting homomorphism $\connectinghomo\colon \mathbb{T}_{[Y]}[\mathcal{U}]\to \mathbb{T}^{Dol}_{[Y]}[\mathcal{U}]$ discussed in \Cref{prop:Dol_vertical_space_isomorphism} is naturally thought of as an isomorphism
$$
\connectinghomo \colon \mathbb{T}^{Bel}_{\eta}[\mathcal{U}]\to \mathbb{T}_{Y}[\mathcal{U}]
$$
when $\eta\in C^0(\mathcal{U},\mathcal{A}^{0,0}(\Theta_{M_0}))$ satisfies $\delta\eta=Y$.
\item
The vertical space in $\mathbb{T}^{Bel}_{[\nu]}[\mathcal{U}]$ is also described by the same argument as that in \S\ref{subsec:vertical_space_Dol} and \eqref{eq:vertical_isomorphism_cohomology_class_dol}. 
\end{itemize}


\section{Model of $T^*_{[Y]}T\teich_g$ and Model of the pairing}
\label{sec:Model_Tstar}


Let $\mathcal{U}=\{U_i\}_{i\in I}$ be a locally finite covering on $M_0$.
Let $x_0=(M_0,f_0)\in \teich_g$ and $Y\in Z^1(\mathcal{U},\Theta_{M_0})$.
We define the linear maps
\begin{align*}
D^{Y,\dagger}_1\colon C^0(\mathcal{U},\Omega_{M_0}^{\otimes 2})^{\oplus 2}\to C^1(\mathcal{U},\Omega_{M_0}^{\otimes 2})^{\oplus 2}
\end{align*}
by
$$
D^{Y,\dagger}_1(\psi,q)=(\delta \psi-L_Y(q),\delta q)
$$
where $L_Y(q)_{ij}=L_{Y_{ij}}(q_0)$ for $i$, $j\in I$.
We call
$$
\mathbb{T}^\dagger_{Y}[\mathcal{U}]=\ker(D^{Y,\dagger}_1)
$$
the \emph{model space of the cotangent space to the tangent bundle for a cocycle $Y\in Z^1(\mathcal{U},\Theta_{M_0})$}. Notice that the inclusion on the first coordinate
$$
T^*_{x_0}\teich_g=\mathcal{Q}_{M_0}=H^0(\mathcal{U},\Omega_{M_0}^{\otimes 2})\ni \psi\mapsto
(\psi,0)\in
\mathbb{T}^\dagger_{Y}[\mathcal{U}]
$$
is well-defined. We call this map the \emph{horizontal inclusion}. The projection, which we call the \emph{vertical projection},
on the second coordinate
$$
\mathbb{T}^\dagger_{Y}[\mathcal{U}]\ni (\psi,q)\mapsto q\in \mathcal{Q}_{M_0}=T^*_{x_0}\teich_g
$$
leads the identification between the cokernel of the horizontal inclusion and $\mathcal{Q}_{M_0}$. Thus, we have the following exact sequence
\begin{equation}
\label{eq:cotangent_space_tangent_bundle_cycle}
\minCDarrowwidth10pt
\begin{CD}
0
@>>>
H^0(\mathcal{U},\Omega_{M_0}^{\otimes 2})
@>>>
\mathbb{T}^\dagger_{Y}[\mathcal{U}] @>>>
H^0(\mathcal{U},\Omega_{M_0}^{\otimes 2})
@>>> 0 \\
@.
\psi
@>>> (\psi,0) \\
@.@. (\psi,q) @>>> q
\end{CD}
\end{equation}
In particular, $\dim_{\mathbb{C}}\mathbb{T}^\dagger_{Y}[\mathcal{U}] = 6g-6$ when $H^1(U_i,\Theta_{M_0})=0$ for $i\in I$ (cf. \Cref{thm:Trivialization_pairing_TTT_g} later).

To define the \emph{model space} $\mathbb{T}^\dagger_{[Y]}[\mathcal{U}]$ \emph{of the cotangent space to the tangent bundle for $[Y]\in H^1(\mathcal{U},\Theta_{M_0})$},
we identify
two spaces $\mathbb{T}^\dagger_{Y}[\mathcal{U}]$ and $\mathbb{T}^\dagger_{Y+\delta\beta}[\mathcal{U}]$ for $\beta\in C^0(\mathcal{U},\Theta_{M_0})$ via the $\mathbb{C}$-linear isomorphism
\begin{equation}
\label{eq:equivalence_cotangent_to_tangent}
\mathcal{L}^\dagger_{\beta;Y}\colon
\mathbb{T}^\dagger_{Y}[\mathcal{U}]\ni (\psi,q)
\to
(\psi+L_\beta(q),q)
\in
\mathbb{T}^\dagger_{Y+\delta\beta}[\mathcal{U}].
\end{equation}
We denote by $\tvdag{\psi,q}{[Y]}\in \mathbb{T}^\dagger_{[Y]}[\mathcal{U}]$ the corresponding element to $(\psi,q)\in  \mathbb{T}^\dagger_{Y}[\mathcal{U}]$. From the definition, $\mathcal{L}^\dagger_{\beta;Y}$ commutes the horizontal inclusion. Thus, the diagram \eqref{eq:cotangent_space_tangent_bundle_cycle} leads the following exact sequence:
\begin{equation}
\label{eq:cotangent_space_tangent_bundle_cohom_class}
\minCDarrowwidth10pt
\begin{CD}
0
@>>>
H^0(\mathcal{U},\Omega_{M_0}^{\otimes 2})
@>>>
\mathbb{T}^\dagger_{[Y]}[\mathcal{U}] @>>>
H^0(\mathcal{U},\Omega_{M_0}^{\otimes 2})
@>>> 0 \\
@.
\psi
@>>> \tvdag{\psi,0}{[Y]}\ \\
@.@. \tvdag{\psi,q}{[Y]}\ @>>> q.
\end{CD}
\end{equation}

\subsection{Model of Pairing between $TT\teich_g$ and $T^*T\teich_g$}
\label{subsec:model_pairing_TT_TstarT}
Let $\mathcal{U}=\{U_i\}_{i\in I}$ be a locally finite covering.
Let $Y\in Z^1(\mathcal{U},\Theta_{M_0})$.
To define the \emph{pairing} between $\mathbb{T}_{Y}[\mathcal{U}]$ and $\mathbb{T}^{\dagger}_{Y}[\mathcal{U}]$, we start with the following lemma.
\begin{lemma}
\label{lem:pairing_tangent_cotangent}
For $(X,\dot{Y})\in \ker(D^Y_1)$, $(\psi, q)\in \ker(D^{Y,\dagger}_1)$, and $\xi\in C^0(\mathcal{U},\mathcal{A}^{0,0}(\Theta_{M_0}))$ with $\delta\xi=X$,
we define
\begin{align*}
\omega_{ij}
&=X_{ij}\psi_i-\dot{Y}_{ji}q+L_{Y_{ij}}(\xi_jq)
\\
&={\color{black}X_{ij}\psi_i+(\dot{Y}_{ij}+[X_{ij},Y_{ij}])q+L_{Y_{ij}}(\xi_jq)}.
\end{align*}
Then, $\{\omega_{ij}\}_{i,j\in I}\in C^1(\mathcal{U},\mathcal{A}^{0,0}(\Omega_{M_0}))$ satisfies the following
\begin{itemize}
\item[(1)]
$\omega=\{\omega_{ij}\}_{i,j\in I}\in Z^1(\mathcal{U},\mathcal{A}^{0,0}(\Omega_{M_0}))$.
\item[(2)]
Let $A=\{A_i\}_{i\in I}\in C^0(\mathcal{U},\mathcal{A}^{0,0}(\Omega_{M_0}))$ with $\delta A=\omega$, and $\eta\in C^0(\mathcal{U},\mathcal{A}^{0,0}(\Theta_{M_0}))$ with $\delta\eta=Y$.
Then, $\overline{\partial}A-L_\eta(d(\xi q))=\{(A_i)_{\overline{z}}-L_{\eta_i}((\xi_i)_{\overline{z}}q)\}_{i\in I}$ defines an area element on $M_0$.
\item[(3)]
Under the notation in (2), the integral
\begin{align}
\label{eq:pairing_definition_TM_TstarM}
\tilde{\mathcal{P}}_{\mathbb{TT}}((X,\dot{Y}), (\psi, q))
&=
-\dfrac{1}{2i}\iint_{M_0}((A_i)_{\overline{z}}-L_{\eta_i}((\xi_i)_{\overline{z}}q))d\overline{z}\wedge dz
\end{align}
is independent of the choice of the representatives $Y$ of $[Y]$, $(X,\dot{Y})$ in $\tv{X,\dot{Y}}{Y}$ and $(\psi, q)\in \ker(D^{Y,\dagger}_1)$ in $\tvdag{\psi,q}{[Y]}$.
\end{itemize}
\end{lemma}
\begin{proof}
(1)
\quad
From the argument in the proof of \Cref{prop:vertical_space} and \eqref{eq:Lie-derivative_1},
\begin{align*}
(\delta \omega)_{ijk}
&=
(X_{ij}\psi_i-\dot{Y}_{ji}q+L_{Y_{ij}}(\xi_jq))
+(X_{jk}\psi_j-\dot{Y}_{kj}q+L_{Y_{jk}}(\xi_jq))
\\
&\quad +(X_{ki}\psi_k-\dot{Y}_{ik}q+L_{Y_{ki}}(\xi_iq))
\\
&=\xi_{i}(\psi_k-\psi_i)+\xi_{j}(\psi_i-\psi_j)+\xi_k(\psi_j-\psi_k)\\
&\quad-([\xi_j,Y_{ji}]+[\xi_k,Y_{kj}]+[\xi_i,Y_{ik}])q \\
&\qquad +L_{Y_{ij}}(\xi_jq)+L_{Y_{jk}}(\xi_jq)+L_{Y_{ki}}(\xi_kq)
\\
&=\xi_{i}L_{Y_{ik}}(q)+\xi_{j}L_{Y_{ji}}(q)+\xi_kL_{Y_{kj}}(q)\\
&\quad-([\xi_j,Y_{ji}]+[\xi_k,Y_{kj}]+[\xi_i,Y_{ik}])q \\
&\qquad -L_{Y_{ji}}(\xi_jq)-L_{Y_{kj}}(\xi_kq)-L_{Y_{ik}}(\xi_iq)
=0
\end{align*}
for $i$, $j$, $k\in I$.

\medskip
(2)\quad
The existence of such an $A=\{A_i\}_{i\in I}$ follows from the fact that the sheaf $\mathcal{A}^{0,0}(\Omega_{M_0})$ is fine (e.g. \cite[\S3.4]{MR815922}).
Since $(\xi_j)_{\overline{z}}=(\xi_i)_{\overline{z}}$ on $U_i\cap U_j$,
for $z\in z_j(U_i\cap U_j)$,
\begin{align*}
&(A_j)_{\overline{z}}(z)
-
(A_i)_{\overline{z}}(z_{ij}(z))\left|\dfrac{dz_{ij}}{dz}(z)\right|^2
=
\left(
A_j(z)-A_i(z_{ij}(z))\dfrac{dz_{ij}}{dz}(z)
\right)_{\overline{z}}
\\
&=\left(
\left(
X_{ij}(z)\psi_i(z_{ij}(z))-\dot{Y}_{ji}(z)q_i(z_{ij}(z))+L_{Y_{ij}}(\xi_jq_j)(z_{ij}(z))
\right)\dfrac{dz_{ij}}{dz}(z)
)
\right)_{\overline{z}}
\\
&=L_{Y_{ij}}((\xi_j)_{\overline{z}}q_j)(z_{ij}(z))\left|\dfrac{dz_{ij}}{dz}(z)\right|^2
\\
&=L_{\eta_j}((\xi_j)_{\overline{z}}q_j)(z)
-L_{\eta_i}((\xi_i)_{\overline{z}}q_i)(z_{ij}(z)))\left|\dfrac{dz_{ij}}{dz}(z)\right|^2,
\end{align*}
where $q=\{q_i\}_{i\in I}\in H^0(\mathcal{U},\Omega_{M_0}^{\otimes 2})=\mathcal{Q}_{x_0}$.
This implies what we wanted.

\medskip
(3)\quad
We will show that the integral is independent of the choices of 
\begin{itemize}
\item[(i)] $A\in C^0(\mathcal{U},\mathcal{A}^{0,0}(\Omega_{M_0}))$ with $\delta A=\omega$;
\item[(ii)] $\eta\in C^0(\mathcal{U},\mathcal{A}^{0,0}(\Theta_{M_0}))$ with $\delta \eta= Y$;
\item[(iii)] $\xi\in C^0(\mathcal{U},\mathcal{A}^{0,0}(\Theta_{M_0}))$ with $\xi = X$;
\item[(iv)] representations $Y\in Z^1(\mathcal{U},\Theta_{M_0}))$ in $[Y]$; and
\item[(v)] representations $(X,\dot{Y})\in \ker(D^Y_1)$ of $\tv{X,\dot{Y}}{Y}$.
\end{itemize}

(i) {\bf The choice of $A$:}
Let $A'=\{A'_i\}_{i\in I}\in C^0(\mathcal{U},\mathcal{A}^{0,0}(\Omega_{M_0}))$ with $\delta A'=\omega$.
Then, $A-A'=(A(z)-A'(z))dz$ is a smooth $1$-form on $M_0$ and
\begin{align*}
&\iint_{M_0}\left((A_i)_{\overline{z}}-L_{\eta_i}((\xi_i)_{\overline{z}} q)\right)d\overline{z}\wedge dz-
\iint_{M_0}\left((A'_i)_{\overline{z}}-L_{\eta_i}((\xi_i)_{\overline{z}} q)\right)d\overline{z}\wedge dz \\
&=\iint_{M_0}\overline{\partial}(A_i-A'_i)d\overline{z}\wedge dz
=\iint_{M_0}d(A_i dz-A'_idz)=0
\end{align*}
from the Stokes theorem.

\medskip
\noindent
(ii) {\bf The choice of $\eta$:}
Let $\eta'\in C^0(\mathcal{U},\mathcal{A}^{0,0}(\Theta_{M_0}))$ with
$\delta \eta'=Y$. Then $H=\eta'-\eta\in H^0(\mathcal{U},\mathcal{A}^{0,0}(\Theta_{M_0}))$. Since $\mu=\overline{\partial}\xi\in H^0(\mathcal{U},\mathcal{A}^{0,1}(\Theta_{M_0}))$ and $H\mu q=H(z)\mu(z) q(z)d\overline{z}$ is a global $(0,1)$-form on $M_0$,
\begin{align*}
&\iint_{M_0}\left((A_i)_{\overline{z}}-L_{\eta_i}(\mu q)\right)d\overline{z}\wedge dz-
\iint_{M_0}\left((A_i)_{\overline{z}}-L_{\eta_i'}(\mu q)\right)d\overline{z}\wedge dz \\
&=\iint_{M_0} L_{H}(\mu q)d\overline{z}\wedge dz=\iint_{M_0} (H\mu q)_z(z)d\overline{z}\wedge dz
\\
&=-\iint_{M_0} d\left((H\mu q) d\overline{z}\right)=0.
\end{align*}

\medskip
\noindent
(iii) {\bf The choice of $\xi$:}
Let $\xi'\in C^0(\mathcal{U},\mathcal{A}^{0,0}(\Theta_{M_0}))$ with $\delta \xi'=X$. Take $A'=\{A'_i\}_{i\in I}\in C^0(\mathcal{U},\mathcal{A}^{0,0}(\Omega_{M_0}))$ satisfying
$$
A'_jdz-A'_idz=(X_{ij}\psi_i-\dot{Y}_{ji}q+L_{Y_{ij}}(\xi'_jq))dz
$$
for $i$, $j\in I$. Then,  $\Xi=\xi'-\xi$ defines a smooth vector field on $M_0$.
Let $\nu=\overline{\partial}\eta\in H^0(\mathcal{U},\mathcal{A}^{0,1}(\Theta_{M_0}))$. Define $B=\{B_idz\}_{i\in I}\in C^0(\mathcal{U},\mathcal{A}^{0,0}(\Omega_{M_0}))$ by
$$
B_idz=(A'_i-A_i-L_{\eta_i}(\Xi q))dz.
$$
Then, 
\begin{align*}
(\delta B)_{ij}
&=B_jdz-B_idz\\
&=L_{Y_{ij}}(\xi'_jq)dz-L_{Y_{ij}}(\xi_jq)dz-(L_{\eta_j}(\Xi q)-L_{\eta_i}(\Xi q))dz \\
&=L_{Y_{ij}}(\Xi q)dz-L_{\eta_j-\eta_i}(\Xi q)dz=0.
\end{align*}
Hence, $B$ defines a global $1$-form on $M_0$. Notice that
\begin{align*}
d(L_{\eta_i}(\Xi q)dz)
&=d((\eta_i\Xi q)_zdz)
=
((\eta_i\Xi q)_{z\overline{z}})d\overline{z}\wedge dz
\\
&=
(-(\nu\Xi q)_{z}+(\eta_i \Xi_{\overline{z}}q)_z)d\overline{z}\wedge dz\\
&=
(-(\nu\Xi q)_{z}+L_{\eta_i}(\Xi_{\overline{z}}q))d\overline{z}\wedge dz.
\end{align*}
Therefore, we obtain
\begin{align*}
&d (B_idz)=
d((A'_i-A_i-L_{\eta_i}(\Xi q))dz) \\
&=((A'_i)_{\overline{z}}-(A_i)_{\overline{z}}
+(\nu\Xi q)_{z}-L_{\eta_i}(\Xi_{\overline{z}}q)
)
d\overline{z}\wedge dz \\
&=
\left((A'_i)_{\overline{z}}
-L_{\eta_i}((\xi'_i)_{\overline{z}}q)
\right)d\overline{z}\wedge dz
-\left((A_i)_{\overline{z}}
-L_{\eta_i}((\xi_i)_{\overline{z}}q)
\right)d\overline{z}\wedge dz+(\nu\Xi q)_{z}d\overline{z}\wedge dz
\\
&=
\left((A'_i)_{\overline{z}}
-L_{\eta_i}((\xi'_i)_{\overline{z}}q)
\right)d\overline{z}\wedge dz
-\left((A_i)_{\overline{z}}
-L_{\eta_i}((\xi_i)_{\overline{z}}q)
\right)d\overline{z}\wedge dz-d((\nu\Xi q)d\overline{z})
\end{align*}
for $i\in I$.
This means that
\begin{align*}
&
\iint_{M_0}
\left(
(A'_i)_{\overline{z}}
-L_{\eta_i}((\xi'_i)_{\overline{z}}q)
\right)d\overline{z}\wedge dz
-
\iint_{M_0}
\left(
(A_i)_{\overline{z}}
-L_{\eta_i}((\xi_i)_{\overline{z}}q)
\right)d\overline{z}\wedge dz \\
&=\iint_{M_0}d(B_idz+(\nu \Xi q)d\overline{z})=0.
\end{align*}

\medskip
\noindent
{\bf (iv) The choice of $Y$ in $[Y]$:}
Take $A\in C^0(\mathcal{U},\mathcal{A}^{0,0}(\Theta_{M_0}))$ and the cocycle $\omega$ for $Y$ as the statement.
Let $\beta\in C^0(\mathcal{U},\Theta_{M_0})$. From \Cref{prop:isomorphism_tile_L} and \eqref{eq:equivalence_cotangent_to_tangent}, for $Y+\delta \beta$, 
the representatives of $\tv{X,\dot{Y}}{[Y]}$ and $\tvdag{\psi, q}{[Y]}$ are $(X,\dot{Y}+K'(\beta,X))$ and $(\psi+L_\beta(q),q)$ respectively. Hence, from \eqref{eq:Lie-derivative_01} and \eqref{eq:Lie-derivative_1}, the cocycle $\omega'$ which we consider here should be
\begin{align*}
\omega'_{ij}dz
&=X_{ij}(\psi_i+L_{\beta_i}(q))dz
-(\dot{Y}_{ji}+[\beta_i,X_{ji}])qdz
+L_{Y_{ij}+\beta_j-\beta_i}(\xi_jq)dz
\\
&=(X_{ij}\psi_i-\dot{Y}_{ji}q+L_{Y_{ij}}(\xi_jq)+X_{ij}L_{\beta_i}(q)
-[\beta_i,X_{ji}]q+L_{\beta_j-\beta_i}(\xi_jq))dz
\\
&=(X_{ij}\psi_i-\dot{Y}_{ji}q+L_{Y_{ij}}(\xi_jq)
+
L_{\beta_j}(\xi_jq)-L_{\beta_i}(\xi_iq))dz.
\end{align*}
From (i) and (ii), we may take $A'=\{A'_idz\}_{i\in I}=\{A_idz+L_{\beta_i}(\xi_iq)dz\}_{i\in I}$. In this case, the integrand becomes
\begin{align*}
((A'_i)_{\overline{z}}-L_{\eta_i+\beta_i}((\xi_i)_{\overline{z}}q))d\overline{z}\wedge dz
&=((A_i+L_{\beta_i}(\xi_iq))_{\overline{z}}-L_{\eta_i+\beta_i}((\xi_i)_{\overline{z}}q))d\overline{z}\wedge dz
\\
&=((A_i)_{\overline{z}}-L_{\eta_i}((\xi_i)_{\overline{z}}q))d\overline{z}\wedge dz,
\end{align*}
which implies the integral does not change.

\medskip
\noindent
{\bf (v) The choice of $(X,\dot{Y})$ in $\ker(D^Y_1)$:}
Take $A\in C^0(\mathcal{U},\mathcal{A}^{0,0}(\Theta_{M_0}))$ and the cocycle $\omega$ for $Y$ as the statement.
Let $\alpha$, $\beta\in C^0(\mathcal{U},\Theta_{M_0})$. We consider $(X',\dot{Y}')=(X,\dot{Y})+D_0^Y(\alpha,\beta)$. In this case, the cocycle $\omega'$ which we consider here is
\begin{align*}
\omega'_{ij}dz&=
(X_{ij}+\alpha_j-\alpha_i)\psi_idz
-(\dot{Y}_{ji}+\beta_i-\beta_j+[\alpha_j,Y_{ji}])qdz
+L_{Y_{ij}}((\xi_j+\alpha_j)q)dz\\
&=\left(X_{ij}\psi_i-\dot{Y}_{ji}q+L_{Y_{ij}}(\xi_jq)
+(\alpha_j\psi_j-\alpha_i\psi_i+\beta_iq-\beta_jq
\right. \\
&\quad 
\left.
-\alpha_jL_{Y_{ij}}(q)-[\alpha_j,Y_{ij}]qdz+L_{Y_{ij}}(\alpha_jq)
\right)dz
\\
&=\left(X_{ij}\psi_i-\dot{Y}_{ji}q+L_{Y_{ij}}(\xi_jq)\right)dz
+(\alpha_j\psi_j-\beta_jq)dz-(\alpha_i\psi_i-\beta_iq)dz
\end{align*}
from \eqref{eq:Lie-derivative_1}. Therefore, from (i) and (ii), we may take $A'=\{A'_i\}_{i\in I}$ defined by
$$
A_i'dz=A_idz+(\alpha_i\psi_i-\beta_iq)dz
$$
for $i\in I$. Then, $\delta A'=\omega'$ and
the integrand satisfies
\begin{align*}
((A'_i)_{\overline{z}}-L_{\eta_i}((\xi_j+\alpha_j)_{\overline{z}}q))d\overline{z}\wedge dz
=(A_i-L_{\eta_i}((\xi_i)_{\overline{z}}q)d\overline{z}\wedge dz
\end{align*}
and the integral does not change.
\end{proof}

\begin{definition}[Model of the pairing]
\label{def:model_pairing}
Let $[Y]\in H^0(\mathcal{U},\Theta_{M_0})$.
The \emph{model of the pairing}
$$
\mathcal{P}_{\mathbb{TT}}\colon \mathbb{T}_{[Y]}[\mathcal{U}]\oplus  \mathbb{T}^\dagger_{[Y]}[\mathcal{U}]\to \mathbb{C}
$$
is defined by 
$$
\mathcal{P}_{\mathbb{TT}}\left(
\tv{X,\dot{Y}}{[Y]},\tvdag{\psi, q}{[Y]}
\right)
=
\tilde{\mathcal{P}}_{\mathbb{TT}}\left((X,\dot{Y}), (\psi, q)\right).
$$
Notice in the calculation that we first fix a representative $Y$ in $[Y]$, and take representatives $(X,\dot{Y})\in \ker(D^Y_1)$ of $\tv{X,\dot{Y}}{[Y]}$ and $(\psi,q)\in \ker(D^{Y,\dagger}_1)$ of $\tvdag{\psi, q}{[Y]}$.
\end{definition}

We claim the non-degeneracy of the model of the pairing. It can be also deduced from the presentation of the model of the pairing under the trivialization (cf. \Cref{thm:Trivialization_pairing_TTT_g}).

\begin{proposition}
\label{prop:non-degenerate}
When the locally finite covering $\mathcal{U}=\{U_i\}_{i\in I}$ satisfies $H^1(U_i,\Theta_{M_0})=0$ for $i\in I$,
the model of the pairing $\mathcal{P}_{\mathbb{TT}}$ is non-degenerate.
\end{proposition}

\begin{proof}
Let $\tv{X,\dot{Y}}{[Y]}\in \mathbb{T}_{[Y]}[\mathcal{U}]$. Suppose
$$
\mathcal{P}_{\mathbb{TT}}\left(
\tv{X,\dot{Y}}{[Y]},\tvdag{\psi, q}{[Y]}
\right)
=0
$$
for all $\tvdag{\psi, q}{[Y]}\in \mathbb{T}_{[Y]}^\dagger[\mathcal{U}]$. 
We will conclude $\tv{X,\dot{Y}}{[Y]}=0$.

Let $(X,\dot{Y})\in \ker(D^Y_1)$ be a representative of $\tv{X,\dot{Y}}{[Y]}$ and a representative $(\psi,q)\in \ker(D^{Y,\dagger}_1)$ of $\tvdag{\psi,q}{[Y]}$. We take $\xi$, $\eta$, $A$, and $\omega$ as \Cref{lem:pairing_tangent_cotangent}.

When $q=0$, from the definition, $\psi=\{\psi_i\}_{i\in I}\in H^1(\mathcal{U},\Omega_{M_0}^{\otimes 2})$ and
\begin{equation}
\label{eq:non-degnerate_0}
\omega_{ij}=X_{ij}\psi
\end{equation}
for $i$, $j\in I$.
This means that 
\begin{align*}
0&=\mathcal{P}_{\mathbb{TT}}\left(
\tv{X,\dot{Y}}{[Y]},\tvdag{\psi, 0}{[Y]}
\right)
=-\dfrac{1}{2i}\iint_{M_0}(A_i)_{\overline{z}}(z)d\overline{z}\wedge dz
\\
&=-\pi{\rm Res}([\{X_{ij}\psi\}_{i,j\in I}])
\end{align*}
(cf. \S\ref{subsec:residue}). Since the residue defines a non-degenerate pairing,
we have $[X]=0$ in $H^0(\mathcal{U},\Theta_{M_0})$. Therefore, there is $\alpha\in C^0(\mathcal{U},\Theta_{M_0})$ with $\delta \alpha=X$. In this case,
we may assume $\xi=\alpha$ from (3) of \Cref{lem:pairing_tangent_cotangent}. From \eqref{eq:Lie-derivative_1} and \eqref{eq:D_1},
\begin{align}
\omega_{ij}
&=(\alpha_j-\alpha_i)\psi_i-\dot{Y}_{ji}q+L_{Y_{ij}}(\alpha_jq)
\label{eq:non-degnerate_1}
\\
&=(\alpha_j-\alpha_i)\psi_i-\dot{Y}_{ji}q+
\alpha_jL_{Y_{ij}}(q)-[\alpha_j,Y_{ij}]q
\nonumber
\\
&{\color{black}=\alpha_j(\psi_i+L_{Y_{ij}}(q))-\alpha_i\psi_i-(\dot{Y}_{ji}-[\alpha_j,Y_{ji}])q}
\nonumber
\\
&=\alpha_j\psi_j-\alpha_i\psi_i-(\dot{Y}_{ji}-[\alpha_j,Y_{ji}])q
\nonumber
\\
&=\alpha_j\psi_j-\alpha_i\psi_i+(\dot{Y}_{ij}-[\alpha_i,Y_{ij}])q.
\nonumber
\end{align}
Notice from \Cref{prop:vertical_space}, $\dot{Y}-K(\alpha,Y)=\{\dot{Y}_{ij}-[\alpha_i,Y_{ij}]\}_{i,j\in I}\in Z^1(\mathcal{U},\Theta_{M_0})$. Therefore,
in this case, the pairing satisfies
\begin{align*}
0&=\mathcal{P}_{\mathbb{TT}}\left(
\tv{\delta \alpha,\dot{Y}}{[Y]},\tvdag{\psi, q}{[Y]}
\right)
=-\dfrac{1}{2i}\iint_{M_0}(A_i)_{\overline{z}}(z)d\overline{z}\wedge dz \\
&=-\pi{\rm Res}([\{(\dot{Y}_{ij}-[\alpha_i,Y_{ij}])q\}_{i,j\in I}]).
\end{align*}
Since $(\psi,q)$ is taken arbitrary in $\ker(D^{Y,\dagger}_1)$, we have
$\dot{Y}-K(\alpha,Y)=\delta\beta$ for some $\beta\in C^0(\mathcal{U},\Theta_{M_0})$. Thus, we obtain 
$$
(X,\dot{Y})=(\delta \alpha, \delta \beta+K(\alpha,Y))=D^Y_0(\alpha,\beta),
$$
and $\tv{X,\dot{Y}}{[Y]}=0$.

Next, let $\tvdag{\psi,q}{[Y]}\in \mathbb{T}^\dagger_{[Y]}[\mathcal{U}]$.
Suppose that 
$$
\mathcal{P}_{\mathbb{TT}}\left(
\tv{X,\dot{Y}}{[Y]},\tvdag{\psi, q}{[Y]}
\right)
=0
$$
for all $\tv{X, \dot{Y}}{[Y]}\in \mathbb{T}_{[Y]}[\mathcal{U}]$. 

As the above discussion, we fix $(X,\dot{Y})\in \ker(D^Y_1)$ be a representative of $\tv{X,\dot{Y}}{[Y]}$ and a representative $(\psi,q)\in \ker(D^{Y,\dagger}_1)$ of $\tvdag{\psi,q}{[Y]}$, and take $\xi$, $\eta$, $A$, and $\omega$ as \Cref{lem:pairing_tangent_cotangent}.

Suppose first that $[X]=0$. Then, there is $\alpha\in C^0(\mathcal{U},\Theta_{M_0})$ such that $\delta\alpha=X$. From the discussion around \eqref{eq:non-degnerate_1}, the cocycle $\omega$ satisfies
\begin{align*}
\omega_{ij}
&=\alpha_j\psi_j-\alpha_i\psi_i+(\dot{Y}_{ij}-[\alpha_i,Y_{ij}])q.
\end{align*}
Hence, 
$$
0=\mathcal{P}_{\mathbb{TT}}\left(
\tv{\delta \alpha,\dot{Y}}{[Y]},\tvdag{\psi, q}{[Y]}
\right)
=-\pi{\rm Res}([(\dot{Y}_{ij}-[\alpha_i,Y_{ij}])q]).
$$
Since $H^1(U_i,\Theta_{M_0})=0$ for $i\in I$, from \Cref{prop:vertical_space},
the totality of $[(\dot{Y}_{ij}-[\alpha_i,Y_{ij}]$ fills $H^1(\mathcal{U},\Theta_{M_0})=H^1(M_0,\Theta_{M_0})$. Therefore, the non-degeneracy of the residue implies that $q=0$. 
Thus, from the discussion around \eqref{eq:non-degnerate_0}, $\psi$ satisfies
$$
0=\mathcal{P}_{\mathbb{TT}}\left(
\tv{\delta \alpha,\dot{Y}}{[Y]},\tvdag{\psi, 0}{[Y]}
\right)
=-\pi{\rm Res}([\{X_{ij}\psi_i\}_{i,j\in I}]).
$$
From the non-degeneracy of the residue again, we obtain $\psi=0$. Thus, we conclude $\tvdag{\psi,q}{[Y]}=0$.
\end{proof}
%
%
%
%
%
%

\subsection{Dolbeaut presentation and Pairing}
We claim
\begin{lemma}
\label{lem:dolbeaut_presention_pairing}
Let $\mathcal{U}=\{U_i\}_{i\in I}$ be a locally finite covering of $M_0$ with $H^1(U_i,\Theta_{M_0})=0$ for $i\in I$.
Let $(X,\dot{Y})\in \ker(D^Y_1)$ and $(\psi,q)\in \ker(D^{Y,\dagger}_1)$.
Let $\xi=\{\xi_i\}_{i\in I}$, $\dot{\eta}=\{\dot{\eta}_i\}_{i\in I}\in C^0(\mathcal{U},\mathcal{A}^{0,0}(\Theta_{M_0}))$ with $\delta\xi=X$ and $\dot{Y}=\delta\dot{\eta}+K(\xi,Y)$. Then
$$
(\xi_j\psi_j+\dot{\eta}_jq)-
(\xi_i\psi_i+\dot{\eta}_iq)
=X_{ij}\psi_i-\dot{Y}_{ji}q+L_{Y_{ij}}(\xi_jq)
$$
\end{lemma}

\begin{proof}
Indeed, from 
\eqref{eq:Lie-derivative_01} and \eqref{eq:Lie-derivative_1},
\begin{align*}
&X_{ij}\psi_i-\dot{Y}_{ji}q+L_{Y_{ij}}(\xi_jq)
\\
&=(\xi_j-\xi_i)\psi_i-(\dot{\eta}_i-\dot{\eta}_j+[\xi_j,Y_{ji}])q+L_{Y_{ij}}(\xi_jq) \\
&=\xi_j\psi_i-\xi_i\psi_i-\dot{\eta}_iq+\dot{\eta}_jq-[\xi_j,Y_{ji}]q+L_{Y_{ij}}(\xi_jq) \\
&=\xi_j(\psi_j-L_{Y_{ij}}(q))-\xi_i\psi_i-\dot{\eta}_iq+\dot{\eta}_jq+[\xi_j,Y_{ij}]q+L_{Y_{ij}}(\xi_jq) \\
&=(\xi_j\psi_j+\dot{\eta}_jq)-(\xi_i\psi_i+\dot{\eta}_iq)
{\color{black}-\xi_jL_{Y_{ij}}(q)+[\xi_j,Y_{ij}]q}+L_{Y_{ij}}(\xi_jq)
\\
&=(\xi_j\psi_j+\dot{\eta}_jq)-(\xi_i\psi_i+\dot{\eta}_iq)
{\color{black}-L_{\xi_j}(Y_{ij}q)}+L_{Y_{ij}}(\xi_jq) \\
&=(\xi_j\psi_j+\dot{\eta}_jq)-(\xi_i\psi_i+\dot{\eta}_iq),
\end{align*}
which implies what we wanted.
\end{proof}


\begin{theorem}[Dolbeaut presentation and Pairing]
\label{thm:doulbeaut_presentation_pairing}
Let $\mathcal{U}=\{U_i\}_{i\in I}$ be a locally finite covering of $M_0$ with $H^1(U_i,\Theta_{M_0})=0$ for $i\in I$.
Let $\tv{X,\dot{Y}}{Y}\in \mathbb{T}_{Y}[\mathcal{U}]$ and $\tvdag{\psi,q}{Y}\in \mathbb{T}^\dagger_{Y}[\mathcal{U}]$. Let $\eta\in C^0(\mathcal{U},\mathcal{A}^{0,0}(\Theta_{M_0}))$ with $\delta \eta= Y$ and
$$
\connectinghomo^{-1}(\tv{X,\dot{Y}}{Y})=\tvDBel{\mu,\dot{\nu}}{\eta}\in \mathbb{T}_{\eta}[\mathcal{U}].
$$
Then,
\begin{align*}
\mathcal{P}_{\mathbb{TT}}(\tv{X,\dot{Y}}{[Y]}, \tvdag{\psi, q}{[Y]})
&=
\dfrac{1}{2i}\iint_{M_0}((\mu\psi_i+\dot{\nu}_iq)
-L_{\eta_i}(\mu q))d\overline{z}\wedge dz.
\nonumber
\end{align*}
\end{theorem}

\begin{proof}
Under the notation in \Cref{lem:dolbeaut_presention_pairing},
let $\mu=-\overline{\partial}\xi_i$ and $\dot{\nu}=-\overline{\partial}\dot{\eta}_i$.
When we set $A_idz=(\xi_i\psi_i+\dot{\eta}_iq)dz$,
$$
((A_i)_{\overline{z}}-L_{\eta_i}((\xi_i)_{\overline{z}}q))d\overline{z}\wedge dz
=-(\mu \psi_i+\dot{\nu}q-L_{\eta_i}(\mu q))d\overline{z}\wedge dz.
$$
Hence, the formula follows from the definition of the pairing.
\end{proof}

\section{Model space of $T_{q_0}\mathcal{Q}_g$}
\label{sec:tangent-Qg}
Let $\mathcal{Q}_g$ be the vector bundle of holomorphic quadratic differentials over $\teich_g$. Let $\Pi_{\mathcal{Q}_g}\colon \mathcal{Q}_g\to \teich_g$ is the projection. As discussed in \S\ref{sec:infinitesimal_deformation_RS}, the bundle $\mathcal{Q}_g$ is canonically identified with the cotangent bundle $T^*\teich_g$ over $\teich_g$. In this section, we review Hubbard and Masur's description of the holomorphic tangent space to $\mathcal{Q}_g$ given in \cite{MR523212}.

\subsection{Tangent space}
\label{subsec:tangent_space_to_T_Tstar}
The tangent space $T_{q_0}\mathcal{Q}_g$ of $\mathcal{Q}_g$ at a holomorphic quadratic differential $q_0\in \mathcal{Q}_{M_0}=H^1(M_0,\Omega_{M_0}^{\otimes 2})$ is described as the first hyper-cohomology group of the complex of sheaves $\mathbb{L}_{q_0}$:
\begin{equation}
\label{eq:complex_of_sheaves_L}
\begin{CD}
0 @>>> \Theta_{M_0} @>{L_{\cdot}(q_0)}>> \Omega_{M_0}^{\otimes 2} @>>> 0,
\end{CD}
\end{equation}
where $L{\cdot}(q_0)$ is the Lie derivative along a vector field given at \eqref{eq:lie_derivative_definition}
(cf. \cite[Proposition 3.1]{MR523212}).

Fix a locally finite covering $\mathcal{U}=\{U_i\}_{i\in I}$ of $M_0$. Then, the tangent space $T_{q_0}\mathcal{Q}_g$ is the first cohomology group ${\bf H}^1(\mathbb{L}_{q_0})$ of the following double complex:
\begin{equation}
\label{eq:double_complex}
\begin{CD}
0
\\
@AAA \\
C^{0}(\mathcal{U},\Omega_{M_0}^{\otimes 2})
@>{\delta}>>
C^{1}(\mathcal{U},\Omega_{M_0}^{\otimes 2}) \\
@A{{\color{black}L_{\cdot}q_{0}}}AA
@A{{\color{black}-L_{\cdot}q_{0}}}AA
\\
C^{0}(\mathcal{U},\Theta_{M_0})
@>{\delta}>>
C^{1}(\mathcal{U},\Theta_{M_0})
@>{\delta}>>
C^{2}(\mathcal{U},\Theta_{M_0}).
\end{CD}
\end{equation}
Then,
\begin{align*}
{\bf Z}^1(\mathcal{U},\mathbb{L}_{q_0})
&=\{(X,\varphi)\in Z^1(\mathcal{U},\Theta_{M_0})\oplus C^0(\mathcal{U},\Omega^{\otimes 2}_{M_0})\mid \delta\varphi-L_X(q_0)=0\}
\\
{\bf B}^1(\mathcal{U},\mathbb{L}_{q_0})
&=\{(\delta \alpha, \{{\color{black}L_{\alpha_{i}}(q_0)}\}_i,)\in 
Z^1(\mathcal{U},\Theta_{M_0})
\oplus
C^0(\mathcal{U},\Omega^{\otimes 2}_{M_0})
\mid \alpha=\{\alpha_i\}_i\in C^{0}(\mathcal{U},\Theta_{M_0})\}
\end{align*}
and
$$
{\bf H}^1(\mathcal{U},\mathbb{L}_{q_0})
={\bf Z}^1(\mathcal{U},\mathbb{L}_{q_0})/{\bf B}^1(\mathcal{U},\mathbb{L}_{q_0}).
$$
The first \emph{hyper-cohomology group}
${\bf H}^1(\mathbb{L}_{q_0})$ is defined by taking the direct limit of ${\bf H}^1(\mathcal{U},\mathbb{L}_{q_0})$ in terms of the refinement of (locally finite) coverings. We denote by $[X,\varphi]_{q_0}\in {\bf H}^1(\mathcal{U},\mathbb{L}_{q_0})$ the cohomology class of $(X,\varphi)\in {\bf Z}^1(\mathcal{U},\mathbb{L}_{q_0})$.

The following proposition might be well-known. Indeed, Hubbard and Masur deal with the case of the direct limits in \cite[Proposition 4.5]{MR523212}. However, they do not indicate conditions of the covering in the statement (but, it is implicitly given). Hence, we give a proof for completeness and for confirmation, although the idea here is the same as that given by Hubbard and Masur.

\begin{proposition}
\label{prop:some_covering_hypercohomology}
When a locally finite covering $\mathcal{U}=\{U_i\}_{i\in I}$ satisfies that $H^1(U_i,\Theta_{M_0})=0$ for all $i\in I$,
${\bf H}^1(\mathcal{U},\mathbb{L}_{q_0})$ is isomorphic to ${\bf H}^1(\mathbb{L}_{q_0})$. In particular, ${\bf H}^1(\mathcal{U},\mathbb{L}_{q_0})$  is isomorphic to $T_{q_0}\mathcal{Q}_g$.
\end{proposition}

\begin{proof}
We consider the following commutative diagram:
$$
\minCDarrowwidth10pt
\begin{CD}
0
@>>>
0
@>>>
C^0(\mathcal{U},\Theta_{M_0})
@>{id}>>
C^0(\mathcal{U},\Theta_{M_0})
@>>>
0
\\
@.
@VVV
@V{\delta\oplus {\color{black}L_{\cdot}(q_0)}}VV
@VV{\delta_2}V
\\
0
@>>> 
C^0(\mathcal{U},\Omega_{M_0}^{\otimes 2})
@>{inc}>>
C^1(\mathcal{U},\Theta_{M_0})
\oplus
C^0(\mathcal{U},\Omega_{M_0}^{\otimes 2})
@>{pr_1}>>
C^1(\mathcal{U},\Theta_{M_0})
@>>>
0
\\
@.
@V{\delta_1}VV
@V{\delta_2\oplus (\delta_1-L_\cdot(q_0))}VV
@VV{\delta_2}V
\\
0
@>>>
C^1(\mathcal{U},\Omega_{M_0}^{\otimes 2})
@>{inc}>>>
C^2(\mathcal{U},\Theta_{M_0})
\oplus
C^1(\mathcal{U},\Omega_{M_0}^{\otimes 2})
@>{pr_1}>>
C^2(\mathcal{U},\Theta_{M_0})
@>>> 0,
\end{CD}
$$
where, in the above diagram, $\delta_1$ and $\delta_2$ are coboundary map for cochains of sheaves $\Omega_{M_0}^{\oplus 2}$ and $\Theta_{M_0}$, $inc$ means the inclusion to the second coordinate, and $pr_1$ is the projection on the first coordinate
(cf. the diagram in \cite[p.264]{MR523212}). This diagram leads the exact sequence
$$
\minCDarrowwidth10pt
\begin{CD}
H^0(\mathcal{U},\Theta_{M_0})
@>>>
H^0(\mathcal{U},\Omega_{M_0}^{\otimes 2})
@>>>
{\bf H}^1(\mathcal{U},\mathbb{L}_{q_0})
@>>>
H^1(\mathcal{U},\Theta_{M_0})
@>>>
H^1(\mathcal{U},\Omega_{M_0}^{\otimes 2}).
\end{CD}
$$
Since $M_0$ does not admit a (global) holomorphic vector field, we have $H^0(\mathcal{U},\Theta_{M_0})=0$. Let $(q_0)$ be the divisor of $q_0$. The sheaf $\Omega_{M_0}^{\otimes 2}$ is naturally isomorphic to the sheaf $\mathcal{O}_{M_0}((q_0))$ of meromorphic functions on $M_0$ which is multiples of the divisor $(q_0)$, where a section on $U\subset M_0$ is a meromorphic function on $U$ satisfying $(f)+D\ge 0$ (e.g. \cite[\S16.4, \S17.4]{MR648106}).
From the Serre duality,
$$
H^1(M_0,\Omega_{M_0}^{\otimes 2})\cong H^1(M_0,\mathcal{O}_{M_0}((q_0)))
\cong H^0(M_0,\Omega_{M_0}(-(q_0)))^*.
$$
An element in $H^0(M_0,\Omega_{M_0}(-(q_0))$ is a holomorphic $1$-form $\omega$ on $M_0$ with $(\omega)-(q_0)\ge 0$, and $2g-2=\deg((\omega))\ge \deg((q_0))=4g-4$. Therefore, $H^1(M_0,\Omega_{M_0}^{\otimes 2})=0$ if $g\ge 2$.
Since the canonical map $H^1(\mathcal{U},\Omega_{M_0}^{\otimes 2})\to H^1(M_0,\Omega_{M_0}^{\otimes 2})$ is injective, we conclude that $H^1(\mathcal{U},\Omega_{M_0}^{\otimes 2})=0$.
Thus, when $H^1(U_i,\Theta_{M_0})=0$ for all $i\in I$, we have the following exact sequence
$$
\minCDarrowwidth10pt
\begin{CD}
0
@>>>
\mathcal{Q}_{M_0}\cong H^0(\mathcal{U},\Omega_{M_0}^{\otimes 2})
@>>>
{\bf H}^1(\mathcal{U},\mathbb{L}_{q_0})
@>>>
H_1(\mathcal{U},\Theta_{M_0})\cong T_{x_0}\teich_g
@>>>
0,
\end{CD}
$$
which canonically represents the fibration
$Q_{M_0}\to T_{q_0}\mathcal{Q}_g\to T_{x_0}\teich_g$ via the isomorphism ${\bf H}^1(\mathbb{L}_{q_0})\cong T_{q_0}\mathcal{Q}_g$ (cf. \cite[Proposition 4.5]{MR523212}).
\end{proof}
%


\begin{remark}
\label{remark:differential_projection_cotangent}
From the above discussion, the map
$$
T_{q_0}\mathcal{Q}_g\cong {\bf H}^1(\mathbb{L}_{q_0})\ni [X,\varphi]_{q_0}\to [X]\in H^1(M_0,\Theta_{M_0})
\cong T_{x_0}\teich_g
$$
is the horizontal projection (cf. \S\ref{sec:tangent_on_cotangent_bundles}), which
presents the differential 
$$
D\Pi^\dagger_{\teich_g}=D\Pi_{\mathcal{Q}_g}|_{q_0}\colon T_{q_0}\mathcal{Q}_g\to T_{x_0}\teich_g
$$
of the projection $\Pi^\dagger_{\teich_g}=\Pi_{\mathcal{Q}_g}\colon T^*\teich_g=\mathcal{Q}_g\to \teich_g$, where $x_0=\Pi_{\mathcal{Q}_g}(q_0)$. As discussed in the proof of \Cref{prop:some_covering_hypercohomology}, the kernel, which is nothong but the vertical space, is isomorphic to $T^*_{x_0}\teich_g\cong H^0(\mathcal{U},\Omega_{M_0}^{\otimes 2})$.
\end{remark}

\subsection{Description with holomorphic families}
\label{subsec:description_Hol}
Let $B$ be a neighborhood of the origin in $\mathbb{C}$.
Let $(\mathcal{M},\pi,B)$ be a holomorphic family of compact Riemann surfaces of genus $g$. For $t\in B$, set $M_t=\pi^{-1}(t)$. By taking $B$ sufficiently small, we identify $\mathcal{M}$ with the product $M_{0}\times B$ via the local trivialization as $C^\infty$-manifolds (cf. \cite[Theorem 2.4]{MR815922}). By the implicit mapping theorem, when $B$ is taken sufficiently small if necessary again, there is a covering $\{U_i\}_{i\in I}$ of $M_{0}$ and an injective holomorphic map $Z_i\colon U_i\times B_0\to \mathbb{C}\times B_0$ such that $z_i^t=Z_i\mid_{U_i\times \{t\}}\colon U_i\times \{t\}\to \mathbb{C}$ ($i\in I$) make an analytic chart of $M_t$ for all $t\in B$. We identify $U_i\times \{t\}$ with $U_i$ and define $z_{ij}^t=z^{t}_i\circ (z_j^{t})^{-1}$ on $U_i\cap U_j$.
Let $\{q_t\}_{t\in B}$ be a holomorphic family of holomorphic quadratic differentials with $q_t\in  M_t$. Suppose that $q_{t}$ is presented as $q_{t}^t(z)dz^2$ on $z^0_i(U_i)$. Set
\begin{align*}
X_{ij}(z) & = \left.\dfrac{\partial z_{ij}^t}{\partial t}\right\mid_{t=0}\circ (z_{ij}^0)^{-1}(z)\quad (z\in z_i^0(U_i\cap U_j))
\\
\varphi_i(z)
&=\left.\dfrac{\partial q_{i}^t}{\partial t}\right|_{t=0}(z)
\quad (z\in z_i^0(U_i))
\end{align*}
(cf. \S\ref{subsec:cochain}).
Then, $(X,\varphi)=(\{X_{ij}\}_{i,j\in I}, \{\varphi_i\}_{i\in I})\in {\bf Z}^1(\mathcal{U},\mathbb{L}_{q_0})$ and it represents the tangent vector of the family $\{q_t\}_{t\in B}$ at $t=0$. Indeed, since
$$
z_{ij}^t(z)=z_{ij}^0(z)+tX_{ij}(z_{ij}(z))+o(t),
$$
for $z\in z_j(U_i\cap U_j)$,
\begin{align*}
&q_i^t(z_{ij}^t(z))\dfrac{dz_{ij}^t}{dz}(z)^2 \\
&=q_i^t\left(
z_{ij}^0(z)+tX_{ij}(z_{ij}^0(z))+o(t)
\right)\dfrac{dz_{ij}^0}{dz}(z)^2\left(1+tX_{ij}'(z_{ij}^0(z))\dfrac{dz_{ij}^0}{dz}(z)+o(t)\right)^2
\\
&=
q_0(z_{ij}^0(z))\dfrac{dz_{ij}^0}{dz}(z)^2
+t
\left(
\varphi_i(z_{ij}^0(z))
+L_{X_{ij}}(q_0)(z_{ij}^0(z))
\right)
\dfrac{dz_{ij}^0}{dz}(z)^2
+o(t) \\
q_j^t(z)&=q_0(z)+t\varphi_j(z)+o(t)
\end{align*}
as $t\to 0$. This means that
$$
\varphi_j(z)-\varphi_i(z_{ij}^0(z))\dfrac{dz_{ij}^0}{dz}(z)^2=
L_{X_{ij}}(q_0)(z_{ij}^0(z))
\dfrac{dz_{ij}^0}{dz}(z)^2
$$
for $z\in z_j(U_i\cap U_j)$, which is written as 
$$
\varphi_j-\varphi_i=L_{X_{ij}}(q_0)\quad (i,j\in I)
$$
or $\delta\varphi=L_X(q_0)$
for short.

\subsection{Tangent spaces to strata}
The following is essentially due to Hubbard-Masur \cite{MR523212} and Dumas \cite{MR3413977}.
For the sake of completeness, we give a brief proof.

\begin{proposition}[Tangent space to the strata]
\label{prop:tangent_space_strata}
Let $q_0\in \mathcal{Q}_g(k_1,\cdots,k_n;\epsilon)$.
Let $\mathcal{U}=\{U_i\}_{i\in I}$ be a locally finite covering of $M_0$ such that each $U_i$ contains at most $1$ zero of $q_0$, and $H^1(U_i,\Theta_{M_0})=0$ for $i\in I$.
A tangent vector $[X,\varphi]_{q_0}\in H^1(\mathcal{U},\mathbb{L}_{q_0})\cong T_{q_0}\mathcal{Q}_g$ is in $T_{q_0}\mathcal{Q}_g(k_1,\cdots,k_n;\epsilon)$, then for any $i\in I$, a meromorphic function $\varphi_i/q_0$ has at most simple poles.
\end{proposition}

\begin{proof}
The following argument is a modification of that given by Dumas \cite[Lemma 5.2]{MR3413977}.

As we discussed in \S\ref{subsec:description_Hol}, each $\varphi_i$ is appeared by differentiating a holomorphic family $\{q_t\}_{|t|<\epsilon}$ of holomorphic quadratic differentials on a local chart $(U_i,z_i)$. For the simplicity of the argument, we may assume that $z^t_i(z_0)=0$ and $q_0=z^kdz^2$ on $z_i(U_i)$, where $k$ is the order of zero at $z_0$ of $q_0$.
Since $q_t$ has the same symbol as $q_0$, we can write
$$
q_t=\alpha_t^*(z^kdz^2)
$$
for $|t|<\epsilon$ and some holomorphic function $\alpha_t$ which depends holomorphically in $t$ and $\alpha_0(z)=z$.
Then,
$$
\varphi_i=z^{k-1}(k\alpha'+2z\dot{\alpha}')dz^2
$$
and $\varphi_i$ has a zero of order at least $k-1$ at $z_0$. Hence $\varphi_i/q_0$ has at most simple poles on $U_i$.
\end{proof}

\section[The model space of $T^*_{q_0}\mathcal{Q}_g$]{The model space of $T^*_{q_0}\mathcal{Q}_g$ and Models of the pairing and the holomorphic symplectic form}
\label{sec:ModelTTster_TstarTstar}
\subsection{Model space}
We set
$$
{\bf H}^{1,\dagger}(\mathcal{U},\mathbb{L}_{q_0})=\{[\Phi,Y]^\dagger_{q_0}\mid [-Y,\Phi]_{q_0}\in {\bf H}^1(\mathcal{U},\mathbb{L}_{q_0})\}.
$$
and ${\bf H}^{1,\dagger}(\mathbb{L}_{q_0})$ as its direct limit. We will adopt the space ${\bf H}^{1,\dagger}(\mathbb{L}_{q_0})$ as the \emph{model space} of the cotangent space $T^*_{q_0}\mathcal{Q}_g$ to $\mathcal{Q}_g$ at $q_0\in \mathcal{Q}_g$. A representative $(\Phi, Y)$ of $[\Phi,Y]^\dagger_{q_0}\in {\bf H}^{1,\dagger}(\mathcal{U},\mathbb{L}_{q_0})$ satisfies
$$
\Phi_j-\Phi_i=-L_{Y_{ij}}(q_0)
$$
for $i,j\in I$.

\subsection{Model of pairing}
\label{subsec:model_pairing_cotangent}
Let $\mathcal{U}$ be a locally finite covering of $M_0$. We define the \emph{pairing} between ${\bf H}^{1}(\mathcal{U},\mathbb{L}_{q_0})$ and ${\bf H}^{1,\dagger}(\mathcal{U},\mathbb{L}_{q_0})$ by
\begin{align}
\label{eq:pairing_cotangent}
&\paircot([X,\varphi]_{q_0},[\Phi,Y]^\dagger_{q_0})
\\
&=-\dfrac{1}{2i}
\iint_{M_0}(\xi_i)_{\overline{z}}
\left(
(\Phi_i+L_{\eta_i}(q_0))+(\eta_i)_{\overline{z}}(\varphi_i-L_{\xi_i}(q_0))
\right)d\overline{z}\wedge dz,
\nonumber
\end{align}
where $\xi$, $\eta\in C^0(\mathcal{U}, \mathcal{A}^{0,0}(\Theta_{M_0}))$ with $\delta\xi=X$ and $\delta\eta=Y$.
We check the pairing \eqref{eq:pairing_cotangent} is well-defined.

We first check that the integral in \eqref{eq:pairing_cotangent} is independent of the choices of $\xi$ and $\eta$.
Let $\xi'=\{\xi'_i\}_{i\in I}\in C^0(\mathcal{U}, \mathcal{A}^{0,0}(\Theta_{M_0}))$ with $\delta\xi'=X$. Then, $\Xi=\xi_i-\xi'_i$ on $U_i$ defines a global vector field on $M_0$. The difference between the integrals in the right-hand side \eqref{eq:pairing_cotangent} defined from $\xi$ and $\xi'$ is equal to
\begin{align*}
&\dfrac{1}{2i}\iint_{M_0}
\left(
\Xi_{\overline{z}}(\Phi_i+L_{\eta_i}(q_0))
-(\eta_i)_{\overline{z}}L_{\Xi}(q_0)\right)d\overline{z}\wedge dz
\\
&=\dfrac{1}{2i}\iint_{M_0}d(\Xi (\Phi_i+L_{\eta_i}(q_0))dz+
2(q_0 \Xi (\eta_i)_{\overline{z}})d\overline{z})=0,
\end{align*}
since the one form 
$$
\Xi (\Phi_i+L_{\eta_i}(q_0))dz
+2q_0 \Xi (\eta_i)_{\overline{z}}d\overline{z}
$$
in the integral is the global one form on $M_0$.
The well-definedness for $\eta$ is treated by the same argument.

We next check that the integral in \eqref{eq:pairing_cotangent} is independent of the choices of the cocycles.
The cohomology class for the cocycle $(X,\varphi)$ consists of $(X+\delta\alpha, \varphi+L_{\alpha}(q_0))$ 
for $\alpha\in C^0(\mathcal{U},\Theta_{M_0})$. 
Let $\xi'\in C^0(\mathcal{U},\mathcal{A}^{0,0}(\Theta_{M_0}))$
with $\delta\xi'=X+\delta\alpha$. Since $\delta(\xi'-\alpha)=X$, from the previous discussion, the integrand in the right-hand side of \eqref{eq:pairing_cotangent} is
\begin{align*}
&(\xi'_i)_{\overline{z}}
(\Phi_i+L_{\eta_i}(q_0))
+
(\eta_i)_{\overline{z}}
(\varphi_i+L_{\alpha_i}(q_0)-L_{\xi'_i}(q_0))
\\
&(\xi'_i-\alpha_i)_{\overline{z}}
(\Phi_i+L_{\eta_i}(q_0))
+
(\eta_i)_{\overline{z}}
(\varphi_i-L_{\xi'_i-\alpha_i}(q_0))
\\
&=(\xi_i)_{\overline{z}}
(\Phi_i+L_{\eta_i}(q_0))
+
(\eta_i)_{\overline{z}}
(\varphi_i-L_{\xi_i}(q_0)).
\end{align*}
The case for $[\Phi,Y]\in {\bf H}^{1,\dagger}(\mathcal{U},\mathbb{L}_{q_0})$ follows from the same argument.

For two locally finite coverings $\mathcal{U}$, $\mathcal{V}$.
Suppose that $\mathcal{V}$ is a refinement of $\mathcal{U}$.
Since the natural maps ${\bf H}^{1}(\mathcal{U},\mathbb{L}_{q_0})\to {\bf H}^{1}(\mathcal{V},\mathbb{L}_{q_0})$ and ${\bf H}^{1,\dagger}(\mathcal{U},\mathbb{L}_{q_0})\to {\bf H}^{1,\dagger}(\mathcal{V},\mathbb{L}_{q_0})$ are defined by the restriction, for the pairing \eqref{eq:pairing_cotangent} descends to the pairing on ${\bf H}^{1}(\mathbb{L}_{q_0})\times {\bf H}^{1,\dagger}(\mathbb{L}_{q_0})$.

\begin{proposition}[Non-degenerate]
\label{prop:non-degenerate_cotang}
When a locally finite covering $\mathcal{U}=\{U_i\}_{i\in I}$ satisfies $H^1(U_i,\Theta_{M_0})=0$ for $i\in I$, the pairing \eqref{eq:pairing_cotangent} is non-degenerate.
\end{proposition}

\begin{proof}
Let $[X,\varphi]_{q_0}\in {\bf H}^{1}(\mathcal{U},\mathbb{L}_{q_0})$.
Suppose that
$$
\paircot([X,\phi]_{q_0},[\Phi,Y]^\dagger_{q_0})=0
$$
for all $[\Phi,Y]^\dagger_{q_0}\in {\bf H}^{1,\dagger}(\mathcal{U},\mathbb{L}_{q_0})$.
Let $\Psi\in \mathcal{Q}_g$. Since $[\Psi,0]^\dagger_{q_0}\in {\bf H}^{1,\dagger}(\mathcal{U},\mathbb{L}_{q_0})$, 
%
$$
0=\paircot([X,\phi],[\Psi,0])=
-\dfrac{1}{2i}
\iint_{M_0}(\xi_i)_{\overline{z}}(z)\Psi(z)d\overline{z}\wedge dz.
$$
From the assumption of the covering, $H^1(\mathcal{U},\Theta_{M_0})\cong H^1(M_0,\Theta_{M_0})$.
Therefore, $[X]=0$ in $H^1(\mathcal{U},\Theta_{M_0})$
(cf. \S\ref{subsec:the_first_cohomology}).
Therefore, there is $\alpha\in C^0(\mathcal{U},\Theta_{M_0})$ such that $\delta\alpha=X$. We may adopt $\alpha$ as $\xi$ in the formula \eqref{eq:pairing_cotangent} of the pairing and 
$$
0=\paircot([\delta \alpha,\varphi]_{q_0},[\Phi,Y]^\dagger_{q_0})=
-\dfrac{1}{2i}
\iint_{M_0}(\eta_i)_{\overline{z}}(\varphi_i(z)-L_{\alpha_i}(q_0)(z))d\overline{z}\wedge dz.
$$
Since $\{\varphi_i-L_{\alpha_i}(q_0)\}_{i\in I}$ defines a (global) holomorphic quadratic differential on $M_0$, we get $\varphi_i-L_{\alpha_i}(q_0)=0$ for all $i$, since the pairing \eqref{eq:pairing} is non-degenerate.
Thus, we get $[X,\varphi]_{q_0}=[\delta\alpha, L_\alpha(q_0)]^\dagger_{q_0}=0$.
The remaining case is dealt with the same argument.
\end{proof}

%
\subsection{Kawai's description of the holomorphic symplectic form}
\label{subsec:Kawai}
In this section, following Kawai \cite{MR1386110}, we recall the holomorphic symplectic form on the tangent space of the space of holomorphic quadratic differentials,
and rewrite Kawai's symplectic form under our setting.

First we recall Kawai's description.
Let $M_0$ be a closed Riemann surface of genus $g$.
Let $q_0\in \mathcal{Q}_{M_0}$
Let $\Gamma_0$ be a cocompact Fuchsian group acting on the upper-half plane $\mathbb{H}$ with $\mathbb{H}/\Gamma_0=M_0$. In the following discussion, we identify differentials on a Riemann surfaces with automorphic forms on its universal covering space.

Let $\mu_j$ ($j=1$, $\cdots$, $3g-3$) be a system of smooth $(-1,1)$-automorphic forms whose Teichm\"uller equivalence classes span the tangent space at $x_0=(M_0,f_0)\in \teich_g$.
In \cite{MR1386110}, Kawai takes such a basis from the space of harmonic Beltrami differentials.
We can check that his calculation is valid for a system of $3g-3$ smooth $(-1,1)$-forms which form a basis on the tangent space to the Teichm\"uller space at $x_0$.

Let $q_s^\alpha$ ($\alpha=1,2$, $|s|<\epsilon$) be a holomorphic family of quadratic differentials with $q_0^1=q_0^2=q_0\in \mathcal{Q}_{x_0}$. Suppose $q_s^\alpha\in \mathcal{Q}_{x^\alpha_s}$ and  $x^\alpha_s=(M^\alpha_s,f^\alpha_s)$ with $x^1_0=x^2_0=x_0$. Suppose that the Beltrami differential of $f^\alpha_s\colon M_0\to M^\alpha_s$ is equal to $\nu_\alpha(s)=\sum_{j=1}^{3g-3}\epsilon^\alpha_j(s)\mu_j$ for some holomorphic functions $\epsilon^\alpha_j$ ($\alpha=1$, $2$ and $j=1$, $\cdots$, $3g-3$).
Let $w^{\nu_\alpha(s)}$ be the normalized quasiconformal mapping on $\mathbb{H}$ satisfying the Beltrami equation $(w^{\nu_\alpha(s)})_{\overline{z}}=\nu_\alpha(s)(w^{\mu^i(s)})_z$, $w^{\nu_\alpha(s)}(0)=w^{\nu_\alpha(s)}(1)-1=0$ and $w^{\mu_\alpha(s)}(\infty)=\infty$.
Then, Kawai describes the holomorphic symplectic form on $\mathcal{Q}_g$ as
\begin{align}
&\omega_{\mathcal{Q}_g}(t_1,t_2)
\label{eq:Kawai_symplectic_form}
\\
&=\dfrac{1}{2}\sum_{j=1}^{3g-3}\left\{
\left(
\iint_{\mathbb{H}/\Gamma_0}[\dot{q}^1_0(z)+(q^1_0)'(z)
\dot{w}^{\nu_1(0)}(z)
+2q_0^1(z)(\dot{w}^{\nu_1(0)})_z(z)]\mu_j(z)dxdy
\right)\dot{\epsilon}^2_j(0)\right.
\nonumber
\\
&\quad
-
\left.
\left(
\iint_{\mathbb{H}/\Gamma_0}[\dot{q}^2_0(z)+(q^2_0)'(z)
\dot{w}^{\nu_2(0)}(z)
+2q_0^2(z)(\dot{w}^{\nu_2(0)})_z(z)]\mu_j(z)dxdy
\right)\dot{\epsilon}^1_j(0)
\right\},
\nonumber
\end{align}
where $t_\alpha\in T_{q_0}\mathcal{Q}_g$ is tangent to the family $\{q^\alpha_s\}_{|s|<1}$ at $s=0$. 
See \cite[(3.4)]{MR1386110}.
%

We claim 

\begin{proposition}[Holomorphic symplectic form]
\label{prop:holomorphic_symplectic_form}
Fix a locally finite covering $\mathcal{U}=\{U_\alpha\}_{i\in I}$ with $H^1(U_i,\Theta_{M_0})=0$ for $i\in I$.
The holomorphic symplectic form $\omega_{\mathcal{Q}_g}$ on $\mathcal{Q}_g$ is described on the model space ${\bf H}^1(\mathbb{L}_{q_0})$ of $T_{q_0}\mathcal{Q}_g$ by
\begin{align*}
&\omega_{\mathcal{Q}_g}([X^1,\varphi^1]_{q_0},[X^2,\varphi^2]_{q_0})
\\
&=
\dfrac{1}{4i}
\left(
\iint_{M_0}(\xi^1_i)_{\overline{z}}(\varphi^2_i-L_{\xi^2_i}(q_0))
-\iint_{M_0}(\xi^2_i)_{\overline{z}}(\varphi^1_i-L_{\xi^1_i}(q_0))
\right)d\overline{z}\wedge dz
\end{align*}
for $[X^\alpha,\varphi^\alpha]_{q_0}\in {\bf H}^1(\mathbb{L}_{q_0})$ and $\xi^\alpha\in C^0(\mathcal{U}, \mathcal{A}^{0,0}(\Theta_{M_0}))$ with $[\delta\xi^\alpha]=[X^\alpha]$ ($\alpha=1$, $2$). 
\end{proposition}

\begin{proof}
As we mentioned above, we rewrite Kawai's holomorphic symplectic form under our setting.
Fix a locally finite covering $\mathcal{U}=\{U_i\}_{i\in I}$.
Suppose that for $\alpha=1,2$, $t_\alpha=[X^\alpha,\varphi^\alpha]_{q_0}\in {\bf H}^1(\mathbb{L}_{q_0})\cong T_{q_0}\mathcal{Q}_g$. 
The restriction $\xi^\alpha_i$ of the minus of the derivative $-\dot{w}^{\nu_i(0)}$ (in Kawai's formula) to $U_i$ defines a cochain $\xi^\alpha=\{\xi^\alpha_i\}_{i\in I}\in C^0(\mathcal{U},\mathcal{A}^{0,0}(\Theta_{M_0}))$ with $[\delta\xi^\alpha]=[X^\alpha]$ (cf. \S\ref{subsec:appendix_1}).
Notice that
$$
(\xi^\alpha)_{\overline{z}}=-\nu_i(0)|_{U_i}=-\sum_{j=1}^{3g-3}\epsilon^\alpha_j(0)\mu_j|_{U_i}
$$
for $\alpha=1$, $2$ and $i\in I$.
From \eqref{eq:Kawai_symplectic_form},
we obtain
\begin{align*}
\omega_{\mathcal{Q}_g}(t_1,t_2)
&=
\dfrac{1}{2}
\left(-\iint_{M_0}[\varphi^1_i(z)-L_{\xi^1_i}(q_0)(z)](\xi^2_i)_{\overline{z}}(z)dxdy
\right.
\\
&\qquad+
\left.
\iint_{M_0}[\varphi^2_i(z)-L_{\xi^2_i}(q_0)(z)](\xi^1_i)_{\overline{z}}(z)dxdy\right),
\end{align*}
which implies what we wanted.
\end{proof}

\section{Remarks on regularity of cochains in the formulae of pairings}
\label{sec:remark_smoothness}
All $0$-chains in the formulae of the pairings in \Cref{lem:pairing_tangent_cotangent}, \Cref{thm:doulbeaut_presentation_pairing}, \S\ref{subsec:model_pairing_cotangent}, and \Cref{prop:holomorphic_symplectic_form} are assumed to be smooth. However, for the application, we may need to think of such $0$-chains with weaker regularity conditions.

Henceforth, let $\sob^{k,s}(D)$ be the function space consisting of continuous functions all of whose $s$-th derivarives are in $L^s$ for all $s\le k$. By definition, $\sob^{0,s}(D)=L^s(D)$.

For a Riemann surface $M$ and a , let $\sob^{k,s:p,q}(\mathcal{S})$ be the sheaf of germs of $(p,q)$-forms of class $\sob^{k,s}$with coefficients in $\mathcal{S}=\Theta_{M}$ or $\Omega_{M}$. 
It is known that the Green formula holds for differentials of class $\sob^{1,1}$. For instance, see \cite[p.150 (6.17)]{MR0344463}.
Threfore, in \Cref{lem:pairing_tangent_cotangent} and \Cref{thm:doulbeaut_presentation_pairing}, we can check that the proofs of the formulae work with 
\begin{itemize}
\item
$\xi$, $\eta\in C^0(\mathcal{U},\sob^{2,1:0,0}(\Theta_{M_0})$, 
\item
$A\in C^0(\mathcal{U},\sob^{1,1:0,0}(\Omega_{M_0}))$,
\item
$\mu\in L^{\infty}_{(-1,1)}(M_0)$
with $-\overline{\partial}\xi=\mu$ for some $\xi\in C^0(\mathcal{U},\sob^{2,1:0,0}(\Theta_{M_0}))$,
\item
$\dot{\nu}\in C^0(\mathcal{U},\sob^{0,\infty:0,1}(\Theta_{M_0}))$
with $-\overline{\partial}\dot{\eta}=\dot{\nu}$ for some $\dot{\eta}\in C^0(\mathcal{U},\sob^{2,1:0,0}(\Theta_{M_0}))$ which satisfies
$$
\delta\dot{\eta}+K(\xi,\delta \eta)=0.
$$
\end{itemize}
For the formulae of the pairing and the symplectic form in \Cref{lem:pairing_tangent_cotangent} and \Cref{thm:doulbeaut_presentation_pairing}, we can also check that the proofs work with $\xi$, $\eta\in C^0(\mathcal{U},\sob^{2,1}(\Theta_{M_0}))$.

Indeed, the invariance of the pairing represented with Dolbeaut presentations in \Cref{thm:doulbeaut_presentation_pairing} follows from the discussion in \Cref{lem:pairing_tangent_cotangent} for
$$
A_i=\xi_i \psi_i+\dot{\eta}_iq.
$$
The other case can be similarly checked.

In general, we need such smoothness of cochains to the validity of the formulae and the invariance in terms of the representatives, since we apply the Green theorem in checking the invariance. However, in some case, the formulae also valid with the slightly weaker assumption.

Indeed, the Dolbeaut presentations in \Cref{thm:doulbeaut_presentation_pairing} also holds with cochains $\xi$ and $\eta$ such that $-\overline{\partial}\xi$ and $-\overline{\partial}\eta$ are Teichm\"uller Beltrami differentials (cf. \S\ref{subsubsec:TB_differenital}). A Teichm\"uller Beltrami differential on $M_0$, by definition, forms $\overline{\varphi}/|\varphi|$ for some $\varphi\in \mathcal{Q}_{M_0}$. Hence, it is real analytic except for the zeros of $\varphi$, but not continuous on whole $M_0$, which means that such $\xi$ and $\eta$ are not in $C^0(\mathcal{U},\sob^{2,1:0,0}(\Theta_{M_0})$. 
However, in general, for a holomorphic function $f$ on a domain $D$,
$$
\left(\dfrac{\overline{f}}{|f|}\right)_z=-\dfrac{(\overline{f})^2}{2|f|^3}f',\quad
\left(\dfrac{\overline{f}}{|f|}\right)_{\overline{z}}=\dfrac{\overline{f'}}{2|f|}\in L^1_{loc}(D)
$$
since zeros of a holomorphic function are discrete.
Hence, for any relatively compact simply conneted subdomain $D'$ in $D$, we can construct $\eta\in \sob^{2,1}(D'-{\rm Zero}(f))$ such that $\eta_{\overline{z}}=-\overline{f}/|f|\in L^\infty(D)$ (cf. \cite[Proposition 4.19]{MR1215481}). The discussion with the Green theorem works for such differentials
by applying the Royden-type argument given in \cite{MR0288254} (cf. \S\ref{subsec:Royden_cal}). Namely, we first take a smooth exhaustion of the complement of the zeros of the quadratic differntials (which defines the Teichm\"uller Beltrami differential), apply the Green theorem for such domain in the exhaustion, and take the limit.


\chapter{Trivializations of Models of Spaces and Pairings}
\label{chap:trivialization_model}
\section{Guiding frame, Good section for the coboundary operator}
\label{sec:good_section}
Henceforth, for discussing trivializations, we fix a $\mathbb{C}$-linear map, which we call a \emph{guiding frame}, 
$\GuaidF\colon H^1(M_0,\Theta_{M_0})\to L_{(-1,1)}^\infty(M_0)$ such that
\begin{itemize}
\item[(i)]
for $X\in Z^1(\mathcal{U},\Theta_{M_0})$, $\GuaidF([X])$ represents $[X]\in H^1(\mathcal{U},\Theta_{M_0})\cong H^1(M_0,\Theta_{M_0})\cong T_{x_0}\teich_g$. Namely $\mathscr{T}_{x_0}([X])=[\GuaidF([X])]$ in $T_{x_0}\teich_g$; and
\item[(ii)]
for $[X]\in H^1(\mathcal{U},\Theta_{M_0})$,
each $\GuaidF([X])$ is a smooth Beltrami differential. Namely, $\GuaidF([X])\in \Gamma(M_0,\mathcal{A}^{0,1}(\Theta_{M_0}))$.
\end{itemize}

\begin{example}[Ahlfors-Weill guiding frame]
Since the Ahlfors-Weill section $\ahlforsW{x_0}$ is $\mathbb{C}$-linear, the differential $D\ahlforsW{x_0}|_{x_0}$ from $T_0\Bers{x_0}=B_2(\mathbb{D}^*,\Gamma_0)$ to $T_0L^\infty(\mathbb{D},\Gamma_0)=L^\infty(\mathbb{D},\Gamma_0)$ coincides with $\ahlforsW{x_0}$. Hence,
\begin{equation}
\label{eq:ahlfors_weill-frame}
\GuaidF=\ahlforsW{x_0}\circ D\Bersemb_{x_0}|_{x_0}\circ \mathscr{T}_{x_0}
\colon H^1(M_0,\Theta_{M_0})\to L^\infty(\mathbb{D},\Gamma_0)\cong L_{(-1,1)}^\infty(M_0)
\end{equation}
is a guiding frame, where $D\Bersemb_{x_0}|_{x_0}\colon T_{x_0}\teich_g\to T_0\Bers{x_0}=B_2(\mathbb{D}^*,\Gamma_0)$ is the differential of the Bers embedding $\Bersemb_{x_0}$ at $x_0\in \teich_g$ (cf. \S\ref{subsec:Bers_embedding}). 
We call such guiding frame the \emph{Ahlfors-Weill guiding frame}.

Indeed, let $[X]\in H^1(\mathcal{U},\Theta_{M_0})$. each $\GuaidF([X])$ is a smooth Beltrami differential. Since
\begin{align*}
[\GuaidF([X])]
&=\left.D\Bersproj_{x_0}\right|_0\circ \GuaidF([X])
=\left.D\Bersproj_{x_0}\right|_0\circ \ahlforsW{x_0}\circ D\Bersemb_{x_0}|_{x_0}\circ \mathscr{T}_{x_0}([X]) \\
&=D\left.(\Bersproj_{x_0}\circ \ahlforsW{x_0}\right)_{0}\circ D\Bersemb_{x_0}|_{x_0}\circ \mathscr{T}_{x_0}([X]) \\
&= (D\Bersemb_{x_0}|_{x_0})^{-1}\circ D\Bersemb_{x_0}|_{x_0}\circ \mathscr{T}_{x_0}([X])
=\mathscr{T}_{x_0}([X])
\end{align*}
from \eqref{eq:right_inv},
$\GuaidF([X])$ represents $[X]$.
\end{example}

\begin{proposition}[Good section for $\delta$]
\label{prop:linear-map-L}
For a guiding frame $\GuaidF$, there is a $\mathbb{C}$-linear map $\GoodS=\GoodS^{\GuaidF}$ from $Z^1(\mathcal{U},\Theta_{M_0})$ to $C^0(\mathcal{U},\mathcal{A}^{0,0}(\Theta_{M_0}))$ such that
\begin{itemize}
\item[(a)]
$\delta\circ \GoodS(X)=X$ for $X\in Z^1(\mathcal{U},\Theta_{M_0})$; and
\item[(b)]
$\GoodS\circ \delta(\alpha)=\alpha$ for $\alpha\in C^0(\mathcal{U},\Theta_{M_0})$.
\end{itemize}
Especially, the linear map $\GoodS$ is uniquely determined from the guiding frame $\GuaidF$ with the above condition {\rm (a)} and the following condition {\rm (c)}:
\begin{itemize}
\item[(c)]
$\overline{\partial}\GoodS(X)_i=-\GuaidF([X])$ on $U_i$ for $i\in I$.
\end{itemize}
Actually, the condition {\rm (b)} follows from the conditions {\rm (a)} and {\rm (c)}.
\end{proposition}

\begin{proof}
Let $X=\{X_{ij}\}_{i,j}\in Z^1(\mathcal{U},\Theta_{M_0})$. For an index $j$, we take $\xi_j\in \Gamma(U_j,\mathcal{A}^{0,0}(\Theta_{M_0}))$ which satisfies $\overline{\partial}\xi_j=-\GuaidF([X])$ on $U_i$. By the assumption, there is an $\{\alpha_j\}_j\in C^0(\mathcal{U},\Theta_{M_0})$ such that
\begin{equation}
\label{eq:xi_alpha}
\xi_j-\xi_i=X_{ij}+(\alpha_j-\alpha_i)
\end{equation}
on $U_{i}\cap U_j$ 
for all indices $i$, $j\in I$ (cf. \S\ref{subsubsec:1st_cohomology_tangent_space}). 

We claim that the difference $\xi_j-\alpha_j\in \Gamma(U_j,\mathcal{A}^{0,0}(\Theta_{M_0}))$ is independent of the choice of $\xi_j$'s. Indeed, take $\xi'_j\in \Gamma(U_j,\mathcal{A}^{0,0}(\Theta_{M_0}))$ and $\alpha'_j\in \Gamma(U_j,\Theta_{M_0})$ satisfying $\overline{\partial}\xi'_j=-\GuaidF([X])$ on $U_i$ and 
$$
\xi'_j-\xi'_i=X_{ij}+(\alpha'_j-\alpha'_i)
$$
on $U_{i}\cap U_j$ for all $i$, $j\in I$.
Then,
$$
\xi'_j-\xi'_i-(\alpha'_j-\alpha'_i)=X_{ij}=\xi_j-\xi_i+(\alpha_j-\alpha_i)
$$
and
$$
(\xi_j-\alpha_j)-(\xi'_j-\alpha'_j)=(\xi_i-\alpha_i)-(\xi'_i-\alpha'_i)
$$
on $U_{i}\cap U_j$ for all $i$, $j$. Since $\overline{\partial}\xi_j=\overline{\partial}\xi'_j=-\GuaidF([X])\mid_{U_j}$, 
$$
\xi=(\xi_j-\alpha_j)-(\xi'_j-\alpha'_j)
$$
on $U_j$ defines a holomorphic vector field on $M$. Since $\chi(M)<0$, $M$ has no non-trivial holomorphic vector field. Therefore
$$
\xi_j-\alpha_j=\xi'_j-\alpha'_j
$$
for all $j$, and we confirm the claim.

Let us return to the proof of \Cref{prop:linear-map-L}.
We define a $\mathbb{C}$-linear map $\GoodS\colon Z^1(\mathcal{U},\Theta_{M_0})\to C^0(\mathcal{U},\mathcal{A}^{0,0}(\Theta_{M_0}))$ by
$$
\GoodS(X)=\{\xi_j-\alpha_j\}_j.
$$

We check that $\GoodS$ satisfies the desired conditions. Indeed, from \eqref{eq:xi_alpha}, $\GoodS$ satisfies the condition (a). 
%
We check that the condition (b) follows from the conditions (a) and (c). Let $\alpha=\{\alpha_j\}_j\in C^0(\mathcal{U},\Theta_{M_0})$. Since $\GuaidF([\delta\alpha])=\GuaidF(0)=0$, the condition (c) implies that $\GoodS(\delta\alpha)_i\in \Gamma(U_i,\Theta_{M_0})$ for each $i\in I$. From the condition (a), $\delta \circ \GoodS(\delta\alpha)=\delta \alpha$. Hence
$$
\xi=\GoodS(\delta\alpha)_i-\alpha_i
$$
on $U_{i}$ defines a holomorphic vector field on $M_0$, and $\GoodS(\delta\alpha)=\alpha$. Therefore, $\GoodS$ satisfies the condition (b).

We check the uniqueness. Suppose $\GoodS'$ satisfies the conditions (a) and (c).
Consider $\GoodS''\colon Z^1(\mathcal{U},\Theta_{M_0})\to C^0(\mathcal{U},\mathcal{A}^{0,0}(\Theta_{M_0}))$ defined by $\GoodS''(X)=\GoodS(X)-\GoodS'(X)$. From the condition (c), $\GoodS''(X)_i\in \Gamma(U_i,\Theta_{M_0})$ for $i\in I$. From the condition (a), 
$$
\delta\circ \GoodS''(X)=\delta\circ \GoodS(X)-\delta\circ \GoodS'(X)=X-X=0.
$$
Hence $\GoodS''(X)$ defines a global holomorphic vector field on $M_0$. Hence $\GoodS''(X)=0$ for all $X\in Z^1(\mathcal{U},\Theta_{M_0})$.
\end{proof}
%
As \eqref{eq:S-L-zeta1}, we have
\begin{equation}
\label{eq:S-L-zeta}
\delta(S(\GoodS(X),Y))=\zeta(X,Y)
\end{equation}
for $X,Y\in Z^1(\mathcal{U},\Theta_{M_0})$.

\section{Trivialization of model spaces of $T_{[Y]}T\teich_g$}
\label{sec:trivialization_of_model_space}
In the following three sections, we assume that a locally finite covering $\mathcal{U}=\{U_i\}_{i\in I}$ satisfies $H^1(U_i,\Theta_{M_0})=0$ for all $i\in I$. In this section, we will show that the model spaces $\mathbb{T}_Y[\mathcal{U}]$ and $\mathbb{T}_{[Y]}[\mathcal{U}]$ defined in \Cref{def:model_space_for_cocycle,def:model_space} are isomorphic to the product space $H^1(\mathcal{U},\mathcal{O}_{M_0})^{\oplus 2}\cong H^1(M_0,\mathcal{O}_{M_0})^{\oplus 2}$. In particular, in this case,
\begin{equation}
\label{eq:dimention_T_Y}
\dim_{\mathbb{C}}\mathbb{T}_{[Y]}[\mathcal{U}]=\dim_{\mathbb{C}}\mathbb{T}_{Y}[\mathcal{U}]=6g-6
\end{equation}
for $[Y]\in H^1(\mathcal{U},\Theta_{M_0})$ (cf. \Cref{thm:trivialization}). 
Let
$$
\Pi \colon C^1(\mathcal{U},\Theta_{M_0})\to C^1(\mathcal{U},\Theta_{M_0})/\delta C^0(\mathcal{U},\Theta_{M_0})
$$
be the projection. Since $\delta\circ \delta=0$, the coboundary operator $\delta\colon C^1(\mathcal{U},\Theta_{M_0})\to Z^2(\mathcal{U},\Theta_{M_0})$
descends to a $\mathbb{C}$-linear map
$$
\delta\colon
 C^1(\mathcal{U},\Theta_{M_0})/\delta C^0(\mathcal{U},\Theta_{M_0})\to Z^2(\mathcal{U},\Theta_{M_0}).
$$
The existence in the following lemma is trivial since $H^2(M_0,\Theta_{M_0})=0$.
We give and fix a concrete construction of the coboundary.

\begin{lemma}[Coboundary of the primary obstraction]
\label{lem:map_E}
There is a $\mathbb{C}$-linear map
$$
E\colon Z^1(\mathcal{U},\Theta_{M_0})^{\oplus 2}\to C^1(\mathcal{U},\Theta_{M_0})/\delta C^0(\mathcal{U},\Theta_{M_0})
$$
such that
\begin{itemize}
\item[(a)]
$\delta\circ E(X,Y)=\zeta(X,Y)$ for $X,Y\in Z^1(\mathcal{U},\Theta_{M_0})$;
\item[(b)]
For $\alpha\in C^0(\mathcal{U},\Theta_{M_0})$, $\Pi\circ S(\alpha,Y)=E(\delta\alpha,Y)$; and
\item[(c)]
For  $X\in Z^1(\mathcal{U},\Theta_{M_0})$ and $\beta\in C^0(\mathcal{U},\Theta_{M_0})$, 
$$
K'(\beta,X)+\dfrac{1}{2}[X,\delta\beta]-E(X,\delta\beta)=0
$$
modulo $\delta C^0(\mathcal{U},\Theta_{M_0})$.
\item[(d)]
For $X,Y\in Z^1(\mathcal{U},\Theta_{M_0})$,
$E(X,Y)=E(Y,X)$ in $C^1(\mathcal{U},\Theta_{M_0})/\delta C^0(\mathcal{U},\Theta_{M_0})$.
\end{itemize}
\end{lemma}

\begin{proof}
Let $\GoodS\colon Z^1(\mathcal{U},\Theta_{M_0})\to C^0(\mathcal{U},\mathcal{A}^{0,0}(\Theta_{M_0}))$ be a linear map defined in \Cref{prop:linear-map-L}. 
Then
\begin{align*}
\overline{\partial}S(\GoodS(X),Y)_{ij}
&=\dfrac{1}{2}\overline{\partial}[\GoodS(X)_i+\GoodS(X)_j,Y_{ij}] \\
&=\dfrac{1}{2}[\overline{\partial}\GoodS(X)_i+\overline{\partial}\GoodS(X)_j,Y_{ij}] \\
&=-[\GuaidF([X]),Y_{ij}] \\
&=-[\GuaidF([X]),\GoodS(Y)_j]-[\GuaidF([X]),\GoodS(Y)_i].
\end{align*}
Take $\sigma_j\in \Gamma(U_j,\mathcal{A}^{0,0}(\Theta_{M_0}))$ with
$\overline{\partial}\sigma_i=-[\GuaidF([X])\mid_{U_i},\GoodS(Y)_i]$ on $U_i$. We define
\begin{equation}
\label{eq:map_E_1}
E_{ij}(X,Y)=S(\GoodS(X),Y)_{ij}-(\sigma_j-\sigma_i)
\in \Gamma(U_{i}\cap U_{j},\Theta_{M_0})
\end{equation}
for $X$, $Y\in Z^1(\mathcal{U},\Theta_{M_0})$. Then, $E(X,Y)=\{E_{ij}(X,Y)\}_{i,j}$ is determined up to coboundaries in $\delta C^0(\mathcal{U},\Theta_{M_0})$ (which is caused only from the choice of $\sigma=\{\sigma_i\}_i\in C^0(\mathcal{U},\mathcal{A}^{0,0}(\Theta_{M_0}))$). 
From the linearity of $\GoodS$, $E$ defines a $\mathbb{C}$-linear map
$$
E\colon Z^1(\mathcal{U},\Theta_{M_0})^{\oplus 2}
\to C^1(\mathcal{U},\Theta_{M_0})/\delta C^0(\mathcal{U},\Theta_{M_0}).
$$
We check that the map $E$ satisfies the desired conditions.
Indeed, from \eqref{eq:S-L-zeta},
$$
\delta\circ E(X,Y)=\delta \circ T(X,Y)=\delta\circ S(\GoodS(X),Y)=\zeta(X,Y).
$$
Hence $E$ satisfies the condition (a).

Let $\alpha=\{\alpha_j\}_j\in C^0(\mathcal{U},\Theta_{M_0})$. Since $\GuaidF([\delta \alpha])=\GuaidF(0)=0$ and $\overline{\partial}\sigma_j=-[\GuaidF([X]),\GoodS(Y)_j]=0$ on $U_j$, the vector field $\sigma_j$ in \eqref{eq:map_E_1} is a holomorphic vector field on $U_j$ for all $j$. From \Cref{prop:linear-map-L} and \eqref{eq:map_E_1}, 
$$
E(\delta \alpha,Y)\equiv T(\delta \alpha,Y)=S(\GoodS(\delta \alpha),Y)=S(\alpha,Y)
$$
modulo $\delta C^0(\mathcal{U},\Theta_{M_0})$. Therefore, $E$ satisfies the condition (b).

We check (c). Indeed, from the definition of $E$ and (b) of \Cref{prop:linear-map-L}, 
\begin{equation}
\label{eq:definition_E}
E_{ij}(X,\delta\beta)=S(\GoodS(X),\delta\beta)_{ij}-(\sigma'_j-\sigma'_i)
\end{equation}
where $\overline{\partial}\sigma'_j=-[\GuaidF([X]),\beta_j]$ on $U_j$.
By (a) of \Cref{prop:linear-map-L},
\begin{align*}
&K'(\beta,X)_{ij}+\dfrac{1}{2}[X,\delta\beta]_{ij}-S(\GoodS(X),\delta\beta)_{ij}+(\sigma'_j-\sigma'_i)
 \\
&=[\beta_j,\GoodS(X)_j-\GoodS(X)_i]+\dfrac{1}{2}[\GoodS(X)_j-\GoodS(X)_i,\beta_j-\beta_i] \\
&\qquad -\dfrac{1}{2}[\GoodS(X)_i+\GoodS(X)_j,\beta_j-\beta_i]+(\sigma'_j-\sigma'_i) \\
&=([\beta_j,\GoodS(X)_j]+\sigma'_j)-([\beta_i,\GoodS(X)_i]+\sigma'_i).
\end{align*}
Since $[\beta_i,\GoodS(X)_i]+\sigma'_i\in \Gamma(U_i,\Theta_{M_0})$ for $i\in I$,
we have done.

Finally we check (d).
Let $\sigma$, $\sigma'\in C^0(\mathcal{U},\mathcal{A}^{0,0}(\Theta_{M_0}))$ such that $\overline{\partial}\sigma_i=-[\GuaidF([X]),\GoodS(Y)_i]$ and
$\overline{\partial}\sigma'_i=-[\GuaidF([Y]),\GoodS(X)_i]$ on $U_i$.
From the definition \eqref{eq:map_E_1} of $E$,
\begin{align*}
E_{ij}(X,Y)-E_{ij}(Y,X)
&=S(\GoodS(X),Y)-S(\GoodS(Y),X)-(\sigma_j-\sigma'_j)+(\sigma_i-\sigma'_i)\\
&=([\GoodS(X)_j,\GoodS(Y)_j]-(\sigma_j-\sigma'_j))\\
&\quad-([\GoodS(X)_i,\GoodS(Y)_i]-(\sigma_i-\sigma'_i)).
\end{align*}
Since for $i\in I$, 
\begin{align*}
\overline{\partial}
([\GoodS(X)_i,\GoodS(Y)_i]-(\sigma_i-\sigma'_i))
&=-[\GuaidF([X]),\GoodS(Y)_i]-[\GoodS(X)_i,\GuaidF([Y])] \\
&\quad-(-[\GuaidF([X]),\GoodS(Y)_i]+[\GuaidF([Y]),\GoodS(X)_i])=0,
\end{align*}
$E(X,Y)-E(Y,X)\equiv 0$ modulo $\delta C^0(\mathcal{U},\Theta_{M_0})$.
\end{proof}
%

We claim the following:
\begin{theorem}[Trivialization of model spaces]
\label{thm:trivialization}
Let $E$ be the map defined in \Cref{lem:map_E}.
Then, the $\mathbb{C}$-linear map $\trivialization_Y\colon \mathbb{T}_Y[\mathcal{U}] \to H^1(\mathcal{U},\Theta_{M_0})^{\oplus 2}$
 defined by
\begin{equation}
\label{eq:trivialization}
\trivialization_Y\left(
\tv{X,\dot{Y}}{Y}\right)
=\left([X],
\left[\dot{Y}+\dfrac{1}{2}[X,Y]-E(X,Y)\right]
\right) 
\end{equation}
is an isomorphism which satisfies
\begin{equation}
\label{eq:commute_L_Y_beta}
\trivialization_{Y+\delta\beta}\circ \mathcal{L}_{\beta;Y}=\trivialization_Y
\end{equation}
for all $\beta\in C^0(\mathcal{U},\Theta_{M_0})$. 
In particular, the isomorphism $\trivialization_Y$ descends to the isomorphism
$$
\trivialization_{[Y]}\colon \mathbb{T}_{[Y]}[\mathcal{U}]\to  H^1(\mathcal{U},\Theta_{M_0})^{\oplus 2}. 
$$
\end{theorem}

\begin{proof}
We first check that the map \eqref{thm:trivialization} is well-defined. From \Cref{lem:map_E}, the cochain in the bracket in the second coordinate of the map \eqref{thm:trivialization} is cocycle. We check that the restriction of the $\mathbb{C}$-linear map $Z^1(\mathcal{U},\Theta_{M_0})\oplus C^1(\mathcal{U},\Theta_{M_0})
\to H^1(\mathcal{U},\Theta_{M_0})\oplus H^1(\mathcal{U},\Theta_{M_0})$ defined by
$$
(X,\dot{Y})\mapsto
\left([X],
\left[\dot{Y}+\dfrac{1}{2}[X,Y]-E(X,Y)\right]
\right) 
$$
to the image of $D^Y_0$ (defined in \eqref{eq:D_0}) is trivial. Let $\alpha$, $\beta\in C^0(\mathcal{U},\Theta_{M_0})$. The first coordinate of the image of $D^Y_0(\alpha,\beta)$ under the above map is trivial because of the definition of the cohomology group $H^1(\mathcal{U},\Theta_{M_0})$.
For the second coordinate,
\begin{align*}
&
\left(\delta \beta+K(\alpha,Y)+\dfrac{1}{2}[\delta\alpha,Y]-E(\delta\alpha,Y)
\right)_{ij} 
\\
&\equiv \beta_j-\beta_i+[\alpha_i,Y_{ij}]+\dfrac{1}{2}[\alpha_j-\alpha_i,Y_{ij}]-
\dfrac{1}{2}[\alpha_i+\alpha_j,Y_{ij}]
=\beta_j-\beta_i\equiv 0
\end{align*}
modulo $\delta C^0(\mathcal{U},\Theta_{M_0})$. Hence, \eqref{eq:trivialization} is well-defined.

We construct the inverse of the map \eqref{thm:trivialization}. For $X$, $W\in Z^1(\mathcal{U},\Theta_{M_0})$, we define 
$$
\dot{W}=W-\dfrac{1}{2}[X,Y]+E(X,Y).
$$
Then, from \Cref{lem:map_E}, 
\begin{align*}
\delta\left(\dot{W}+\dfrac{1}{2}[X,Y]\right)-\zeta(X,Y) &=
\delta\left(W+E(X,Y)\right)-\zeta(X,Y)=0.
\end{align*}
Hence we have a well-defined $\mathbb{C}$-linear map
\begin{equation}
\label{eq:inverse}
(X,W)\mapsto
\tv{
X,
W-\dfrac{1}{2}[X,Y]+E(X,Y)}{Y}
\end{equation}
from $Z^1(\mathcal{U},\Theta_{M_0})\oplus Z^1(\mathcal{U},\Theta_{M_0})$ to $\mathbb{T}_Y$. We check that the image of the direct sum of two copies of $\delta C^0(\mathcal{U},\Theta_{M_0})$ under \eqref{eq:inverse} is trivial. Let $\alpha$, $\beta\in C^0(\mathcal{U},\Theta_{M_0})$. When we substitute $\delta\alpha$ and $\delta \beta$ to $X$ and $W$ in \eqref{eq:inverse}, the $ij$-component of the second coordinate is equal to
\begin{align*}
&\beta_j-\beta_i-\dfrac{1}{2}[\alpha_j-\alpha_i,Y_{ij}]+E(\delta \alpha,Y)_{ij} \\
&\equiv \beta_j-\beta_i-\dfrac{1}{2}[\alpha_j-\alpha_i,Y]+\dfrac{1}{2}[\alpha_i+\alpha_j,Y_{ij}]
=\beta_j-\beta_i+K(\alpha,Y)_{ij}
\end{align*}
modulo $\delta C^0(\mathcal{U},\Theta_{M_0})$,
which implies that the image of $(\delta \alpha,\delta\beta)$ under the map \eqref{eq:inverse} is zero in $\mathbb{T}_Y$.
Hence, the map \eqref{eq:inverse} descends to the $\mathbb{C}$-linear map 
from $H^1(\mathcal{U},\Theta_{M_0})^{\oplus 2}$ to $\mathbb{T}_Y$. From the definition, this map is the inverse of \eqref{thm:trivialization}.

\Cref{eq:commute_L_Y_beta} follows from (c) of \Cref{lem:map_E}.
\end{proof}

\subsection{Trivialization for Dolbeault type presentation}
We continue to use the linear maps $\GuaidF$, $\GoodS$ and $E$ in \Cref{prop:linear-map-L} and \Cref{lem:map_E}.
Let $Y\in Z^1(\mathcal{U},\Theta_{M_0})$.
We define a $\mathbb{C}$-linear map 
$\tilde{\trivializationD}_Y\colon \ker(D^{Y,1}_0) \to \Gamma(\mathcal{U},\mathcal{A}^{0,1}(\Theta_{M_0}))^{\oplus 2}$ by
\begin{equation}
\label{eq:definition_Xi}
\tilde{\trivializationD_Y}(\mu,\dot{\nu})=(\GuaidF([X]), \hat{\nu})
\end{equation}
where $[X]\in H^1(\mathcal{U},\Theta_{M_0})$ with $[\GuaidF([X])]=[\mu]$ (cf. \eqref{eq:Teichmuller_lemma}) and the differential $\hat{\nu}\in \Gamma(\mathcal{U},\mathcal{A}^{0,1}(\Theta_{M_0}))$ in \eqref{eq:definition_Xi} is defined by
\begin{equation}
\label{eq:definition_hat_nu}
\hat{\nu} = 
(\dot{\nu}_i-\overline{\partial}[\chi,\GoodS(Y)_i])+[\GuaidF([X]),\GoodS(Y)_i]
\end{equation}
on $U_i$ where $\chi\in \Gamma(M_0,\mathcal{A}^{0,0}(\Theta_{M_0}))$ with $\overline{\partial}\chi=\GuaidF([X])-\mu$.
Since
\begin{align*}
[\chi,\GoodS(Y)_j]-[\chi,\GoodS(Y)_i]
+[-\chi,Y_{ij}]=0
\end{align*}
on $U_i\cap U_j$,
from \Cref{prop:equivalence_class},
$(\mu,\dot{\nu})$ and $(\GuaidF([X]),\{\dot{\nu}_i-\overline{\partial}[\chi ,\GoodS(Y)_i]\}_{i\in I})$ are
equivalent in $\mathbb{T}^{Dol}_Y[\mathcal{U}]$.
We claim
\begin{theorem}[Trivialization for Dolbeault type presentation]
\label{thm:trivialization_Dol}
The map $\tilde{\trivializationD_Y}$ descends to a $\mathbb{C}$-linear isomorphism
$\trivializationD_Y\,\colon\, \mathbb{T}^{Dol}_Y \to (T_{x_0}\teich_g)^{\oplus 2}$
which commutes the following diagram
\begin{equation}
\label{eq:trivialization_model_spaces}
\begin{CD}
\mathbb{T}_Y[\mathcal{U}] @<{\connectinghomo}<<  \mathbb{T}^{Dol}_Y[\mathcal{U}] \\
@V{\trivialization_Y}VV @VV{\trivializationD_Y}V \\
H^1(\mathcal{U},\Theta_{M_0})^{\oplus 2}
@>{\mathscr{T}_{x_0}\oplus \mathscr{T}_{x_0}}>>
(T_{x_0}\teich_g)^{\oplus 2},
\end{CD}
\end{equation}
where $\connectinghomo$ is the connecting homomorphism defined at \eqref{eq:exact_T_Y}, and $\mathscr{T}_{x_0}$ is the isomorphism defined at \eqref{eq:identify_H1_T_g}. Furthermore, in the commutative diagram \eqref{eq:trivialization_model_spaces},
the vertical spaces satisfy
\begin{equation}
\label{eq:trivialization_vertical_spaces}
\begin{CD}
\mathbb{T}^{V}_Y[\mathcal{U}] @<{\left.\connectinghomo\right|_{\mathbb{T}^{Dol,V}_Y}}<<  \mathbb{T}^{Dol,V}[\mathcal{U}]_Y \\
@V{\left.\trivialization_Y\right|_{\mathbb{T}_Y^{V}}=\{[0]\}\oplus \verticalincmodel{Y}:\mathcal{U}}VV @VV{\left.\trivializationD_Y\right|_{\mathbb{T}^{Dol}_Y}}V \\
\{[0]\}\oplus H^1(\mathcal{U},\Theta_{M_0})
@>{\mathscr{T}_{x_0}\oplus \mathscr{T}_{x_0}}>>
\{[0]\}\oplus T_{x_0}\teich_g.
\end{CD}
\end{equation}
\end{theorem}

\begin{proof}
Since
\begin{align*}
\overline{\partial}[\chi,\GoodS(Y)_j]
-
\overline{\partial}[\chi,\GoodS(Y)_i]
&=\overline{\partial}[\chi,Y_{ij}] 
=[\overline{\partial}\chi,Y_{ij}] \\
&=[\GuaidF([X])-\mu,Y_{ij}]
\end{align*}
on $U_i\cap U_j$,
we have
\begin{align*}
0
&=\dot{\nu}_j-\dot{\nu}_i+[\mu,Y_{ij}] \\
&=(\dot{\nu}_j-\overline{\partial}[\chi,\GoodS(Y)_j])-
(\dot{\nu}_i-\overline{\partial}[\chi,\GoodS(Y)_i])
+[\GuaidF([X]),Y_{ij}]
\end{align*}
on $U_i\cap U_j$ for $i$, $j\in I$.
Let $\beta\in C^0(\mathcal{U},\mathcal{A}^{0,0}(\Theta_{M_0}))$ with
$-\overline{\partial}\beta_i=\dot{\nu}_i-\overline{\partial}[\chi,\GoodS(Y)_i]$ on $U_i$,
and $\dot{Y}=\delta\beta+K(\GoodS(X),Y)$.
Then,
$$
\connectinghomo\left(\tvD{\mu,\dot{\nu}}{Y}\right)=
\connectinghomo\left(\tvD{\GuaidF([X]),\{\dot{\nu}_i-\overline{\partial}[\chi,\GoodS(Y)_i]\}_{i\in I}}{Y}\right)=\tv{X,\dot{Y}}{Y}
$$
from the definition of the connecting homomorphism.
On the other hand, 
\begin{align*}
&\dot{Y}_{ij}+\dfrac{1}{2}[X_{ij},Y_{ij}]-E(X,Y)_{ij} \\
&=\beta_j-\beta_i+[\GoodS(X)_i,Y_{ij}]+\dfrac{1}{2}[\GoodS(X)_j-\GoodS(X)_i,\GoodS(Y)_j-\GoodS(Y)_i] \\
&\quad -\dfrac{1}{2}[\GoodS(X)_j+\GoodS(X)_i,\GoodS(Y)_j-\GoodS(Y)_i]+\sigma_j-\sigma_i \\
&=(\beta_j+\sigma_j)-(\beta_i+\sigma_i)
\end{align*}
modulo $\delta C^0(\mathcal{U},\Theta_{M_0})$,
where $\sigma_i\in \Gamma(U_i,\mathcal{A}^{0,0}(\Theta_{M_0}))$ satisfies $-\overline{\partial}\sigma_i=[\GuaidF([X]),\GoodS(Y)_i]$ (cf. \eqref{eq:map_E_1}).
Since
$$
-\overline{\partial}(\beta_i+\sigma_i)=
(\dot{\nu}_i-\overline{\partial}[\chi,\GoodS(Y)_i])+[\GuaidF([X]),\GoodS(Y)_i]=\hat{\nu}_i
$$
on $U_i$,
$$
\mathscr{T}_{x_0}\left(
\left[
\dot{Y}+\dfrac{1}{2}[X,Y]-E(X,Y)
\right]
\right)
=[\hat{\nu}]\in T_{x_0}\teich_g.
$$
As a consequence,
$(\mathscr{T}_{x_0}\oplus \mathscr{T}_{x_0})\circ \trivialization_Y\circ \connectinghomo\left(\tvD{\mu,\dot{\nu}}{Y}\right)=([\mu],[\hat{\nu}])$, which
is nothing but the equivalence class of $\tilde{\trivializationD}_Y(\mu,\dot{\nu})$
in $T_{x_0}\teich_g\oplus T_{x_0}\teich_g$.

The commutative diagram \eqref{eq:trivialization_vertical_spaces} follows from the diagram \eqref{eq:trivialization_model_spaces}, the definition of $E$ (\eqref{eq:K-bra-S}, \eqref{eq:definition_E}), the definitions of the vertical spaces (\Cref{def:vertical_space_TY,def:vertical_space_Dol}),
and \Cref{prop:Dol_vertical_space_isomorphism}.
\end{proof}

\begin{remark}
\label{remark:1}
The map $E$ defined in \Cref{lem:map_E} are dependent on the choices of the linear map $\GuaidF\colon H^1(M_0,\Theta_{M_0})\to L_{(-1,1)}^\infty(M_0)$.
Therefore, so are the trivializations in \cref{thm:trivialization,thm:trivialization_Dol}.
\end{remark}

\section{Trivialization of the model space of $T^*_{[Y]}T\teich_g$}
\label{sec:trivialization_of_model_space_T*}
Fix a guiding frame $\GuaidF\colon H^1(M_0,\Theta_{M_0})\to L_{(-1,1)}^\infty(M_0)$ and take $\GoodS\colon Z^1(\mathcal{U},\Theta_{0})\to C^0(\mathcal{U},\mathcal{A}^{0,0}(\Theta_{M_0}))$ defined in \Cref{prop:linear-map-L}.

Let $\tvdag{\psi,q}{[Y]}\in \mathbb{T}^\dagger_{[Y]}[\mathcal{U}]$.
Let $Y\in Z^1(\mathcal{U},\Theta_{M_0})$ and $(\psi,q)=(\{\psi_i\}_{i\in I}, q)\in \ker(D^{Y,\dagger}_1)$ be representatives of $[Y]$ and $\tvdag{\psi,q}{[Y]}$ respectively.
Since
$$
\psi_j-\psi_i=L_{Y_{ij}}(q)=L_{\GoodS(Y)_j}(q)-L_{\GoodS(Y)_i}(q)
$$
for $i$, $j\in I$, $\psi-L_{\GoodS(Y)}(q)=\{\psi_i-L_{\GoodS(Y)_i}(q)\}_{i\in I}\in H^0(\mathcal{U},\mathcal{A}^{0,0}(\Omega_{M_0}^{\otimes 2}))$.

The exact sequence
$$
\begin{CD}
0 @>>> \Omega_{M_0}^{\otimes 2} @>{inc}>>
\mathcal{A}^{0,0}(\Omega_{M_0}^{\otimes 2})
@>{\overline{\partial}}>>
\mathcal{A}^{0,1}(\Omega_{M_0}^{\otimes 2})
@>>> 0
\end{CD}
$$
leads an exact sequence
\begin{equation}
\label{eq:exact_cotangent_smooth_holo}
\minCDarrowwidth10pt
\begin{CD}
@.H^0(\mathcal{U},\Omega_{M_0}^{\otimes 2})
@>>>H^0(\mathcal{U},\mathcal{A}^{0,0}(\Omega_{M_0}^{\otimes 2}))
@>{\overline{\partial}}>>
H^0(\mathcal{U},\mathcal{A}^{0,1}(\Omega_{M_0}^{\otimes 2}))
\\
@>>> H^1(\mathcal{U},\Omega_{M_0}^{\otimes 2}).
\end{CD}
\end{equation}
Since $\delta L_{\GoodS(Y)}(q)=L_Y(q)\in C^1(\mathcal{U},\Omega_{M_0}^{\otimes 2})$, $\overline{\partial} L_{\GoodS(Y)}(q)\in H^0(\mathcal{U},\mathcal{A}^{0,1}(\Omega_{M_0}^{\otimes 2}))$.
Since $H^1(\mathcal{U},\Omega_{M_0}^{\otimes 2})=0$,
there is a unique smooth quadratic differential $Q=\{Q_i\}_{i\in I}\in H^0(\mathcal{U},\mathcal{A}^{0,0}(\Omega_{M_0}^{\otimes 2}))$ which satisfies
\begin{equation}
\label{eq:condition_Q}
\overline{\partial}Q_i=\overline{\partial}L_{\GoodS(Y)_i}(q)
\end{equation}
for $i\in I$,
and the \emph{uniqueness condition}
\begin{equation}
\label{eq:uniqueness_lift_Q}
\iint_{M_0}\GuaidF([Z])Q=\dfrac{1}{2i}\iint_{M_0}\GuaidF([Z])(z)Q(z)d\overline{z}\wedge dz=0
\end{equation}
for all $[Z]\in H^1(M_0,\Theta_{M_0})$.
The uniqueness of such $Q$ follows from the fact that $\mathcal{Q}_{M_0}$ is the dual of $T_{x_0}\teich_g$ and 
$$
T_{x_0}\teich_g\cong H^1(M_0,\Theta_{M_0})\ni [Z]\mapsto \iint_{M_0}\GuaidF([Z])Q\in \mathbb{C}
$$
is a $\mathbb{C}$-linear functional.
We claim:
\begin{theorem}[Trivialization]
\label{thm:trivialization_cotangent_to_tangent}
Under the above notation,
\begin{equation}
\label{eq:trivialization_t_dagger}
\trivialization^{\dagger}_{[Y]}
\colon
\mathbb{T}^\dagger_{[Y]}[\mathcal{U}]\ni 
\tvdag{\psi,q}{[Y]}
\mapsto (\{\psi_i-L_{\GoodS(Y)_i}(q)+Q\}_{i\in I},q)\in \mathcal{Q}_{M_0}\oplus \mathcal{Q}_{M_0}
\end{equation}
is a $\mathbb{C}$-isomorphism.
\end{theorem}

\begin{proof}
Since $\dim_{\mathbb{C}}\mathbb{T}^\dagger_{[Y]}[\mathcal{U}]=2\dim_{\mathbb{C}}\mathcal{Q}_{M_0}=6g-6$,
it suffices to show that the map \eqref{eq:trivialization_t_dagger} is injective.

Suppose that $\tvdag{\psi,q}{[Y]}$ is in the kernel of the map \eqref{eq:trivialization_t_dagger}.
Let $(\psi,q)\in \ker(D^{Y,\dagger}_1)$ be the representative of $\tvdag{\psi,q}{[Y]}$.
From the second coordinate, we have $q=0$. From \eqref{eq:condition_Q}, $\overline{\partial}Q=\overline{\partial}\psi=0$. Hence, $Q$ is holomorphic.
However, from the uniqueness condition \eqref{eq:uniqueness_lift_Q}, we have $Q=0$,
and hence $\psi=0$. As a conclusion, we obtain $\tvdag{\psi,q}{[Y]}=0$.
\end{proof}

\section{Presentations of the pairings under trivializations}
\label{sec:presentation_pairing_trivialization}
Let $x_0=(M_0,f_0)\in \teich_g$ and $[Y]\in H^1(M_0,\Theta_{M_0})\cong T_{x_0}\teich_g$.
Let $\mathcal{U}=\{U_i\}_{i\in I}$ be a locally finite covering of $M_0$ with $H^1(U_i,\Theta_{M_0})=0$ for $i\in I$.
Fix a guiding frame  $\GuaidF\colon H^1(\mathcal{U},\Theta_{M_0})\to L^\infty_{(-1,1)}(M_0)$
and the good section $\GoodS\colon Z^1(\mathcal{U},\Theta_{M_0})\to C^0(\mathcal{U},\mathcal{A}^{0,0}(\Theta_{M_0}))$ in \Cref{prop:linear-map-L}.
Consider the trivializations $\trivialization_{[Y]}$ on $\mathbb{T}_{[Y]}[\mathcal{U}]$ 
in \Cref{thm:trivialization}
and 
$\trivialization^{\dagger}_{[Y]}$ on $\mathbb{T}^\dagger_{[Y]}[\mathcal{U}]$
in \Cref{thm:trivialization_cotangent_to_tangent}
defined from the guiding frame $\GuaidF$ and the good section $\GoodS$.

Let $\tv{X,\dot{Y}}{Y}\in \mathbb{T}_Y[\mathcal{U}]$ and 
$\tvdag{\psi,q}{[Y]}\in \mathbb{T}^\dagger_{[Y]}[\mathcal{U}]$.
Let $(X,\dot{Y})$ and $(\psi,q)$ be the representatives.
We continue to use the notation in \S\ref{sec:trivialization_of_model_space} and \S\ref{sec:trivialization_of_model_space_T*} frequently.
In the following discussion, we notice that though $E(X,Y)$ is determined up to $C^0(\mathcal{U},\Theta_{M_0})$, the ambiguity does not affect in the calculation.

For $i,j\in I$, 
\begin{align*}
&X_{ij}(\psi_i-L_{\GoodS(Y)_i}(q)+Q)+
\left(\dot{Y}_{ij}+\dfrac{1}{2}[X_{ij},Y_{ij}]-E(X,Y)_{ij}\right)q \\
&=X_{ij}\psi_i-\dot{Y}_{ji}q+L_{Y_{ij}}(\GoodS(X)_jq) \\
&\quad
-X_{ij}L_{\GoodS(Y)_i}(q)+X_{ij}Q-\dfrac{1}{2}[X_{ij},Y_{ij}]q-E(X,Y)_{ij}q-L_{Y_{ij}}(\GoodS(X)_jq)
\end{align*}
and the last five terms become
\begin{align*}
&-X_{ij}L_{\GoodS(Y)_i}(q)+X_{ij}Q-\dfrac{1}{2}[X_{ij},Y_{ij}]q-E(X,Y)_{ij}q-L_{Y_{ij}}(\GoodS(X)_jq)
\\
&=-(\GoodS(X)_j-\GoodS(X)_i)L_{\GoodS(Y)_i}(q)+(\GoodS(X)_j-\GoodS(X)_i)Q\\
&\quad
-\dfrac{1}{2}[\GoodS(X)_j-\GoodS(X)_i,\GoodS(Y)_j-\GoodS(Y)_i]q
-\dfrac{1}{2}[\GoodS(X)_j+\GoodS(X)_i,\GoodS(Y)_j-\GoodS(Y)_i]q \\
&\quad +(\sigma_j-\sigma_i)q-L_{\GoodS(Y)_j-\GoodS(Y)_i}(\GoodS(X)_jq)
\\
&=-\GoodS(X)_jL_{\GoodS(Y)_i}(q)+\GoodS(X)_iL_{\GoodS(Y)_i}(q)+\GoodS(X)_jQ-\GoodS(X)_iQ\\
&\quad
-[\GoodS(X)_j,\GoodS(Y)_j]q+[\GoodS(X)_j,\GoodS(Y)_i]q \\
&\quad +(\sigma_jq-\sigma_iq-L_{\GoodS(Y)_j}(\GoodS(X)_jq)-L_{\GoodS(Y)_i}(\GoodS(X)_jq)
\\
&=
(\GoodS(X)_jQ+\sigma_jq-\GoodS(X)_jL_{\GoodS(Y)_j}(q))-(\GoodS(X)_iQ+\sigma_iq-\GoodS(X)_iL_{\GoodS(Y)_i}(q))
\\
&\quad
-\GoodS(X)_jL_{\GoodS(Y)_i}(q)+[\GoodS(X)_j,\GoodS(Y)_i]q+L_{\GoodS(Y)_i}(\GoodS(X)_jq)
\\
&\quad
+\GoodS(X)_jL_{\GoodS(Y)_j}(q)-[\GoodS(X)_j,\GoodS(Y)_j]q-L_{\GoodS(Y)_j}(\GoodS(X)_jq) \\
&=
(\GoodS(X)_jQ+\sigma_jq-\GoodS(X)_jL_{\GoodS(Y)_j}(q))
-(\GoodS(X)_iQ+\sigma_iq-\GoodS(X)_iL_{\GoodS(Y)_i}(q))
\end{align*}
from \eqref{eq:Lie-derivative_1}
since $Q$ and $q$ are global sections on $M_0$.
Hence, for $\Omega=\{\Omega_i\}_{i\in I}\in C^0(\mathcal{U},\mathcal{A}^{0,0}(\Omega_{M_0}))$ with
$$
\Omega_j-\Omega_i=
X_{ij}(\psi_i-L_{\GoodS(Y)_i}(q)+Q)+
\left(\dot{Y}_{ij}+\dfrac{1}{2}[X_{ij},Y_{ij}]-E(X,Y)_{ij}\right)q,
$$
we define $A=\{A_i\}_{i\in I}\in C^0(\mathcal{U},\mathcal{A}^{0,0}(\Omega_{M_0}))$
by 
$$
A_i=\Omega_i+(-\GoodS(X)_iQ-\sigma_iq+\GoodS(X)_iL_{\GoodS(Y)_i}(q)).
$$
on $U_i$.
Then, $A=\{A_i\}_{i\in I}$ satisfies
$$
A_j-A_i=X_{ij}\psi_i-\dot{Y}_{ji}q+L_{Y_{ij}}(\GoodS(X)_jq)
$$
and
\begin{align*}
&\overline{\partial}A_i-L_{\GoodS(Y)_i}((\GoodS(X)_i)_{\overline{z}}q)
\\
&=\overline{\partial}\Omega_i+\GuaidF([X])Q
-\GoodS(X)_i\overline{\partial}(L_{\GoodS(Y)_i}(q))
-[\GuaidF([X]),\GoodS(Y)_i]]q
\\
&\quad
+\GuaidF([X])L_{\GoodS(Y)_i}(q)+\GoodS(X)_i\overline{\partial}(L_{\GoodS(Y)_i}(q))
-L_{\GoodS(Y)_i}(\GuaidF([X])q)
\\
&=\overline{\partial}\Omega_i+\GuaidF([X])Q.
\end{align*}
from \eqref{eq:Lie-derivative_04}.
Therefore, we deduce
\begin{align*}
&-\dfrac{1}{2i}\iint_{M_0}((A_i)_{\overline{z}}(z)-L_{\GoodS(Y)_i}((\GoodS(X)_i)_{\overline{z}}q)(z))d\overline{z}\wedge dz \\
&= -\dfrac{1}{2i}\iint_{M_0}((\Omega_i)_{\overline{z}}(z)+\GuaidF([X])(z)Q(z))d\overline{z}\wedge dz \\
&=-\dfrac{1}{2i}\iint_{M_0}d\Omega_i
\end{align*}
from the uniqueness condition \eqref{eq:uniqueness_lift_Q} of $Q$.
Thus,
from \eqref{eq:pairing_2}, we obtain the following.

\begin{theorem}[Trivialization and Pairing]
\label{thm:Trivialization_pairing_TTT_g}
Under the above notation,
for $\tv{X,\dot{Y}}{Y}\in \mathbb{T}_Y[\mathcal{U}]$ and 
$\tvdag{\psi,q}{[Y]}\in \mathbb{T}^\dagger_{[Y]}[\mathcal{U}]$,
we set
\begin{align*}
\trivialization_{[Y]}\left(\tv{X,\dot{Y}}{Y}\right)
&=([X],[Z])\in H^1(\mathcal{U},\Theta_{M_0})^{\oplus 2} \\
(\mathscr{T}_{x_0}\oplus \mathscr{T}_{x_0})(([X],[Z]))
&=([\mu],[\hat{\nu}])
\in T_{x_0}\teich_g\oplus T_{x_0}\teich_g \\
\trivialization^{\dagger}_{[Y]}\left(\tvdag{\psi,q}{[Y]}\right)
&=(\varphi,q)\in \mathcal{Q}_{x_0}\oplus \mathcal{Q}_{x_0}.
\end{align*}
Then,
\begin{align*}
\mathcal{P}_{\mathbb{TT}}\left(
\tv{X,\dot{Y}}{[Y]},\tvdag{\psi, q}{[Y]}
\right)
&=-\pi\left(
{\rm Res}([\{X_{ij}\varphi\}_{i,j\in I}])+{\rm Res}([\{Z_{ij}q\}_{i,j\in I}])
\right) \\
&=\langle [\mu],\varphi\rangle+\langle [\hat{\nu}],q\rangle.
\end{align*}
\end{theorem}

%
%
%
%
%

\section[Trivialization and Pairing]{Trivialization of the model space of $T_{q_0}T^*\teich_g$ and $T^{*}_{q_0}T^*\teich_g$, and the Pairing under the trivializations}
\label{sec:trivialization_TTstar_TstarTstar}
As well as the previous sections,
we fix a guiding frame $\GuaidF\colon H^1(M_0,\Theta_{M_0})\to L_{(-1,1)}^\infty(M_0)$,
and take $\GoodS$ as \Cref{prop:linear-map-L}.
For $[X,\varphi]\in {\bf H}^1(\mathcal{U},\mathbb{L}_{q_0})$
and $[\Phi,Y]\in {\bf H}^{1,\dagger}(\mathcal{U},\mathbb{L}_{q_0})$, we take $Q$,
$Q'\in H^0(\mathcal{U},\mathcal{A}^{0,0}(\Omega_{M_0}^{\otimes 2}))$ satisfying
\begin{align*}
\overline{\partial}Q_i
&=\overline{\partial}(\varphi-L_{\GoodS(X)_i}(q_0))=\overline{\partial}(L_{\GoodS(X)_i}(q_0))
\\
\overline{\partial}Q'_i
&=\overline{\partial}(\Phi_i+L_{\GoodS(Y)_i}(q_0))=-\overline{\partial}(L_{\GoodS(Y)_i}(q_0))
\end{align*}
on $U_i$
and
$$
\dfrac{1}{2i}\iint_{M_0}\GuaidF([Z])Qd\overline{z}\wedge dz=\dfrac{1}{2i}\iint_{M_0}\GuaidF([Z])Q'd\overline{z}\wedge dz=0
$$
for all $[Z]\in H^1(M_0,\Theta_{M_0})$ as \S\ref{sec:trivialization_of_model_space_T*}. 
We claim
\begin{theorem}[Trivialization]
\label{thm:Trivialization_T_Tstar}
Under the above notation,
\begin{align*}
\trivialization^{*}_{q_0}
&\colon
{\bf H}^1(\mathcal{U},\mathbb{L}_{q_0})\ni [X,\varphi]_{q_0}
\mapsto ([X], \{\varphi_i-L_{\GoodS(X)_i}(q_0)+Q\}_{i\in I})\in H^1(\mathcal{U},\Theta_{M_0})\oplus \mathcal{Q}_{M_0}
\\
\trivialization^{*\dagger}_{q_0}
&\colon
{\bf H}^{1,\dagger}(\mathcal{U},\mathbb{L}_{q_0})\ni 
[\Phi, Y]_{q_0}\mapsto (\Phi+L_{\GoodS(Y)}(q_0)+Q',[Y])\in \mathcal{Q}_{M_0}\oplus H^1(\mathcal{U},\Theta_{M_0})
\end{align*}
are $\mathbb{C}$-isomorphisms.
\end{theorem}

\begin{proof}
We only check that $\trivialization^{*}_{q_0}$ is $\mathbb{C}$-isomorphic. Since the dimensions of the both sides are same, it suffices to show that the map is injective.
Suppose that $ [X,\varphi]_{q_0}$ is in the kernel of the map.
Since $[X]=0$, $X=\delta\alpha$ for some $\alpha\in C^0(\mathcal{U},\Theta_{M_0})$.
From \Cref{prop:linear-map-L}, $\GoodS(X)=\GoodS(\delta\alpha)=\alpha$. Thus,
$$
\overline{\partial}Q=-\overline{\partial}(\psi-L_{\alpha}(q_0))=0
$$ 
and hence $Q=0$ and $\varphi=L_\alpha(q_0)$.
Thus, we obtain
$$
[X,\varphi]_{q_0}=[\delta\alpha,L_{\alpha}(q_0)]_{q_0}=0
$$
and we have done.
\end{proof}

We define the pairing $\paircot$ between models ${\bf H}^1(\mathcal{U},\mathbb{L}_{q_0})$
and ${\bf H}^{1,\dagger}(\mathcal{U},\mathbb{L}_{q_0})$, and the symplectic form $\omega_{\mathcal{Q}_g}$ on the model ${\bf H}^1(\mathcal{U},\mathbb{L}_{q_0})$
at \eqref{eq:pairing_cotangent} and \Cref{prop:holomorphic_symplectic_form}.
We claim

\begin{theorem}[Pairing and Symplectic form under the trivializations]
\label{thm:pairing_cotangent_over_cotangent}
Under the above notation, for $[X,\varphi]_{q_0}$, $[X',\varphi']_{q_0}\in {\bf H}^1(\mathcal{U},\mathbb{L}_{q_0})$
and $[\Phi,Y]_{q_0}\in {\bf H}^{1,\dagger}(\mathcal{U},\mathbb{L}_{q_0})$, we set
\begin{align*}
\trivialization^{*}_{q_0}([X,\varphi]_{q_0})
&=([X],\psi)=(\{[\{X_{ij}\}_{i,j\in I}], \{\psi_i\}_{i\in I})
\in H^1(\mathcal{U},\Theta_{M_0})\oplus \mathcal{Q}_{x_0} \\
\trivialization^{*}_{q_0}([X',\varphi']_{q_0})
&=([X'],\psi')=(\{[\{X'_{ij}\}_{i,j\in I}], \{\psi'_i\}_{i\in I})
\in H^1(\mathcal{U},\Theta_{M_0})\oplus \mathcal{Q}_{x_0}
\\
\trivialization^{*\dagger}_{q_0}([\Phi,Y]_{q_0})
&=(\Psi,[Y])=(\{\Psi_i\}_{i\in I}, [\{Y_{ij}\}_{i,j\in I}])\in \mathcal{Q}_{x_0} \oplus H^1(\mathcal{U},\Theta_{M_0})
\end{align*}
and 
$\mathscr{T}_{x_0}([X])=[\mu]$, $\mathscr{T}_{x_0}([X'])=[\mu']$,
$\mathscr{T}_{x_0}([Y])=[\nu]\in T_{x_0}\teich_g$. 
Then
\begin{align*}
\paircot([X,\varphi],[\Phi,Y])
&=-\pi\left({\rm Res}([\{X_{ij}\Psi_i\}_{i,j\in I}])+{\rm Res}([\{Y_{ij}\psi_i\}_{i,j\in I}])\right)
\\
&=\langle [\mu],\Psi\rangle+\langle [\nu],\psi\rangle
\\
\omega_{\mathcal{Q}_g}([X,\varphi],[X',\varphi'])
&=\dfrac{\pi}{2}
\left({\rm Res}([\{X_{ij}\psi'_i\}_{i,j\in I}])-{\rm Res}([\{X'_{ij}\psi_i\}_{i,j\in I}])\right)
\\
&=-\dfrac{1}{2}(\langle [\mu],\psi'\rangle-\langle [\mu'],\psi\rangle).
\end{align*}
\end{theorem}

\begin{proof}
We only check the formula for the pairing. The formula for the symplectic form is also deduced by the similar argument.
We use the notation given above frequently.

We define $\{\Omega_i\}_{i\in I}\in C^0(\mathcal{U},\mathcal{A}^{0,0}(\Omega_{M_0}))$ by
$$
\Omega_i=\GoodS(X)_{i}\Psi_i+\GoodS(Y)_i\psi_i.
$$
Then,
$$
\Omega_j-\Omega_i=X_{ij}\Psi_i+Y_{ij}\psi_i
$$
and
\begin{align*}
&-\pi\left({\rm Res}([\{X_{ij}\Psi_i\}_{i,j\in I}])+{\rm Res}([\{Y_{ij}\psi_i\}_{i,j\in I}])\right)
\\
&=-\dfrac{1}{2i}\iint_{M_0}d\Omega_i
=-\dfrac{1}{2i}\iint_{M_0}(\GuaidF([X])\Psi_i+\GuaidF([Y])\psi_i) \\
&
=-\dfrac{1}{2i}\iint_{M_0}(\GuaidF([X])(\Phi_i+L_{\GoodS(Y)_i}(q_0)+Q')+\GuaidF([Y])(\varphi_i-L_{\GoodS(X)_i}(q_0)+Q)) \\
&=\paircot([X,\varphi],[\Phi,Y])
\end{align*}
from the uniqueness condition for $Q$ and $Q'$.
%
%
\end{proof}

\section{Dualities in model spaces}
\label{sec:duals_in_model_spaces}
As we discussed in \S\ref{sec:infinitesimal_deformation_RS} and \S\ref{subsec:residue},
the pairing defined by the residue
$$
H^1(\mathcal{U},\Theta_{M_0})\times \mathcal{Q}_{x_0}\ni ([X],\varphi)\mapsto
-\pi{\rm Res}([\{X_{ij}\varphi\}_{i,j\in I}])\in\mathbb{C}
$$
is non-degenerate, and gives a duality between $H^1(\mathcal{U},\Theta_{M_0})$
and $\mathcal{Q}_{x_0}$. Under the identification $H^1(\mathcal{U},\Theta_{M_0})$
with the tangent space $T_{x_0}\teich_g$, the duality gives a recognition of $\mathcal{Q}_{x_0}$ as the cotangent space $T_{x_0}^*\teich_g$.

Let $[Y]\in H^1(\mathcal{U},\Theta_{M_0})$ and $q_0\in \mathcal{Q}_{x_0}$.
As we observed that the models of pairings $\mathcal{P}_{\mathbb{TT}}$ between $\mathbb{T}_{[Y]}[\mathcal{U}]$ and $\mathbb{T}^\dagger_{[Y]}[\mathcal{U}]$, and $\paircot$ between ${\bf H}^{1}(\mathcal{U},\mathbb{L}_{q_0})$ and  ${\bf H}^{1,\dagger}(\mathcal{U},\mathbb{L}_{q_0})$ are non-generate.
Hence, these pairing makes $\mathbb{T}^\dagger_{[Y]}[\mathcal{U}]$ and ${\bf H}^{1,\dagger}(\mathcal{U},\mathbb{L}_{q_0})$ as the dual spaces of $\mathbb{T}_{[Y]}[\mathcal{U}]$ and ${\bf H}^{1}(\mathcal{U},\mathbb{L}_{q_0})$, respectively. 

Recall from Linear algebra, for vector spaces $V$ and $W$, a natural isomorphism $(V\oplus W)^*\cong V^*\oplus W^*$ of the dual spaces of the direct sum is induced by the pairing 
$$
(V\oplus W)\times (V^*\oplus W^*)\ni ((v,w),(v^*,w^*))\mapsto v^*(v)+w^*(w)
$$
(e.g. \cite[II.5, p.90]{MR1009162}).

We now fix a guiding frame $\GuaidF$. The we obtain the isomorphisms
\begin{align*}
\mathbb{T}_{[Y]}[\mathcal{U}]
&\to H^1(\mathcal{U},\Theta_{M_0})\oplus H^1(\mathcal{U},\Theta_{M_0}) \\
\mathbb{T}^\dagger_{[Y]}[\mathcal{U}]
&\to H^1(\mathcal{U},\Theta_{M_0})\oplus \mathcal{Q}_{x_0}
\end{align*}
via the trivializations defined from the guiding frame $\GuaidF$.
From \Cref{thm:Trivialization_pairing_TTT_g} and \Cref{thm:pairing_cotangent_over_cotangent}, the trivializations of the model spaces are natural in terms of the pairings. 

As a consequence, the trivializations for $\mathbb{T}^\dagger_{[Y]}[\mathcal{U}]$ and ${\bf H}^{1,\dagger}(\mathcal{U},\mathbb{L}_{q_0})$ defined in \Cref{thm:trivialization_cotangent_to_tangent} and \Cref{thm:Trivialization_T_Tstar} defined from the guiding frame $\GuaidF$ naturally coincide with the isomorphisms of the dual spaces
\begin{align*}
\mathbb{T}^\dagger_{[Y]}[\mathcal{U}]
&\to \mathcal{Q}_{x_0}\oplus \mathcal{Q}_{x_0}
=(H^1(\mathcal{U},\Theta_{M_0})\oplus H^1(\mathcal{U},\Theta_{M_0}))^* \\
\mathbb{T}^\dagger_{[Y]}[\mathcal{U}]
&\to \mathcal{Q}_{x_0}\oplus H^1(\mathcal{U},\Theta_{M_0})
=
(H^1(\mathcal{U},\Theta_{M_0})\oplus \mathcal{Q}_{x_0})^*
\end{align*}
via the pairiings.

%

%
%
%
%
%
%
%
%
%
%
%
%


\chapter{Direct limits}
\label{chap:direct_limits}
In this section, we define $\mathbb{C}$-vector spaces $\mathbb{T}_{[Y]}$ and $\mathbb{T}^\dagger_{[Y]}$ for a cohomology class $[Y]\in H^1(M_0,\Theta_{M_0})$ by taking the direct limits, in a similar manner as the definition of the cohomologies of sheaves on spaces (cf. \cite[\S12.5]{MR648106} or \cite[\S3.3]{MR815922}). This spaces $\mathbb{T}_{[Y]}$ and $\mathbb{T}^\dagger_{[Y]}$ are thought of the (ideal) model spaces of the double tangent space and the cotangent space at $[Y]\in H^1(M_0,\Theta_{M_0})\cong T_{x_0}\teich_g$ defined independently of the choice of locally finite coverings.

As we already discussed in \S\ref{subsec:tangent_space_to_T_Tstar}, the direct limit  ${\bf H}^1(\mathbb{L}_{q_0})$ of the model of the tangent space $T_{q_0}\mathcal{Q}_g$ is already discussed by Hubbard and Masur in \cite{MR523212}. The direct limit ${\bf H}^{1,\dagger}(\mathbb{L}_{q_0})$ of the model  of the cotangent space $T^*_{q_0}\mathcal{Q}_g$ is also treated in the same way. Moreover, by the standard argument, we see that the pairing  and the holomorphic symplectic forms also descends to the direct limits (cf. \S\ref{sec:direct_limit_cotangent_space}). Hence, we treat only here the infinitesimal spaces over $T\teich_g$.

\section{Direct limit of the model space $\mathbb{T}_{[Y]}[\mathcal{U}]$}
\subsection{Refinements and induced homomorphisms}
Let $\mathcal{U}=\{U_i\}_{i\in I}$ be a locally finite covering on $M_0$ and $\mathcal{V}=\{V_\lambda\}_{j\in \Lambda}$ be a refinement of $\mathcal{U}$. In this case, we denote by $\mathcal{V}\prec \mathcal{U}$.
For a sheaf $\mathscr{S}$ on $M_0$, let $\Pi^{\mathcal{U}}_{\mathcal{V}}\colon C^k(\mathcal{U},\mathscr{S})\to C^k(\mathcal{V},\mathscr{S})$ be the homomorphism induced by the restriction (cf. \cite[\S3.3, Lemma 3.2]{MR815922}).
The map is defined with a \emph{refining map} $k\colon \Lambda\to I$ between indices with $V_\lambda\subset U_{k(\lambda)}$, and by restricting sections on $U_{k(\lambda)}$ to $V_\lambda$. The map $\Pi^{\mathcal{U}}_{\mathcal{V}}$ is extended to a homomorphism between products of the groups of cochains and the cohomology groups. 
For the simplicity of the notation, we use the same symbol to denote the extension and the induced homomorphism between cohomology groups. One can easily check that
$$
D^Y_q\circ \Pi^{\mathcal{U}}_{\mathcal{V}}=\Pi^{\mathcal{U}}_{\mathcal{V}}\circ D^Y_q,
\quad
D^{Y;s}_q\circ \Pi^{\mathcal{U}}_{\mathcal{V}}=\Pi^{\mathcal{U}}_{\mathcal{V}}\circ D^{Y;s}_q
$$
for $s,q=0,1$.  
We first claim

\begin{lemma}
\label{lem:commute_refinement}
The map $\Pi^{\mathcal{U}}_{\mathcal{V}}$ induces a canonical homomorphism $\check{\Pi}^{\mathcal{U}}_{\mathcal{V}}\colon \mathbb{T}_{[Y]}[\mathcal{U}]\to \mathbb{T}_{{[\Pi^{\mathcal{U}}_{\mathcal{V}}(Y)]}}[\mathcal{V}]$,
which are uniquely defined by $\mathcal{U}$ and $\mathcal{V}$.
\end{lemma}

\begin{proof}
Since the map $\Pi^{\mathcal{U}}_{\mathcal{V}}$ is defined by restricting, 
if the refining map $k\colon \Lambda \to I$ 
is fixed, the map $\Pi^{\mathcal{U}}_{\mathcal{V}}$ (defined by $k$) defines well-defined linear maps
\begin{equation}
\label{eq:refine_ok}
\mathbb{T}_Y[\mathcal{U}]\to \mathbb{T}_{\Pi^{\mathcal{U}}_{\mathcal{V}}(Y)}[\mathcal{V}],\quad
\mathbb{T}_{[Y]}[\mathcal{U}]\to \mathbb{T}_{[\Pi^{\mathcal{U}}_{\mathcal{V}}(Y)]}[\mathcal{V}].
\end{equation}

Let $k, j\colon \Lambda\to I$ be maps with $V_\lambda\subset U_{k(\lambda)}\cap U_{j(\lambda)}$ for $\lambda\in \Lambda$. 
Let $(X,\dot{Y})\in \ker(D^Y_1)$.
We define $Z$, $Z'$, $W$, $W'\in C^1(\mathcal{V},\Theta_{M_0})$, $\dot{W}$, $\dot{W}'\in C^1(\mathcal{V},\Theta_{M_0})$, $\alpha$, $\beta$, $\dot{\gamma}\in C^0(\mathcal{V},\Theta_{M_0})$ by
\begin{itemize}
\item $Z_{\lambda\mu}=X_{k(\lambda)k(\mu)}|_{V_{\lambda}\cap V_{\mu}}$,
$Z'_{\lambda\mu}=X_{j(\lambda)j(\mu)}|_{V_{\lambda}\cap V_{\mu}}$,
\item
$W_{\lambda\mu}=Y_{k(\lambda)k(\mu)}|_{V_{\lambda}\cap V_{\mu}}$,
$W'_{\lambda\mu}=Y_{j(\lambda)j(\mu)}|_{V_{\lambda}\cap V_{\mu)}}$,
\item
$\dot{W}_{\lambda\mu}=\dot{Y}_{k(\lambda)k(\mu)}|_{V_{\lambda}\cap V_{\mu}}$,
$\dot{W}'_{\lambda\mu}=\dot{Y}_{j(\lambda)j(\mu)}|_{V_{\lambda}\cap V_{\mu}}$,
\item
$\alpha_{\lambda}=X_{k(\lambda)j(\lambda)}|_{V_{\lambda}}$,
$\beta_{\lambda}=Y_{k(\lambda)j(\lambda)}|_{V_{\lambda}}$,
\item
$\dot{\gamma}_{\lambda}=\dot{Y}_{k(\lambda)j(\lambda)}|_{V_{\lambda}}$, and
$\dot{\gamma}^*_{\lambda}=\dot{Y}_{j(\lambda)k(\lambda)}|_{V_{\lambda}}$
\end{itemize}
for $\lambda,\mu\in \Lambda$.
Then, it is easy to see that $(Z,\dot{W})\in \ker(D^{W}_1)$ and $(Z',\dot{W}')\in \ker(D^{W'}_1)$.
It is known that $[W]=[W']=[\Pi^{\mathcal{U}}_{\mathcal{V}}(Y)]$ in $H^1(\mathcal{V},\Theta_{M_0})$ (cf. \cite[Lemma 3.2]{MR815922}).

We will show
\begin{equation}
\label{eq:well-defined-refinement}
\begin{cases}
Z'=Z+\delta \alpha,\quad W'=W+\delta\beta
\\
\dot{W}'
=(\dot{W}+K'(\beta,Z))+\delta\dot{\gamma}+K(\alpha,W').
\end{cases}
\end{equation}
Indeed, \eqref{eq:well-defined-refinement} is equivalent to
$$
(Z',\dot{W}')=\tilde{L}_{-\beta;W}(Z,\dot{W})+D^{W'}_0(\alpha,\dot{\gamma})
$$
and hence $\tv{Z',\dot{W}'}{W'}$ is equivalent to $\tv{Z,\dot{W}}{W}$ in $\mathbb{T}_{{[\Pi^{\mathcal{U}}_{\mathcal{V}}(Y)]}}[\mathcal{V}]$ by \Cref{prop:isomorphism_tile_L}. As a consequence, the induced linear map $\mathbb{T}_{[Y]}[\mathcal{U}]\to \mathbb{T}_{[\Pi^{\mathcal{U}}_{\mathcal{V}}(Y)]}[\mathcal{V}]$
is independent of the choice of the map $k\colon \Lambda\to I$ with $V_\lambda\subset U_{k(\lambda)}$.

We proceed to the proof of  \eqref{eq:well-defined-refinement}.
Since $X$ and $Y$ are cocycles,
$$
Z'=Z+\delta\alpha, \quad W'=W+\delta\beta.
$$
Since $(X,\dot{Y})\in \ker(D^Y_1)$, $(Z,\dot{W})$ and $(Z',\dot{W}')$ satisfy
\begin{equation}
\label{eq:refinement_0}
\dot{\gamma}_\lambda+\dot{\gamma}_{\lambda}^*+[\alpha_\lambda,\beta_\lambda]=0
\end{equation}
and
\begin{align}
\dot{W}_{\lambda\mu}+\dot{Y}_{k(\mu)j(\lambda)}+\dot{\gamma}^*_\lambda &
+\dfrac{1}{2}[Z_{\lambda\mu},W_{\lambda\mu}]+\dfrac{1}{2}[X_{k(\mu)j(\lambda)},Y_{k(\mu)j(\lambda)}]+
\dfrac{1}{2}[-\alpha_\lambda,-\beta_\lambda] 
\label{eq:refinement_1}
\\
&-\dfrac{1}{2}[Z_{\lambda\mu},Y_{k(\mu)j(\lambda)}]-\dfrac{1}{2}[W_{\lambda\mu},X_{k(\mu)j(\lambda)}]=0
\nonumber
\\
\dot{\gamma}^*_{\mu}+\dot{Y}_{k(\mu)j(\lambda)}+\dot{W}'_{\lambda\mu} &
+\dfrac{1}{2}[-\alpha_\mu,-\beta_\mu]
+\dfrac{1}{2}[X_{k(\mu)j(\lambda)},Y_{k(\mu)j(\lambda)}] 
+\dfrac{1}{2}[Z'_{\lambda\mu},W'_{\lambda\mu}]
\label{eq:refinement_2}\\
&-\dfrac{1}{2}[-\alpha_{\mu},Y_{k(\mu)j(\lambda)}]-\dfrac{1}{2}[-\beta_{\mu},X_{k(\mu)j(\lambda)}]=0
\nonumber
\end{align}
for $\lambda$, $\mu\in \Lambda$. 
Since
\begin{align*}
X_{k(\mu)j(\lambda)}&=\alpha_\mu-Z'_{\lambda\mu}=\alpha_\lambda-Z_{\lambda\mu} \\
Y_{k(\mu)j(\lambda)}&=\beta_\mu-W'_{\lambda\mu}=\beta_\lambda-W_{\lambda\mu},
\end{align*}
by subtracting
\eqref{eq:refinement_2} from \eqref{eq:refinement_1}, we obtain
\begin{align*}
0
&=
\dot{W}_{\lambda\mu}-\dot{W}'_{\lambda\mu}
+\left(\dot{\gamma}^*_\lambda 
+\dfrac{1}{2}[\alpha_\lambda,\beta_\lambda]
\right)
-
\left(\dot{\gamma}^*_\mu
+\dfrac{1}{2}[\alpha_\mu,\beta_\mu]
\right)
+\dfrac{1}{2}[Z_{\lambda\mu},W_{\lambda\mu}] \\
&\quad
-\dfrac{1}{2}[Z'_{\lambda\mu},W'_{\lambda\mu}]
+\dfrac{1}{2}[\beta_\lambda,Z_{\lambda\mu}]+\dfrac{1}{2}[\alpha_\lambda,W_{\lambda\mu}]
+\dfrac{1}{2}[\alpha_\mu,W'_{\lambda\mu}]+\dfrac{1}{2}[\beta_\mu,Z'_{\lambda\mu}].
\end{align*}
The last six terms of the right-hand side becomes
\begin{align*}
&-[Z_{\lambda\mu},\beta_\mu]+[\alpha_\lambda,W'_{\lambda\mu}]-
\dfrac{1}{2}[\alpha_\mu,\beta_\mu]+\dfrac{1}{2}[\alpha_{\lambda},\beta_{\lambda}] \\
&=K'(\beta,Z)_{\lambda\mu}+K(\alpha,W')_{\lambda\mu}-
\dfrac{1}{2}[\alpha_\mu,\beta_\mu]+\dfrac{1}{2}[\alpha_{\lambda},\beta_{\lambda}].
\end{align*}
From \eqref{eq:refinement_0}, 
we obtain \eqref{eq:well-defined-refinement}.
%
%
\end{proof}

We claim the following (Compare \cite[Theorem 3.4]{MR815922}).
\begin{lemma}
\label{lem:injectivity_refinement}
Let $\mathcal{U}$ and $\mathcal{V}$ be locally finite coverings of $M_0$ with $\mathcal{V}\prec \mathcal{U}$.
Let $\check{\Pi}^{\mathcal{U}}_{\mathcal{V}}\colon \mathbb{T}_{[Y]}[\mathcal{U}]\to \mathbb{T}_{[\Pi^{\mathcal{U}}_{\mathcal{V}}(Y)]}[\mathcal{V}]$ be the induced homomorphism. Then
\begin{itemize}
\item[{\rm (a)}]
$\check{\Pi}^{\mathcal{U}}_{\mathcal{V}}$ maps 
$\mathbb{T}^V_{[Y]}[\mathcal{U}]$ to 
$\mathbb{T}^V_{[\Pi^{\mathcal{U}}_{\mathcal{V}}(Y)]}[\mathcal{V}]$; and
\item[{\rm (b)}]
$\check{\Pi}^{\mathcal{U}}_{\mathcal{V}}$ is injective.
\end{itemize}
\end{lemma}

\begin{proof}
Let $\mathcal{U}=\{U_i\}_{i\in I}$ and $\mathcal{V}=\{V_\lambda\}_{\lambda\in \Lambda}$. We denote by $k\colon \Lambda\to  I$ the map defining the induced map $\Pi^{\mathcal{U}}_{\mathcal{V}}$. Let $W=\Pi^{\mathcal{U}}_{\mathcal{V}}(Y)$.

\medskip
\noindent
(a)
Let $\tv{X,\dot{Y}}{Y}\in \mathbb{T}^V_{Y}[\mathcal{U}]$. There is an $\alpha\in C^0(\mathcal{U},\Theta_{M_0})$ such that $X=\delta\alpha$. Then, 
$\Pi^{\mathcal{U}}_{\mathcal{V}}(X)=\delta(\Pi^{\mathcal{U}}_{\mathcal{V}}(\alpha))$
and hence
$$
\tv{\Pi^{\mathcal{U}}_{\mathcal{V}}(X),\Pi^{\mathcal{U}}_{\mathcal{V}}(\dot{Y})}{Y}
=\tv{\delta(\Pi^{\mathcal{U}}_{\mathcal{V}}(\alpha)),\Pi^{\mathcal{U}}_{\mathcal{V}}(\dot{Y})}{Y}
\in \mathbb{T}^V_{\Pi^{\mathcal{U}}_{\mathcal{V}}(Y)}[\mathcal{V}].
$$
\quad

\medskip
\noindent
(b)
\quad
Let $(X,\dot{Y})\in \ker (D^Y_1)$ with $\check{\Pi}^{\mathcal{U}}_{\mathcal{V}}(\tv{X,\dot{Y}}{[Y]})=0$ in $\mathbb{T}_{[\Pi^{\mathcal{U}}_{\mathcal{V}}(Y)]}[\mathcal{V}]$. Take $\alpha$, $\beta\in C^0(\mathcal{V},\Theta_{M_0})$ such that
$$
\Pi^{\mathcal{U}}_{\mathcal{V}}(X)=\delta \alpha, \quad
\Pi^{\mathcal{U}}_{\mathcal{V}}(\dot{Y})=\delta\beta+K(\alpha,W).
$$
For $i\in I$, 
$$
X_{k(\lambda)i}+X_{ik(\mu)}=X_{k(\lambda)k(\mu)}=\alpha_\mu-\alpha_\lambda
$$
and $X_{k(\lambda)i}+\alpha_\lambda=X_{k(\mu)i}+\alpha_\mu$ on $U_i\cap V_\lambda\cap V_\mu$. Hence
$A=\{A_i\}_{i\in I}\in C^0(\mathcal{U},\Theta_{M_0})$ is defined by 
$$
A_i=X_{k(\lambda)i}+\alpha_\lambda
$$
on $U_i\cap V_\lambda$ for $i\in I$. The cochain $A$ satisfies
$$
A_j-A_i=(X_{k(\lambda)j}+\alpha_\lambda)-(X_{k(\lambda)i}+\alpha_\lambda)=X_{ij}
$$
on $U_i\cap U_j\cap V_\lambda$ for $i,j\in I$ and $\lambda\in \Lambda$. This mean that $\delta A=X$.
Notice from the definition that $A_{k(\lambda)}=X_{k(\lambda)k(\lambda)}+\alpha_\lambda=\alpha_\lambda$ on $V_\lambda= U_{k(\lambda)}\cap V_\lambda$ for $\lambda\in \Lambda$.

Since $\delta A=X$,
from the definition of $D^W_1$ (see \eqref{eq:D_1}) and \eqref{eq:K-bra-S}, we obtain
\begin{align*}
0&=\dot{Y}_{k(\lambda)i}+\dot{Y}_{ik(\lambda)}+[X_{ik(\lambda)},Y_{ik(\lambda)}] \\
&=(\dot{Y}_{k(\lambda)i}-[A_{k(\lambda)},Y_{k(\lambda)i}])+
(\dot{Y}_{ik(\lambda)}-[A_{i},Y_{ik(\lambda)}])\\
0&=\left(\delta\left(\Pi^{\mathcal{U}}_{\mathcal{V}}(\dot{Y})+\dfrac{1}{2}[\Pi^{\mathcal{U}}_{\mathcal{V}}(X),\Pi^{\mathcal{U}}_{\mathcal{V}}(Y)]\right)-\zeta(\Pi^{\mathcal{U}}_{\mathcal{V}}(X),\Pi^{\mathcal{U}}_{\mathcal{V}}(Y))\right)_{ik(\lambda)k(\mu)}
\\
&=\dot{Y}_{ik(\lambda)}+\dot{Y}_{k(\lambda)k(\mu)}+\dot{Y}_{k(\mu)i}
-[A_i, Y_{ik(\lambda)}]
-[A_{k(\lambda)}, Y_{k(\lambda)k(\mu)}]-[A_{k(\mu)},Y_{k(\mu)i}] \\
&=(\dot{Y}_{ik(\lambda)}-[A_i, Y_{ik(\lambda)}])+(\beta_\mu-\beta_\lambda+[A_{k(\lambda)},Y_{k(\lambda)k(\mu)}]))-(\dot{Y}_{ik(\mu)}-[A_i, Y_{ik(\mu)}]) \\
&\quad-[A_{k(\lambda)}, Y_{k(\lambda)k(\mu)}] \\
&=(\dot{Y}_{ik(\lambda)}-[A_i, Y_{ik(\lambda)}]-\beta_\lambda)-(\dot{Y}_{ik(\mu)}-[A_i, Y_{ik(\mu)}]-\beta_\mu). 
\end{align*}
on $U_i\cap V_\lambda\cap V_\mu$
for $\lambda$, $\mu\in \Lambda$ and $i\in I$.
Therefore, $B=\{B_i\}_{i\in I}\in C^0(\mathcal{U},\Theta_{M_0})$ is defined by
$$
B_i=-\dot{Y}_{ik(\lambda)}+[A_i, Y_{ik(\lambda)}]+\beta_\lambda
$$
on $U_i\cap V_\lambda$, and the cochain $B$ satisfies
\begin{align*}
(\delta B)_{ij}
&=-(\dot{Y}_{jk(\lambda)}-[A_j, Y_{jk(\lambda)}]-\beta_\lambda)+(\dot{Y}_{ik(\lambda)}-[A_i, Y_{ik(\lambda)}]-\beta_\lambda) \\
&=(\dot{Y}_{k(\lambda)j}-[A_{k(\lambda)}, Y_{jk(\lambda)}])+(\dot{Y}_{ik(\lambda)}-[A_i, Y_{ik(\lambda)}])
\\
&=-(\dot{Y}_{ji}-[A_j,Y_{ji}])=\dot{Y}_{ij}-[A_i,Y_{ij}].
\end{align*}
on $U_i\cap U_j\cap V_\lambda$. Therefore
$$
\dot{Y}=\delta B+K(A,Y),
$$
and $D^Y_0(A,B)=(X,\dot{Y})$.
Thus, we obtain $\tv{X,\dot{Y}}{[Y]}=0$ in $\mathbb{T}_{[Y]}[\mathcal{U}]$.
\end{proof}

%
%

\subsection{The direct limit $\mathbb{T}_{[Y]}$ for $[Y]\in H^1(M_0,\Theta_{M_0})$}
\label{subsec:inductice_limit_def}
It is known that a canonical homomorphism $\Pi^{\mathcal{U}}\colon H^1(\mathcal{U},\Theta_{M_0})\to H^1(M_0,\Theta_{M_0})$ defined from the direct limit via refinements is injective (cf. \cite[Theorem 3.4]{MR815922}). For the simplicity of notation, for a cocycle $Y\in Z^1(\mathcal{U},\Theta_{M_0})$, we denote by $[Y]$ the cohomology class of $Y$ in both $H^1(\mathcal{U},\Theta_{M_0})$ and $H^1(M_0,\Theta_{M_0})$.

Let $\mathcal{U}$, $\mathcal{V}$, and $\mathcal{W}$ be locally finite coverings of $M_0$
with $\mathcal{W}\prec\mathcal{V}\prec \mathcal{U}$. It is easy to see that the homomorphism $\check{\Pi}^{\mathcal{U}}_{\mathcal{V}}$ in \Cref{lem:commute_refinement} satisfies
that $\check{\Pi}^{\mathcal{U}}_{\mathcal{U}}=id$ and
$$
\check{\Pi}^{\mathcal{V}}_{\mathcal{W}}\circ \check{\Pi}^{\mathcal{U}}_{\mathcal{V}}=\check{\Pi}^{\mathcal{U}}_{\mathcal{W}}.
$$
Therefore, the \emph{direct limit}
$$
\mathbb{T}_{[Y]}=\varinjlim_{\mathcal{U}}\mathbb{T}_{[Y]}[\mathcal{U}]
$$
is well-defined (in this case, the totality of the locally finite coverings of $M_0$ is thought of as a directed set such that $\mathcal{U}$ precedes $\mathcal{V}$ if $\mathcal{V}$ is a refinement of $\mathcal{U}$). From \Cref{lem:injectivity_refinement}, the induced homomorphism
$$
\check{\Pi}^{\mathcal{U}}\colon \mathbb{T}_{[Y]}[\mathcal{U}]\to \mathbb{T}_{[Y]}
$$
is injective (cf. \cite[II, Proposition 0.3]{MR842190}).

\subsection{Structure of the direct limit $\mathbb{T}_{[Y]}$}
\label{subsec:structure_direct_limit}
From \Cref{prop:vertical_space}, \eqref{eq:K-bra-S}, (b) of \Cref{lem:map_E}, \Cref{thm:trivialization}, we notice the following.

\begin{proposition}
\label{prop:structure_T_Y_U}
Let $\alpha\in C^0(\mathcal{U},\Theta_{M_0})$.
Then,
$$
\minCDarrowwidth20pt
\begin{CD}
0 @>>> H^1(\mathcal{U},\Theta_{M_0})@>>> \mathbb{T}_Y[\mathcal{U}] @>>> H^1(\mathcal{U},\Theta_{M_0}) \\
@. [W] @ >>> \tv{\delta\alpha,W+K(\alpha,Y)}{Y} @. \\
@. @. \tv{X,\dot{Y}}{Y} @>>> [X]
\end{CD}
$$
is exact. In particular, the dimension of $\mathbb{T}_Y[\mathcal{U}] $ is at most the twice of that of $H^1(\mathcal{U},\Theta_{M_0})$. In addition, when $H^1(U_i,\Theta_{M_0})=0$ for all $i$,
$$
\minCDarrowwidth20pt
\begin{CD}
0 @>>> H^1(\mathcal{U},\Theta_{M_0})@>>> \mathbb{T}_Y[\mathcal{U}] @>>> H^1(\mathcal{U},\Theta_{M_0})
@>>>0
\end{CD}
$$
is exact.
\end{proposition}
We discuss the exact sequences in \Cref{prop:structure_T_Y_U} for various locally finite coverings. Let $\mathcal{U}$ and $\mathcal{V}$ are locally finite coverings with $\mathcal{V}\prec \mathcal{U}$. From (a) of \Cref{lem:injectivity_refinement}, the following diagram is commutative:
$$
\minCDarrowwidth20pt
\begin{CD}
0 
@>>>
H^1(\mathcal{U},\Theta_{M_0})
@<{\verticalincmodel{[Y]}}<{\cong}<
\mathbb{T}^V_{[Y]}[\mathcal{U}]
@>>> 
\mathbb{T}_{[Y]}[\mathcal{U}] 
@>>> 
H^1(\mathcal{U},\Theta_{M_0}) 
\\
@.
@V{\Pi^{\mathcal{U}}_{\mathcal{V}}}VV
@V{\check{\Pi}^{\mathcal{U}}_{\mathcal{V}}}VV
@V{\check{\Pi}^{\mathcal{U}}_{\mathcal{V}}}VV
@V{\Pi^{\mathcal{U}}_{\mathcal{V}}}VV
\\
0 
@>>> 
H^1(\mathcal{V},\Theta_{M_0})
@<{\verticalincmodel{[Y]}}<{\cong}<
\mathbb{T}^V_{[Y]}[\mathcal{V}] 
@>>> 
\mathbb{T}_{[Y]}[\mathcal{V}]
@>>>
H^1(\mathcal{V},\Theta_{M_0}) .
\end{CD}
$$
Therefore, 
there is a $\mathbb{C}$-linear isomorphism
$$
\verticalincmodel{[Y]}\colon \mathbb{T}^V_{[Y]}\to H^1(M_0,\Theta_{M_0})
$$
which satisfies the following commutative diagram
\begin{equation}
\label{eq:direct_limit-exact1}
\minCDarrowwidth20pt
\begin{CD}
0
@>>>
H^1(\mathcal{U},\Theta_{M_0})
@<{\verticalincmodel{[Y]}}<{\cong}<
\mathbb{T}^V_{[Y]}[\mathcal{U}]
@>>> 
\mathbb{T}_{[Y]}[\mathcal{U}] @>>> H^1(\mathcal{U},\Theta_{M_0}) 
\\
@.
@V{\Pi^{\mathcal{U}}}VV
@V{\check{\Pi}^{\mathcal{U}}}VV
@V{\check{\Pi}^{\mathcal{U}}}VV
@V{\Pi^{\mathcal{U}}}VV
\\
0
@>>>
H^1(M_0,\Theta_{M_0})
@<{\verticalincmodel{[Y]}}<{\cong}<
\mathbb{T}^V_{[Y]}
@>>>
\mathbb{T}_{[Y]}
@>>>
H^1(M_0,\Theta_{M_0}).
\end{CD}
\end{equation}
When $H^1(U_i,\Theta_{M_0})=0$ for all $i\in I$, from \eqref{eq:dimention_T_Y} and \Cref{prop:structure_T_Y_U},  
$$
\dim_{\mathbb{C}} \mathbb{T}_{[Y]}=6g-6,
$$
and the last arrows in the two horizontal sequences in \eqref{eq:direct_limit-exact1} are surjective.

\subsection{Direct limit of Dolbeault type presentation}
In this section, we discuss the direct limit of the Dolbeaut-type presentations. We only deal with the case of $\{\mathbb{T}^{Dol}_{[Y]}[\mathcal{U}]\}_{\mathcal{U}}$. The direct limit of $\{\mathbb{T}^{Bel}_{[\nu]}[\mathcal{U}]\}_{\mathcal{U}}$ is obtained in the similar manner.

Though the connecting homomorphism $\connectinghomo$ might not be defined for arbitrary coverings, we can define the connecting homomorphism
$$
\connectinghomo\colon \mathbb{T}^{Dol}_{[Y]}:=
\varinjlim_{\mathcal{U}}\mathbb{T}^{Dol}_{[Y]}[\mathcal{U}]\to  \mathbb{T}_{[Y]}
$$
between the direct limits since any covering $\mathcal{U}=\{U_i\}_{i\in I}$ admits a refinement $\mathcal{V}=\{V_j\}_{j\in J}$ with $H^1(V_j,\Theta_{M_0})=0$ for all $j\in J$. Notice from \Cref{prop:Dol_vertical_space_isomorphism} and \Cref{lem:commute_refinement,lem:injectivity_refinement}, the linear map
$$
\check{\Pi}^{\mathcal{U}}_{\mathcal{V}}\colon \mathbb{T}^{Dol}_{[Y]}[\mathcal{U}]\to \mathbb{T}^{Dol}_{[\Pi^{\mathcal{U}}_{\mathcal{V}}(Y)]}[\mathcal{V}]
$$
for locally finite coverings $\mathcal{U}=\{U_i\}_{i\in I}$ and $\mathcal{V}=\{V_j\}_{j\in J}$ of $M_0$ with $\mathcal{V}\prec \mathcal{U}$ and $H^1(U_i,\Theta_{M_0})=0$, $H^1(V_j,\Theta_{M_0})=0$ for $i\in I$ and $j\in J$ is well-defined and injective (in fact, it is isomorphic).

From \Cref{lem:Dol_base_change} and \Cref{coro:Dol_tangent_connecing_homo}, we deduce the following.

\begin{theorem}[Direct limit and Connecting homomorphism]
\label{thm:Dol_direct_limit}
The connecting homomorphism $\connectinghomo$ induces an isomorphism $\connectinghomo\colon \mathbb{T}^{Dol}_{[Y]}\to \mathbb{T}_{[Y]}$ which satisfies the following commutative diagram:
$$
\begin{CD}
\mathbb{T}^{Dol}_{[Y]}[\mathcal{U}] @>{\connectinghomo}>> \mathbb{T}_{[Y]}[\mathcal{U}] \\
@V{\check{\Pi}^{\mathcal{U}}}VV @V{\check{\Pi}^{\mathcal{U}}}VV \\
\mathbb{T}^{Dol}_{[Y]}@>{\connectinghomo}>> \mathbb{T}_{[Y]}
\end{CD}
$$
for a locally finite covering $\mathcal{U}=\{U_i\}_{i\in I}$ with $H^1(U_i,\Theta_{M_0})=0$ for $i\in I$. In addition, the commutative diagram respects the vertical spaces.
\end{theorem}
%

\subsection{Structure of Dolbeault type presentation}
From \eqref{eq:vertical_isomorphism_cohomology_class_dol},
the following is well-defined and exact for $\alpha\in H^0(\mathcal{U},\mathcal{A}^{0,0}(\Theta_{M_0}))$:
\begin{equation}
\label{eq:exact_Dol}
\minCDarrowwidth10pt
\begin{CD}
0@>>>
T_{x_0}\teich_g @>{\cong}>> \mathbb{T}^{Dol,V}_Y[\mathcal{U}] @>{inc}>>\mathbb{T}^{Dol}_Y[\mathcal{U}]@>>> T_{x_0}\teich_g @>>> 0\\
@.[\dot{\nu}] @>>> \tvD{-\overline{\partial}\alpha ,\dot{\nu}+\overline{\partial}[\alpha,\eta_i]}{Y} \\
@.@. @. [\tvD{\mu,\dot{\nu}}{Y} @>>> [\mu]
\end{CD}
\end{equation}
and the following diagram is commute:
\begin{equation}
\label{eq:commute_Dol}
\minCDarrowwidth20pt
\begin{CD}
T_{x_0}\teich_g @>{\cong}>> 
\mathbb{T}^{Dol,V}_Y[\mathcal{U}]@>{inc}>>\mathbb{T}^{Dol}_Y[\mathcal{U}]@>{onto}>> T_{x_0}\teich_g \\
@A{\mathscr{T}_{x_0}}AA @V{\connectinghomo}VV @V{\connectinghomo}VV @A{\mathscr{T}_{x_0}}AA \\
H^1(\mathcal{U},\Theta_{M_0}) @>{\cong}>> 
\mathbb{T}^{V}_Y[\mathcal{U}]@>{inc}>>\mathbb{T}_Y[\mathcal{U}]@>{onto}>> H^1(\mathcal{U},\Theta_{M_0}),
\end{CD}
\end{equation}
where the first horizontal line is the exact sequence \eqref{eq:exact_Dol} and ``inc" means the inclusion.

\section{Direct limit of $\mathbb{T}^\dagger_{[Y]}$ for $[Y]\in H^0(M_0,\Theta_{M_0})$}
\label{sec:direct_limit_cotangent_space}
For locally finite coverings $\mathcal{U}$ and $\mathcal{V}$ of $M_0$ with $\mathcal{V}\prec \mathcal{U}$, we can see that the refinement $\Pi^{\mathcal{U}}_{\mathcal{V}}$ commutes $D^{Y,\dagger}_1$ and $\tilde{L}^\dagger_{\beta;Y}$
%
on $C^0(\mathcal{U},\Omega_{M_0}^{\otimes 2})^{\oplus 2}$
for $\beta\in C^0(\mathcal{U},\Theta_{M_0})$.
Therefore, we have a well-defined map
$$
\check{\Pi}^{\mathcal{U}}_{\mathcal{V}}\colon 
\mathbb{T}^\dagger_{[Y]}[\mathcal{U}]
\to \mathbb{T}^\dagger_{[\Pi^{\mathcal{U}}_{\mathcal{V}}(Y)]}[\mathcal{V}].
$$
and $\check{\Pi}^{\mathcal{V}}_{\mathcal{W}}\circ \check{\Pi}^{\mathcal{U}}_{\mathcal{V}}=\check{\Pi}^{\mathcal{U}}_{\mathcal{W}}$
for locally finite coverings $\mathcal{U}$, $\mathcal{V}$, $\mathcal{W}$ with $\mathcal{W}\prec \mathcal{V}\prec \mathcal{U}$. Therefore, the direct limit
$$
\mathbb{T}^\dagger_{[Y]}=\varinjlim_{\mathcal{U}}\mathbb{T}^\dagger_{[Y]}[\mathcal{U}]
$$
exists.
The isomorphism
$$
\mathbb{T}^\dagger_{[Y]}[\mathcal{U}]\to \mathcal{Q}_{M_0}\oplus \mathcal{Q}_{M_0}
$$
given in \Cref{thm:trivialization_cotangent_to_tangent} naturally induces a $\mathbb{C}$-isomorphism
$$
\mathbb{T}^\dagger_{[Y]}\to \mathcal{Q}_{M_0}\oplus \mathcal{Q}_{M_0}
$$
which commutes the following diagram:
$$
\begin{CD}
\mathbb{T}^\dagger_{[Y]}[\mathcal{U}]@>>> \mathcal{Q}_{M_0}\oplus \mathcal{Q}_{M_0} \\
@V{\check{\Pi}^{\mathcal{U}}}VV @| \\
\mathbb{T}^\dagger_{[Y]} @>>>\mathcal{Q}_{M_0}\oplus \mathcal{Q}_{M_0}.
\end{CD}
$$
We can also check that the pairing
$$
\mathcal{P}_{\mathbb{TT}}\colon \mathbb{T}_{[Y]}[\mathcal{U}]\oplus \mathbb{T}^\dagger_{[Y]}[\mathcal{U}]\to \mathbb{C}
$$
defined at \Cref{def:model_pairing}
also descends to the non-degenerate pairing $\mathcal{P}_{\mathbb{TT}}$ between the direct limits $\mathbb{T}_{[Y]}\oplus \mathbb{T}^\dagger_{[Y]}$
(cf. \Cref{prop:non-degenerate}).

\chapter{Double tangent spaces to Teichm\"uller space}
\label{chap:double_tangent_space_model}
One of the main purpose of this chapter is to show \Cref{thm:main2}, which says that the direct limit $\mathbb{T}_{[Y]}$  of the model space for $[Y]\in H^1(M_0,\Theta_{M_0})$ defined in \S\ref{subsec:inductice_limit_def} naturally stands for the double tangent space $T_vT\teich_g$ of $\teich_g$ at the tangent vector $v\in T_{x_0}\teich_g$ corresponding to $[Y]$.


\section{Holomorphic families, charts and cochains}
In this section, we shall discuss the cocycles for the double tangent space defined from holomorphic families of Riemann surfaces.
We will see that the maps $D_0^Y$ and $D_1^Y$ defined in \S\ref{subsec:model_space_definition} naturally appear in the calculation.
Moreover, the calculations given in this section will be used in \S\ref{sec:double_tangent_main} to show that the model spaces discussed in \S\ref{sec:Model_space_definition} and \S\ref{sec:Dolbeaut_type_presentation} are exactly the models of the double tangent space.
\label{sec:families_charts_cochains}

\subsection{Families and charts}
\label{subsec:charts}
Let $\epsilon_0>0$ and $D=\{(t,s)\in \mathbb{C}^2\mid |t|,|s|<\epsilon_0\}$. Let $\mathcal{M}_0\to D$ be a holomorphic family of Riemann surfaces of genus $g$ with the fiber at $(0,0)\in D$ is biholomorphically equivalent to $M_0$. 
Let $M_{t,s}$ be the fiber over $(t,s)\in D$.
After a trivialization (as the differential family) $\mathcal{M}_0\cong M_0\times D$, 
when we take $\epsilon_0>0$ sufficiently small,  there is a complex analytic chart $\mathcal{U}=\{U_i\}_{i\in I}$ of $M_0$ and a family of diffeomorphisms $U_i\times D_i\ni (p, t,s)\mapsto \hat{z}_i(p,t,s):=(z^{t,s}_i(p),t,s)\in \mathbb{C}\times D$ (onto its image) such that 
the chart $\{(U_i\times D, \hat{z}_i\}_{i\in I}$ makes $M_0\times D$ the total space of holomorphic family of Riemann surfaces of genus $g$ which is isomorphic to $\mathcal{M}_0$ with fixing the parameter space $D$.
For $(t,s)\in D$, the analytic coordinate chart $\mathcal{U}^{t,s}=\{(U_i, z_i^{t,s})\}_{i\in I}$ on $M_0$ makes a base surface of $M_0$ a Riemann surface conformally equivalent to $M_{t,s}$.
We may assume that $z^{0,0}_i(U_i)=\mathbb{D}$ for each $i\in I$.
We fix a notation:
\begin{align*}
W_i^{t,s}(z)&=z^{t,s}_i\circ z^{0,0}_i(z) \quad (z\in \mathbb{D})\\
U^{t,s}_i &=z^{t,s}_i(U_i) \\
U^{t,s}_{ij} &=z^{t,s}_i(U_i\cap U_j) \\
z_{ij}^{t,s} &=z_i^{t,s}\circ (z_j^{t,s})^{-1}
\colon U^{t,s}_{ji}\to U^{t,s}_{ij}.
\end{align*}
When $s=0$, we set $z_i^t=z_i^{t,0}$, $U^{t}_i =U^{t,0}_i$,
$U^{t}_{ij}=U^{t,0}_{ij}$, and $z_{ij}^{t}=z_{ij}^{t,0}$ for simplicity.
 


\subsection{Cochains}
\label{subsec:cochain}
For $i$, $j\in I$ with $i\ne j$ and $U_i\cap U_j\ne \emptyset$, let
$$
\tilde{U}_{ij}=\{((t,s),z)\in D\times \mathbb{C}\mid z\in U^{t,s}_{ij}\}.
$$
Then $\tilde{U}_{ij}$ is a domain in $\mathbb{C}^3$. From the assumption,
$$
\tilde{U}_{ji}\ni ((t,s),z)\mapsto ((t,s),z_{ij}^{t,s}(z))\in \tilde{U}_{ij}
$$
is a biholomorphism. For $t\in D$, we define
\begin{align*}
X^t_{ij}&=\dfrac{\partial z_{ij}^{t}}{\partial t}\circ z_{ji}^t(z)\partial_{z_i^t}\in \Gamma(U^t_{ij},\Theta_{M_0}) \\
Y^{t}_{ij}&=\left.\dfrac{\partial z_{ij}^{t,s}}{\partial s}\right|_{s=0}\circ z_{ji}^t(z)\partial_{z_i^t} \in \Gamma(U^t_{ij},\Theta_{M_0}) \\
\dot{Y}_{ij}&=\left.\dfrac{\partial Y_{ij}^{t}}{\partial t}\right|_{t=0}(z)\partial_{z_i^0}
\in \Gamma(U^0_{ij},\Theta_{M_0}).
\end{align*}
We set $\dot{Y}_{ii}=0$,
$X_{ij}=X^0_{ij}$, $Y_{ij}=Y^0_{ij}$. Then
\begin{equation}
\label{eq:double_tangent_space-2}
X=\{X_{ij}\}_{i,j}, \ Y=\{Y_{ij}\}_{i,j}\in Z^1(\mathcal{U},\Theta_{M_0}), \
\dot{Y}=\{\dot{Y}_{ij}\}_{i,j}\in C^1(\mathcal{U},\Theta_{M_0}).
\end{equation}
From \eqref{eq:identify_H1_T_g}, the cohomology classes $[X]$, $[Y]\in H^1(\mathcal{U},\Theta_{M_0})$
are corresponds to $\dot{\varphi}$, $\psi_0\in B_2=T_0\Bers{x_0}$ via tha natural identification $T_0\Bers{x_0}\cong H^1(M_0,\Theta_{M_0})\cong H^1(\mathcal{U},\Theta_{M_0})$.

\subsection{Relations between cochains}
The following proposition implies that $(X,\dot{Y})\in \ker(D^Y_1)$
for $X$ and $\dot{Y}$ defined in \eqref{eq:double_tangent_space-2}.

\begin{proposition}
\label{prop:cocycle_conditions}
Under the above notation, $\dot{Y}$ satisfies
\begin{itemize}
\item[{\rm (a)}] $\dot{Y}_{ij}+\dot{Y}_{ji}+[X_{ij},Y_{ij}]=0$ for $i$, $j\in I$; and
\item[{\rm (b)}] Let $Z_{ij}=\dot{Y}_{ij}+\dfrac{1}{2}[X_{ij},Y_{ij}]$. Then
\begin{align}
Z_{ij} +Z_{ji} &= 0 
\label{eq:double_tangent_3}
\\
Z_{ij}+Z_{jk}+Z_{ki}-\zeta(X,Y)_{ijk} &=0
\label{eq:double_tangent_4}
\end{align}
for $i$, $j$, $k\in I$.
\end{itemize}
\end{proposition}

\begin{proof}
(a)\quad From the definition, for $z\in U^0_{ij}=z_i^0(U_i\cap U_j)$, 
$z\in U^t_{ij}$ for sufficiently small $t\in D$ and
\begin{align*}
& Y_{ji}^t(z_{ji}^t(z))\dfrac{1}{(z_{ji}^t)'(z)}=Y_{ji}(z_{ji}^0(z))\dfrac{1}{(z_{ji}^0)'(z)} \\
&\quad +t
(\dot{Y}_{ji}(z_{ji}^0(z))+Y'_{ji}(z_{ji}^0(z))X_{ji}(z_{ji}^0(z))-Y_{ji}(z_{ji}^0(z))X'_{ji}(z_{ji}^0(z)))\dfrac{1}{(z_{ji}^0)'(z)} \\
&\quad+o(|t|) \\
&=Y_{ji}(z_{ji}^0(z))\dfrac{1}{(z_{ji}^0)'(z)}+t
(\dot{Y}_{ji}(z_{ji}^0(z))+[X_{ji},Y_{ji}](z_{ji}^0(z)))\dfrac{1}{(z_{ji}^0)'(z)} \\
&\quad+o(|t|) \\
&Y_{ij}^t(z)
=Y_{ij}(z)+t\dot{Y}_{ij}(z)+o(|t|).
\end{align*}
Since $Y_{ij}^t(\partial/\partial z_i^0)+Y_{ji}^t(\partial/\partial z_j^0)=0$ for $t\in D$, we obtain
$$
\dot{Y}_{ij}(z)+\left(\dot{Y}_{ji}(z_{ji}^0(z))+[X_{ji},Y_{ji}](z_{ji}^0(z))\right)\dfrac{1}{(z_{ji}^0)'(z)}=0,
$$
which implies what we wanted.

\medskip
\noindent
(b) \eqref{eq:double_tangent_3} follows from (a) since
$$
[X_{ji},Y_{ji}]=[-X_{ij},-Y_{ij}]=[X_{ij},Y_{ij}].
$$
We shall check \eqref{eq:double_tangent_4}. Let $z\in z^0_i(U_i\cap U_j\cap U_k)$.
Since $\{Y^t_{ij}\}_{i,j}$ is a $1$-cocycle for $t\in D$, $Y^t_{ij}+Y^t_{jk}+Y^t_{ki}=0$,
which is equivalent to
$$
Y^t_{ij}(z)
+
Y_{jk}(z^t_{ji}(z))\dfrac{1}{(z^t_{ji})'(z)}
+
Y^t_{ki}(z^t_{ki}(z))\dfrac{1}{(z^t_{ki})'(z)}=0.
$$
From the calculation, we can see 
\begin{align*}
Y^t_{ij}(z)&=Y_{ij}(z)+t\dot{Y}_{ij}(z)+o(|t|) \\
Y_{jk}(z^t_{ji}(z))\dfrac{1}{(z^t_{ji})'(z)}
&=Y_{jk}(z^0_{ji}(z))\dfrac{1}{(z^0_{ji})'(z))} \\
&\qquad +t
\left(
\dot{Y}_{jk}(z^0_{ji}(z))+[X_{ji},Y_{jk}](z^0_{ji}(z))
\right)\dfrac{1}{(z^0_{ji})'(z)}+o(|t|)
\\
Y^t_{ki}(z^t_{ki}(z))\dfrac{1}{(z^t_{ki})'(z)}
&=Y_{ki}(z^0_{ki}(z))\dfrac{1}{(z^0_{ki})'(z)} \\
&\qquad+t
\left(
\dot{Y}_{ki}(z^0_{ki}(z))+[X_{ki},Y_{ki}](z^0_{ki}(z))
\right)\dfrac{1}{(z^0_{ki})'(z)}+o(|t|)
\end{align*}
as $t\to 0$. Therefore,
\begin{equation}
\label{eq:double_tangent_5}
\dot{Y}_{ij}+\dot{Y}_{jk}+\dot{Y}_{ki}+[X_{ji},Y_{jk}]+[X_{ki},Y_{ki}]=0.
\end{equation}
On the other hand, since
\begin{align*}
& \dfrac{1}{2}[X_{ij},Y_{ij}]+\dfrac{1}{2}[X_{jk},Y_{jk}]+\dfrac{1}{2}[X_{ki},Y_{ki}]
-\dfrac{1}{2}([X_{ij},Y_{jk}]+[Y_{ij},X_{jk}]) \\
&=\dfrac{1}{2}[X_{ij},Y_{ik}+Y_{kj}]+\dfrac{1}{2}[X_{jk},Y_{ji}+Y_{ik}]+\dfrac{1}{2}[X_{ki},Y_{ki}]
+\dfrac{1}{2}[X_{ji},Y_{jk}]+\dfrac{1}{2}[X_{jk},Y_{ij}] \\
&=
[X_{ji},Y_{jk}]+[Y_{ki},Y_{ki}]-\dfrac{1}{2}[X_{ki},Y_{ki}]+\dfrac{1}{2}[X_{ij},Y_{ik}]
+\dfrac{1}{2}[X_{jk},Y_{ik}] \\
&=[X_{ji},Y_{jk}]+[Y_{ki},Y_{ki}],
\end{align*}
\eqref{eq:double_tangent_5} is equivalent to \eqref{eq:double_tangent_4}.
\end{proof}

\subsection{Connection to Dolbeaut presentation}
\label{subsec:Notation_2}
We continue to use the symbols defined and discussed in the previous sections.
For $(t,s)\in D$, we set
\begin{align*}
\xi^{s}_{i}&=-\left.\dfrac{\partial W^{t,s}_i}{\partial t}\right|_{t=0}
\circ (W_i^{0,s})^{-1}(z)\partial_{z_i^{0,s}}
\in \Gamma(U^{0,s}_i,\mathcal{A}^{0,0}(\Theta_{M_0})) \\
\eta^{t}_{i}
&=
-\left.\dfrac{\partial W_i^{t,s}}{\partial s}\right|_{s=0}
\circ
(W_i^{t,0})^{-1}(z)\partial_{z_i^{t}}
\in \Gamma(U^t_i,\mathcal{A}^{0,0}(\Theta_{M_0})) \\
\dot{\xi}_{i}&=\left.\dfrac{\partial \xi_{i}^{s}}{\partial s}\right|_{s=0}(z)\partial_{z_i^0},\quad \dot{\xi}^*_{i}=\left.\dfrac{\partial \xi_{i}^{s}}{\partial \overline{s}}\right|_{s=0}(z)\partial_{z_i^0}
\in \Gamma(U^0_{i},\mathcal{A}^{0,0}(\Theta_{M_0}))
\\
\dot{\eta}_{i}&=\left.\dfrac{\partial \eta_{i}^{t}}{\partial t}\right|_{t=0}(z)\partial_{z_i^0},\quad \dot{\eta}^*_{i}=\left.\dfrac{\partial \eta_{i}^{t}}{\partial \overline{t}}\right|_{t=0}(z)\partial_{z_i^0}
\in \Gamma(U^0_{i},\mathcal{A}^{0,0}(\Theta_{M_0}))
\\
X^{s}_{ij}&=\left.\dfrac{\partial z_{ij}^{t,s}}{\partial t}\right|_{t=0}\circ z_{ji}^{0,s}(z)\partial_{z_i^{0,s}} \in \Gamma(U^{0,s}_{ij},\Theta_{M_0}) \\
\dot{X}_{ij}&=\left.\dfrac{\partial X_{ij}^{s}}{\partial s}\right|_{s=0}(z)\partial_{z_i^0}
\in \Gamma(U^0_{ij},\Theta_{M_0}).
\end{align*}
For signs of $\xi_i^s$ and $\eta^s_i$, see \S\ref{subsec:appendix_1}.
We set $\dot{X}_{ii}=0$ and $\dot{X}=\{\dot{X}_{ij}\}_{i,j}\in C^1(\mathcal{U},\Theta)$. From the definition,
$$
\xi^s_j-\xi_i^s=X^s_{ij},\quad
\eta^t_j-\eta_i^t=Y^t_{ij}.
$$
\begin{remark}
\label{remark:2}
$\xi_i^s$ and $\eta_i^t$ may not depend holomorphically in $s$, $t$ in general. 
For instance, see Masumoto's example in \cite[\S3]{MR1123804}).
\end{remark}

\begin{proposition}
\label{prop:dot_X_Y}
Under the above notation,
$\dot{X}$ and $\dot{Y}$ satisfy the following.
\begin{align}
\dot{Y}_{ij} &= \dot{\eta}_j-\dot{\eta}_i-[\eta^0_j,X_{ji}]=\dot{\eta}_j-\dot{\eta}_i+[\eta^0_j,X_{ij}]
\label{eq:dot_Y_ij}
\\
\dot{X}_{ij} &= \dot{\xi}_j-\dot{\xi}_i-[\xi^0_j,Y_{ji}]=\dot{\xi}_j-\dot{\xi}_i+[\xi^0_j,Y_{ij}]
\label{eq:dot_X_ij}
\\
\dot{\eta}^*_j-\dot{\eta}^*_i
&=(\eta^0_i)_{\overline{z}}\overline{X_{ij}}
\label{eq:eta-star}
\\
\dot{\xi}^*_j-\dot{\xi}^*_i
&=(\xi^0_i)_{\overline{z}}\overline{Y_{ij}}
\label{eq:xi-star}
\\
\dot{X}_{ij}-\dot{Y}_{ij}&=[{\color{black}X}_{ij},{\color{black}Y}_{ij}].
\label{eq:Xij-Yij}
\end{align}
In particular,
\begin{align}
\label{eq:Yij_xi}
\dot{Y}_{ij}
&=\dot{\xi}_j-\dot{\xi}_i+[\xi^0_{i},Y_{ij}]
\\
\dot{\eta}_i&=\dot{\xi}_i+[\xi_i,\eta_i]+\chi,
\label{eq:eta_xi}
\end{align}
where $\chi$ is a (global) vector field on $M_0$.
\end{proposition}

\begin{proof}
For $z\in U^0_{ij}$, $z\in U_{ij}^{t,s}$ for sufficiently small $(t,s)\in D$ and
\begin{align*}
&\eta_j^t(z_{ji}^t(z))\dfrac{1}{(z_{ji}^t)'(z)}
=\eta_j^0(z^0_{ji}(z))\dfrac{1}{(z^0_{ji})'(z)}
\\
&\quad +t\left(\dot{\eta}_{j}(z_{ji}(z))+(\eta^0_j)_z(z_{ji}(z))X_{ji}(z^0_{ji}(z))-
\eta^0_j(z^0_{ji}(z))X_{ji}'(z^0_{ji}(z))
\right)\dfrac{1}{(z^0_{ji})'(z)} \\
&\qquad
+\overline{t} 
\left(
\dot{\eta}^*_j(z_{ji}(z))+(\eta^0_j)_{\overline{z}}(z^0_{ji}(z))\overline{X_{ji}(z^0_{ji}(z))}
\right)\dfrac{1}{(z^0_{ji})'(z)}+o(|t|)
\\
&\eta_i^t(z)
=\eta_i^0(z)+t\dot{\eta}_{i}(z)+\overline{t}\dot{\eta}^*_{i}(z)+o(|t|) \\
&Y_{ij}^t(z)
=Y_{ij}(z)+t\dot{Y}_{ij}(z)+o(|t|).
\end{align*}
Hence, the relation $\eta_j^t-\eta^t_i=Y^t_{ij}$ implies
\begin{align*}
\dot{Y}_{ij}(z)&=\left(\dot{\eta}_{j}(z^0_{ji}(z))+[X_{ji},\eta^0_j](z^0_{ji}(z))\right)\dfrac{1}{(z^0_{ji})'(z)} -\dot{\eta}_{i}(z) \\
0 &=\left(
\dot{\eta}^*_j(z_{ji}(z))+(\eta^0_j)_{\overline{z}}(z_{ji}^0(z))\overline{X_{ji}(z)}
\right)\dfrac{1}{(z_{ji}^0)'(z)}-\dot{\eta}^*_{i}(z),
\end{align*}
which are equivalent to
\eqref{eq:dot_Y_ij} and \eqref{eq:eta-star}.
By the same calculation, we also have
\eqref{eq:dot_X_ij} and \eqref{eq:xi-star}.

Next we prove \eqref{eq:Xij-Yij}.
As the discussion above, for $z\in U^0_{ij}$, $z\in U_{ij}^{t,s}$ for sufficiently small $(t,s)\in D$. Then,
\begin{align*}
\dot{Y}_{ij}(z)
&=\left.\dfrac{\partial Y_{ij}^t}{\partial t}\right|_{t=0}(z) =\left.\partial_{t}
\left(
\left.\dfrac{\partial z^{t,s}_{ij}}{\partial s}\right|_{s=0}\circ z^{t,0}_{ji}(z)
\right)
\right|_{t=0} \\
&=
\left.\dfrac{\partial^2 z^{t,s}_{ij}}{\partial s\partial t}\right|_{((t,s)=(0,0)}\circ z^0_{ji}(z)
+Y'_{ij}(z)\dfrac{1}{(z^0_{ji})'(z)}\cdot X_{ji}(z^0_{ji}(z))
\\
&=
\left.\dfrac{\partial^2 z^{t,s}_{ij}}{\partial s\partial t}\right|_{((t,s)=(0,0)}\circ z^0_{ji}(z)
-Y'_{ij}(z)X_{ij}(z)
\\
\dot{X}_{ij}(z)
&=
\left.\dfrac{\partial^2 z^{t,s}_{ij}}{\partial t\partial s}\right|_{((t,s)=(0,0)}\circ z^0_{ji}(z)
-X'_{ij}(z)Y_{ij}(z),
\end{align*}
which implies \eqref{eq:Xij-Yij}.
\eqref{eq:Yij_xi} follows from \eqref{eq:dot_X_ij} and \eqref{eq:Xij-Yij}.
\eqref{eq:eta_xi} is deduced from \eqref{eq:dot_Y_ij} and \eqref{eq:dot_X_ij}.
\end{proof}

\begin{remark}
Equation \eqref{eq:Xij-Yij} leads the formulation of the flip on $TT\teich_g$ (cf. \Cref{thm:canonical_flip}).
Equation \eqref{eq:Yij_xi} means that $\dot{Y}$ is in the image of the connection homomorphism $\connectinghomo$ which is discussed in \S\ref{subsec:Dol_type_representation}.
\end{remark}

\section{Double tangent space over Teichm\"uller space}
\label{sec:double_tangent_main}
In this section, we will show that the model space $\mathbb{T}_{[Y]}[\mathcal{U}]$ for $[Y]\in H^1(M_0,\Theta_{M_0})$ defined in \Cref{def:model_space} is naturally regarded as the double tangent space $T_vT\teich_g$ of $\teich_g$ for $v\in T_{x_0}\teich_g$ corresponding to $[Y]$, when the covering $\mathcal{U}$ satisfies an appropriate condition. Namely, throughout this section, we assume that  a locally covering $\mathcal{U}=\{U_i\}_{i\in I}$ of $M_0$ satisfies that for each $i\in I$, $\overline{U_i}$ is an embedded closed topological disk in $M_0$ with smooth boundary. In this case, $H^1(M_0,\Theta_{M_0})\cong H^1(\mathcal{U},\Theta_{M_0})$ since $H^1(U_i,\Theta_{M_0})=0$ for all $i\in I$.

\subsection{Double tangent spaces on domains revisited}
Suppose $M$ is a domain in $\mathbb{C}^n$. Let $(z^1,\cdots,z^n)$ be the standard coordinates of $\mathbb{C}^n$.
Then, $TM$ is biholomorphic to $M\times \mathbb{C}^n$ by
$$
M\times \mathbb{C}^n\ni (p,\eta)\mapsto \sum_{k=1}^n\eta^k\left(\partial_{z^k}\right)_p\in TM
$$
$TTM$ is biholomorphic to $(M\times \mathbb{C}^n)\times (\mathbb{C}^n\times \mathbb{C}^n)$ by
\begin{equation}
\label{eq:trivialization_TTM}
(M\times \mathbb{C}^n)\times (\mathbb{C}^n\times \mathbb{C}^n)\ni (p,\eta,\alpha,\beta)\mapsto 
\sum_{k=1}^n\alpha^k\left(\partial_{z^k}\right)_u+\sum_{k=1}^n\beta^k\left(\partial_{\eta^k}\right)_u\in TTM
\end{equation}
for $\displaystyle u=\sum_{k=1}^n \eta^k\left(\partial_{z^k}\right)_p\in TM$.
Namely, in the trivialization \eqref{eq:trivialization_TTM}, a tangent vector $(p,\eta,\alpha,\beta)\in TTM=(M\times \mathbb{C}^n)\times (\mathbb{C}^n\times \mathbb{C}^n)$ is represented by a holomorphic disk
$$
\{|t|<\epsilon\}\ni t\mapsto (\chi_1(t),\chi_2(t))=(p+t\alpha+o(|t|),\eta+t\beta+o(|t|))\quad (t\to 0).
$$
in $TM= M\times \mathbb{C}^n$. Under these notation, a holomorphic polydisk
\begin{align*}
\{|t|<\epsilon\}\times \{|s|<\epsilon\}
\in (t,s)\mapsto 
&\chi(t,s)
=
\chi_1(t)+s\chi_2(t) \\
&=p+t\alpha+s(\eta+t\beta)+\beta't^2+\beta''s^2+o((|t|+|s|)^2)
\end{align*}
for some $\beta'$, $\beta''\in \mathbb{C}^n$ as $t,s\to 0$ in $M$ satisfies that
the partial derivative $\dfrac{\partial \chi}{\partial s}(t,0)$ at $(t,0)$ represents the tangent vector $\chi_2(t)$ at $\chi_1(t)\in M$.

\subsection{Double tangent spaces over Teichm\"uller spaces}
We use symbols defined in \S\ref{subsec:Bers_embedding} frequently. As \S\ref{subsec:Bers_embedding}, we fix $x_0=(M_0,f_0)\in \teich_g$, the Fuchsian group $\Gamma_0$ acting on $\mathbb{D}$ with $\mathbb{D}/\Gamma_0=M_0$.

We identify the Teichm\"uller space $\teich_g$ with a bounded domain $\Bers{x_0}\subset B_2$ via the Bers embedding $\Bersemb_{x_0}$. The tangent bundle $T\teich_g$ of $\teich_g$ is (holomorphically) isomorphic to a trivial holomorphic vector bundle $T\Bers{x_0}=\Bers{x_0}\times B_2\to \Bers{x_0}$ as we already discussed in \S\ref{subsec:trivialization_tangent_bundle_cotangent_bundle}.
Since the Ahlfors-Weill section $\ahlforsW{x_0}$ defined in \eqref{eq:Ahlfors-Weill-map} is $\mathbb{C}$-linear, after identifying $L^\infty_{(-1,1)}(M_0)\cong L^\infty(\mathbb{D},\Gamma_0)$,
\begin{equation}
\label{eq:Bers_tangent}
\ahlforsWhat{x_0}\colon
T_0\Bers{x_0}=B_2\ni \psi \mapsto [\ahlforsW{x_0}(\psi)]=\left.D\Bersproj_{x_0}\right|_0\circ \ahlforsW{x_0}(\psi)\in T_{x_0}\teich_g
\end{equation}
is the inverse of the differential $D\Bersemb_{x_0}|_{x_0}\colon T_{x_0}\teich_g\to T_0\Bers{x_0}$ of the Bers embedding $\Bersemb_{x_0}$ at $x_0\in \teich_g$.

Since $\Bersemb_{x_0}\colon \teich_g\to \Bers{x_0}$ is biholomorphic, the double tangent space $TT\teich_g$ is also a trivial holomorphic vector bundle which is isomorphic to the product spaces $TT\Bers{x_0}=(\Bers{x_0}\times B_2)\times (B_2\times B_2)$. 

Let
\begin{equation}
\label{eq:double_tangent_space-1}
(0, \psi_0,\dot{\varphi},\dot{\psi})\in TT\Bers{x_0}=(\Bers{x_0}\times B_2)\times (B_2\times B_2),
\end{equation}
where $(0, \psi_0)\in T_0\Bers{x_0}$ and $(\dot{\varphi},\dot{\psi})\in T_{(0,\psi_0)}T\Bers{x_0}$ and $\dot{\varphi}\in T_0\Bers{x_0}$. 

\subsection{Holomorphic families}
Consider a holomorphic map $\chi(t)=(\varphi_t,\psi_t)$ from $D_1=\{|t|<\epsilon_0\}$ to $T\Bers{x_0}=\Bers{x_0}\times B_2$ such that
\begin{equation}
\label{eq:definition_families}
\begin{cases}
\varphi_t=t\dot{\varphi}+o(|t|) \\
\psi_t=\psi_0+t\dot{\psi}+o(|t|)
\end{cases}
\end{equation}
as $t\to 0$. The map $\chi$ is a holomorphic disk tangent to $(\dot{\varphi},\dot{\psi})\in T_{(0,\psi_0)}T\Bers{x_0}$ at the origin $0\in D_1$.

Let $D_2=\{|s|<\epsilon_0\}$. For $(t,s)\in D=D_1\times D_2$,
\begin{equation}
\label{eq:mu-teichmuller}
\mu(t,s)(z)=\ahlforsW{x_0}(\varphi_t+s\psi_t)=-\dfrac{1}{2\overline{z}^4}(|z|^2-1)^2\left(\varphi_t(1/\overline{z})+s\psi_t(1/\overline{z})\right)
\end{equation}
is regarded as a (real analytic) $(-1,1)$-form on $M_0$ (cf. \eqref{eq:Ahlfors-Weill-map}).
When $\epsilon_0$ is sufficiently small, $\|\mu(t,s)\|_\infty<1$ for $(t,s)\in D$. Hence, the bundle $\mathcal{M}=\cup_{(t,s)}M_{t,s}$ of the $\mu(t,s)$-quasiconformal deformations $M_{t,s}$ of $M_0$ defines a holomorphic family $\mathcal{M}\to D$ of Riemann surfaces of genus $g$. 

For $(t,s)\in D$, let $\mu_i(t,s)\in L^\infty(\mathbb{D})$ be the representation of $\mu(t,s)$ via the chart $(U_i,z^0_i)$. 
Let $W_i^{t,s}$ be a quasiconformal mapping on $z^0_i(U_i)=\mathbb{D}$ such that 
\begin{itemize}
\item[(i)]
$\overline{\partial}W_i^{t,s}=\mu_i(t,s)\partial W_i^{t,s}$ on $\mathbb{D}$; and
\item[(ii)]
each $W^{t,s}_i$ ($i\in I$, $(t,s)\in D$) is normalized in the sense that $W^{t,s}_i$ is extended to a quasiconformal mapping on $\mathbb{C}$ with $W^{t,s}_i(z)=z+o(1)$ as $z\to \infty$.
\end{itemize}
We shrink $U_i$ so that each $\mu_i(t,s)$ is smooth on the closure of $\mathbb{D}$.
Let $z^{t,s}_i=W_i^{t,s}\circ z^{0}_i$ and define $\hat{z}_i\colon U_i\times D\to \mathbb{C}\times D$ by $\hat{z}_i(0,t,s)=(z^{t,s}(p),t,s)$. Then, $\{U_i\times D, \hat{z}_i\}$ is a complex analytic chart on $M_0\times D$, which is total space of the holomorphic family of Riemann surfaces of genus $g$ which isomorphic to $\mathcal{M}\to D$
(cf. \cite{MR0430318}).

\subsection{Double tangent space and its model}
From the discussion in \S\ref{subsec:cochain}, we define the cochains $X^t$, $Y^s$, $\dot{Y}$ and $\xi^0$, $\eta^0$, $\dot{\xi}\in C^0(\mathcal{U},\mathcal{A}^{0,0}(\Theta_{M_0}))$
from the coordinate system defined in the previous section. By definition
$$
\delta\xi^0 = X \quad \mbox{\and}\quad
\delta\eta^0 = Y.
$$
Notice from \eqref{eq:identify_H1_T_g} that $[X]$, $[Y]\in H^1(M_0,\Theta_{M_0})$ naturally correspond to $\dot{\varphi}$, $\psi_0\in B_2$ under the isomorphism $B_2=T_0\Bers{x_0}\cong H^1(M_0,\Theta_{M_0})$.

\begin{theorem}[Double tangent space and its model]
\label{thm:main1}
Under the above symbols, the map
\begin{equation}
\label{eq:main_map}
\isomorphism_{[Y]}^{x_0}\colon T_{[Y]}T\teich_g\cong T_{(0,\psi_0)}T\Bers{x_0}= B_2\times B_2\ni (\dot{\varphi},\dot{\psi})\mapsto \tv{X,\dot{Y}}{Y}\in \mathbb{T}_{Y}[\mathcal{U}]
\end{equation}
is a $\mathbb{C}$-linear isomorphism.

Furthermore, when we choose the Ahlfors-Weill guiding frame $\GuaidF=\ahlforsW{x_0}\circ D\Bersemb_{x_0}|_{x_0}\circ \mathscr{T}_{x_0}$ (discussed in \eqref{eq:ahlfors_weill-frame}) as the guiding frame for the trivializations, the isomorphism $\isomorphism_{[Y]}^{x_0}$ satisfies the following commutative diagram:
\begin{equation}
\label{eq:commutative_DT}
\begin{CD}
T_{(0,\psi_0)}T\Bers{x_0} @>{\isomorphism_{[Y]}^{x_0}}>>\mathbb{T}_{Y}[\mathcal{U}] 
@<{\connectinghomo}<< 
\mathbb{T}^{Dol}_Y
\\
@VV{\ahlforsWhat{x_0}\oplus \ahlforsWhat{x_0}}V @VV{{\trivialization_{Y}}}V 
@VV{\trivializationD_Y}V \\
(T_{x_0}\teich_g)^{\oplus 2}
@<{\mathscr{T}_{x_0}\oplus \mathscr{T}_{x_0}}<< H^1(\mathcal{U},\Theta_{M_0})^{\oplus 2}
@>{\mathscr{T}_{x_0}\oplus \mathscr{T}_{x_0}}>> (T_{x_0}\teich_g)^{\oplus 2}
\end{CD}
\end{equation}
(cf. See \Cref{remark:1}).
\end{theorem}

\begin{proof}
We set
$$
\mu_i(t,s)=t\mu_i^1+s\mu^2_i+ts\mu^3_i+t^2 \mu^4_i+s^2\mu^5_i+o((|t|+|s|)^2)
$$
as $|t|+|s|\to 0$, 
where $\mu_i^1$, $\mu_i^2$ and $\mu_i^3$ are the presentations in the coordinate chart $(U_i,z^{0}_i)$ of $\ahlforsW{x_0}(\dot{\varphi})$, $\ahlforsW{x_0}(\psi_0)$ and $\ahlforsW{x_0}(\dot{\psi})$. $\mu^4_i$ and $\mu^5_i$ are some smooth Beltrami differentals.

From the condition (ii) of $W^{t,s}_i$, 
\begin{align}
W^{t,s}_i(z)
&=z+T(\mu_i(t,s))+T\left(\mu_i(t,s)H(\mu_i(t,s))\right)+o((|t|+|s|)^2) 
\label{eq:W-expantion}
\\
&=z+tT(\mu_i^1)+sT(\mu_i^2)
+t^2\left(T(\mu^4_i)+T(\mu_i^1 H(\mu_i^1)\right) 
\nonumber\\
&\quad +ts \left(T(\mu^3_i)+T(\mu^1_i H(\mu^2_i))+T(\mu^2_i H(\mu^1_i))\right)
\nonumber\\
&\qquad +s^2\left(T(\mu^5_i)+T(\mu_i^2 H(\mu_i^2)))+o((|t|+|s|)^2\right)
\nonumber
\end{align}
as $|t|+|s|\to 0$, where $T$ and $H$ are defined by
\begin{align}
T(\omega)(z)&=-\dfrac{1}{\pi}\iint_{\mathbb{C}}\dfrac{\omega(\zeta)}{\zeta-z}
\frac{i}{2}d\zeta\wedge d\overline{\zeta}
\label{eq:operator_T} \\
H(\omega)(z)&=-\dfrac{1}{\pi}\iint_{\mathbb{C}}\dfrac{\omega(\zeta)}{(\zeta-z)^2}
\frac{i}{2}d\zeta\wedge d\overline{\zeta}
\label{eq:operator_H}
\end{align}
for $\omega\in L^p(\mathbb{C})$ ($1<p<\infty$) on $\mathbb{C}$ (cf. \cite[Theorem 4.3]{MR867407}), and $H(\omega)$ is defined as the Cauchy principal value.  It is known that for a $C^\infty$-function $\omega$ on the closure of a bounded domain $D$ with smooth boundary (and set $\omega\equiv 0$ for $\mathbb{C}\setminus \overline{D}$), 
$T(\omega)$ and $H(\omega)$ are $C^\infty$ on $D$ and satisfy
\begin{equation}
\label{eq:proof_main_4}
T(\omega)_{\overline{z}}=\omega, \
T(\omega)_z=H(\omega), 
\
\mbox{and}
\
H(\omega)_{\overline{z}}=\omega_z
\end{equation}
(see \S\ref{appendix:regularity}. See also\cite[p.26]{MR867407}, \cite[p.160]{MR0344463}). 

In our case,
\begin{equation}
\label{eq:proof_main_3}
\xi^0_i=-T(\mu^1_i)\quad \mbox{and}\quad \eta^0_i=-T(\mu^2_i)
\end{equation}
on $U^0_i$ for $i\in I$ from \eqref{eq:W-expantion}.

From the definition (cf. \S\ref{subsec:Notation_2}), 
$\xi^s_i\circ W_i^{0,s}=-\left.\dfrac{\partial W_i^{t,s}}{\partial t}\right|_{t=0}$, and
\begin{align*}
\left.\dfrac{\partial^2W_i^{t,s}}{\partial s\partial t}
\right|_{t=0,s=0}(z)
&=-\dot{\xi}_i(z)+(T(\mu_i^1))_z(z)T(\mu^2_i)(z) \\
&=-\dot{\xi}_i(z)+H(\mu_i^1)(z)T(\mu^2_i)(z)
\end{align*}
for $z\in U^0_i$ since $W_i^{t,s}(z)$ varies holomorphically in $t,s$ when $z\in U^0_i$ is fixed.
Hence, from \eqref{eq:W-expantion},
$$
\dot{\xi}_i=-\left(T(\mu^3_i)+T(\mu^1_i H(\mu^2_i))+T(\mu^2_i H(\mu^1_i))-H(\mu_i^1)T(\mu^2_i)\right)
$$
on $U^0_i$. From \Cref{prop:cocycle_conditions} and \eqref{eq:Yij_xi},
$$
T_{(0,\psi_i)}T\Bers{x_0}=B_2\oplus B_2\ni (\dot{\varphi}, \dot{\psi})\mapsto 
(X,\dot{Y})\in \ker(D^Y_1)
$$
is a $\mathbb{C}$-linear which descends to $\isomorphism_{[Y]}^{x_0}$ in \eqref{eq:main_map}, and hence the map $\isomorphism_{[Y]}^{x_0}$ is also $\mathbb{C}$-linear.

We will show that \eqref{eq:main_map} is an isomorphism. 
We use the Dolbeault type presentation, and continue to discuss with the notation given above. 
Suppose $\isomorphism_{[Y]}^{x_0}(\dot{\varphi},\dot{\psi})=0$ in $\mathbb{T}_{Y}$. From \Cref{thm:trivialization_Dol}, 
\begin{equation}
\label{eq:proof-main_thm_1}
\connectinghomo^{-1}\circ \isomorphism_{[Y]}^{x_0}(\dot{\varphi},\dot{\psi})
=\tvD{-\overline{\partial}\xi, -\overline{\partial}\dot{\xi}}{Y}
\end{equation}
is trivial. Since $\eta^0_i=-T(\mu^2_i)$ for $i\in I$ and $\delta\eta^0=Y$, 
from \Cref{prop:equivalence_class}, there are $\alpha$, $\beta\in \Gamma(M_0,\mathcal{A}^{0,0}(\Theta_{M_0}))$ such that 
\begin{align}
\overline{\partial}\xi_i &= \overline{\partial}\alpha
\label{eq:dol-1} \\
\overline{\partial}\dot{\xi}_i 
&=-\overline{\partial}[\alpha, \eta^0_i]-\overline{\partial}\beta,
\label{eq:dol-2}
\end{align}
where
\begin{align}
\overline{\partial}\dot{\xi}_i 
&= 
-\overline{\partial}
\left(
T(\mu^3_i)+T(\mu^1_i H(\mu^2_i))+T(\mu^2_i H(\mu^1_i))-H(\mu_i^1)T(\mu^2_i)
\right)
\label{eq:dol-3}
\\
&= 
-\mu^3_i-\mu^1_i (T(\mu^2_i))_z-\mu^2_i H(\mu^1_i)+(\mu_i^1)_zT(\mu^2_i)+H(\mu_i^1)\mu^2_i
\nonumber
\\
&= 
-\mu^3_i+[\mu^1_i,T(\mu^2_i)]
\nonumber
\end{align}
on $U_i$ (cf. \eqref{eq:proof_main_4}).

From \eqref{eq:proof_main_3}, \eqref{eq:dol-1} implies that $\ahlforsW{x_0}(\dot{\varphi})=-\overline{\partial}\alpha$, which means that $\ahlforsW{x_0}(\dot{\varphi})$ is infinitesimally trivial by \eqref{eq:Dou}. Since $\ahlforsWhat{x_0}$ in \eqref{eq:Bers_tangent} is isomorphism, we obtain $\dot{\varphi}=0$ (and hence $\mu^1_i=0$ for all $i\in I$). Therefore, $\overline{\partial}\alpha=0$, that is, $\alpha$ is a holomorphic vector field on $M_0$, which also implies that $\alpha=0$.
Thus, \eqref{eq:dol-2} and \eqref{eq:dol-3} lead
$$
\mu^3_i=\overline{\partial}\beta
$$
on $U^0_i$. Since $\mu^3_i$ is the presentation of $\ahlforsW{x_0}(\dot{\psi})$ on $U_i$, applying the above discussion, we deduce that $\dot{\psi}=0$. As a consequence, $\isomorphism_{[Y]}^{x_0}$ is injective. Since the dimensions of $T_{(0,\psi_0)}T\Bers{x_0}$ and $\mathbb{T}_{Y}$ are equal, $\isomorphism_{[Y]}^{x_0}$ is an isomorphism.

Next, we will show that the diagram commutes. To show this, we shall check that
\begin{equation}
\label{eq:proof_main_2}
\trivializationD_Y\circ \connectinghomo^{-1}\circ \isomorphism_{[Y]}^{x_0}=\ahlforsWhat{x_0}\oplus \ahlforsWhat{x_0}
\end{equation}
on $T_{(0,\psi_0)}T\Bers{x_0}=B_2\oplus B^2$ since $\trivializationD_Y\circ \connectinghomo^{-1}=(\mathscr{T}_{x_0}\oplus \mathscr{T}_{x_0})\circ \trivialization_{Y}$ from \Cref{thm:trivialization_Dol}. 

Let $\GoodS$ be the good section for the coboundary operator $\delta$ defined from the Ahlfors-Weill guiding frame $\GuaidF=\ahlforsW{x_0}\circ D\Bersemb_{x_0}|_{x_0}\circ \mathscr{T}_{x_0}$ (cf. \Cref{prop:linear-map-L}). From the definition,
$$
\GoodS(X)_i=-T(\GuaidF([X]))+\alpha_i
$$
for $X\in Z^1(\mathcal{U},\Theta_{M_0})$ and some $\alpha\in C^0(\mathcal{U},\Theta_{M_0})$ (depending on $X$). Since $X=\delta \xi^0$, $Y=\delta \eta^0$, $\GuaidF([X])=\ahlforsW{x_0}(\dot{\varphi})$ and $\GuaidF([Y])=\ahlforsW{x_0}(\psi_0)$,
we deduce
\begin{align*}
\GoodS(X)_i&=\xi^0_i=-T(\mu^1_i) \\
\GoodS(Y)_i&=\eta^0_i=-T(\mu^2_i)
\end{align*}
on $U_i$. 
From the definition of $\trivializationD_Y$ in \eqref{eq:definition_Xi} and \eqref{eq:proof-main_thm_1}, the Beltrami differential in the first coordinate of $\trivializationD_Y\circ \connectinghomo^{-1}\circ \isomorphism_{[Y]}^{x_0}(\dot{\varphi}, \dot{\psi})$ satisfies
$$
\overline{\partial}\GoodS(X)_i=-\mu^1_i
$$
on $U_i$, which implies that the first coordinate is equal to $\ahlforsW{x_0}(\dot{\varphi})$.

Notice that the first coordinate of $\trivializationD_Y\circ \connectinghomo^{-1}\circ \isomorphism_{[Y]}^{x_0}(\dot{\varphi}, \dot{\psi})$ satisfies $\GuaidF([X])=\ahlforsW{x_0}(\dot{\varphi})$,
Hence, from \eqref{eq:definition_hat_nu}, the second coordinate $\hat{\nu}$ of $\trivializationD_Y\circ \connectinghomo^{-1}\circ \isomorphism_{[Y]}^{x_0}(\dot{\varphi},\dot{\psi})$ is defined by
\begin{equation}
\label{eq:mu3}
\hat{\nu}
=-\overline{\partial}\dot{\xi}_i+[\GuaidF([X]),\GoodS(Y)_i]
=\mu^3_i-[\mu^1_i,T(\mu^2_i)]+[\mu^1_i, T(\mu^2_i)]=\mu^3_i
\end{equation}
on $U_i$. This means that the second coordinate $\hat{\nu}$ is equal to $\ahlforsW{x_0}(\dot{\psi})$. Thus \eqref{eq:proof_main_2} is confirmed.
\end{proof}

\section{Summary}
Suppose a locally finite covering $\mathcal{U}=\{U_i\}_{i\in I}$ of $M_0$ satisfies that each $U_i$ is the interior of closed topological disk with smooth boundary.
Since $H^1(U_i,\Theta_{M_0})=0$ for all $i\in I$, from the discussion in \S\ref{subsec:structure_direct_limit},
$$
\dim_{\mathbb{C}} \mathbb{T}_{[Y]}=6g-6,
$$
and the last arrows in the two horizontal sequences in \eqref{eq:direct_limit-exact1} are surjective.
Hence, the isomorphism $\isomorphism_{[Y]}^{x_0}\colon T_{[Y]}T\teich_g\cong B_2\times B_2\to T_{Y}[\mathcal{U}]$ given in  \Cref{thm:main1} induces a linear isomorphism
$$
\check{\isomorphism}_{[Y]}^{x_0}\colon T_{[Y]}T\teich_g\to \mathbb{T}_{[Y]}.
$$
In summary, we conclude the following.

\begin{theorem}[Double tangent space and its models]
\label{thm:main2}
Under the above notation,
the following diagram is commutative and all horizontal sequences of the diagram are exact:
\begin{equation}
\label{eq:direct_limit-exact2}
\minCDarrowwidth8pt
\begin{CD}
H^1(\mathcal{U},\Theta_{M_0})
@<{\verticalincmodel{Y}}<{\cong}<
\mathbb{T}^V_{Y}[\mathcal{U}] 
@>{inc}>>
\mathbb{T}_{Y}[\mathcal{U}] 
@>{onto}>>
H^1(\mathcal{U},\Theta_{M_0})
\\
@|
@VV{\cong}V
@VV{\cong}V
@|
\\
H^1(\mathcal{U},\Theta_{M_0})
@<{\verticalincmodel{[Y]}}<{\cong}<
\mathbb{T}^V_{[Y]}[\mathcal{U}] 
@>{inc}>>
\mathbb{T}_{[Y]}[\mathcal{U}] 
@>{onto}>>
H^1(\mathcal{U},\Theta_{M_0})
\\
@V{\Pi^{\mathcal{U}}}V{\cong}V
@V{\check{\Pi}^{\mathcal{U}}}V{\cong}V
@V{\check{\Pi}^{\mathcal{U}}}V{\cong}V
@V{\Pi^{\mathcal{U}}}V{\cong}V
\\
H^1(M_0,\Theta_{M_0})
@<{\verticalincmodel{[Y]}}<{\cong}<
\mathbb{T}^V_{[Y]}
@>{inc}>> 
\mathbb{T}_{[Y]}
@>{onto}>>
H^1(M_0,\Theta_{M_0})
\\
@V{{\mathscr{T}_{x_0}}}V{\cong}V
@.
@A{\check{\isomorphism}_{[Y]}^{x_0}}A{\cong}A
@V{{\mathscr{T}_{x_0}}}V{\cong}V
\\
T_{x_0}\teich_g 
@<{\verticalincmodel{[Y]}}<{\cong}<
T_{[Y]}T_{x_0}\teich_g 
@>{inc}>>
T_{[Y]}T\teich_g
@>{onto}>>
T_{x_0}\teich_g,
\end{CD}
\end{equation}
where $\verticalincmodel{[Y]}$ is the vertical projection (cf. \eqref{eq:vertical_projection}), and ``inc" means the inclusion.
\end{theorem}
The first horizontal sequence in \eqref{eq:direct_limit-exact2} follows from \eqref{eq:commute_vertical_TYU}.
From \Cref{thm:Dol_direct_limit} and \eqref{eq:commute_Dol}, we also have the following commutative diagram.

\begin{equation}
\label{eq:direct_limit-exact3}
\minCDarrowwidth8pt
\begin{CD}
T_{x_0}\teich_g
@<{\verticalincmodelDol{[Y]}}<{\cong}<
\mathbb{T}^{Dol,V}_{[Y]}
@>{inc}>>
\mathbb{T}^{Dol}_{[Y]}
@>{onto}>>
T_{x_0}\teich_g
\\
@A{{\mathscr{T}_{x_0}}}A{\cong}A
@V{\connectinghomo}V{\cong}V
@V{\connectinghomo}V{\cong}V
@A{{\mathscr{T}_{x_0}}}A{\cong}A
\\
H^1(M_0,\Theta_{M_0})
@<{\verticalincmodel{[Y]}}<{\cong}<
\mathbb{T}^V_{[Y]}
@>{inc}>> 
\mathbb{T}_{[Y]}
@>{onto}>>
H^1(M_0,\Theta_{M_0})
\\
@V{{\mathscr{T}_{x_0}}}V{\cong}V
@.
@A{\check{\isomorphism}_{[Y]}^{x_0}}A{\cong}A
@V{{\mathscr{T}_{x_0}}}V{\cong}V
\\
T_{x_0}\teich_g 
@<{\verticalincmodel{[Y]}}<{\cong}<
T_{[Y]}T_{x_0}\teich_g 
@>{inc}>>
T_{[Y]}T\teich_g
@>{onto}>>
T_{x_0}\teich_g,
\end{CD}
\end{equation}
where each horizontal line is exact.
We also obtain the same diagram for $\mathbb{T}^{Bel}_{[\nu]}$.

\chapter{Variational formula of the pairing function}
\label{chap:variation_formula_pairing}
In this chapter, we shall give a variational formula for the pairing between $T\teich_g$ and $T^*\teich_g\cong \mathcal{Q}_g$.
As given in \S\ref{eq:pairing_Teichmuller}, the pairing
\begin{equation}
\label{eq:pairing_as_function}
T\teich_g\oplus\mathcal{Q}_{g}\ni (v,q)\mapsto 
\pairingsymb(v,q):=\langle v,q\rangle
\end{equation}
is a canonical pairing defined on the total space of the Whitney sum
$$
T\teich_g\oplus\mathcal{Q}_{g}=T\teich_g\oplus T^*\teich_g
=\{(v,q)\in T\teich_g\times T^*\teich_g\mid \pi_{\teich_g}(v)=\pi^*_{\teich_g}(q)\}
$$
of the (holomorphic) tangent bundle and the (holomorphic) cotangent bundle of the Teichm\"uller space, where $\pi^*_{\teich_g}\colon \mathcal{Q}_g\cong T^*\teich_g\to \teich_g$ is the projection. Notice that the total space of the Whitney sum is a complex submanifold of the product manifold $T\teich_g\times T^*\teich_g$, since the projections are submersions.
The pairing defines a holomorphic function on $T\teich_g\oplus T^*\teich_g$.

To this end, we first give a variational formula of the pairing function and a presentation of the variational formula by the model of the pairing in \S\ref{sec:differential_formula_for_the_pairing}. Then, using the variational formula, we will conclude that our model spaces are actually models of the second order infinitesimal spaces over the Teichm\"uller space in \S\ref{sec:trivialization_via_Bers_embedding}.


\section{Variations of $2$ forms}
A family $\{E_j\}_{j\in J}$ of subsets of $M_0$ is called a \emph{reasonable decomposition} of $M_0$ if each $E_i$ is the closure of a (topological) disk in $M_0$, the interiors of $E_{j_1}$ and $E_{j_2}$ are disjoint for distinct $j_1$, $j_2\in J$, $\cup_{j\in J}E_j=M_0$, and $M_0\setminus \cup_{j\in J}{\rm Int}(E_j)$ has measure zero. By definition, for any $2$-form $\Omega=\Omega(z)d\overline{z}\wedge dz$ on $M_0$, 
$$
\iint_{M_0}\Omega = \sum_{j\in J}\iint_{E_j}\Omega(z)d\overline{z}\wedge dz.
$$
By the Lebesgue theorem,
for any covering $\mathcal{U}=\{U_i\}_{i\in I}$ of $M_0$, there is a reasonable decomposition $\{E_j\}_{j\in J}$ of $M_0$ such that $E_j\subset U_{i(j)}$ for all $j\in J$ and some $i(j)\in I$.

From \S\ref{sec:Lie_derivative}, the Lie derivative of a smooth $2$-form $\Omega = \Omega(z)d\overline{z}\wedge dz$ along a smooth vector field $X=X(z)\partial_{z}$ is given by
$$
L_X\Omega=\left(\Omega_z(z)X(z)+\Omega(z)X_z(z)\right)d\overline{z}\wedge dz.
$$

\subsection{First derivative}
We use the notation on holomorphic families, families of local coordinates, and cocycles given in \S\ref{subsec:description_Hol} frequently. Let $\Omega^t=\{\Omega_i^t\}_{i\in I}$ be a smooth family of $2$-forms where $\Omega^t$ is a $2$-form on $M_t$. Let
\begin{align*}
\xi_i(z) & = -\left.\dfrac{\partial z_{i}^t}{\partial t}\right\mid_{t=0}\circ (z_{i}^0)^{-1}(z) \\
\Omega^t_i(z)&=\Omega^0_i(z)+t\dot{\Omega}_i(z)+\overline{t}\dot{\Omega}^*_i(z)+o(t)
\end{align*}
for $z\in z_i^0(U_i)$.
We can check that
\begin{align*}
\dot{\Omega}_j(z_{ij}(z))\left|\dfrac{dz_{ij}}{dz}(z)\right|^2-\dot{\Omega}_i(z)
&=L_{X_{ij}}(\Omega^0_i)(z)
\\
\dot{\Omega}^*_j(z_{ij}(z))\left|\dfrac{dz_{ij}}{dz}(z)\right|^2-\dot{\Omega}^*_i(z)
&=\overline{L_{X_{ij}}(\overline{\Omega^0_i})(z)}
\end{align*}
for $z\in z_i^0(U_i\cap U_j)$. Therefore,
$$
\dot{\Omega}_i-L_{\xi_i}(\Omega^0_i), \quad 
\dot{\Omega}^*_i-\overline{L_{\xi_i}(\overline{\Omega^0_i})}
$$
on $z_i(U_i)$ define well-defined $2$-forms on $M_0$.
The following is well-known. However, we give a proof for completeness.

\begin{proposition}
\label{prop:variation_two_forms}
Under the above notation, we have
\begin{align*}
\left.\partial_{t}\left(\iint_{M_t}\Omega_t\right)\right|_{t=0}
&=\iint_{M_0}(\dot{\Omega}_i(z)-L_{\xi_i}(\Omega^0_i)(z))d\overline{z}\wedge dz
\\
\left.\partial_{\overline{t}}\left(\iint_{M_t}\Omega_t\right)\right|_{t=0}
&=\iint_{M_0}(\dot{\Omega}^*_i(z)-\overline{L_{\xi_i}(\overline{\Omega^0_i})(z)})d\overline{z}\wedge dz.
\end{align*}
\end{proposition}

\begin{proof}
Let $\{E_j\}_{j\in J}$ be a reasonable decomposition of $M_0$ such that for each $j\in J$ there is an $i(j)\in I$ with $E_j\subset U_{i(j)}$. Let $W_i^t(z)=z^t_i\circ (z^0_i)^{-1}(z)=z-t\xi_i(z)+o(t)$ as $t\to 0$. Since $W_i^t$ defines a holomorphic motion, $W_t^t$ is a quasiconformal mapping. Hence, $\{W_{i(j)}^t(E_j)\}_{j\in J}$ is a reasonable deomposition on $M_t$ for sufficiently small $t$. Therefore,
$$
\iint_{M_t}\Omega_t=\sum_{j\in J}\iint_{W_{i(j)}^t(E_j)}\Omega^t_{i(j)}(w)d\overline{w}\wedge dw
$$
On the other hand,
\begin{align*}
&\iint_{W_{i(j)}^t(E_j)}\Omega^t_{i(j)}(w)d\overline{w}\wedge dw\\
&=\iint_{E_j}\Omega^t(W_{i(j)}^t(z))(|(W_{i(j)}^t)_z|^2-|(W_{i(j)}^t)_{\overline{z}}|^2)d\overline{z}\wedge dz
\\
&=\iint_{E_j}\Omega^0_{i(j)}(z)d\overline{z}\wedge dz \\
&\quad +t\iint_{E_j}\left(\dot{\Omega}_{i(j)}(z)-L_{\xi_i}(\Omega^0_{i(j)})(z)\right)d\overline{z}\wedge dz
\\
&\quad +\overline{t}\iint_{E_j}\left(\dot{\Omega}^*_{i(j)}(z)-
\overline{L_{\xi_i}(\overline{\Omega^0_{i(j)}})(z)}\right)d\overline{z}\wedge dz
+o(t),
\end{align*}
which implies what we wanted.
\end{proof}


\subsection{Notation}
\label{subsec:notation}
Let $F$ a differentiable mapping on a complex manifold $M$ around $p_0\in M$.
Let $V\in T^{1,0}_{p_0}M$ and $\{p(t)\}_{|t|<\epsilon}$ be an analytic (holomorphic) disk in $M$ with $p(0)=p_0$. We denote by
\begin{align*}
DF|_{p_0}[V] &=\dfrac{\partial F\circ p}{\partial t}(0) \\
\overline{D}F|_{p_0}[V] &=\dfrac{\partial F\circ p}{\partial \overline{t}}(0).
\end{align*}
The following lemma is well-known. However, we give a proof for a completeness.

\begin{lemma}
\label{lem:non-degenerate}
For a differentiable mapping $F\colon M\to \mathbb{C}^n$,
The differential of $F$ at $p_0\in M$ (as a real manifold) is non-generate if and only if $DF|_{p_0}[V]+\overline{D}F|_{p_0}[V]=0$ implies $V=0$ for $V\in T^{1,0}_{p_0}M$.
\end{lemma}

\begin{proof}
We may assume that $M=\mathbb{C}^m$.
Let $z_k=x_k+iy_k$ for $k=1$, $\cdots$, $m$ and $x=(x_1,\cdots, x_m)$ and $y=(y_1,\cdots,y_m)$. Set $F=u+iv$, $F=(F_1,\cdots,F_n)$ and $F_s=u_s+iv_s$ for $s=1$, $\cdots$, $n$. 
Let $V=V_1+iV_2$.
Let $JF$ be the Jacobi matrix of $F$ at $p_0$ after identifying $\mathbb{C}^m\ni z\mapsto (x,y)\in \mathbb{R}^m\times \mathbb{R}^m$. Then,
$JF\begin{bmatrix} V_1 \\ V_2\end{bmatrix}=0$ means
\begin{align*}
\sum_{k=1}^m\left(\dfrac{\partial u_s}{\partial x_k}(V_1)_k+\dfrac{\partial u_s}{\partial y_k}(V_2)_k\right)&=0 \\
\sum_{k=1}^m\left(\dfrac{\partial v_s}{\partial x_k}(V_1)_k+\dfrac{\partial v_s}{\partial y_k}(V_2)_k\right)&=0
\end{align*}
for $s=1$, $\cdots$, $n$. This is equivalent to
\begin{align*}
0&=\sum_{k=1}^m\left(\dfrac{\partial F_s}{\partial x_k}(V_1)_k+\dfrac{\partial F_s}{\partial y_k}(V_2)_k\right)
=\sum_{k=1}^m\left(\dfrac{\partial F_s}{\partial z_k}V_k+\dfrac{\partial F_s}{\partial \overline{z}_k}\overline{V_k}\right)
=DF|_{p_0}[V]+\overline{D}F|_{p_0}[V]
\end{align*}
for $s=1$, $\cdots$, $n$. 
\end{proof}

\section{Variation of the pairing function}
\label{sec:differential_formula_for_the_pairing}
In the following theorem, 
we assume that $\mathcal{U}=\{U_i\}_{i\in I}$ is a locally finite covering of $M_0$ such that $H^1(U_i,\Theta_{M_0})=0$ for $i\in I$.

\begin{theorem}[Variational formula]
\label{thm:variational_formula_pairing}
Let $x_0=(M_0,f_0)\in \teich_g$.
Let $[Y]\in H^1(\mathcal{U},\Theta_{M_0})\cong H^1(M_0,\Theta_{M_0})\cong T_{x_0}\teich_g$ and $q_0\in Q_{M_0}\cong T^*_{x_0}\teich_g$.
Let $V_1=\tv{X,\dot{Y}}{Y}\in \mathbb{T}_{Y}[\mathcal{U}]\cong T_{[Y]}T\teich_g$ and $V_2=[X,\varphi]_{q_0}\in {\bf H}^1(\mathcal{U},\mathbb{L}_{q_0})\cong T_{q_0}\mathcal{Q}_g$.
Then, the differential of the pairing \eqref{eq:pairing_as_function} is given by
\begin{align}
\label{eq:differential_formula}
\left.D\pairingsymb\right|_{([Y],q_0)}\left[
V_1,V_2
\right]
&=-\dfrac{1}{2i}\iint_{M_0}(dA_i-L_{\xi_i}((\eta_i)_{\overline{z}}q_0))d\overline{z}\wedge dz,
\end{align}
where 
\begin{itemize}
\item
$\xi\in C^0(\mathcal{U},\sob^{1,1}(\Theta_{M_0}))$ and $\eta\in C^0(\mathcal{U},\sob^{2,1}(\Theta_{M_0}))$ with $\delta \xi=X$ and $\delta \eta=Y$;
\item
$q_0=\{q_0\}_{i\in I}$ and $\eta_iq_0=\eta_i(z)q_0(z)dz$; and 
\item
$A_i=A_i(z)dz$ satisfies $A_i(z)\in W^{1,1}(U_i)$ and
\begin{align}
\label{eq:cocycle_condition}
A_i(z^0_{ij}(z))\dfrac{dz^0_{ij}}{dz}(z)-A_j(z)
&=
\dot{Y}_{ji}(z)q_0(z)+Y_{ji}(z)\varphi_j(z)
+L_{X_{ji}}(\eta_iq_0)(z)
\end{align}
for $z\in z_j^0(U_i\cap U_j)$ and $i$, $j\in I$,
where $\varphi=\{\varphi_i\}_{i\in I}$.
\end{itemize}
\end{theorem}
The proof of \Cref{thm:variational_formula_pairing} is given in \S\ref{subsec:proof_differential_formula}.

Suppose $A=\{A_i\}_{i\in I}$ satisfies \eqref{eq:cocycle_condition}.
Since $dA_i=(A_i)_{\overline{z}}(z)d\overline{z}\wedge dz$, $X_{ji}$ is holomorphic, and $(\eta_i)_{\overline{z}}=(\eta_j)_{\overline{z}}$ on $U_i\cap U_j$,
\begin{align*}
(A_i)_{\overline{z}}(z^0_{ij}(z))\left|\dfrac{dz^0_{ij}}{dz}(z)\right|^2-(A_j)_{\overline{z}}(z)
&=(L_{X_{ji}}(\eta_iq_0))_{\overline{z}}
=L_{X_{ji}}((\eta_i)_{\overline{z}}q_0)
\end{align*}
on $U_i\cap U_j$. Hence, the integrand in the integral at the right-hand side of \eqref{eq:differential_formula} is an $2$-form on $M_0$.
%
%

\subsection{Well-definedness of the integral}
\label{subsec:well_defined-integral}
Since the tangent vectors are defined as the equivalence classes, before proving \Cref{thm:variational_formula_pairing}, we shall check that the integral in the right-hand side of \eqref{eq:differential_formula} is independent of the choices of the representatives.

\subsubsection{{\bf The choice of $\{A_i\}_{i\in I}$}}
When $B=\{B_i\}_{i\in I}$ satisfies \eqref{eq:cocycle_condition} instead of $A=\{A_i\}_{i\in I}$.
Then, $\Omega_i=A_i-B_i$ defines a $(1,0)$-form $\Omega=\Omega(z)dz$ on $M_0$ with $\Omega_i(z)\in W^{1,1}(U_i)$.
Therefore, we obtain
$$
\dfrac{1}{2i}
\iint_{M_0}(dA_i-L_{\xi_i}((\eta_i)_{\overline{z}}q_0))-\dfrac{1}{2i}\iint_{M_0}(dB_i-L_{\xi_i}((\eta_i)_{\overline{z}}q_0))
=\dfrac{1}{2i}\iint_{M_0}d\Omega=0.
$$
by the Green formula for $W^{1,1}$-forms. See \S\ref{sec:remark_smoothness}.

\subsubsection{{\bf The choices of $\xi$ and $\eta$}}
We first check that the integral is independent of the choice of $\xi=\{\xi_i\}_{i\in I}$.
Take $\Xi\in C^0(\mathcal{U},\sob^{1,1}(\Theta_{M_0}))$ with $\delta \Xi = X$. Then, $\xi-\Xi$ defines a global $\sob^{1,1}$-vector field $\mathcal{X}$ on $M_0$. 
Since the integrand in this case is 
$\{dA_i-L_{\Xi_i}(d(\eta_iq_0))\}_{i\in I}$ and $\nu=-(\eta_i)_{\overline{z}}=-\overline{\partial}\eta_i$ is a global $(-1,1)$-form on $M_0$, the difference of the integrands satisfies
\begin{align*}
\left(dA_i-L_{\xi_i}((\eta_i)_{\overline{z}}q_0)\right)-
\left(dA_i-L_{\Xi_i}((\eta_i)_{\overline{z}}q_0)\right)
&=-L_{\mathcal{X}}((\eta_i)_{\overline{z}}q_0)
=d\left(\nu q_0\mathcal{X}d\overline{z}\right).
\end{align*}
Therefore, we obtain the difference of their integrals vanishes.
%

We shall check that the integral is independent of the choice of $\eta=\{\eta_i\}_{i\in I}$.
Let $H=\{H_i\}_{i\in I}\in C^0(\mathcal{U},\sob^{2,1}(\Theta_{M_0}))$ with $\delta H=Y$.
Take $B=\{B_i\}_{i\in I}$ satisfying \eqref{eq:cocycle_condition} when we consider $H$ instead of $\eta$.
Notice that $\mathcal{Y}=\eta-H$ defines a global vector field on $M_0$. Set $\Omega_i=A_i-B_i$ for $i\in I$.
From \eqref{eq:cocycle_condition}, the difference of the integrands in \eqref{eq:differential_formula} satisfies
\begin{align*}
d\Omega_i-L_{\xi_i}(\mathcal{Y}_{\overline{z}} q_0)
&=d\Omega_i-
L_{\xi_i}((\mathcal{Y} q_0)_{\overline{w}}d\overline{w}\wedge dw)
\\
&=d\left(\Omega_i(w)dw+\mathcal{Y}_{\overline{w}}(w)q_0(w)\xi_i(w)d\overline{w}\right)
\end{align*}
and
\begin{align*}
\left(d\Omega_i(z^0_{ij}(z))-L_{\xi_i}(\mathcal{Y}_{\overline{z}} q_0)(z_{ij}^0(z))\right)
&\left|\dfrac{dz^0_{ij}}{dz}(z)\right|^2
-
\left(
d\Omega_j(z)-L_{\xi_j}(\mathcal{Y}_{\overline{z}} q_0)(z)
\right)
\\
&
=
dL_{X_{ji}}(\mathcal{Y} q_0)(z)-L_{X_{ji}}(\mathcal{Y}_{\overline{z}} q_0))(z)
\\
&=L_\mathcal{Y}((X_{ji}q_0)_{\overline{z}})=0
\end{align*}
for $z\in z_j^0(U_i\cap U_j)$ and $w\in z_i^0(U_i)$ since $X_{ji}q_0$ is a holomorphic $1$-form on $z_j^0(U_i\cap U_j)$.
Therefore, the deference of the integrands is an exact $2$-form on $M_0$, and hence the integral of the difference vanishes.

\subsubsection{{\bf The representative of $[Y]$}}
%
Let $\beta\in C^0(\mathcal{U},\Theta_{M_0})$.
From \Cref{prop:isomorphism_tile_L}, we consider $\tilde{L}_\beta(X,\dot{Y})=(X,\dot{Y}+K'(\beta,X))\in \ker(D^{Y+\delta\beta}_1)$ instead of $(X,\dot{Y})\in \ker(D^Y_1)$, 
the right-hand side of \eqref{eq:cocycle_condition} becomes
\begin{align*}
&(\dot{Y}_{ji}+[\beta_i,X_{ji}])q_0+(Y_{ji}+(\beta_i-\beta_j))\varphi_j+L_{X_{ji}}((\eta_i+\beta_i)q_0)\\
&=(\dot{Y}_{ji}q_0+Y_{ji}\varphi_j+L_{X_{ji}}(\eta_iq_0))+([\beta_i,X_{ji}]q_0+(\beta_i-\beta_j)\varphi_j+L_{X_{ji}}(\beta_iq_0))
\\
&=(\dot{Y}_{ji}q_0+Y_{ji}\varphi_j+L_{X_{ji}}(\eta_iq_0))+(\beta_i\varphi_i-\beta_j\varphi_j).
\end{align*}
Therefore, the integral in \eqref{eq:differential_formula}  in this case satisifes
$$
\dfrac{1}{2i}\iint_{M_0}
\left(
d(A_i+\beta_i\varphi_i)-L_{\xi_i}(d(\eta_i q_0))
\right)=
\dfrac{1}{2i}\iint_{M_0}
\left(dA_i-L_{\xi_i}(d(\eta_i q_0))\right).
$$

\subsubsection{{\bf The representatives of $\tv{X,\dot{Y}}{[Y]}$ and $[X,\varphi]_{q_0}$}}
Take two cochains $\alpha$, $\beta\in C^0(\mathcal{U},\Theta_{M_0})$. From the definitions of ${\bf H}^1(\mathbb{L}_{q_0})$ and $\mathbb{T}_Y$, we choose
$(X+\delta\alpha, \dot{Y}+\delta\beta+K(\alpha,Y))$ and $(\varphi+L_\alpha(q_0), X+\delta\alpha)$ as the representatives in stead of $(X,\dot{Y})\in \ker(D^Y_0)$ and $(\varphi,X)\in C^0(\mathcal{U},\Omega_{M_0}^{\otimes 2})\oplus C^1(\mathcal{U},\Theta_{M_0})$.
In this case, the right-hand side of \eqref{eq:cocycle_condition} becomes
\begin{align*}
&
(\dot{Y}_{ji}+\beta_i-\beta_j+[\alpha_j,Y_{ji}])q_0+
Y_{ji}(\varphi_j+L_{\alpha_j}(q_0))+
L_{X_{ji}+\alpha_i-\alpha_j}(\eta_iq_0)
\\
&=\left(
\dot{Y}_{ji}q_0+Y_{ji}\varphi_j+L_{X_{ji}}(\eta_iq_0)
\right)
\\
&\quad+\left(
\beta_i-\beta_j+[\alpha_j,Y_{ji}])q_0+Y_{ji}L_{\alpha_j}(q_0)+
L_{\alpha_i}(\eta_iq_0)-L_{\alpha_j}(\eta_iq_0)
\right)
\\
&=\left(
\dot{Y}_{ji}q_0+Y_{ji}\varphi_j
+L_{X_{ji}}(\eta_iq_0)
\right)
+
(\beta_iq_0+L_{\alpha_i}(\eta_iq_0))-
(\beta_jq_0+L_{\alpha_j}(\eta_jq_0))
\end{align*}
on $z_j^0(U_i\cap U_j)$.
%
Hence, the integrand in \eqref{eq:differential_formula} in this case satisfies
\begin{align*}
d(A_i+\beta_iq_0+L_{\alpha_i}(\eta_iq_0))-L_{\xi_i+\alpha_i}(\eta_iq_0)
=dA_i-L_{\xi_i}(\eta_iq_0)
\end{align*}
since $\beta_iq_0=\beta_i(z)q_0(z)dz$ is holomorphic on $U_i$.
This implies that the integral in \eqref{eq:differential_formula} is independent of the choices of representatives.

\subsection{Proof of \Cref{thm:variational_formula_pairing}}
\label{subsec:proof_differential_formula}
We use the symbols defined from \S\ref{subsec:charts} to \S\ref{subsec:Notation_2}, frequently.
Let $M_t=M_{t,0}$ for $|t|<\epsilon_0$. Let $\{|t|<\epsilon_0\}\to x_t\in \teich_g$ be the representation of the holomorphic family.
Let $q^t=\{q^t_i\}_{i\in I}$ be a holomorphic family of holomorphic quadratic differentials such that $q^t\in \mathcal{Q}_{M_t}$.
Let $v_t=[\{Y^t_{ij}\}_{i,j\in I}]\in H^1(M_t,\Theta_{M_t})\cong T_{x_t}\teich_g$.
Then, from \eqref{eq:pairing_2},
\begin{align*}
\langle v_t,q_t\rangle
&=-\pi {\rm Res}([Y_{ij}^tq_t])
=-\dfrac{1}{2i}\iint_{M_t}d(\eta_i^tq_t dz).
\end{align*}
Since
$$
\eta_i^t(z)q_t(z)=\eta_i^0(z)q_0(z)+tA_i(z)+\overline{t}B_i(z)+o(t)
$$
on $z_i^0(U_i)$ as $t\to 0$ where $B$ is some function on $z^0_i(U_i)$,
and
\begin{equation}
\label{eq:A-eta}
A_i(z)=\dot{\eta_i}(z)q_0(z)+\eta_i^0(z)\varphi_i(z).
\end{equation}
By \Cref{prop:dot_X_Y},
\begin{align*}
&A_i(z_{ij}^0(z))\dfrac{dz_{ij}^0}{dz}(z)-A_j(z)
\\
&=
\left(\dot{\eta_i}(z^0_{ij}(z))q_0(z_{ij}^0(z))+\eta_i^0(z_{ij}^0(z))\varphi_i(z_{ij}^0(z))\right)\dfrac{dz_{ij}^0}{dz}(z)
\\
&\qquad\qquad
-
(\dot{\eta_j}(z)q_0(z)+\eta_j^0(z)\varphi_j(z))
\\
&=\left(
\dot{\eta_i}(z^0_{ij}(z))\left(\dfrac{dz_{ij}^0}{dz}(z)\right)^{-1}-\dot{\eta}_j(z)\right)q_0(z)
\\
&\qquad
+\eta_j^0(z)L_{X_{ji}}(q_0)+Y_{ji}(z)\varphi_j(z)+Y_{ji}(z)L_{X_{ji}}(q_0)(z)
\\
&=\dot{Y}_{ji}(z)q_0(z)+Y_{ji}(z)\varphi_j(z)-[\eta_i^0,X_{ji}](z)q_0(z)+\eta_j^0(z)L_{X_{ji}}(q_0)+Y_{ji}(z)L_{X_{ji}}(q_0)(z) \\
&=\dot{Y}_{ji}(z)q_0(z)+Y_{ji}(z)\varphi_j(z)+L_{X_{ji}}(\eta^0_jq_0)(z)+L_{X_ji}(Y_{ji}q_0)
\\
&=\dot{Y}_{ji}(z)q_0(z)+Y_{ji}(z)\varphi_j(z)+L_{X_{ji}}(\eta_i^0q_0)(z),
\end{align*}
which means $\{A_i\}_{i\in I}$ satisfies \eqref{eq:cocycle_condition}.
From \Cref{prop:variation_two_forms} and the discussion in \S\ref{subsec:well_defined-integral}, we obtain the formula \eqref{eq:differential_formula}.

\subsection{Variational formula with the Dolbeault representation}
From \Cref{thm:Dolbeault_pre} and \eqref{eq:A-eta}, we obtain the following.

\begin{corollary}
\label{coro:variational_pairing_Dol}
Under the assumption and the notation in \Cref{thm:Dolbeault_pre}, let $\tvD{\mu,\dot{\nu}}{Y}\in \mathbb{T}^{Dol}_Y[\mathcal{U}]$ with $\connectinghomo(\tvD{\mu,\dot{\nu}}{Y})=\tv{X,\dot{Y}}{Y}$, $\mu_i=-(\xi_i)_{\overline{z}}$, $\nu_i=-(\eta_i)_{\overline{z}}$ and $\dot{\nu}_i=-(\dot{\eta}_i)_{\overline{z}}$ for $i\in I$. 
Then,
\begin{align*}
\left.D\pairingsymb\right|_{([Y],q_0)}\left[
V_1,V_2
\right]
&=\dfrac{1}{2i}\iint_{M_0}
\left((\dot{\nu}_i- [\xi_i,\eta_i]_{\overline{z}})q_0+\nu_i\varphi_i
-
L_{\xi_i}(\nu_iq_0)\right)
d\overline{z}\wedge dz,
\end{align*}
where $\connectinghomo$ is the connecting homomorphism (cf. \Cref{prop:Dol_vertical_space_isomorphism} and \Cref{prop:dot_X_Y}).
\end{corollary}

\begin{remark}
\label{remark:dolbeault_rep_variation}
From the discussion in \S\ref{subsec:well_defined-integral}, when $[X]=0$, $\xi$ may be assumed to be $\xi=0$. In this case, in the formula in \Cref{coro:variational_pairing_Dol}, $\dot{\nu}_i$ and $\varphi_i$ can be assumed to be a globally defined on $M_0$, and the formula becomes
\begin{align}
\label{eq:differential_formula_dol_2}
\left.D\pairingsymb\right|_{([Y],q_0)}\left[
V_1,V_2
\right]
&=\dfrac{1}{2i}\iint_{M_0}
(\dot{\nu}_i q_0+\nu_i\varphi_i)d\overline{z}\wedge dz.
\end{align}
\end{remark}

\subsection{Variation of Pairing and Pairing on second order infinitesimal space}
The following formula will be applied in defining the model map of the switch later
(cf. \Cref{prop:Derivative_of_pairing_and_flip} and \Cref{prop:switch-cotangent-tangent}).

\begin{proposition}[A formula of the variation of Pairing]
\label{prop:variational_formula_and_pairing}
The variational formula for the pairing is presented by
\begin{equation}
\label{eq:differential_and_pairing_on_TT}
D\mathcal{P}|_{([Y],q)}
\left[
\tv{X,\dot{Y}}{[Y]},[X,\varphi]_q
\right]
=\mathcal{P}_{\mathbb{TT}}|_{[X]}
\left(
\tv{Y,\dot{Y}+[X,Y]}{[X]}, \tvdag{\varphi,q}{[X]}
\right)
\end{equation}
for $\tv{X,\dot{Y}}{[Y]}\in \mathbb{T}_{[Y]}[\mathcal{U}]\cong T_{[Y]}T\teich_g$,
$[X,\varphi]_q\in {\bf H}^1(\mathcal{U},\mathbb{L}_{q})\cong T_{q}\mathcal{Q}_g$,
$[Y]\in H^1(\mathcal{U},\Theta_{M_0})\cong H^1(M_0,\Theta_{M_0})\cong T_{x_0}\teich_g$ and $q\in Q_{M_0}\cong T^*_{x_0}\teich_g$.
\end{proposition}
\begin{proof}
%
From \eqref{eq:pairing_definition_TM_TstarM},
the right-hand side of \eqref{eq:differential_and_pairing_on_TT} is equal to
\begin{align*}
-\dfrac{1}{2i}\iint_{M_0}((A_i)_{\overline{z}}-L_{\xi_i}((\eta_i)_{\overline{z}}q))d\overline{z}\wedge dz,
\end{align*}
where $\xi=\{\xi_i\}_{i\in I}$ and $\eta=\{\eta_i\}_{i\in I}\in C^0(\mathcal{U},\mathcal{A}^{0,0}(\Theta_{M_0}))$
with $\delta\xi=X$ and $\delta\eta=Y$, and $A=\{A_i\}_{i\in I}\in C^0(\mathcal{U},\mathcal{A}^{0,0}(\Omega_{M_0}))$ with
\begin{align*}
A_j-A_i
&=Y_{ij}\varphi_i-(\dot{Y}_{ji}+[X_{ji},Y_{ji}])q+L_{X_{ij}}(\eta_jq)
\\
&=Y_{ij}\varphi_i+\dot{Y}_{ij}q+L_{X_{ij}}(\eta_jq).
\end{align*}
From \eqref{eq:cocycle_condition}, we have the equality \eqref{eq:differential_and_pairing_on_TT}.
%
%
%
%
\end{proof}

\section{Models are models of the second order infinitesimal spaces}
\label{sec:trivialization_via_Bers_embedding}
In this section, we fix $x_0\in \teich_g$ and identify the Teichm\"uller space $\teich_g$ with the image of the Bers embedding with basepoint $x_0$, and hence it is a bounded domain in $B_2=B_2(\mathbb{D}^*,\Gamma_0)$. Hence, the holomorphic tangent bundle $T\teich_g$ is naturally identified with the trivial bundle $\teich_g\times B_2\to \teich_g$ via the Bers embedding.

Let $\GuaidF=\ahlforsW{x_0}\circ D\Bersemb_{x_0}|_{x_0}\circ \mathscr{T}_{x_0}$ be the Ahlfors-Weill guiding frame, and $\GoodS$ is a good section defined from $\GuaidF$.
Let
$$
\trivialization_{[Y]}\colon \mathbb{T}_{[Y]}[\mathcal{U}]\to H^1(\mathcal{U},\Theta_{M_0})^{\oplus 2}
$$
be the trivialization  defined from the Ahlfors-Weill guiding frame $\GuaidF$.
Since the differential $D\Bersemb_{x_0}|_{x_0}\colon T_{x_0}\teich_g\to B_2$ is the inverse of $\ahlforsWhat{x_0}=\left.D\Bersproj_{x_0}\right|_0\circ \ahlforsW{x_0}$ (cf. \eqref{eq:Bers_tangent}), from the left square of the commutative diagram \eqref{eq:commutative_DT}, the map
\begin{equation}
\label{eq:Bers_trivialization_model}
(D\Bersemb_{x_0}|_{x_0})^{\oplus 2}\circ 
(\mathscr{T}_{x_0})^{\oplus 2} \circ \trivialization_{[Y]}
\, \colon \, \mathbb{T}_{[Y]}[\mathcal{U}]\to B_2\times B_2
\end{equation}
coincides with the  the triviallization (identification) $\mathbb{T}_{[Y]}[\mathcal{U}]\to T_{[Y]}T\teich_g\cong B_2\times B_2$ via the Bers embedding.

The purpose of this section is to observe that when we choose  the trivializations defined in \Cref{thm:trivialization}, \Cref{thm:trivialization_cotangent_to_tangent} and \Cref{thm:Trivialization_T_Tstar}
are natural in the sense that when we choose the Ahlfors-Weill guiding frame,  these trivializations for the second order infinitesimal spaces are naturally thought of as the trivialization via the Bers embedding.

In the following discussion, we notice that $B_2$ and $\mathcal{Q}_{x_0}$ have a natural non-degenerate pairing:
$$
B_2\times \mathcal{Q}_{x_0}\ni (\varphi,q)\mapsto \langle [\ahlforsW{x_0}(\varphi)],q\rangle.
$$

\subsection{Models of the tangent space to $T^*\teich_g=\mathcal{Q}_g$}
In \S\ref{subsec:trivialization_tangent_bundle_cotangent_bundle}, we discuss the trivialization of the cotangent bundle $\mathcal{Q}_g\cong T^*\teich_g$ as $T^*\teich_g\cong \teich_g\times \mathcal{Q}_{x_0}$.

Let $(\varphi_t,q_t)=(t\dot{\varphi},q_0+t\dot{q})+o(t)\in \teich_g\times \mathcal{Q}_{x_0}\subset B_2\times \mathcal{Q}_{x_0}$ and $(\varphi_t,\psi_t)=(t\dot{\varphi}, \psi_0+t\dot{\psi})+o(t)\in \teich_g\times B_2\subset B_2\times B_2$ be holomorphic curves.
Suppose that $[X,\Phi]_{q_0}\in H_1(\mathcal{U},\mathbb{L}_{q_0})$ and $\tv{X,\dot{Y}}{[Y]}\in \mathbb{T}_{[Y]}[\mathcal{U}]$ be the tangent vectors of the holomorphic curves at $t=0$.
By definition, $D\Bersemb_{x_0}|_{x_0}\circ \mathscr{T}_{x_0}([X])=\dot{\varphi}$ and $D\Bersemb_{x_0}|_{x_0}\circ \mathscr{T}_{x_0}([Y])=\psi_0$.

From the trivializations $T\teich_g\cong \teich_g\times B_2$ and $T^*\teich_g\cong \teich_g\times \mathcal{Q}_{x_0}$ discussed in \S\ref{subsec:trivialization_tangent_bundle_cotangent_bundle},
we obtain
$$
\pairingsymb_{T\teich_g}|_{\varphi_t}((\varphi_t,\psi_t), (\varphi_t,q_t))=\langle [\ahlforsW{x_0}(\psi_t)],q_t\rangle.
$$
From \Cref{thm:variational_formula_pairing},
\begin{equation}
\label{eq:Bers_emb_pairing}
\left.D\pairingsymb\right|_{([Y],q_0)}\left[
\tv{X,\dot{Y}}{[Y]},[X,\Phi]_{q_0}
\right]
=
\langle [\ahlforsW{x_0}(\psi_0)],\dot{q}\rangle+
\langle [\ahlforsW{x_0}(\dot{\psi})],q_0\rangle.
\end{equation}

Since $\GuaidF([X])=\ahlforsW{x_0}(\dot{\varphi})$ and $\GuaidF([Y])=\ahlforsW{x_0}(\psi_0)$,
from \Cref{thm:Trivialization_pairing_TTT_g} and \Cref{prop:variational_formula_and_pairing},
the left-hand side of \eqref{eq:Bers_emb_pairing} is equal to
$$
\mathcal{P}_{\mathbb{TT}}|_{[X]}
\left(
\tv{Y,\dot{Y}+[X,Y]}{[X]}, \tvdag{\Phi,q_0}{[X]}
\right)
=
\langle [\ahlforsW{x_0}(\psi_0)],\dot{Q}\rangle+
\langle [\hat{\nu}],q_0\rangle
$$
where 
\begin{align*}
(\mathscr{T}_{x_0}\oplus \mathscr{T}_{x_0})
\circ \trivialization_{[X]}\left(\tv{Y,\dot{Y}+[X,Y]}{[X]}\right)
&=([\ahlforsW{x_0}(\psi_0)],[\hat{\nu}])
\in T_{x_0}\teich_g\oplus T_{x_0}\teich_g \\
\trivialization^{\dagger}_{[X]}\left(\tvdag{\Phi,q_0}{[X]}\right)
&=(\dot{Q},q_0)\in \mathcal{Q}_{x_0}\oplus \mathcal{Q}_{x_0}
\end{align*}
and these trivializations are defined with the guiding frame $\GuaidF$.
Notice from the discussion around \eqref{eq:mu3} that $\hat{\nu}=\ahlforsW{x_0}(\dot{\psi})$. Therefore, we conclude
$$
\langle [\ahlforsW{x_0}(\psi_0)],\dot{q}\rangle=\langle [\ahlforsW{x_0}(\psi_0)],\dot{Q}\rangle
$$
and
$\dot{Q}=\dot{q}$ since we can choose $\psi_0$ arbitrarilly.

We summarize the above discussion as follows.
\begin{proposition}
Under the trivialization $T^*\mathcal{Q}_g\cong \teich_g\times \mathcal{Q}_{x_0}$ defined in \S\ref{subsec:trivialization_tangent_bundle_cotangent_bundle}, the tangent space $H^1(M_0,\mathbb{L}_{q_0})$ at $q_0\in \mathcal{Q}_{x_0}\subset \mathcal{Q}_g$ is trivialized by the trivialization in \Cref{thm:Trivialization_T_Tstar} induced by the guiding frame $\GuaidF=\ahlforsW{x_0}\circ D\Bersemb_{x_0}|_{x_0}\circ \mathscr{T}_{x_0}$ defined from the Ahlfors-Weill section.
\end{proposition}

\subsection{Models of cotangent spaces}
\label{subsec:models_cotangent_spaces}
As we discussed in \S\ref{sec:duals_in_model_spaces}, 
the trivializations of $\mathbb{T}_{[Y]}[\mathcal{U}]$ and ${\bf H}^{1,\dagger}(\mathcal{U},\mathbb{L}_{q_0})$ ($[Y]\in H^1(\mathcal{U},\Theta_{M_0})$ and $q_0\in \mathcal{Q}_{x_0}$) defined in \Cref{thm:trivialization_cotangent_to_tangent} and \Cref{thm:Trivialization_T_Tstar} are thought of as the isomorphisms between the dual spaces. Namely, the trivializations are natural in terms of the pairings. For the record, we summarize the following.
\begin{proposition}[Models of duals are models of cotangent spaces]
\label{thm:dual_spaces_are_duals}
The trivializations defined by the Ahlfors-Weill guiding frames gives natural identifications
\begin{align*}
\mathbb{T}_{[Y]}^\dagger[\mathcal{U}]
&\to T^*_{[Y]}T\teich_g\cong \mathcal{Q}_{x_0}\oplus \mathcal{Q}_{x_0}
\\
{\bf H}^{1,\dagger}(\mathcal{U},\mathbb{L}_{q_0})
&\to
T^*_{q_0}T^*\teich_g
\cong \mathcal{Q}_{x_0}\oplus H^1(\mathcal{U},\mathbb{L}_{q_0}).
\end{align*}
\end{proposition}

\chapter{Flip, Switch, Dualization and Lie bracket}
\label{chap:flip_switch_dual_lie}

The purpose of this chapter is to define the models of the flip, the switch, and the dualizaiton for the models of the second order infinitesimal spaces.
We also give a presentation of the Lie bracket for $C^1$-vector fields of type $(1,0)$ on the Teichm\"uller space.

\section{Model spaces over underlying manifolds}
For a complex manifold $M$, the flip on $TTM$, the switch from $TT^*\!M$ to $T^*TM$, and the dualization between $T^*\!T^*\!M$ and $TT^*\!M$ preserves the fibers over points of $M$. From this observation, we define
\begin{align*}
(\mathbb{T}\mathbb{T})_{x_0}[\mathcal{U}]
&=
\left\{
\left(
[Y],\tv{X,\dot{Y}}{[Y]}
\right)\mid [Y]\in H^1(\mathcal{U},\Theta_{M_0}), \tv{X,\dot{Y}}{[Y]}\in \mathbb{T}_{[Y]}[\mathcal{U}]
\right\}
\\
(\mathbb{T}^\dagger\mathbb{T})_{x_0}[\mathcal{U}]
&=
\left\{
\left(
[Y],\tvdag{\psi,q}{[Y]}
\right)\mid [Y]\in H^1(\mathcal{U},\Theta_{M_0}), \tvdag{\psi,q}{[Y]}\in \mathbb{T}^\dagger_{[Y]}[\mathcal{U}]
\right\}
\\
(\mathbb{T}\mathbb{T}^\dagger)_{x_0}[\mathcal{U}]
&=
\left\{
\left(
q,[X,\varphi]_q
\right)\mid q\in \mathcal{Q}_{x_0}, [X,\varphi]_q\in {\bf H}^1(\mathcal{U},\mathbb{L}_q)
\right\}
\\
(\mathbb{T}^\dagger\mathbb{T}^\dagger)_{x_0}[\mathcal{U}]
&=
\left\{
\left(
q,[\Phi,Y]_q
\right)\mid q\in \mathcal{Q}_{x_0}, [\Phi,Y]_q\in {\bf H}^{1,\dagger}(\mathcal{U},\mathbb{L}_q)
\right\}.
\end{align*}
where $x_0=(M_0,f_0)\in \teich_g$ and $\mathcal{U}$ be a locally finite covering of $M_0$
with $H^1(U_i,\Theta_{M_0})=0$ for $i\in I$.

Throughtout this chapter, we fix a guiding frame $\GuaidF$ and the good section $\GoodS$ for trivializations (cf. \S\ref{sec:good_section} and \Cref{thm:trivialization}).
From \Cref {thm:trivialization}, \Cref{thm:trivialization_cotangent_to_tangent}, and \Cref{thm:Trivialization_T_Tstar}, we have the following bijections:
\begin{align*}
\trivialization_{\mathbb{T}}
&
\colon (\mathbb{T}\mathbb{T})_{x_0}[\mathcal{U}]
\to H^1(\mathcal{U},\Theta_{M_0})^{\oplus 3} \\
\trivialization^\dagger_{\mathbb{T}}
&
\colon (\mathbb{T}^\dagger\mathbb{T})_{x_0}[\mathcal{U}]
\to H^1(\mathcal{U},\Theta_{M_0}) \oplus\mathcal{Q}_{x_0}\oplus \mathcal{Q}_{x_0}\\
\trivialization^*_{\mathbb{T}}
&
\colon (\mathbb{T}\mathbb{T}^\dagger)_{x_0}[\mathcal{U}]
\to \mathcal{Q}_{x_0}\oplus H^1(\mathcal{U},\Theta_{M_0}) \oplus\mathcal{Q}_{x_0}\\
\trivialization^{*\dagger}_{\mathbb{T}}
&
\colon (\mathbb{T}\mathbb{T}^\dagger)_{x_0}[\mathcal{U}]
\to \mathcal{Q}_{x_0}\oplus\mathcal{Q}_{x_0}\oplus H^1(\mathcal{U},\Theta_{M_0})
\end{align*}
by
\begin{align*}
\trivialization_{\mathbb{T}}([Y],\tv{X,\dot{Y}}{[Y]})
&=\left([Y],\trivialization_{[Y]}\left(\tv{X,\dot{Y}}{[Y]}\right)\right) \\
\trivialization^\dagger_{\mathbb{T}}([Y],\tvdag{\psi,q}{[Y]}
&=\left([Y],\trivialization^\dagger_{[Y]}\left(\tvdag{\psi,q}{[Y]}\right)\right) \\
\trivialization^*_{\mathbb{T}}(q,[X,\varphi]_q)
&=\left(q,\trivialization^*_{q}\left([X,\varphi]_q\right)\right) \\
\trivialization^{*\dagger}_{\mathbb{T}}(q,[\Phi,Y]_q)
&=\left(q,\trivialization^{*\dagger}_{q}\left([\Phi,Y]_q\right)\right).
\end{align*}

\section{Flip}
\label{sec:involutive_automorphism}
In this section, we discuss the flip on the double tangent space.

\subsection{Action on the model spaces}
We define
\begin{align*}
(\mathbb{Z}\mathbb{Z})_{x_0}[\mathcal{U}]
&=
\left\{
\left(
Y,(X,\dot{Y})
\right)\mid Y\in Z^1(\mathcal{U},\Theta_{M_0}), (X,\dot{Y})\in \ker(D^Y_1)
\right\},
\end{align*}
where $D^Y_1$ is a linear map defined in \S\ref{subsec:model_space_definition}.
We claim
\begin{proposition}
\label{prop:involution_cocycles}
Let $\mathcal{U}$ be a locally finite covering of $M_0$.
We define $\canoflip_{\mathbb{Z},x_0}\colon
(\mathbb{Z}\mathbb{Z})_{x_0}[\mathcal{U}]\to (\mathbb{Z}\mathbb{Z})_{x_0}[\mathcal{U}]$ and $\canoflip_{\mathbb{T}}=\canoflip_{\mathbb{T},x_0}\colon(\mathbb{T}\mathbb{T})_{x_0}[\mathcal{U}]\to (\mathbb{T}\mathbb{T})_{x_0}[\mathcal{U}]$ by
\begin{align}
\canoflip_{\mathbb{Z},x_0}
\left(Y,(X,\dot{Y})\right)
&=\left(X,(Y,\dot{Y}+[X,Y])\right)
\label{eq:flip1}
\\
\canoflip_{\mathbb{T},x_0}
\left([Y],\tv{X,\dot{Y}}{[Y]}\right)
&=\left([X],\tv{Y,\dot{Y}+[X,Y]}{[X]}\right).
\label{eq:flip2}
\end{align}
Then, $\canoflip_{\mathbb{Z},x_0}$ and $\canoflip_{\mathbb{T},x_0}$ are well-defined self-maps on $(\mathbb{ZZ})_{x_0}[\mathcal{U}]$ and $(\mathbb{TT})_{x_0}[\mathcal{U}]$ satisfying the commutative diagram:
$$
\begin{CD}
(\mathbb{ZZ})_{x_0}[\mathcal{U}] @>{\canoflip_{\mathbb{Z}}}>>  (\mathbb{ZZ})_{x_0}[\mathcal{U}]  \\
@VVV @VVV \\
(\mathbb{ZZ})_{x_0}[\mathcal{U}]  @>{\canoflip_{\mathbb{T}}}>>  (\mathbb{TT})_{x_0}[\mathcal{U}],
\end{CD}
$$
where vertical arrows mean the natural projection.
\end{proposition}
\begin{proof}
We need to check that $\canoflip_{\mathbb{Z},x_0}$ and $\canoflip_{\mathbb{T},x_0}$ are well-defined.

Let $(X,\dot{Y})\in \ker(D^Y_1)$.  For $i$, $j\in I$,
\begin{align*}
&(\dot{Y}_{ij}+[X_{ij},Y_{ij}])+(\dot{Y}_{ij}+[X_{ij},Y_{ij}])+[Y_{ij},X_{ij}]
=\dot{Y}_{ij}+\dot{Y}_{ji}+[X_{ij},Y_{ij}]=0.
\\
&
\delta\left((\dot{Y}+[X,Y])+\dfrac{1}{2}[Y,X]\right)-\zeta(Y,X)
=
\delta\left((\dot{Y}+\dfrac{1}{2}[X,Y]\right)-\zeta(X,Y)=0,
\end{align*}
which means that $(Y,\dot{Y}+[X,Y])\in \ker(D^X_1)$. Therefore, $\canoflip_{\mathbb{Z},x_0}$ is well-defined.

Next, we check $\canoflip_{\mathbb{T},x_0}$ is well-defined.
For $\alpha$, $\beta\in C^0(\mathcal{U},\Theta_{M_0})$,
the second cooridnate of $\canoflip_{\mathbb{Z}}(Y, (X,\dot{Y})+D_0^Y(\alpha,\beta))$ is equal to
\begin{align*}
(Y,\dot{Y}+\delta\beta+K(\alpha,Y)+[X+\delta \alpha,Y])
&=(Y,\dot{Y}+[X,Y]+K'(\alpha,Y)+(0,\delta\beta) \\
&=\tilde{\mathcal{L}}_{\alpha;X}(Y,\dot{Y}+[X,Y])+D_0^{X+\delta\alpha}(0,\beta).
\end{align*}
Since $\tilde{\mathcal{L}}_{\alpha;X}(Y,\dot{Y}+[X,Y])\in \ker(D^{Y+\delta\alpha}_1)$,
from \Cref{prop:isomorphism_tile_L} and \Cref{def:model_space},
the equivalence class $\tv{Y,\dot{Y}+[X,Y]}{[X]}\in \mathbb{T}_{[X]}$ is well-defined from $\tv{X,\dot{Y}}{Y}\in \mathbb{T}_Y$.
Notice that $\tilde{L}_{\beta;Y}(X,\dot{Y})=(X,\dot{Y}+K'(\beta,X))$
($\in \ker(D^{Y+\delta\beta}_1)$) corresponds to
\begin{align*}
(Y+\delta\beta,\dot{Y}+K'(\beta,X)+[X,Y+\delta\beta])
&=(Y+\delta\beta,\dot{Y}+[\beta_i,X]+[X,Y])\\
&=(Y,\dot{Y}+[X,Y])+D_0^{X}(\beta,0).
\end{align*}
In the sequel, for $\tv{X,\dot{Y}}{[Y]}\in \mathbb{T}_{[Y]}$,
the equivalence class $\tv{Y,\dot{Y}+[X,Y]}{[X]}$ is well-defined in $\mathbb{T}_{[X]}$.
\end{proof}

\subsection{The action of the Flip $\canoflip_{\teich_g}$}
We define 
a self-map $\canoflip^{triv}_{\mathbb{T},x_0}$ of $(\mathbb{TT})_{x_0}[\mathcal{U}]$ by
\begin{align*}
\canoflip^{triv}_{\mathbb{T},x_0}([Y],[X],[W])
&=([X],[Y],[W]),
\end{align*}
where $\trivialization_{[Y]}$ is the trivialization defined in \Cref{thm:trivialization}.
We define $\isomorphism_{x_0}\colon (TT)_{x_0}\teich_g\to(\mathbb{TT})_{x_0}\teich_g$ by
$$
\isomorphism_{x_0}(v)
=
\left(
\mathscr{T}_{x_0}^{-1}(\Pi_{T\teich_g}(v)), 
\isomorphism_{\Pi_{T\teich_g}(v)}^{\Pi_{\teich_g}(\Pi_{T\teich_g}(v))}(v)
\right),
$$
where $\isomorphism_{\cdot}^{\cdot}$ is the isomorphism defined in \Cref{thm:main1}.

We discuss the flip on the double tangent space on Teichm\"uller space. For $x_0\in \teich_g$, let $\Bersemb_{x_0}\colon \teich_g\to B_2$ be the Bers embedding and $\Bers{x_0}$ the image (\S\ref{subsec:Bers_embedding}). Then the twice differentials of the Bers embedding $\Bersemb_{x_0}$ induces the canonical isomorphism
$$
TT\teich_g\to \Bers{x_0}\times B_2\times B_2\times B_2
$$
since $\Bers{x_0}$ is a domain in $B_2$ ($\cong \mathbb{C}^{3g-3}$). Under the identification, the action of the flip $\canoflip_{\teich_g}$ on $TT\teich_g$ is presented by
$$
\begin{CD}
\Bers{x_0}\times B_2\times B_2\times B_2 @>{\canoflip_{\teich_g}}>> \Bers{x_0}\times B_2\times B_2\times B_2 \\
(x,\psi_0,\dot{\varphi},\dot{\psi})
@>>>
(x,\dot{\varphi},\psi_0,\dot{\psi})
\end{CD}
$$
from the definition of the flip (cf. \S\ref{subsubsec:canonical_flip}).

Our model of the flip \eqref{eq:flip2} presents the action of the flip on the double tangent space. Actually we have the following.

\begin{theorem}[Flip]
\label{thm:canonical_flip}
The flip $\canoflip_{\mathbb{T}}$ on $(\mathbb{TT})_{x_0}[\mathcal{U}]$ satisfying the commutative diagram:
$$
\begin{CD}
(TT)_{x_0}\teich_g
@>{\isomorphism_{x_0}}>>
(\mathbb{TT})_{x_0}[\mathcal{U}] 
@>{\trivialization_{\mathbb{T}}}>> 
H^1(\mathcal{U},\Theta_{M_0})^{\oplus 3} \\
@V{\canoflip_{\teich_g}}VV
@V{\canoflip_{\mathbb{T}}}VV 
@VV{\canoflip^{triv}_{\mathbb{T}}}V \\
(TT)_{x_0}\teich_g
@>{\isomorphism_{x_0}}>>
(\mathbb{TT})_{x_0}[\mathcal{U}]
@>{\trivialization_{\mathbb{T}}}>>  
H^1(\mathcal{U},\Theta_{M_0})^{\oplus 3},
\end{CD}
$$
where $\canoflip_{\teich_g}$ is the flip on $TT\teich_g$.
\end{theorem}

\begin{proof}
From \Cref{thm:main1}, the commutativity of the left-square of the diagram follows from that of the right-square by applying a special guiding frame defined in \eqref{eq:ahlfors_weill-frame}. Hence, we only check the commutativity of the right-square.

Let $([X],[W])\in H^1(\mathcal{U},\Theta_{M_0})^{\oplus 2}$ and $\tv{X,\dot{Y}}{Y}\in \mathbb{T}_Y[\mathcal{U}]$ with
$\trivialization_Y\left(\tv{X,\dot{Y}}{Y}\right)=([X],[W])$. Suppose $\trivialization_X\left(\tv{Y,\dot{Y}+[X,Y]}{X}\right)=([Y],[W'])$. Then
\begin{align*}
W&=\dot{Y}+\dfrac{1}{2}[X,Y]-E(X,Y) \\
W'&=\dot{Y}+[X,Y]+\dfrac{1}{2}[Y,X]-E(Y,X)=
\dot{Y}+\dfrac{1}{2}[X,Y]-E(Y,X)
\end{align*}
modulo $\delta C^0(\mathcal{U},\Theta_{M_0})$, where $\mathcal{U}$ is an appropriate covering on $M$. 
From (d) of \Cref{lem:map_E},
$$
W-W'=E(X,Y)-E(Y,X)\equiv 0
$$
modulo $\delta C^0(\mathcal{U},\Theta_{M_0})$.
Hence $[W]=[W']$ in $H^1(\mathcal{U},\Theta_{M_0})$.
\end{proof}

\subsection{Flip via Dolbeaut presentations}
\label{subsec:Flip_Dolbeaut}
From the \eqref{eq:eta_xi} 
and \Cref{coro:variational_pairing_Dol}, it is natural to define, in the cocycle level, the flip under the two Doubeaut-type presentations (given in \S\ref{sec:Dolbeaut_type_presentation})
for $\tvD{\mu,\dot{\nu}}{Y}\in \mathbb{T}_{Y}^{Dol}[\mathcal{U}]$ and $\tvDBel{\mu,\dot{\nu}}{\eta}\in \mathbb{T}_\eta^{Bel}[\mathcal{U}]$
by
\begin{align*}
\canoflip^{Dol}_{\mathbb{T}}(\tvD{\mu,\dot{\nu}}{Y})=
\tvD{-\overline{\partial}\eta, \dot{\nu}-\overline{\partial}[\xi,\eta]}{X}
\in \mathbb{T}_{X}^{Dol}[\mathcal{U}]
\\
\canoflip^{Bel}_{\mathbb{T}}(\tvDBel{\mu,\dot{\nu}}{\eta})=
\tvDBel{-\overline{\partial}\eta, \dot{\nu}-\overline{\partial}[\xi,\eta]}{\xi}
\in \mathbb{T}_{\xi}^{Bel}[\mathcal{U}]
\end{align*}
for $\mathbb{T}^{Dol}_{Y}[\mathcal{U}]$ and $\mathbb{T}^{Bel}_{\eta}[\mathcal{U}]$ , after fixing cochains $\xi$, $\eta\in C^0(\mathcal{U},\mathcal{A}^{0,0}(\Theta_{M_0}))$ with $\delta\xi=X$, $-\overline{\partial}\xi=\mu$, and $\delta\eta=Y$, where $[\xi,\eta]=\{[\xi_i,\eta_i]\}_{i\in I}$.

Indeed,
take a cochain $\dot{\eta}$ with $-\overline{\partial}\dot{\eta}=\dot{\nu}$.
Let
$\dot{Y}=\delta\dot{\eta}+K(\xi,Y)$.
Then
\begin{align*}
\dot{Y}+[X,Y]
&=
\delta\dot{\eta}+(K'(\xi,Y)-K(\eta,X))+K(\eta,X) \\
&=\delta\dot{\eta}+([\xi_j,\eta_j]-[\xi_j,\eta_i]-[\eta_i,\xi_j]+[\eta_i,\xi_i])+K(\eta,X) \\
&=\delta(\dot{\eta}+[\xi,\eta])+K(\eta,X).
\end{align*}
From \eqref{eq:flip2},
we deduce
$$
-\overline{\partial}(\dot{\eta}+[\xi,\eta])
=\dot{\nu}-\overline{\partial}[\xi,\eta]
$$
is the second coordinate of $\canoflip^{Dol}_{\mathbb{T}}(\tvD{\mu,\dot{\nu}}{Y})$ (and hence, of $\canoflip^{Bel}_{\mathbb{T}}(\tvDBel{\mu,\dot{\nu}}{\eta})$).
The commutatice diagrams like as those given in \Cref{prop:involution_cocycles} and \Cref{thm:canonical_flip} also hold in the Dolbeaut-type representations from above observations.
%



\section{Dualization, Pairing and Symplectic form}
\subsection{Trivialization}
Let $\mathcal{U}=\{U_i\}_{i\in I}$ be a locally finite covering with $H^1(U_i, \Theta_{M_0})=0$ for $i \in I$.
We define the model of the \emph{dualization} by
$$
\canoflipdagger_{\mathbb{T}}\colon 
{\bf H}^{1}(\mathcal{U},\mathbb{L}_{q_0})
\ni [X,\varphi]_{q_0}
\to
[\varphi,-X]_{q_0}
\in
{\bf H}^{1,\dagger}(\mathcal{U},\mathbb{L}_{q_0}).
$$
We already discussed the model $\paircot$ of the pairing between $T_{q_0}T^*\teich_g\cong {\bf H}^{1}(\mathcal{U},\mathbb{L}_{q_0})$ and $T^*_{q_0}T^*\teich_g\cong {\bf H}^{1,\dagger}(\mathcal{U},\mathbb{L}_{q_0})$, and the model $\omega_{\mathcal{Q}_g}$ of the symplectic form on $T_{q_0}T^*\teich_g\cong {\bf H}^{1}(\mathcal{U},\mathbb{L}_{q_0})$ in \Cref{sec:ModelTTster_TstarTstar}.

We set
\begin{align*}
\canoflipdagger_{\mathbb{T}}
&\colon
(\mathbb{T}\mathbb{T}^\dagger)_{x_0}[\mathcal{U}]
\to 
(\mathbb{T}^\dagger\mathbb{T}^\dagger)_{x_0}[\mathcal{U}]
\\
\canoflipdagger_{\mathbb{T}}{}^{triv}
&
\colon \mathcal{Q}_{x_0}\oplus H^1(\mathcal{U},\Theta_{M_0})\oplus \mathcal{Q}_{x_0}\to \mathcal{Q}_{x_0}\oplus
 \mathcal{Q}_{x_0}\oplus H^1(\mathcal{U},\Theta_{M_0})
\end{align*}
by
\begin{align*}
\canoflipdagger_{\mathbb{T}}(q,[X,\varphi]_q)&=
(q,[\varphi,-X]_q)
\\
\canoflipdagger_{\mathbb{T}}{}^{triv}(q,[X],\varphi)
&=(q,\varphi, -[X]).
\end{align*}
We can easily check that these maps are well-defined.
From the definitions and \Cref{thm:pairing_cotangent_over_cotangent}, we deduce the following
(cf. \Cref{prop:dualization_symplectic_form_M}).
\begin{theorem}[Dualization, Pairing and Symplectic form]
\label{thm:dualization_pairing_symplectic_form}
The diagram
$$
\begin{CD}
(\mathbb{T}\mathbb{T}^\dagger)_{x_0}[\mathcal{U}]
@>{\trivialization^{*}_{\mathbb{T}}}>>
\mathcal{Q}_{x_0}\oplus  H^1(\mathcal{U},\Theta_{M_0})\oplus \mathcal{Q}_{x_0} \\
@V{\canoflipdagger_{\mathbb{T}}}VV @VV{\canoflipdagger_{\mathbb{T}}{}^{triv}}V
\\
(\mathbb{T}\dagger\mathbb{T}^\dagger)_{x_0}[\mathcal{U}]
@>{\trivialization^{*\dagger}_{\mathbb{T}}}>> 
\mathcal{Q}_{x_0}\oplus\mathcal{Q}_{x_0}\oplus H^1(\mathcal{U},\Theta_{M_0})
\end{CD}
$$
 is commutative.
Furthermore,
$$
\paircot([X,\varphi]_q,\canoflipdagger_{\mathbb{T}}([X,\varphi']_q))
=2\omega_{\mathcal{Q}_g}([X',\varphi']_q,[X,\varphi]_q)
$$
holds for $[X,\varphi]_q$, $[X',\varphi']_q\in (\mathbb{T}\mathbb{T}^\dagger)_{x_0}[\mathcal{U}]$.
\end{theorem}

\section{Switch}
\label{sec:model_switch}
\subsection{Action on the model spaces}
We check the following (cf. \Cref{prop:Derivative_of_pairing_and_flip})
\begin{proposition}[Switch]
\label{prop:switch-cotangent-tangent}
The mapping
%
$$
\switch_{\mathbb{T}}\colon
(\mathbb{T}\mathbb{T}^\dagger )_{x_0}[\mathcal{U}]
\ni
(q,[X,\varphi]_q)
\mapsto
([X],\tvdag{\varphi,q}{[X]}) 
\in (\mathbb{T}^\dagger  \mathbb{T})_{x_0}[\mathcal{U}]
$$
is a well-defined bijection, and satisfies the following with the differential of the pairing:
\begin{equation}
\label{eq:differential_and_pairing}
D\mathcal{P}|_{(v,q)}[V_1,V_2]
=\mathcal{P}_{\mathbb{TT}}|_{[X]}(\canoflip_{\mathbb{T}}(V_1),\switch_{\mathbb{T}}(V_2))
\end{equation}
for $V_1=\tv{X,\dot{Y}}{Y}\in \mathbb{T}_{Y}[\mathcal{U}]\cong T_{[Y]}T\teich_g$, $V_2=[\varphi, X]\in {\bf H}^1(\mathcal{U},\mathbb{L}_{q_0})\cong T_{q_0}\mathcal{Q}_g$,
$[Y]\in H^1(\mathcal{U},\Theta_{M_0})\cong H^1(M_0,\Theta_{M_0})\cong T_{x_0}\teich_g$ and $q\in Q_{M_0}\cong T^*_{x_0}\teich_g$.
\end{proposition}
\begin{proof}
We only check the well-definedness of the map $\switch_{\mathbb{T}}$. The relation \eqref{eq:differential_and_pairing} follows from \Cref{prop:variational_formula_and_pairing}.

Let $(X,\varphi)$ and $(X',\varphi')$ be two representatives of $[X,\varphi]_q\in H^1(\mathcal{U},\mathbb{L}_{q})$.
There is an $\alpha\in C^0(\mathcal{U},\Theta_{M_0})$ such that $X'=X+\delta \alpha$, $\delta\varphi = L_X(q)$, and $\varphi'=\varphi+L_\alpha(q)$. Therefore,
$(\varphi,q)\in \ker(D^{X,\dagger}_1)$, 
$(\varphi',q)\in \ker(D^{X',\dagger}_1)$, 
and 
$$
(\varphi',q)=(\varphi+L_{\alpha}(q),q)=\mathcal{L}^\dagger_{\alpha;X}(\varphi,q)
$$
and hence $\tvdag{\varphi,q}{[X]}$ is well-defined in $\mathbb{T}^\dagger_{[X]}[\mathcal{U}]$.
The bijectivity follows from the definition.
%
\end{proof}
We define the map $\switch^{triv}_{\mathbb{T}}$ by
$$
\begin{CD}
\mathcal{Q}_{x_0}\oplus H^1(\mathcal{U},\Theta_{M_0})\oplus \mathcal{Q}_{x_0}
@>{\switch^{triv}_{\mathbb{T}}}>>
H^1(\mathcal{U},\Theta_{M_0})\oplus \mathcal{Q}_{x_0}\oplus \mathcal{Q}_{x_0}
\\
(q,[X],\psi)
@>>>
([X],\psi,q).
\end{CD}
$$
From the definition of the switch, we have the following commutative diagram:
$$
\begin{CD}
(\mathbb{T}\mathbb{T}^\dagger)_{x_0}[\mathcal{U}]
@>{\trivialization^{*}_{\mathbb{T}}}>>
\mathcal{Q}_{x_0}\oplus H^1(\mathcal{U},\Theta_{M_0})\oplus \mathcal{Q}_{x_0}
\\
@V{\switch_{\mathbb{T}}}VV @VV{\switch^{triv}_{\mathbb{T}}}V \\
(\mathbb{T}^\dagger\mathbb{T})_{x_0}[\mathcal{U}] @>{\trivialization^{\dagger}_{\mathbb{T}}}>>
H^1(\mathcal{U},\Theta_{M_0})\oplus \mathcal{Q}_{x_0}\oplus \mathcal{Q}_{x_0}
\end{CD},
$$
where the trivializations in the vertical directions of the diagram can be chosen arbitrarily.

\subsection{Maps on the second order infinitesimal spaces}
As noticed in \S\ref{sec:trivialization_via_Bers_embedding}, our model spaces $(\mathbb{T}\mathbb{T})_{x_0}[\mathcal{U}]$ and $(\mathbb{T}\mathbb{T}^*)_{x_0}[\mathcal{U}]$ are thought of as a model of $(TT)_{x_0}\teich_g$ and $(TT^*)_{x_0}\teich_g$, respectively, via the Bers embedding (or the Ahlfors-Weill section).
We denote by $\isomorphism^\dagger_{x_0}$ the natural identification from $(TT^*)_{x_0}\teich_g$ to $(\mathbb{T}\mathbb{T}^*)_{x_0}[\mathcal{U}]$.
As discussed in \S\ref{subsec:models_cotangent_spaces}, there are natural bijections
\begin{align*}
\isomorphism^*_{x_0} &\colon (T^*T)_{x_0}\teich_g \to (\mathbb{T}^*\mathbb{T})_{x_0}[\mathcal{U}] \\
\isomorphism^{*\dagger}_{x_0} &\colon (T^*T^*)_{x_0}\teich_g \to (\mathbb{T}^*\mathbb{T}^*)_{x_0}[\mathcal{U}]
\end{align*}
defined from the dual isomorphisms.

From \Cref{prop:switch-cotangent-tangent},
we have the following commutative diagram for the switch between $T^*T\teich_g$ and $TT^*\teich_g$:
\begin{equation}
\label{eq:flip_model}
\begin{CD}
(T^*T)_{x_0}\teich_g @>{\switch_{\teich_g}}>> (TT^*)_{x_0}\teich_g \\
@V{\cong}VV @V{\cong}VV \\
(\mathbb{T}^\dagger\mathbb{T})_{x_0}[\mathcal{U}] @>{\switch_{\mathbb{T}}}>> (\mathbb{T}\mathbb{T}^\dagger)_{x_0}[\mathcal{U}]
\end{CD}
\end{equation}
where two vertical arrows are canonical identification (cf. \S\ref{sec:flip_TstarTMand TTstarM} and \S\ref{sec:trivialization_via_Bers_embedding}).

\begin{remark}
For the record, we notice the following:
In the correspondence \eqref{eq:flip_model} that for $([X],\tvdag{\psi,q}{[X]})$ in the left-hand side, we first choose $X$ of a representative of $[X]$ and the representaive $(\psi,q)\in\ker(D^{X,\dagger}_1)$ of $\tvdag{\psi,q}{[X]}$. Then, $(X,\psi)$ is a cocycle in $Z^1(\mathcal{U},\mathbb{L}_q)$.
\end{remark}

%
%
%

\section{Lie bracket between vector fields}
\label{sec:lie_bracket}
Let $\mathcal{X}$ and $\mathcal{Y}$ be $C^1$-vector fields of type $(1,0)$ defined around $x_0$. We think of $\mathcal{X}$ and $\mathcal{Y}$ as smooth maps from a neighborhood of $x_0$ in $\teich_g$ to $T\teich_g$.

\begin{theorem}[Lie bracket]
\label{thm:lie_bracket}
Let $\mathcal{U}$ be a locally finite covering of $M$ satisfying that each $U_i$ is an embedded closed disk with smooth boundary.
Under the above notation,
let $\mathcal{X}_{x_0}=[X]$ and $\mathcal{Y}_{x_0}=[Y]\in T_{x_0}\teich_g\cong H^1(M_0,\Theta_{M_0})\cong H^1(\mathcal{U},\Theta_{M_0})$. Suppose that
\begin{align*}
\left(D\mathcal{X}(\mathcal{Y})_{[X]}\right)^{10}
&=\tv{Y, \dot{X}}{X}\in \mathbb{T}_X[\mathcal{U}]\cong \mathbb{T}_{[X]}[\mathcal{U}]
\cong T_{[X]}T\teich_g \\
\left(D\mathcal{Y}(\mathcal{X}\right)_{[Y]})^{01}
&=\tv{X, \dot{Y}}{Y}
\in \mathbb{T}_Y[\mathcal{U}]\cong \mathbb{T}_{[Y]}[\mathcal{U}]
\cong T_{[Y]}T\teich_g.
\end{align*}
Then, the Lie bracket of $\mathcal{X}$ and $\mathcal{Y}$ at $x$ is represented by
\begin{equation}
\label{eq:lie_bracket_formula}
\left[\mathcal{X},\mathcal{Y}\right]_{x_0}=[\dot{Y}-\dot{X}+[X,Y]]
\in H^1(M_0,\Theta_{M_0}).
\end{equation}
In particular, $\left[\mathcal{X},\mathcal{Y}\right]_{x_0}=0$ if and only if
$$
\dot{X}-\dot{Y}=[X,Y]+\delta\beta
$$
for some $\beta\in C^0(\mathcal{U},\Theta_{M_0})$.
\end{theorem}

\begin{proof}
From \Cref{prop:involution_cocycles}, 
$$
\canoflip_{\mathbb{Z}}(Y,(X,\dot{Y}))-(X,(Y,\dot{X}))=
(X,(0,\dot{Y}-\dot{X}+[X,Y]))
$$
is a lift of 
$$
\canoflip_{\mathbb{T}}\left([Y],\tv{X, \dot{Y}}{[Y]}\right)-
\left([X],\tv{Y, \dot{X}}{[X]}\right),
$$
and its vertical projection is presented as $[\dot{Y}-\dot{X}+[X,Y]]\in H^1(M_0,\Theta_{M_0})$ (cf. \Cref{prop:vertical_space} and \Cref{remark:sum_double_tangent_space}). The formula \eqref{eq:lie_bracket_formula} follows from \Cref{prop:Lie_bracket}.
\end{proof}

\begin{remark}
Since $\dot{Y}^*=-\dot{Y}-[X,Y]$, \eqref{eq:lie_bracket_formula} is restated as 
\begin{equation}
\label{eq:lie_bracket_formula2}
\left[\mathcal{X},\mathcal{Y}\right]_{x_0}=[\dot{Y}^*+\dot{X}]
\in H^1(M_0,\Theta_{M_0})
\end{equation}
where $*$ means the filp of subscripts defined in \eqref{eq:filp}.
\end{remark}

\section{Lie bracket via the Dolbeaut presentation}
We rephrase \Cref{thm:lie_bracket} under the Dolbeaut presentations.
Let $\mathcal{X}$ and $\mathcal{Y}$ as above.
Fix cochains $\xi$, $\eta\in C^0(\mathcal{U},\mathcal{A}^{0,0}(\Theta_{M_0}))$ with $\delta \xi=X$ and $\delta \eta=Y$.
Suppose that
\begin{align*}
\left(D\mathcal{X}(\mathcal{Y})_{[X]}\right)^{10}
&=\tvD{-\partial \eta, \dot{\nu}}{X}\in \mathbb{T}^{Dol}_X[\mathcal{U}]
\\
&=\tvDBel{-\partial \eta, \dot{\nu}}{\xi}\in \mathbb{T}^{Bel}_\xi[\mathcal{U}]
\\
\left(D\mathcal{Y}(\mathcal{X}\right)_{[Y]})^{01}
&=\tvD{-\partial \xi, \dot{\mu}}{Y}
\in \mathbb{T}^{Dol}_Y[\mathcal{U}]
\\
&=\tvDBel{-\partial \xi, \dot{\mu}}{\eta}
\in \mathbb{T}^{Bel}_\eta[\mathcal{U}].
\end{align*}
From the discussion in \S\ref{subsec:Flip_Dolbeaut},
we have
\begin{align*}
\left[\mathcal{X},\mathcal{Y}\right]_{x_0}=
[\dot{\nu}-\dot{\mu}+\overline{\partial}[\xi,\eta]]
\in T_{x_0}\teich_g,
\end{align*}
which implies that 
$$
\left\langle \left[\mathcal{X},\mathcal{Y}\right]_{x_0},q\right\rangle
=
\iint_{M_0}(\dot{\nu}_i-\dot{\mu}_i+[\xi_i,\eta_i]_{\overline{z}})q
$$
for all $q\in \mathcal{Q}_{x_0}=T^*_{x_0}\teich_g$.
Thus, $\left[\mathcal{X},\mathcal{Y}\right]_{x_0}=0$ if and only if 
there is a global vector field $\chi\in \Gamma(M_0,\mathcal{A}^{0,0}(\Theta_{M_0}))$ such that
$$
\dot{\nu}-\dot{\mu}+\overline{\partial}[\xi,\eta]=-\overline{\partial}\chi.
$$


\chapter[Variational formula]{Variational formulae of $L^1$-Norm function and Teichm\"uller metric}
\label{chap:variational_formula_L1-teich}

\section{$L^1$-norm for holomorphic quadratic differentials}
In this section, we discuss the variational formula of the $L^1$-norm for holomorphic quadratic differentials 
$$
T^*\teich_g\cong \mathcal{Q}_g\ni q\mapsto \LOneNorm(q):=\|q\|
$$
on the cotangent bundle. It is known that the $L^1$-norm function $\LOneNorm$ is of class $C^1$ on the cotangent bundle  except for the zero section (cf. \cite[Lemma 2]{MR0288254}).
For reader's convenience, we summarize Royden's argument briefly. 

Let $x_0=(M_0,f_0)\in \teich_g$. Fix a symplectic basis $\{A_k,B_k\}_{k=1}^g$ of $H_1(M_0,\mathbb{Z})$.
Let $w_0$ be an Abelian differential on $M_0$ with simple zeros. For $x=(M,f)\in \teich_g$, we denote by $w^x_0$ the Abelian differential on $M$ with same $A$-periods as $w_0$. Take a small neighborhood $U$ of $x_0$ so that each zero of $w^x_0$ for $x\in U$ are simple. We label the zeros of $w_0^x$. Let $w^x_i$ be the normalized Abelian differential on $M$ whose $A_k$-period is $\delta_{ik}$ for $k=1,\cdots,g$. For $i=g+1$, $\cdots$, $3g-3$, let $\gamma_i^\mu$ be the Abelian differential of the third kind which has simple pole of residue $1$ at the $(i-g)$-th zero of $w_0^x$, and a simple pole of residue $-1$ at the $2g-2$-nd zero of $w_0^x$.
Then,
$$
U\times \mathbb{C}\mapsto (x,(t_1,\cdots,t_{3g-3}))\mapsto
\left(\sum_{i=1}^gt_iw_i^x+\sum_{i=g+1}^{3g-3}t_j\gamma_j^x\right)w_0^x\in \mathcal{Q}_g
$$
gives a differentiable (holomorphic) local trivialization. We can check that the differential in the right of the above equation is differentiable (cf. Lemma 1 in \cite{MR0288254}).  Hence, the norm function varies as a function of $C^1$. 
We apply (a part of) Royden's argument to obtain our variational formula and present the variational formula in our own setting.

\subsection{Variational formula derived from Royden's calculation}
\label{subsec:Royden_cal}
The calculation here almost follows from that by Royden in \cite{MR0288254}. We start with the following lemma.

\begin{lemma}[Lemma 1 in \cite{MR0288254}]
\label{lem:royden}
Let $q_0$ be a holomorphic quadratic diffetential on a Riemann surface $W$, and $\varphi$ a smooth (not necessarily holomorphic) quadratic differential. Let $K$ be a compact set in $W$. Then the function
$$
f(t)=\iint_K|q_0(z)+t\varphi(z)|dxdy
$$
is differentiable as a function of $t$, and satisfies
$$
f(t)=f(0)+\dfrac{t}{2}\iint_{K}\dfrac{\overline{q_0}\varphi}{|q_0|}
dxdy
+
\dfrac{\overline{t}}{2}\iint_{K}\dfrac{q_0\overline{\varphi}}{|q_0|}dxdy+o(t).
$$
\end{lemma}
In the proof of \Cref{lem:royden}, Royden first discusses the variation
$$
|q_0+t\varphi|=|q_0|+t\dfrac{\overline{q_0}}{2|q_0|}\varphi+\overline{t}\dfrac{q_0}{2|q_0|}\overline{\varphi}+o(t)
$$
on a compact set on which $|q_0|>0$, and take the limit via the exhaustion of $K$ by relatively compact domains of $K\cap \{|q_0|>0\}$.
By applying Royden's argument and \Cref{prop:variation_two_forms}, we also prove the following
(see also the proof of \cite[Theorem 5.3]{MR3413977}).

\begin{theorem}[Variation of $L^1$-norm]
\label{thm:variation_L1-norm}
Let $\mathcal{U}=\{U_i\}_{i\in I}$ be a locally finite covering of $M_0$ with $H^1(U_i,\Theta_{M_0})=0$ for all $i\in I$.
For $V=[X,\varphi]_{q_0}\in {\bf H}^1(\mathcal{U},\mathbb{L}_{q_0})\cong T_{q_0}T^*\teich_g$, we have
\begin{equation}
\label{eq:differential_L1norm}
D \LOneNorm|_{q_0}[V]=
\dfrac{1}{4i}\iint_{M_0}
\dfrac{\overline{q_0}}{|q_0|}\left(\varphi_i-L_{\xi_i}(q_0)\right)d\overline{z}\wedge dz,
\end{equation}
where $q_0\in \mathcal{Q}_{x_0}$ and $\xi=\{\xi_i\}_{i\in I}\in C^0(\mathcal{U},\sob^{1,1}(\Theta_{M_0}))$ with $\delta\xi = X$.
\end{theorem}

We notice the following.

\begin{enumerate}
\item
Since the zeros of $q_0$ is discrete, the integral \eqref{eq:differential_L1norm} makes sense.
\item
In the calculation of the formula \eqref{eq:differential_L1norm},  the following observation may be useful : For a differentiable vector field $X=X(z)\partial_{z}$ on $U_i$,
$$
L_{X}(|q_0|)
=\left(\dfrac{\overline{q_0}}{2|q_0|}(q_0)'X+|q_0|X_z\right)d\overline{z}\wedge dz
=\dfrac{\overline{q_0}}{2|q_0|}L_X(q_0)
$$
on $z_i^0(U_i)\cap \{|q_0|>0\}$. 
\item
The formula \eqref{eq:differential_L1norm} is independent of the choice of $\xi$. Indeed, take $\Xi=\{\Xi_i\}_{i\in I}\in C^0(\mathcal{U},\sob^{1,1}(\Theta_{M_0}))$ with $\delta \Xi=X$. Then $\mathcal{X}=\xi-\Xi$ is a global $\sob^{1,1}$-vector field on $M_0$. In this case
\begin{align*}
&\iint_{M_0}\dfrac{\overline{q_0}}{|q_0|}\left(\varphi_i-L_{\Xi_i}(q_0)\right)-\iint_{M_0}\dfrac{\overline{q_0}}{|q_0|}\left(\varphi_i-L_{\xi_i}(q_0)\right)
\\
&=\iint_{M_0}\dfrac{\overline{q_0}}{|q_0|}L_{\mathcal{X}}(q_0)=-2\iint_{M_0}d(\mathcal{X}|q_0|d\overline{z})=0
\end{align*}
by the Green theorem.
\item
The formula \eqref{eq:differential_L1norm} is independent of the choice of the representative of $V=[X,\varphi]_{q_0}$. Let $\alpha\in C^0(\mathcal{U},\Theta_{M_0})$. We now choose $(X+\delta\alpha, \varphi+L_{\alpha}(q_0))$ as the representative of $V=[X,\varphi]_{q_0}$. In this case, we may think $\xi+\alpha$ instead of $\xi$, and the inside of the parentheses of the integrand in \eqref{eq:differential_L1norm} becomes
\begin{align*}
\varphi_i+L_{\alpha_i}(q_0)-L_{\xi_i+\alpha_i}(q_0)=\varphi-L_{\xi_i}(q_0).
\end{align*}
This means that the integral is invariant.
\item
The formula \eqref{eq:differential_L1norm} recovers Royden's formula in \Cref{lem:royden}. Indeed, when $[X]\in H^1(\mathcal{U},\Theta_{M_0})$ is trivial, from the above observation, we may assume that $\xi\in C^0(\mathcal{U},\Theta_{M_0})$. Hence, $\varphi_i-L_{\xi_i}(q_0)$ in the integrand of the right-hand side of \eqref{eq:differential_L1norm} is defined from a (global) holomorphic quadratic differential on $M_0$.
\label{enumerate:L1-derivateive5}
\end{enumerate}

From the definition of the pairing \eqref{eq:pairing_cotangent}, we conclude the following:

\begin{corollary}[$\partial$-derivative of $L^1$-norm]
\label{coro:gradient_vector_L1}
Let $q_0\in \mathcal{Q}_g$. Then,
$$
\partial \LOneNorm|_{q_0}=
\dfrac{1}{2}[-L_{\eta}(q_0), Y]^\dagger_{q_0}=-\dfrac{1}{2}[L_{\eta}(q_0), -Y]^\dagger_{q_0}
\in {\bf H}^{1,\dagger}(\mathcal{U},\mathbb{L}_{q_0})\cong T^*_{q_0}\mathcal{Q}_g
$$
where $L_\eta(q_0)=\{L_{\eta_i}(q_0)\}_{i\in I}$, $\eta_i=\{\eta_i\}_{i\in I}\in C^0(\mathcal{U},\sob^{1,1}(\Theta_{M_0}))$ with 
$(\eta_i)_{\overline{z}}=-\overline{q_0}/|q_0|$ for $i\in I$ 
and $Y=\delta\eta\in Z^1(\mathcal{U},\Theta_{M_0})$.
\end{corollary}

Namely, \Cref{coro:gradient_vector_L1} implies
$$
\partial \LOneNorm([X,\varphi]_{q_0})=\paircot\left([X,\varphi]_{q_0},-\dfrac{1}{2}[L_{\eta}(q_0), -Y]^\dagger_{q_0}\right)
$$
for $[X,\varphi]_{q_0}\in {\bf H}^{1}(\mathcal{U},\mathbb{L}_{q_0})\cong T_{q_0}\mathcal{Q}_g$ and $q_0\in \mathcal{Q}_g$.

\begin{remark}
\label{remark:holomorphicity_cocycle}
Under the notation in the above discussion,
we notice that $L_{\eta_i}(q_0)$ is holomorphic on $U_i$ for all $i\in I$ and satisfies
$L_{\eta_j}(q_0)-L_{\eta_i}(q_0)=L_{Y_{ij}}(q_0)$ for $i$, $j\in I$. Indeed,
\begin{align*}
L_{\eta_i}(q_0)_{\overline{z}}
&=(q_0)'(\eta_i)_{\overline{z}}+2q_0(\eta_i)_{\overline{z}z}
=-(q_0)'\dfrac{\overline{q_0}}{|q_0|}
+
2q_0\dfrac{\overline{q_0}^2(q_0)'}{2|q_0|^3}=0
\end{align*}
on $U_i-{\rm Zero}(q_0)$.
Since $\overline{q_0}/|q_0|\in L^{\infty}(U_i)$, we may think that $\eta_i$ is H\"older continuous with exponent $1-(2/p)$ and $(\eta_i)_z\in L^p(U_i)$ for arbitrary $p>2$ by applying Calder{\'o}n-Zygmund's theorem after identifying $U_i$ with a bounded domain in $\mathbb{C}$ (e.g. \cite[Lemma 4.20, Proposition 4.23]{MR1215481}).
In particular, we have $L_{\eta_i}(q_0)\in L_{loc}^2(U_i)$ under the chart of $U_i$. Therefore, the Riemann removable singularity theorem (or the Weyl lemma) asserts that $L_{\eta_i}(q_0)$ is holomorphic on $U_i$.
\end{remark}

\subsection{Variational formula  from the universal deformation}
In \cite{MR523212}, Hubbard and Masur discuss the representation of the tangent vectors at $q_0\in \mathcal{Q}_g$ adapted to the condition of critical points of $q_0$. In this section, we  restate the differential formula in this setting. The subject of this subsection is not needed for further development in this paper. However, it is interesting by itself.

Let $q_0\in \mathcal{Q}_{M_0}$ and $[X,\varphi]_{q_0}\in {\bf H}^1(\mathcal{U},\mathbb{L}_{q_0})\cong T_{q_0}\mathcal{Q}_g$, where $\mathcal{U}=\{U_i\}_{i\in I}$ is a locally finite covering with $H^1(U_i,\Theta_{M_0})=0$ for $i\in I$. Let $\xi=\{\xi_i\}_{i\in I}$ with $\delta \xi=X$. We consider the reasonable decomposition $\{R_j\}_{j\in J}$ of $M_0$ such that
\begin{itemize}
\item
each $R_j$ contains at most one zero of $q_0$ in its interior, and $|q_0|>0$ on $\partial R_j$; and
\item
each $R_j$ is contained in $U_{i(j)}$ for some $i(j)\in I$.
\end{itemize}
Let $J$ be the set of indices $j\in J$ such that the interior of $R_j$ contains a zero $p_j\in M_0$ of $q_0$. For $j\in J\setminus J_0$, we fix $p_j\in {\rm Int}(R_j)$. We also assume
\begin{itemize}
\item
for $j\in J_0$, there is a local chart $(V_j,w_j)$ around $p_j$ such that $\overline{R}_j\subset V_j$, $w_j(p_j)=0$, and $q_0=w^{k_j}dw^2$ under the coordinate $(V_j,w_j)$.
\end{itemize}
For $j\in J\setminus J_0$, we set $k_j=0$. 
For $k\ge 0$, we define
$$
P_k=\{(a_0+a_1w+\cdots+a_{k-2}w^{k-2})dw^2\mid a_i\in \mathbb{C}\}
$$
for $k\ge 2$, and $P_k=\{0\}$ for $k=0$, $1$. 
From the universal deformation theorem by Hubbard and Masur,
each $\varphi_{i(j)}$ is represented as
$$
\varphi_{i(j)}=\mathfrak{p}_{\varphi;j}+L_{Z_j}(q_0)
$$
on $R_j$,
where $\mathfrak{p}_{\varphi;j}\in P_{k_j}$ and $Z_j$ is a (bounded) holomorphic vector field on $V_j$ (cf. \cite[Proposition 3.1]{MR523212}).

Under the above notation, the formula \eqref{eq:differential_L1norm} is rewritten as

\begin{equation}
\label{eq:differential_L1norm_universal_deformation}
D \LOneNorm|_{q_0}[V]=
\dfrac{1}{4i}\sum_{j\in J}\iint_{R_j}
\left(\dfrac{\overline{w}}{|w|}\right)^{k_j}\left(\mathfrak{p}_{\varphi;j}-L_{\xi_{i(j)}-Z_j}(w^{k_j}dw^2)\right)d\overline{w}\wedge dw.
\end{equation}

\subsection{Tangent space to the unit sphere bundle}
\label{subsec:tangent_space_to_the_unit_sphere_bundle}
In this section, we summerize the structure of the real and complex tangent spaces to the unit tangent bundle in $\mathcal{Q}_g$. 
As Royden observed that the $L^1$-norm function is of class $C^1$ on $\mathcal{Q}^\times_g=\mathcal{Q}_g-\{0\}$ (``$0$" means the zero section),
See \S\ref{sec:Royden_theorem_revisited} for detail.

Let 
$$
\mathcal{SQ}_g=\{q\in \mathcal{Q}_g\mid \|q\|=1\}.
$$
The bundle $\mathcal{SQ}_g\to \teich_g$ is the unit sphere bundle over $\teich_g$. Let $x_0=(M_0,f_0)\in \teich_g$ and $q_0\in \mathcal{Q}_{M_0}$. From \Cref{thm:variation_L1-norm}, the real tangent space to $\mathcal{SQ}_g$ at $q_0$ is described by
\begin{align*}
T^{\mathbb{R}}_{q_0}\mathcal{SQ}_g
&=
\left\{
V\in T_{q_0}\mathcal{Q}_g
\mid
{\rm Re}\left(\partial \LOneNorm(V)\right)=0
\right\}
\\
&=\left\{
[X,\varphi]_{q_0}\in T_{q_0}\mathcal{Q}_g
\mid
{\rm Re}\iint_{M_0}\dfrac{\overline{q_0}}{2|q_0|}\left(\varphi_i-L_{\xi_i}(q_0)\right)=0
\right\},
\end{align*}
where
$\xi=\{\xi_i\}_{i\in I}\in C^0(\mathcal{U},\sob^{1,1}(\Theta_{M_0})$ with $\delta\xi=X$
and $\mathcal{U}=\{U_i\}_{i\in I}$ is a locally finite covering with $H^1(U_i,\Theta_{M_0})=0$ for $i\in I$. In the following argument we identify $\mathcal{Q}_{x_0}$ with the vertical space of $T_{q_0}\mathcal{Q}_g={\bf H}^1(\mathcal{U},\mathbb{L}_{q_0})$. Namely,
$$
\begin{CD}
\mathcal{Q}_{x_0}\ni \psi @>{\cong}>> [0,\psi]_{q_0}\in {\bf H}^1(\mathcal{U},\mathbb{L}_{q_0}).
\end{CD}
$$
(cf. \Cref{remark:differential_projection_cotangent}). Let 
$$
(\mathcal{Q}_{x_0})^\perp_{q_0}=
\left\{[0,\psi]_{q_0}\in {\bf H}^1(\mathcal{U},\mathbb{L}_{q_0})
\mid
\psi\in \mathcal{Q}_{x_0},
\iint_{M_0}\dfrac{\overline{q_0}}{|q_0|}\psi=0
\right\}.
$$

Take $[X_{q_0}]\in H^1(\mathcal{U},\Theta_{M_0})$ such that the tangent vector corresponding to $[X_{q_0}]$ is presented by $\overline{q_0}/|q_0|$.
Fix a guiding frame $\GuaidF$ such that $\GuaidF([X_{q_0}])=\overline{q_0}/|q_0|$.
Take $\GoodS$ as \Cref{prop:linear-map-L} for $\GuaidF$. In \Cref{prop:linear-map-L}, we assume that all $\GuaidF([X])$ is smooth. However, we can treat this case in similar way.

We define a \emph{horizontal lift} of $[X]\in H^1(\mathcal{U},\Theta_{M_0})$ by
$$
V^H([X])=[X,L_{\GoodS(X)_i}(q_0)-Q_X]_{q_0}\in T_{q_0}\mathcal{Q}_g={\bf H}^1(\mathcal{U},\mathbb{L}_{q_0}),
$$
where $Q_X\in C^0(\mathcal{U},\mathcal{A}^{0,0}(\Omega_{M_0})^{\otimes 2})$ with
$$
\overline{\partial} Q_X=\overline{\partial}L_{\GoodS(X)_i}(q_0).
$$
and
\begin{equation}
\label{eq:uniquness_horizontal_lift-1}
\iint_{M_0}\GuaidF([Z])Q_X=0
\end{equation}
for all $[Z]\in H^1(\mathcal{U},\Theta_{M_0})$.
Notice that the horizontal lift is independent of choice of the representative $X$.
Indeed, for $\alpha\in C^0(\mathcal{U},\Theta_{M_0})$, from the definition of $Q_X$. $Q_{X+\delta\alpha}=Q_X$ and
\begin{align*}
[X+\delta\alpha,L_{\GoodS(X+\delta\alpha)_i}(q_0)-Q_{X+\delta\alpha}]_{q_0}
&=
[X+\delta\alpha,L_{\GoodS(X)_i+\alpha_i}(q_0)-Q_X]_{q_0}
\\
&=
[X+\delta\alpha,L_{\GoodS(X)_i}(q_0)-Q_X+L_{\alpha_i}(q_0)]_{q_0}
\\
&=[X,L_{\GoodS(X)_i}(q_0)-Q_X]_{q_0}.
\end{align*}
Therefore,
\begin{equation}
\label{eq:horizontal_lift-1}
T_{x_0}\teich_g\ni H^1(\mathcal{U},\Theta_{M_0})\ni [X]\mapsto V^H([X])
\in T_{q_0}\mathcal{Q}_g
\end{equation}
is an injective $\mathbb{C}$-linear map.
We also call the image $T^H_{q_0}\mathcal{Q}_g$ the \emph{horizontal lift of $T_{x_0}\teich_g$ (with the guiding frame $\GuaidF$)}.
We claim

\begin{proposition}
\label{prop:CR_L1}
The maximal complex subspace $T^{1,0}_{q_0}\mathcal{SQ}_g$ of $T^{\mathbb{R}}_{q_0}\mathcal{SQ}_g$ satisfies
\begin{equation}
\label{eq:CR_L1}
T^{1,0}_{q_0}\mathcal{SQ}_g
=T^H_{q_0}\mathcal{Q}_g\oplus (\mathcal{Q}_{x_0})^\perp_{q_0}.
\end{equation}
In particular,
$$
T^{\mathbb{R}}_{q_0}\mathcal{SQ}_g=T^H_{q_0}\mathcal{Q}_g\oplus (\mathcal{Q}_{x_0})^\perp_{q_0}\oplus \langle [0, iq_0]_{q_0}\rangle_{\mathbb{R}}.
$$
\end{proposition}

\begin{proof}
The horizontal lift $T^H_{q_0}\mathcal{Q}_g$ and $(\mathcal{Q}_{x_0})^\perp_{q_0}$ are linearly independent since $(\mathcal{Q}_{x_0})^\perp_{q_0}$ is contained in the kernel of the differential of the projection $D\Pi_{\teich_g}^\dagger\colon T\mathcal{Q}_g\to T\teich_g$. Therefore, the dimension of the right-hand side of \eqref{eq:CR_L1} is $3g-3+3g-4=6g-7$.
For a horizontal lift $V=[X,L_{\GoodS(X)_i}(q_0)-Q]_{q_0}\in T^H_{q_0}\mathcal{Q}_g$, from \Cref{thm:variation_L1-norm},
$$
D\LOneNorm|_{q_0}[V]=\iint_{M_0}\dfrac{\overline{q_0}}{2|q_0|}(L_{\GoodS(X)_i}(q_0)-Q-L_{\GoodS(X)_i}(q_0))=
-\dfrac{1}{2}\iint_{M_0}\GuaidF([X_{q_0}])Q=0.
$$
Therefore, $T^H_{q_0}\mathcal{Q}_g$ is a subspace of $T^{1,0}_{q_0}\mathcal{SQ}_g$.
Since the dimension of $T^{1,0}_{q_0}\mathcal{SQ}_g$ is $6g-7$, \eqref{eq:CR_L1} holds. We can easily check the decomposition of $T^{\mathbb{R}}_{q_0}\mathcal{SQ}_g$. 
\end{proof}

\begin{remark}
The horizontal lift $T^H_{q_0}\mathcal{Q}_g$ is dependent of the choice of the guiding frame $\GuaidF$. In fact, the dependence is caused from an observation that the differential $Q$ in the definition of the horizontal lift determined up to $(\mathcal{Q}_{x_0})^\perp_{q_0}$ when the guiding frame changes.
\end{remark}

\subsection{Conjectures on Levi convexity of unit ball bundles}
\label{subsec:conjecure_Levi}
Suppose $q_0$ is generic, that is, $q_0$ is in the principal stratum.
Under the period coordinates $z=(z_1,\cdots,z_{6g-6})$ via the double (branched) covering around $q_0$, the $L^1$-norm is presented as
$$
\LOneNorm (z) =i\sum_{j=1}^{3g-3}(z_j\overline{z_{g+j}}-z_{g+j}\overline{z_j})
$$
after fixing a symplectic basis on the covering surface. 
Since the period coordinates via the double covering is a complex analytic chart,
the Levi form of $\LOneNorm$ has exactly $3g-3$ positive (negative) eigenvalues (cf. \cite{MR644018}, \cite{MR1094714}. See also \cite[\S5.4]{MR3413977}).

Recall that  a $C^2$ function on in a complex manifold of dimension $n$ is said to be \emph{strictly $k$-pseudoconvex} if its Levi form has at least $n-k+1$ positive eigenvalues. A domain $D$ in a complex manifold is called \emph{$k$-convex} at $x\in \partial D$ if there is a neighborhood $U$ of $x$ and a strictly $k$-pseudoconvex function $\varphi$ on $U$ such that $D\cap U=\{\varphi(x)<\varphi(x)\}$ (cf. \cite[\S2.2]{MR1326623}). In this sence, the $L^1$-norm function is strictly $(3g-2)$-pseudoconvex on the principal stratum and the unit ball bundle $\mathcal{Q}_g^1=\{q\in \mathcal{Q}_g\mid \LOneNorm(q)<1\}$ is $(3g-2)$-convex at each generic boundary point.

Dumas \cite{MR3413977} observes that the complex Hessian (Levi form) of the $L^1$-norm function is positive definite at $q_0$ along $\mathcal{Q}_{x_0}$.
From an observation in the simplest case given in\S\ref{subsec:Sign_of_the_Levi_form}, we pose the following conjecture:

\begin{conjecture}
\label{conj:Levi-L1}
When $q_0\in \mathcal{Q}_g$ is generic, the Levi form of the $L^1$-norm function is negative on a horizontal lift $T^H_{q_0}\mathcal{Q}_g$.
\end{conjecture}
%
%

\begin{conjecture}
\label{conj:Levi-teich}
The Teichm\"uller metric is (strictly) plurisubharmonic on $T\teich_g$.
\end{conjecture}

\Cref{conj:Levi-L1} is possibly rephrased for general $q\in \mathcal{Q}_g$.

%

\section{Royden's theorem revisited}
\label{sec:Royden_theorem_revisited}
We first discuss the regularity of the $L^1$-norm function with our formula.
Namely, the continuity of the total differential $d\LOneNorm$ is derived from the presentation \eqref{eq:differential_L1norm} as follows.

Let $\{[X_n,\varphi_n]_{q_n}\}_{n=1}^\infty$ be a sequence in $T\mathcal{Q}_g$ converges to $[X_0,\varphi_0]_{q_0}$ with $q_0\ne 0$.
from the definition of the tangent vectors to $\mathcal{Q}_g$ (\S\ref{subsec:description_Hol}), we can take an appropriate cochain $\xi_n$ with $\delta \xi_n=X_n$ ($n=0,1,\cdots$) such that $(\varphi_n)_i$ and $(\xi_n)_i$ and their derivatives converge to $(\varphi_0)_i$ on $(\xi_0)_i$ and their derivatives uniformly on any compact sets on $U_i$ with suitable covering $\mathcal{U}=\{U_i\}_{i\in I}$ after trivializing locally the universal family over the Teichm\"uller space (see the discussion in \S\ref{subsec:charts}). Therefore,
$D \LOneNorm([X_n,\varphi_n]_{q_n})$ tends to $D \LOneNorm([X_0,\varphi_0]_{q_0})$ as $n\to \infty$
by Lebesgue theorem. This implies the continuity of $D\LOneNorm$ on $T\mathcal{Q}_g$.

The derivative $D\LOneNorm$ presents the $\partial$-derivative of the $L^1$-norm function (see \Cref{coro:gradient_vector_L1}). Since the $L^1$-norm function $\LOneNorm$ is real-valued, $\overline{\partial}\LOneNorm=\overline{\partial\LOneNorm}$ holds as the identify for sections to the complexified cotangent space 
$$
(T^*)^\mathbb{C}\mathcal{Q}_g=(T^*)^{\mathbb{R}}\mathcal{Q}_g\otimes \mathbb{C}=(T^*)^{1,0}\mathcal{Q}_g\oplus (T^*)^{0,1}\mathcal{Q}_g\to \mathcal{Q}_g^\times =\mathcal{Q}_g-\{0\},
$$
where ``$0$" means the zero section of the bundle $\mathcal{Q}_g\to \teich_g$.
As a consequence, we conclude that the total differential $d\LOneNorm=\partial\LOneNorm+\overline{\partial}\LOneNorm$ is a continuous section on $\mathcal{Q}_g^\times$.
%

For the record, by combining with \Cref{prop:strictly_convex-L1-norm}, we summarize as follows (cf. \cite[Lemma 2]{MR0288254}).

\begin{corollary}[Royden]
\label{coro:L1-norn_C1}
The $L^1$-norm function $\LOneNorm$ on $\mathcal{Q}_g^\times =\mathcal{Q}_g-\{0\}$ is strictly convex and of class $C^1$.
\end{corollary}

As noted by Royden in \cite{MR0288254}, the regularity of the Teichm\"uller metric follows from the regularity of the $L^1$-norm and the following criterion (cf. \cite[Lemma]{MR0288254}). For the convenience of readers, we give a proof of the criterion at \S\ref{sec:Royden_s_criterion}.

We recall that a function $G$ on $\mathbb{C}^n$ is said to be \emph{complex homogeneous} if $G(\alpha\eta)=|\alpha|G(\eta)$ for $\alpha\in \mathbb{C}$ and $\eta\in \mathbb{C}^n$ (See \S\ref{chap:introduction}). 

\begin{proposition}[Royden]
\label{prop:Royden_criterion}
Let $U$ be an open subset of $\mathbb{C}^m$ and $G(x,\eta)$ a continuous function on $U\times \mathbb{C}^n$, which for each $x$ is a positive convex, and complex homogeneous function in $\eta\in \mathbb{C}^n$. Define $F(x,\xi)$ on $U\times \mathbb{C}^n$ by 
$$
F(x,\xi)=\sup\{{\rm Re}(\eta(\xi)) \mid G(x,\eta)=1\}
$$
where $\eta(\xi)$ denotes the $\mathbb{C}$-bilinear pairing
$$
\eta(\xi)=\sum_{i=1}^n\eta_i\xi_i.
$$
Then, $F$ is continuous on $U\times \mathbb{C}^n$, and $F(x,\cdot)$ is a positive convex complex homogeneous function in $\xi\in \mathbb{C}^n$ for each $x\in U$. In addition,
\begin{itemize}
\item for $x\in U$, if $G(x,\cdot)$ is of class $C^1$ on $\mathbb{C}^n-\{0\}$, then $F(x,\cdot)$ is strictly convex; and
\item if $G(x,\cdot)$ is strictly convex for each $x\in U$, and $G$ is totally differentiable in the variable $x\in U$ and the $x$-derivative is continuous on $U\times (\mathbb{C}^n-\{0\})$, $F$ is of class $C^1$ on $U\times (\mathbb{C}^n-\{0\})$. In fact, when
$$
G(x+\Delta x,\eta)=G(x,\eta)+A_{x,\eta}(\Delta x)+\overline{A_{x,\eta}(\Delta x)}+o(|\Delta x|)
$$
as $|\Delta x|\to 0$,
\begin{align*}
F(x+\Delta x,\xi+\Delta \xi)
&=-F(x,\xi)(A_{x,\eta_{x,\xi}}(\Delta x)+\overline{A_{x,\eta_{x,\xi}}(\Delta x)})
\\
&\qquad +\dfrac{1}{2}\eta_{x,\xi}(\Delta \xi)+\dfrac{1}{2}\overline{\eta_{x,\xi}(\Delta \xi)}
+o(|\Delta x|+|\Delta\xi|)
\end{align*}
as $|\Delta x|+|\Delta \xi|\to 0$, where $\eta_{x,\xi}\in \mathbb{C}^n$ is a unique vector satisfying $F(x,\xi)={\rm Re}(\eta_{x,\xi}(\xi))$ and $G(x,\eta_{x,\xi})=1$.
\end{itemize}
\end{proposition}
Actually, Royden stated the proposition for real Finsler metrics. We discuss in complex analytic case for the sake of (more) direct applications to our purpose. However, the proof is essentially same as that of the real case. In summary, we obtain

\begin{corollary}[Royden]
\label{coro:teichmuller_metric_C1}
The Teichm\"uller metric $\teichmullernorm$ on $T\teich_g-\{0\}$ is strictly convex and of class $C^1$.
\end{corollary}

\section{Teichm\"uller metric}
The Teichm\"uller metric is thought of as a function on the tangent bundle:
$$
T\teich_g\ni v\mapsto \teichmullernorm(\pi_{\teich_g}(v),v)
$$
where $\pi_{\teich_g}\colon T\teich_g\to \teich_g$ is the projection.
(cf. \S\ref{subsec:Teichmuller_metric}).
In this section, we compute the differential of the Teichm\"uller metric
on $T\teich_g$ under our setting.

\subsection{Variational formula}
We prove the following.

\begin{theorem}[Derivative of the Teichm\"uller metric]
\label{thm:derivative_T-metric}
Let $x_0=(M_0,f_0)\in \teich_g$ and $v_0\in T_{x_0}\teich_g$. 
Let $\mathcal{U}=\{U_i\}_{i\in I}$ be a locally finite covering of $M_0$ such that each $U_i$ is a closed (topological) disk with smooth boundary in $M_0$.
Suppose thar $v_0$ is represented by the Teichm\"uller differential $\teichmullernorm(x_0,v_0)\dfrac{\overline{q_0}}{|q_0|}$ for $q_0\in \mathcal{Q}_{M_0}-\{0\}$.
Then, for $V=\tv{X,\dot{Y}}{Y}\in \mathbb{T}_Y[\mathcal{U}]\cong T_{v_0}T\teich_g$,
\begin{align*}
D\teichmullernorm|_{v_0}[V]
&=
\dfrac{1}{4i\|q_0\|}
\iint_{M_0}
\left(\dot{\nu}_iq_0+
\mu L_{\eta_i}(q_0)-L_{\eta_i}(\mu q_0)
\right)
d\overline{z}\wedge dz.
\end{align*}
where 
\begin{itemize}
\item
$q_0\in H^0(\mathcal{U},\Omega_{M_0}^{\otimes 2})=\mathcal{Q}_{M_0}$;
\item
$Y\in Z^1(\mathcal{U},\Theta_{M_0})$ with $[Y]=v_0$ in $H^1(\mathcal{U},\Theta_{M_0})\cong T_{x_0}\teich_g$; 
\item
$\xi$, $\eta\in C^0(\mathcal{U},\sob^{2,1}(\Theta_{M_0}))$ with $\delta\xi = X$ and $\delta \eta=Y$;
and
\item
$\tvD{\mu,\dot{\nu}}{Y}\in \mathbb{T}^{Dol}_Y[\mathcal{U}]$ with $\connectinghomo(\tvD{\mu,\dot{\nu}}{Y})=\tv{X,\dot{Y}}{Y}$ and $\mu=-\overline{\partial}\xi_i$ on $U_i$.
\end{itemize}
\end{theorem}

\begin{proof}
By considering $q_0/\|q_0\|$ instead of $q_0$ in the statement, we may assume that $\|q_0\|=1$.
From the discussion in \S\ref{sec:remark_smoothness}, we may also assume that
$$
(\eta_i)_{\overline{z}}=-\teichmullernorm(x_0,v_0)\dfrac{\overline{q_0}}{|q_0|}
$$
on $U_i$.
Consider a non-negative $C^1$-function
\begin{align}
\label{eq:equation_variation_T-metric}
F(v,q)
&=\teichmullernorm(\pi_{\teich_g}(v),v)\LOneNorm(q)-{\rm Re}\pairingsymb(v,q)
\end{align}
on the Whitney sum
$T\teich_g\oplus T^*\teich_g$.
From the assumption, by the Teichm\"uller theorem, $F$ attains the minimum $F(v_0,q_0)=0$ at $(v_0,q_0)$. For the simplicity we let $\teichmullernorm_0=\teichmullernorm(\pi_{\teich_g}(v_0),v_0)$.

Since $F$ attains the minumum at $(v_0,q_0)$ ($\|q_0\|=1$) and the pairing function $\pairingsymb$ is holomorphic,
$$
0=DF[V_1,V_2]=D\teichmullernorm|_{v_0}[V_1]+\teichmullernorm_{0}D \LOneNorm|_{q_0}[V_2]-
\dfrac{1}{2}\left.D\pairingsymb\right|_{(v_0,q_0)}\left[
V_1,V_2
\right]
$$
for all $(V_1,V_2)\in T_{(v_0,q_0)}(T\teich_g\oplus T^*\teich_g)$. 
From \Cref{coro:variational_pairing_Dol} and \Cref{thm:variation_L1-norm},
\begin{align*}
D\teichmullernorm|_{v_0}[V_1]
&=
\dfrac{1}{2}\left.D\pairingsymb\right|_{(v_0,q_0)}\left[
V_1,V_2
\right]
-
\teichmullernorm_{0}D \LOneNorm|_{q_0}[V_2]
\\
&=
\dfrac{1}{4i}
\iint_{M_0}
\left(
\left(
(\dot{\nu}_i-[\xi_i,\eta_i]_{\overline{z}}) q_0+\teichmullernorm_0\dfrac{\overline{q_0}}{|q_0|}\varphi_i
\right)
-
L_{\xi_i}
\left(\teichmullernorm_0\dfrac{\overline{q_0}}{|q_0|}q_0
\right)
\right)
d\overline{z}\wedge dz
\\
&\qquad
-
\teichmullernorm_0
\dfrac{1}{4i}\iint_{M_0}
\dfrac{\overline{q_0}}{|q_0|}\left(\varphi_i-L_{\xi_i}(q_0)\right)d\overline{z}\wedge dz
\\
&=
\dfrac{1}{4i}
\iint_{M_0}
\left(
\left(
(\dot{\nu}_i-[\xi_i,\eta_i]_{\overline{z}}) q_0+\teichmullernorm_0\dfrac{\overline{q_0}}{|q_0|}\varphi_i
\right)
-
\dfrac{\teichmullernorm_0\overline{q_0}}{2|q_0|}L_{\xi_i}(q_0)
\right)
d\overline{z}\wedge dz
\\
&\qquad
-
\teichmullernorm_0
\iint_{M_0}
\dfrac{\overline{q_0}}{|q_0|}\left(\varphi_i-L_{\xi_i}(q_0)\right)d\overline{z}\wedge dz
\Bigg)
\\
&=
\dfrac{1}{4i}
\iint_{M_0}
\left((\dot{\nu}_i-[\xi_i,\eta_i]_{\overline{z}})q_0+\dfrac{\teichmullernorm_0\overline{q_0}}{2|q_0|}L_{\xi_i}(q_0)
\right)
d\overline{z}\wedge dz.
\end{align*}
Since $[\xi_i,\eta_i]_{\overline{z}}=-[\mu,\eta_i]+[\xi_i,(\eta_i)_{\overline{z}}]$ and
\begin{align*}
[\mu,\eta_i]q_0
&=(\mu(\eta_i)_z-\mu_z\eta_i)q_0\\
&=\mu(\eta_i)_zq_0+\mu(\eta_i)_zq_0+\eta_i\mu (q_0)'-L_{\eta_i}(\mu q_0) \\
&=\mu L_{\eta_i}(q_0)-L_{\eta_i}(\mu q_0) \\
[\xi_i,(\eta_i)_{\overline{z}}]q_0
&=\left(\xi_i(\eta_i)_{\overline{z}z}-(\xi_i)_z(\eta_i)_{\overline{z}}\right)q_0
\\
&=
\xi_i\dfrac{\teichmullernorm_0\overline{q_0}}{2|q_0|}(q_0)'+
\teichmullernorm_0(\xi_i)_z|q_0|
=\dfrac{\teichmullernorm_0\overline{q_0}}{2|q_0|}L_{\xi_i}(q_0),
\end{align*}
we obtain
\begin{align*}
D\teichmullernorm|_{v_0}[V_1]&=
\dfrac{1}{4i}
\iint_{M_0}
\left(\dot{\nu}_iq_0+
\mu L_{\eta_i}(q_0)-L_{\eta_i}(\mu q_0)
\right)
d\overline{z}\wedge dz,
\end{align*}
which is what we wanted.
\end{proof}

Hence, from \Cref{thm:doulbeaut_presentation_pairing}, we conclude the following
(cf. \Cref{remark:holomorphicity_cocycle}).

\begin{corollary}[$\partial$-differential of the Teichm\"uller metric]
\label{coro:Differential_Teichmuller_metric}
Let $v_0=[Y]\in H^1(\mathcal{U},\Theta_{M_0})\cong T_{x_0}\teich_g$.
Suppose that $\teichmullernorm_0=\teichmullernorm(\pi_{\teich_g}(v_0),v_0)$ and  $v_0$ is represented by $\teichmullernorm_0\overline{q_0}/|q_0|$, where $q_0\in \mathcal{Q}_{x_0}-\{0\}$. Then,
$$
\partial \teichmullernorm|_{v_0}=\dfrac{1}{2\|q_0\|}\tvdag{L_{\eta_i}(q_0),q_0}{[Y]}
\in \mathbb{T}^\dagger_{[Y]}[\mathcal{U}]\cong T^*_{v_0}T\teich_g
$$
as a $1$-form on $T\teich_g$,
where $\eta=\{\eta_i\}_{i\in I}\in C^0(\mathcal{U},\sob^{2,1}(\Theta_{M_0-{\rm Zero}(q_0)}))$ with $\delta \eta=Y$ and $(\eta_i)_{\overline{z}}=-\teichmullernorm_0\overline{q_0}/|q_0|$.
\end{corollary}

\Cref{coro:gradient_vector_L1} is equivalent to
$$
\partial \teichmullernorm
\left(\tvD{\mu,\dot{\nu}}{Y}\right)=
\mathcal{P}_{\mathbb{TT}}\left(
\tv{X,\dot{Y}}{[Y]},\dfrac{1}{2\|q_0\|}\tvdag{L_{\eta_i}(q_0),q_0}{[Y]}
\right)
$$
for $Y=\delta \eta$, $\tv{X,\dot{Y}}{Y}\in \mathbb{T}_{Y}[\mathcal{U}]\cong T_{v_0}T\teich_g$, $\tvD{\mu,\dot{\nu}}{Y}=\connectinghomo(\tv{X,\dot{Y}}{Y})$, and $v_0=\mathscr{T}_{x_0}([Y])$.

\begin{remark}
\label{remark:teichmuller_metric_section}
From \Cref{remark:dolbeault_rep_variation} and Remark \eqref{enumerate:L1-derivateive5} after the statement of \Cref{thm:variation_L1-norm}, when $[X]=0$, $\dot{\nu}=\dot{\nu}_i$ can be assumed to be a globally defined Beltrami differential on $M_0$ and the formula \eqref{thm:derivative_T-metric} becomes 
$$
D\teichmullernorm|_{v_0}[V]
=\dfrac{1}{4i\|q_0\|}
\iint_{M_0}
\dot{\nu}_i(z) q_0(z)
d\overline{z}\wedge dz
=\dfrac{1}{2}
\iint_{M_0}
\dot{\nu}(z) \dfrac{q_0(z)}{\|q_0\|}
dxdy
=\dfrac{1}{2}\langle \dot{\nu},q_0/\|q_0\|\rangle.
$$
In particular, in this case, we derive the following formula:
\begin{equation}
\label{eq:derivative_Teichmuller_vertical}
\left.\dfrac{d}{dt}\kappa(x_0,v_0+tV)\right|_{t=0}=2{\rm Re}\left(D\teichmullernorm|_{v_0}[V]\right)={\rm Re}\langle V,q_0/\|q_0\|\rangle,
\end{equation}
which is probably well-known to experts
(see \cite[Lemma 2.1]{https://doi.org/10.48550/arxiv.2211.16132}).
\end{remark}

\begin{remark}
In \Cref{thm:derivative_T-metric}, the condition of the cochain $\eta$ can be weakened.
Indeed, we may consider the cochain $\eta\in C^0(\mathcal{U},\sob^{2,1}(\Omega_{M_0}))$ with only the condition $\delta\eta=Y$ by applying the argument for the independence of choices of representatives for defining the model of pairing given in \S\ref{subsec:model_pairing_TT_TstarT}.
\end{remark}
%

\section{Infinitesimal Duality}
\label{sec:infiniteismal_duality}
From the definition, the Teichm\"uller metric is thought of as the dual (Finsler) metric on $\teich_g$ to the $L^1$-norm on $\mathcal{Q}_g$ (cf. \S\ref{subsec:Teichmuller_metric}).
The above mentioned variational formulae of the $L^1$-norm and the Teichm\"uller metric implies that these also satisfy the following duality in the infinitesimal sense:
\begin{theorem}[Infinitesimal duality]
\label{thm:infinitesimal_duality}
Let $q_0\in \mathcal{Q}_{x_0}$ and let $v_0\in T_{x_0}\teich_g$ which is presented by the Teichm\"uller Beltrami differential $\|q_0\|\overline{q}/|q|$. Then,
the $\partial$-derivatives of the squares of the $L^1$-norm and the Teichm\"uller metric satsify
\begin{equation}
\label{eq:infinitesimal_duality}
\switch_{\teich_g}\circ (\canoflipdagger_{\teich_g})^{-1}(-\partial \LOneNorm^2|_{q_0})=
\switch_{\teich_g}\left(
-\dfrac{1}{2}\mathscr{X}_{\LOneNorm^2}|_{q_0}
\right)=
\partial \teichmullernorm^2|_{v_0},
\end{equation}
where $\mathscr{X}_{\LOneNorm^2}$ is the Hamiltonian vector field of the square $\LOneNorm^2$ of the $L^1$-norm function.
\end{theorem}

\begin{proof}
Let $\eta=\{\eta_i\}_{i\in I}$ be the cochain with $(\eta_i)_{\overline{z}}=-\overline{q_0}/|q_0|$. 
Let $Y=\delta \eta$. 
Notice that
\begin{equation}
\label{eq:teichmullernorm_tb}
\teichmullernorm(x_0,v_0)=\sup_{\|q\|=1,q\in \mathcal{Q}_{x_0}}{\rm Re}(\langle v_0,q\rangle)
=\sup_{\|q\|=1,q\in \mathcal{Q}_{x_0}}{\rm Re}\left(\iint_{M_0}\dfrac{\overline{q_0}}{|q_0|}q\right)
=\|q_0\|.
\end{equation}
Since a cochain $\eta'=\|q_0\|\eta$ and $Y'=\|q_0\|Y$ satisfies
$$
\delta \eta'=\|q_0\|Y=Y',
\quad
(\eta'_i)_{\overline{z}}=-\|q_0\|\dfrac{\overline{q_0}}{|q_0|}
=-\teichmullernorm(x_0,v_0)
\dfrac{\overline{q_0}}{|q_0|},
$$
$v_0\in T_{x_0}\teich_g$ is the tangent vector presented by 
$\teichmullernorm(x_0,v_0)\overline{q_0}/|q_0|$ and $[Y']\in H^1(\mathcal{U},\Theta_{M_0})$.
From \Cref{coro:gradient_vector_L1} and \Cref{coro:Differential_Teichmuller_metric},
\begin{align*}
-\partial \LOneNorm^2|_{q_0}
&=-2\LOneNorm(q_0)\partial \LOneNorm|_{q_0}
=\|q_0\|[L_{\eta}(q_0),-Y]_{q_0}\\
&=[L_{\|q_0\|\eta}(q_0),-\|q_0\|Y]_{q_0}\\
&=[L_{\eta'}(q_0),-Y']_{q_0}
\\
%
\partial \teichmullernorm^2|_{v_0}&=
2\teichmullernorm(x_0,v_0)\partial \teichmullernorm|_{v_0}
=2\|q_0\|\partial \teichmullernorm|_{v_0}
\\
&=\tvdag{L_{\eta'}(q_0),q_0}{[Y']}
\end{align*}
in the models.
From \Cref{thm:dualization_pairing_symplectic_form} and \Cref{prop:switch-cotangent-tangent},
we obtain
\begin{align*}
\switch_{\mathbb{T}}\circ (\canoflipdagger_{\mathbb{T}})^{-1}
\left([L_{\eta'}(q_0),-Y']_{q_0}\right)
&=
\switch_{\mathbb{T}}
\left([Y',L_{\eta'}(q_0)]_{q_0}\right)
\\
&=\tvdag{L_{\eta'}(q_0),q_0}{[Y']}.
\end{align*}
From \Cref{prop:H-v},
$$
(\canoflipdagger_{\mathbb{T}})^{-1}
\left(-\partial \LOneNorm^2|_{q_0}\right)
=\dfrac{1}{2}\mathscr{X}_{-\LOneNorm^2}|_{q_0}
=-\dfrac{1}{2}\mathscr{X}_{\LOneNorm^2}|_{q_0}.
$$
Since the model spaces and the model bundle maps actually describe the actual spaces and the actual bundle maps (cf.\S\ref{sec:trivialization_via_Bers_embedding}), 
we obtain the formula which we wanted. 
\end{proof}

From the calculation in the proof of \Cref{thm:infinitesimal_duality}, we also obtain the following.

\begin{corollary}[Canonical sections]
\label{coro:canonical_section}
We have
\begin{align*}
D\Pi^\dagger_{\teich_g}\left((\canoflipdagger_{\mathbb{T}})^{-1}
\left(-\partial \LOneNorm^2|_{q_0}\right)\right)
&=
D\Pi^\dagger_{\teich_g}\left(
-\dfrac{1}{2}\mathscr{X}_{\LOneNorm^2}|_{q_0}
\right) \\
&=\Pi_{\teich_g}\left(\partial \teichmullernorm^2|_{v_0}\right) \\
&=v_0=\left[
\|q_0\|\dfrac{\overline{q_0}}{|q_0|}
\right]\in T\teich_g
\end{align*}
for $q_0\in \mathcal{Q}_g$, where $\Pi_{\teich_g}\colon T\teich_g\to \teich_g$ and $\Pi^\dagger_{\teich_g}\colon \mathcal{Q}_g\to \teich_g$ are the projections of the bundles, $D\Pi^\dagger_{\teich_g}\colon T\mathcal{Q}_g\to T\teich_g$ is the differential.
\end{corollary}

\chapter{Teichm\"uller Beltrami differentials}
\label{chap:TB_map}
In this chapter, we discuss the maps
\begin{align*}
\tb & \colon \mathcal{Q}_g^\times:=\mathcal{Q}_g-\{0\}\ni q\mapsto  \|q\|\left[\dfrac{\overline{q}}{|q|}\right]\in T^\times \teich_g=T\teich_g-\{0\}
\\
\tb_0 & \colon \mathcal{Q}_g^\times\ni q\mapsto\left[\dfrac{\overline{q}}{|q|}\right]\in T^\times\teich_g
\end{align*}
defined by the Teichm\"uller Beltrami differentials (\S\ref{subsubsec:TB_differenital}),
where ``$0$" in the equation means the zero sections.
We call the maps the \emph{Teichm\"uller Beltrami maps}.
From the Teichm\"uller theorem, the map $\tb_0$ is continuous, and the map $\tb$ is a homeomorphism.
The Teichm\"uller Beltrami maps satisfy
\begin{align}
\teichmullernorm
\left(\Pi_{T\teich_g}^\dagger(q_0), \tb(q_0)
\right)
&=\LOneNorm(q_0)
\label{eq:L1-norm_teichmuller_TB}
 \\
\tb_0(q_0)&=\LOneNorm(q_0)^{-1}\tb(q_0)
\label{eq:L1-norm_teichmuller_TB2}
\end{align}
for $q_0\in \mathcal{Q}_{x_0}$
(cf. \eqref{eq:teichmullernorm_tb}).

The aim of this chapter is to prove \Cref{thm:TB_map_is_diffeomorphism}.

\section{Teichm\"uller Beltrami maps are real-analytic on strata}
\label{sec:smoothness_TBmap}
We first discuss the smoothness of the Teichm\"uller Beltrami maps $\tb$ and $\tb_0$. 
From \Cref{coro:canonical_section}, we have the following commutative diagram:
$$
\minCDarrowwidth90pt
\xymatrix@C=50pt@R=40pt{
T\mathcal{Q}_g \ar[r]^{D\Pi^\dagger_{\teich_g}} & T\teich_g \ar[d]^{\Pi_{\teich_g}} \\
\mathcal{Q}^\times_g 
\ar[u]^{-\frac{1}{2}\mathscr{X}_{\LOneNorm^2}}
\ar[ru]_{\tb} \ar[r]_{\Pi^\dagger_{\teich_g}} 
& \teich_g.
%
%
}
$$
Notice that $-\mathscr{X}_{\LOneNorm^2}=2\canoflipdagger_{\teich_g}(-\partial\LOneNorm^2)$ from  \Cref{prop:H-v}, and the bundle map $\canoflipdagger_{\teich_g}\colon T\mathcal{Q}_g\to T^*\mathcal{Q}_g$ is biholomorphic.
Since the $L^1$-norm function $\LOneNorm$ on $\mathcal{Q}_g$ is real-analytic on each stratum of $\mathcal{Q}_g$, so is $\partial \LOneNorm^2$ (cf. \S\ref{sec:statification_space_QD}).
Therefore,
$$
\tb=D\Pi_{T\teich_g}^\dagger\circ \left(-\dfrac{1}{2}\mathscr{X}_{ \LOneNorm^2}\right)
$$
is real-analytc on each stratum.
Since $\tb_0=\LOneNorm^{-1}\tb$,
$\tb$ is also real-analytic on each stratum. 
We summarize as follows.

\begin{proposition}[Teichm\"uller Beltrami maps are real-analytic]
\label{prop:TB_differenials_are-RA}
The Teichm\"uller Beltrami maps $\tb$, $\tb_0\colon \mathcal{Q}^\times_g\to T^\times \teich_g$ are real-analytic on each stratum of $\mathcal{Q}_g$.
\end{proposition}


\section{Teichm\"uller Beltrami map is diffeomorphic on strata}
We deal with the following map defined by the pairing:
\begin{align}
\label{eq:variation_TB}
&\mathcal{Q}_g^*\oplus \mathcal{Q}_g\ni (q,\alpha)\mapsto 
\tbb(q,\alpha)=\|q\|\iint_M\dfrac{\overline{q}}{|q|}\alpha
=
\dfrac{\|q\|}{2i}\iint_M\dfrac{\overline{q(z)}}{|q(z)|}\alpha(z)d\overline{z}\wedge dz
\\
\label{eq:variation_TB0}
&\mathcal{Q}_g^*\oplus \mathcal{Q}_g\ni (q,\alpha)\mapsto 
\tbb_0(q,\alpha)=\iint_M\dfrac{\overline{q}}{|q|}\alpha
=\dfrac{1}{2i}\iint_M\dfrac{\overline{q(z)}}{|q(z)|}\alpha(z)d\overline{z}\wedge dz,
\end{align}
where $M$ is the underlying surface of $q$ (and $\alpha$). The map \eqref{eq:variation_TB0} is decomposed into
\begin{equation}
\label{eq:variation_TB2}
\begin{CD}
\mathcal{Q}_g^*\oplus \mathcal{Q}_g
@>{\tb_0\oplus id}>>
T\teich_g\oplus \mathcal{Q}_g
@>{\pairingsymb_{\teich_g}}>> \mathbb{C}\\
(q,\alpha) @>>> \left(\left[\dfrac{\overline{q}}{|q|}\right],\alpha\right) \\
@. (v,\alpha) @>>> \langle v,\alpha\rangle.
\end{CD}
\end{equation}

\subsection{Variational formula}
Let $\{(q_t,\alpha_t)\}_{|t|<\epsilon}$ be a holomorphic curve in $\mathcal{Q}^*_g\oplus \mathcal{Q}_g$.
Let $V_1=[X,\varphi]_{q_0}$ and $V_2=[X,\psi]_{\alpha_0}$ be the tangent vectors in
${\bf H}^1(\mathcal{U},\mathbb{L}_{q_0})\cong T_{q_0}\mathcal{Q}_g$ and
${\bf H}^1(\mathcal{U},\mathbb{L}_{\alpha_0})\cong T_{\alpha_0}\mathcal{Q}_g$,
where $\mathcal{U}=\{U_i\}_{i\in I}$ be a locally finite covering of the underlying surface $M_0$ of $q_0$ (and $\alpha_0$) with $H^1(U_i,\Theta_{M_0})=0$ for $i\in I$. 
Then,
\begin{align*}
\dfrac{\overline{q_0+t\varphi+o(t)}}{|q_0+t\varphi+o(t)|}(\alpha_0+t\psi+o(t)) 
&=\dfrac{\overline{q_0}}{|q_0|}
\alpha_0
+
t
\left(
-\dfrac{\overline{q_0}\alpha_0}{2q_0|q_0|}\varphi_i
+
\dfrac{\overline{q_0}}{|q_0|}
\psi_i
\right)
+\overline{t}
\dfrac{\alpha_0\overline{\varphi_i}}{2|q_0|}
+o(t)
\end{align*}
on any compact set in $U_i$ ($i\in I$) outside of the zeros of $q_0$, where $\xi=\{\xi_i\}_{i\in I}\in C^0(\mathcal{U}, \sob^{1,1}(\Theta_{M_0}))$ with $\delta\xi = X$. 
From \Cref{prop:tangent_space_strata}, the above calculation also works even when the family $\{q_t\}_{|t<\epsilon}$ is contained in a stratum in $\mathcal{Q}_g$. Thus, we obtain the following.

\begin{proposition}[Differential of Teichm\"uller Beltrami differentials]
\label{prop:derivative_pairing_Teichmuller_Beltrami}
Under the above notation, 
when $q_0$ and the family $\{q_t\}_{|t|<\epsilon}$ are in a stratum $\mathcal{Q}_g(k_1,\cdots,k_n;\epsilon)$,
\begin{align*}
&D\tbb_0|_{(q_0,\alpha_0)}
[V_1,V_2]
\\
&=
\dfrac{1}{2i}
\iint_{M_0}
\left(
-\dfrac{\overline{q_0}\alpha_0}{2q_0|q_0|}(\varphi_i-L_{\xi_i}(q_0))+
\dfrac{\overline{q_0}}{|q_0|}(\psi_i-L_{\xi_i}(\alpha_0))
\right)
d\overline{z}\wedge dz \\
&\overline{D} \tbb_0|_{(q_0,\alpha_0)}
[V_1,V_2]
\\
&=
\dfrac{1}{4i}
\iint_{M_0}
\dfrac{\alpha_0}{|q_0|}
\overline{
\left(
\varphi_i-L_{\xi_i}(q_0)
\right)
}
d\overline{z}\wedge dz
\end{align*}
for  $V_1=[X,\varphi]_{q_0}\in T_{q_0}\mathcal{Q}_g(k_1,\cdots,k_n;\epsilon)$
and $V_2=[X,\psi]_{\alpha_0}\in T_{\alpha_0}\mathcal{Q}_g$.
\end{proposition}

Before proving \Cref{prop:derivative_pairing_Teichmuller_Beltrami}, we notice that the integrals in \Cref{prop:derivative_pairing_Teichmuller_Beltrami} are independent of the choice of the cochain $\xi\in C^0(\mathcal{U},\sob^{1,1}(\Theta_{M_0}))$ with $\delta\xi=X$. From a bit modification of the argument in Remark (3) below \Cref{thm:variation_L1-norm}, the independence follows from the following identities:
\begin{align}
\left(
-\dfrac{\overline{q_0}\alpha_0}{2q_0|q_0|}L_{\chi}(q_0)+\dfrac{\overline{q_0}}{|q_0|}L_{\chi}(\alpha_0)
\right)d\overline{z}\wedge dz
\label{eq:pairing_TB1}
&=-d\left(
\dfrac{\overline{q_0}\alpha_0}{|q_0|}\chi d\overline{z}
\right)
\\
\dfrac{\overline{\alpha_0}}{|q_0|}L_{\chi}(q_0)d\overline{z}\wedge dz
&=
-2d\left(
\dfrac{q_0\overline{\alpha_0}}{|q_0|}\chi d\overline{z}
\right)
\label{eq:pairing_TB2}
\end{align}
on $M_0-{\rm Zero}(q_0)$
for a global vector field $\chi$ on $M_0$ of class $\sob^{1,1}$. We should be careful in the discussion because the $(0,1$)-differentials in the rights of the above two identities are discontinuous on $M_0$. However, as was discussed before, the Green formula varies for such differentials
since the differentials are bounded and continuous except for the zeros of $q_0$
(cf. \S\ref{sec:remark_smoothness}).

The differentiable formulae are obtained by \Cref{prop:variation_two_forms} with applying the argument in \Cref{lem:royden}
and the following identifies
\begin{align*}
L_{\xi_i}\left(\dfrac{\overline{q_0}}{|q_0|}\alpha_0\right)
&=
-\dfrac{\overline{q_0}\alpha_0}{2q_0|q_0|}
L_{\xi_i}(q_0)
+
\dfrac{\overline{q_0}}{|q_0|}L_{\xi_i}(\alpha_0)
\\
L_{\xi_i}\left(
\overline{\dfrac{\overline{q_0}}{|q_0|}\alpha_0}\right)
&=
L_{\xi_i}\left(\dfrac{q_0}{|q_0|}\overline{\alpha_0}\right)
=
\dfrac{\overline{\alpha_0}}{2|q_0|}L_{\xi_i}(q_0)
\end{align*}
on $U_i$ for $i\in I$, which are equivalent to \eqref{eq:pairing_TB1} and \eqref{eq:pairing_TB2}, respectively.

\subsection{Proof of \Cref{thm:TB_map_is_diffeomorphism}}
\label{subsec:proof_TB_map_is_diffeomorphism}
From \Cref{prop:TB_differenials_are-RA}, we only discuss the regularity of the differentials of the Teichm\"uller Beltrami map $\tb$.
Let $N_0$ be a sufficiently small neighborhood of $x_0$ in $\teich_g$ and $\{\alpha_1$, $\cdots$, $\alpha_{3g-3}\}$ be a system of holomorphic sections of the bundle $\mathcal{Q}_g|_{N_0}\to N_0$ such that $\alpha_1(x)$, $\cdots$, $\alpha_{3g-3}(x)$ generate the fiber $\mathcal{Q}_x$ for any $x\in N_0$.
Then, the map
$$
T\teich_g|_{N_0}\ni v\mapsto (x,(\langle v,\alpha_1(x)\rangle,\cdots,\langle v,\alpha_{3g-3}(x)\rangle))\in N_0\times \mathbb{C}^{3g-3}
$$
($x\in N_0$ with $v\in T_x\teich_g$) is a holomorphic (local) trivialization.

\subsection*{Case : $q_0$ is generic}
We first discuss the case where $q_0$ is generic, i.e. $q_0$ is in the principal stratum $\mathcal{Q}_1=\mathcal{Q}_g(1,\cdots,1;-1)$.

Let $V=[X,\varphi]_{q_0}\in {\bf H}^1(\mathcal{U},\mathbb{L}_{q_0})$.
Let $V=[X,\varphi]_{q_0}\in {\bf H}^1(\mathcal{U},\mathbb{L}_{q_0})$. Let $\xi\in C^0(\mathcal{U},\mathcal{A}^{0,0}(\Theta_{M_0}))$ with $\delta\xi=X$. Let $\{q_t\}_{|t|<\epsilon}$ be a holomorphic family of holomorphic quadratic differentials which is tangent to $V$ at $t=0$. Let $x(t)\in \teich_g$ with $q_t\in \mathcal{Q}_{x(t)}$ ($x(0)=x_0$). Let $\alpha_i=\alpha_i(x_0)$ and $V_k=[X,\psi^k]_{\alpha_k}\in {\bf H}^1(\mathcal{U},\mathbb{L}_{\alpha_k})$ ($k=1$, $\cdots$, $3g-3$) the tangent vector to the family $\{\alpha_k(x(t))\}_{|t|<\epsilon}$ at $t=0$. 

Suppose $V=[X,\varphi]_{q_0}\in {\bf H}^1(\mathcal{U},\mathbb{L}_{q_0})$ satisfies
\begin{align*}
&D \tbb|_{(q_0,\alpha_k)}[V,V_k]+\overline{D} \tbb|_{(q_0,\alpha_k)}[V,V_k]
\\
&=
D\LOneNorm(q_0)|_{q_0}[V]\tbb_0(q_0,\alpha_k)+
\LOneNorm(q_0)D \tbb_0|_{(q_0,\alpha_k)}[V,V_k]
\\
&\quad
+
\overline{D\LOneNorm(q_0)|_{q_0}[V]}\tbb_0(q_0,\alpha_k)
+
\LOneNorm(q_0)\overline{D} \tbb_0|_{(q_0,\alpha_k)}[V,V_k]=0
\end{align*}
for $k=1$, $\cdots$, $3g-3$ (cf. \Cref{lem:non-degenerate}). 
These equations are equivalent to
\begin{align}
&
\iint_{M_0}\dfrac{\overline{q_0}\alpha_k}{|q_0|}
\iint_{M_0}
\dfrac{\overline{q_0}}{2|q_0|}\left(\varphi_i-L_{\xi_i}(q_0)\right)
+
\iint_{M_0}\dfrac{\overline{q_0}\alpha_k}{|q_0|}
\iint_{M_0}
\dfrac{q_0}{2|q_0|}\overline{\left(\varphi_i-L_{\xi_i}(q_0)\right)}
\label{eq:non_deg1}
\\
&=
-\iint_{M_0}|q_0|
\iint_{M_0}
\left(
-\dfrac{\overline{q_0}\alpha_k}{2q_0|q_0|}(\varphi_i-L_{\xi_i}(q_0))+
\dfrac{\overline{q_0}}{|q_0|}(\psi^k_i-L_{\xi_i}(\alpha_k))
\right)
\nonumber
\\
&\quad
-
\iint_{M_0}|q_0|
\iint_{M_0}
\dfrac{\alpha_k}{2|q_0|}
\overline{
\left(
\varphi_i-L_{\xi_i}(q_0)
\right)
}
\nonumber
\end{align}
for $k=1$, $\cdots$, $3g-3$, where we omit to write the Euclid measure $d\overline{z}\wedge dz/2i$ in the equation for the simplicity.
Since we now discuss about the non-degeneracy of the differential of the Teichm\"uller Beltrami map, we may further suppose that these equations hold no matter how we choose initially the sections $\alpha_k$ around $x_0$.

Suppose to the contrary that $[X]\ne 0$. Take $a\in \mathbb{C}$ arbitrary.
By definition, $D\alpha_k|_{x_0}[[X]]=V_k=[X,\psi^k]_{\alpha_k}$. 
Since $[X,\psi^k+\beta]_{\alpha_k}\in {\bf H}^1(\mathcal{U},\mathbb{L}_{\alpha_k})$ for $\beta\in \mathcal{Q}_{x_0}$, by modifying the sections $\alpha_k$ in the vertical direction in the beginning, we may assume that
\begin{equation}
\label{eq:non-deg2}
\iint_{M_0}\dfrac{\overline{q_0}}{|q_0|}(\psi_i^k-L_{\xi_i}(\alpha_k))=a
\end{equation}
for $k=1$, $\cdots$, $3g-3$.
However, this contradicts to \eqref{eq:non_deg1} because the terms except for the second in the first  term of the right-hand side are independent of $a\in \mathbb{C}$. Therefore, $[X]=0$ and $V_k=0$ for $k=1$, $\cdots$, $3g-3$. Hence, we may also assume that $\xi\in C^0(\mathcal{U},\Theta_{M_0})$
and $\psi^k=L_{\xi}(q_0)$.

By substituting $\beta=\varphi_i-L_{\xi_i}(q_0)\in \mathcal{Q}_{x_0}$ into \eqref{eq:non_deg1}, we obtain
\begin{equation}
\label{eq:non-deg4}
\iint_{M_0}\dfrac{\overline{q_0}\gamma}{|q_0|}
{\rm Re}\left(\iint_{M_0}
\dfrac{\overline{q_0}}{|q_0|}\beta
\right)
=
-\iint_{M_0}|q_0|
\iint_{M_0}
\dfrac{\gamma}{2|q_0|}
\left(
\overline{\beta}-\dfrac{\overline{q_0}}{q_0}\beta
\right)
\end{equation}
for all $\gamma\in \mathcal{Q}_{x_0}$ since $\{\alpha_k\}_{k=1}^{3g-3}$ is a basis of $\mathcal{Q}_{x_0}$.
By substituting $\gamma=q_0$ into \eqref{eq:non-deg4}, we get
\begin{equation}
\label{eq:non-deg3}
\iint_{M_0}\dfrac{\overline{q_0}}{|q_0|}\beta=\overline{\iint_{M_0}\dfrac{q_0}{|q_0|}\overline{\beta}}=0.
\end{equation}
Suppose to the contrary that $\beta \ne 0$. By substituting $\gamma=\beta$ into \eqref{eq:non-deg4}, we get
$$
0=\iint_{M_0}
\dfrac{\beta}{|q_0|}
\left(
\overline{\beta}-\dfrac{\overline{q_0}}{q_0}\beta
\right)
=
\iint_{M_0}
\left(
1-\dfrac{\overline{q_0}}{q_0}\dfrac{\beta}{\overline{\beta}}
\right)\dfrac{|\beta|^2}{|q_0|}.
$$
Since ${\rm Re}(1-(\overline{q_0}/q_0)(\beta/\overline{\beta}))\ge 0$, we obtain
$$
1-\dfrac{\overline{q_0}}{q_0}\dfrac{\beta}{\overline{\beta}}=0
$$
almost everywhere on $M_0$. Since $q_0$ and $\beta$ are holomorphic quadratic differentials, the equation means that the imaginary part of a meromorphic function $\beta/q_0$ vanishes. Therefore, we obtain $\beta=\lambda q_0$ for some $\lambda\in \mathbb{R}$. On the other hand, from \eqref{eq:non-deg3}, $\lambda=0$ and $\beta=0$. This is a contradiction. Thus, we conclude that $V=[X,\varphi]_{q_0}=[\delta\xi, L_\xi(q_0)]_{q_0}=0$ in ${\bf H}^1(\mathcal{U},\mathbb{L}_{q_0})$.

\subsection*{Case : $q_0$ is in another stratum}
The strategy of the proof of the non-degeneracy is same. We give a sketch by comparing with the discussion of the previous case.

Suppose $q_0\in \mathcal{Q}_g(k_1,\cdots,k_n;\epsilon)$. We denote by $\mathcal{Q}_{x_0}(q_0)$ the subspace of $\mathcal{Q}_{x_0}$ consisting of $\psi\in \mathcal{Q}_{x_0}$ such that $\psi/q_0$ is at most simple poles on $M_0$. Dumas \cite[\S5]{MR3413977} shows that $\mathcal{Q}_{x_0}(q_0)$ is thought of as the tangent space at $q_0$ to the stratum containing $q_0$ of a natural stratification on $\mathcal{Q}_{x_0}$ induced from the stratification on $\mathcal{Q}_g$. In particular, for $\psi\in \mathcal{Q}_{x_0}(q_0)$, $[0,\psi]_{q_0}\in T_{q_0}\mathcal{Q}_g(k_1,\cdots,k_n;\epsilon)$. Notice that $q_0\in \mathcal{Q}_{x_0}(q_0)$ and hence $\mathcal{Q}_{x_0}(q_0)$ is not trivial.
Dumas also discuss an Hermitian inner product
\begin{equation}
\label{eq:Hermitian_Dumas}
\iint_{M_0}\dfrac{\overline{\alpha}\beta}{|q_0|}\quad (\alpha,\beta\in \mathcal{Q}_{x_0}(q_0))
\end{equation}
on $\mathcal{Q}_{x_0}(q_0)$. Let $d_0=\dim_{\mathbb{C}}\mathcal{Q}_{x_0}(q_0)$.

From the beginning of the proof in the previous discussion, we modify the basis $\{\alpha_k\}_{k=1}^{3g-3}$ such that $\alpha_k\in \mathcal{Q}_{x_0}(q_0)$ for $k=1$, $\cdots$, $d_0$. Furthermore, $V_k\in T_{\alpha_k}\mathcal{Q}_g(k_1,\cdots,k_n;\epsilon)$ for $k=1$, $\cdots$, $d_0$.

Let $V=[X,\varphi]_{q_0}\in T_{q_0}\mathcal{Q}_g(k_1,\cdots,k_n;\epsilon)$, and suppose 
\begin{align*}
D \tbb|_{(q_0,\alpha_k)}[V,V_k]
+\overline{D} \tbb|_{(q_0,\alpha_k)}[V,V_k]
=
0
\end{align*}
for $k=1$, $\cdots$, $d_0$. Take an arbitrary $a\in \mathbb{C}$. From \Cref{prop:tangent_space_strata} and the above discussion, $[X,\psi^k+\beta]_{\alpha_k}\in T_{\alpha_k}\mathcal{Q}_g(k_1,\cdots,k_n;\epsilon)$ for $\beta\in \mathcal{Q}_{x_0}(q_0)$. Since the stratum $\mathcal{Q}_g(k_1,\cdots,k_n;\epsilon)$ is a complex manifold, after modifying the analytic disk $\{\alpha_k(x(t))\}_{|t|<\epsilon}$ in the vertical direction in the stratum, we may assume that \eqref{eq:non-deg2} holds for $k=1$, $\cdots$, $d_0$. Hence, we get $X=\delta \xi$ for some $\xi\in C^0(\mathcal{U},\Theta_{M_0})$ and $V_k=[X,\psi^k]_{\alpha_k}=0$ in ${\bf H}^1(\mathcal{U},\mathbb{L}_{\alpha_k})$. By thinking the Hermitian product \eqref{eq:Hermitian_Dumas} on $\mathcal{Q}_{x_0}(q_0)$, the rest of the discussion is almost the same as that given in the case where $q_0$ is generic to obtain $V=[X,\varphi]_{q_0}=0$.

\section{Conjectural picture on  unit sphere bundles}
\label{sec:onjectural_picture2}
Notice from \eqref{eq:T-N-dual} that the Teichm\"uller Beltrami map $\tb_0$ is a bundle isomorphism from the unit sphere bundle $\mathcal{S}\mathcal{Q}_g$ to the unit sphere bundle $\mathcal{S}\teich_g=\{v\in T\teich_g\mid \teichmullernorm (\Pi_{\teich_g}(v),v)=1\}$.
We discuss the Teichm\"uller Beltrami map $\tb_0$ in the infinitesimal level.
Here, we only treat the map $\tb_0$ on the generic differentials. The purpose of this section is to make a conjectural picture of the Teichm\"uller Beltrami map from Complex analysis point of view.

Let $q_0\in \mathcal{SQ}_g$ be a generic differential.
Let $x_0=\Pi^\dagger_{\teich_g}(q_0)$.
As the proof of \Cref{thm:TB_map_is_diffeomorphism}, let $\alpha_k$ be a local holomorphic section of $T\teich_g\to \teich_g$ defined on a neighborhood $U_0$ of $x_0$ such that $\{\alpha_k(x)\}_{k=1}^{3g-3}$ generates $T_x\teich_g$ for $x\in U_0$. 
Let $\tbb_{0,k}(q)=\langle \tb_0(q),\alpha_k(\Pi^\dagger_{\teich_g}(q))\rangle$ for $q\in (\Pi^\dagger_{\teich_g})^{-1}(U_0)$.
Then
\begin{align}
\label{eq:TB-last}
&(\Pi^\dagger_{\teich_g})^{-1}(U_0)
\ni q\mapsto 
(\Pi_{\teich_g}(\tb(q)),
(\tbb_{1}(q),\cdots,\tbb_{3g-3}(q))
)
\in U_0\times \mathbb{C}^{3g-3}
\end{align}
is a local trivialization.

\subsection{Horizontal lifts}
\label{subsec:horizontal_lifts_SQg}
Let $[X_{q_0}]\in H^1(\mathcal{U},\Theta_{M_0})$ which presents $\overline{q_0}/{|q_0|}$.
Take a guiding frame $\GuaidF$ satisfying $\GuaidF([X_{q_0}])=\overline{q_0}/|q_0|$.
Let $\GoodS$ be the good section defined from $\GuaidF$.
We assume that any $[X]\in H^1(\mathcal{U},\Theta_{M_0})$ is of form
\begin{equation}
\label{eq:horizontal_lift-good_section}
\GoodS(X)=\alpha \GoodS(X_{q_0})+\beta \xi,
\end{equation}
where $\alpha$, $\beta\in \mathbb{C}$ and $\xi\in C^0(\mathcal{U},\sob^{1,1}(\Theta_{M_0}))$.
For $[X]\in H^1(\mathcal{U},\Theta_{M_0})$, we define a horizontall lift of $[X]$ by
$$
V^H([X])=[X,\{L_{\GoodS(X)_i}(q_0)-Q_X\}_{i\in I}]_{q_0},
$$
where $Q_X\in \Gamma(M_0,\mathcal{A}^{0,1}(\Omega_{M_0}^{\otimes 2}))$ satisfies
\begin{align}
\overline{\partial}Q_X&=\overline{\partial}L_{\GoodS(X)_i}(q_0) \quad (\mbox{on $U_i$})
\nonumber \\
\iint_{M_0}\dfrac{\overline{\varphi}}{|q_0|}Q_X&=0
\label{eq:uniquness_horizontal_lift-2}
\end{align}
for $i\in I$ and $\varphi\in \mathcal{Q}_{x_0}$.
Though the uniqueness condition \eqref{eq:uniquness_horizontal_lift-2} is different from the condition given in \eqref{eq:uniquness_horizontal_lift-1},
we can also check that the map 
$$
T_{x_0}\teich_g\cong H^1(\mathcal{U},\Theta_{M_0})\ni [X]\mapsto V^H([X])\in T_{q_0}\mathcal{Q}_g
$$
is a well-defined injective map (cf. \S\ref{subsec:tangent_space_to_the_unit_sphere_bundle}). From the proof of \Cref{prop:CR_L1}, the image is thought of as a horizontal lift of $T_{x_0}\teich_g$ contained the real tangent space $T^{\mathbb{R}}_{q_0}\mathcal{SQ}_g$ of $\mathcal{SQ}_g$ at $q_0$.

The horizontal lift $V^H([X])$ here is independent of the choice of the guiding frame. Indeed, let $\GuaidF'$ be another guiding frame with $\GuaidF'([X_{q_0}])=\overline{q_0}/|q_0|$ such that the good section $\GoodS'$ defined from $\GuaidF'$ satisfies the condition \eqref{eq:horizontal_lift-good_section}. Let $[X]\in H^1(\mathcal{U},\Theta_{M_0})$ and set $\xi=\GoodS(X)$ and $\xi'=\GoodS'(X)$. Take $Q'_X\in  \Gamma(M_0,\mathcal{A}^{0,1}(\Omega_{M_0}^{\otimes 2}))$ defined from $\xi'$ with the uniqueness condition \eqref{eq:uniquness_horizontal_lift-2}. Notice from $\delta\xi=\delta\xi'=X$ that the difference
$$
(L_{\xi_i}(q_0)-Q_X)-(L_{\xi'_i}(q_0)-Q'_X)
$$
is a holomorphic quadratic differential on $M_0$, and $\chi=\xi_i-\xi'_i$ on $U_i$ is a global section (cf. \Cref{prop:linear-map-L}). 

Take an arbitrary $\varphi\in \mathcal{Q}_{x_0}$. From the uniqueness condition \eqref{eq:uniquness_horizontal_lift-2},
\begin{align*}
&\iint_{M_0}\dfrac{\overline{\varphi}}{|q_0|}(
(L_{\xi_i}(q_0)-Q_X)-(L_{\xi'_i}(q_0)-Q'_X)
)=\iint_{M_0}\dfrac{\overline{\varphi}}{|q_0|}L_{\chi}(q_0).
\end{align*}
Here, the integrand is 
\begin{align*}
\dfrac{\overline{\varphi}}{|q_0|}L_{\chi}(q_0)d\overline{z}\wedge dz
&=\dfrac{\overline{\varphi}}{|q_0|}(q_0'\chi+2q_0\xi_z)d\overline{z}\wedge dz
\\
&=
2\overline{\left(\dfrac{\varphi}{q_0}\right)}\left(\dfrac{\overline{q_0}q_0'}{2|q_0|}\chi+|q_0|\chi_z\right)d\overline{z}\wedge dz
\\
&=-2d\left(
\left(\overline{\left(\dfrac{\varphi}{q_0}\right)}|q_0|\chi 
\right)d\overline{z}
\right),
\end{align*}
which is the $d$-differential of a bounded $1$-form which is differentiable except for zeros of $q_0$. Therefore, by applying Royden's argment (\S\ref{sec:remark_smoothness}), we get 
\begin{align*}
&\iint_{M_0}\dfrac{\overline{\varphi}}{|q_0|}((L_{\xi_i}(q_0)-Q_X)-(L_{\xi'_i}(q_0)-Q'_X)) \\
&=\lim_{r\to 0}\iint_{M_r}\dfrac{\overline{\varphi}}{|q_0|}L_{\chi}(q_0)
=-2\lim_{r\to 0}\int_{\partial M_r}\overline{\varphi}\dfrac{q_0}{|q_0|}\chi \,d\overline{z}=0
\end{align*}
by Green's theorem, where $\{M_r\}_{r>0}$ is an exhaustion of $M-{\rm Zero}(q_0)$ defined by removing $r$-disks (with respect to some local charts) around each zero of $q_0$.
Since $\varphi\in \mathcal{Q}_{x_0}$ is taken arbitrary, we have that $L_{\xi_i}(q_0)-Q_X=L_{\xi'_i}(q_0)-Q'_X$ on $U_i$ for $i\in I$ (cf. \cite[\S5]{MR3413977}).

\subsection{Conjectural picture}
\label{subsec:conjectual_picture}
For discussing the following proposition, we recall the infinitesimal correspondence in  complex variables.
Let $Z$ be a complex manifold. The real tangent bundle $T^{\mathbb{R}}Z$ (the tangent bundle of the underlying differential manifold of $Z$) is embedded in the complexified tangent bundle $T^{\mathbb{C}}Z=TZ\otimes \mathbb{C}$.
The holomorphic tangent bundle $TZ$ is identified with $T^{\mathbb{R}}Z$ by $TZ\ni v\mapsto v+\overline{v}\in T^{\mathbb{R}}Z$ as $\mathbb{R}$-vector bundles (cf. \S\ref{sec:tangent_cotangent_space}).

Let $Z'$ be another complex manifold and $F\colon Z\to Z'$ be a $C^1$ map. The tangent map $DF\colon T^\mathbb{R}Z\to T^\mathbb{R}Z'$ extends naturally as a bundle map $DF\colon T^\mathbb{C}Z\to T^\mathbb{C}Z'$ between the complexified tangent bundles. Then, we have a sequence of bundle maps
$$
\begin{CD}
TZ @>>> T^{\mathbb{R}}Z\subset T^{\mathbb{C}}Z @>{DF}>> T^{\mathbb{C}}Z'=T^{1,0}Z'\oplus T^{0,1}Z' @>>> TZ'=T^{1,0}Z' \\
v @>>> v+\overline{v} @>>> w @>>> w^{10}
\end{CD}
$$
where $w^{10}$ is the $(1,0)$-part of $w\in T^{\mathbb{C}}Z'$.
In this paper, we call the correspondence $v\mapsto w^{10}=DF^{10}[v]$ the \emph{$(1,0)$-part} of $DF$. In the local expression, that is, when $Z$ and $Z'$ are the complex Euclidean spaces, the $(1,0)$-part is 
\begin{equation}
\label{eq:C-linear-10}
DF^{10}[v]=
\left.\dfrac{\partial}{\partial t}F(z(t))\right|_{t=0}
+
\left.\dfrac{\partial}{\partial \overline{t}}F(z(t))\right|_{t=0}
=\partial F|_{z(0)}\cdot v+\overline{\partial} F|_{z(0)}\cdot \overline{v},
\end{equation}
where $z(t)$ is a holomorphic curve in $Z$ tangent to $v$ at $t=0$
(cf. \Cref{lem:non-degenerate}).
From this observation, we see that the $(1,0)$-part of $DF$ is $\mathbb{R}$-linear, not $\mathbb{C}$-linear in general. When $F$ is (local) diffeomorphic, $DF^{10}$ is $\mathbb{R}$-isomorphic.

Notice that the horizontal lift $T^H_{q_0}\mathcal{Q}_g$ is a subspace of the holomorphic tangent space $T_{q_0}\mathcal{Q}_g$ at $q_0\in \mathcal{Q}_g$.
We claim

\begin{proposition}[$\mathbb{C}$-linearity on horizontal lifts]
\label{prop:C-linear-tb}
The $(1,0)$-part of the differential of the Teichm\"uller Beltrami map $\tb$ is $\mathbb{C}$-linear on $T^H_{q_0}\mathcal{Q}_g$. 
Furthermore, $T^H_{q_0}\mathcal{Q}_g$ is the maximal $\mathbb{C}$-subspace of $T^{\mathbb{R}}_{q_0}\mathcal{S}\mathcal{Q}_g$ where the $(1,0)$-part $D\!\tb^{10}$ of the differenital of the Teichm\"uller Beltrami map $\tb$ is $\mathbb{C}$-linear.
\end{proposition}

\begin{proof}
Let $[X]\in H^1(\mathcal{U},\Theta_{M_0})$, and set $W_j=D\alpha_j|_{x_0}([X])$.
Let $\{q_t\}_{|t|<\epsilon_0}$ be a holomorphic curve tangent to $V^H([X])$ at $t=0$.
Since $\Pi_{\teich_g}\circ\tb_0=\Pi^\dagger_{\teich_g}$, the first coordinate of \eqref{eq:TB-last} is holomorphic. Hence
$$
\left.
\dfrac{\partial \Pi_{\teich_g}\circ\tb(q_t)}{\partial \overline{t}}
\right|_{t=0}
=
\left.
\dfrac{\partial \Pi^\dagger_{\teich_g}(q_t)}{\partial \overline{t}}
\right|_{t=0}=0.
$$ 
As we discussed above, $V^H([X])$ is contained in the maximal complex subspace of $T^\mathbb{R}_{q_0}\mathcal{SQ}_g$, which is equivalent to $D\LOneNorm|_{q_0}[V^H([X])]=0$. Hence 
$$
\overline{D}\LOneNorm|_{q_0}[V^H([X])]=\overline{D\LOneNorm|_{q_0}[V^H([X])]}=0
$$
since the $L^1$-norm function is a real valued funciton.
From \Cref{prop:derivative_pairing_Teichmuller_Beltrami} and the uniqueness condition \eqref{eq:uniquness_horizontal_lift-2}, 
for $j=1$, $\cdots$, $3g-3$,
\begin{align*}
\left.
\dfrac{\partial \tbb_{j}(q_t)}{\partial \overline{t}}
\right|_{t=0}
&=
\overline{D}\LOneNorm|_{q_0}[V^H([X])] \tbb_0(q_0,\alpha_j(x_0))
\\
&\qquad+
\LOneNorm(q_0)\overline{D} \tbb_0|_{(q_0,\alpha_j(x_0))}
[V^H([X]),W_j ] \\
&=
\dfrac{\LOneNorm(q_0)}{4i}
\iint_{M_0}
\dfrac{\alpha_j(x_0)}{|q_0|}
\overline{
\left(
L_{\GoodS(X)_i}(q_0)-Q_X-L_{\GoodS(X)_i}(q_0)
\right)
}d\overline{z}\wedge dz
=0.
\end{align*}
Since \eqref{eq:TB-last} is a complex analytic local trivialization,
the above calculation means that the $(1,0)$-part of the differential of $\tb$ on $T^H_{q_0}\mathcal{Q}_g$ is obtained from its $\partial$-derivatives, and hence it is a $\mathbb{C}$-linear map from $T^H_{q_0}\mathcal{Q}_g$ to $T\teich_g$. 

Let $v=[X,\varphi]_{q_0}\in T^{\mathbb{R}}_{q_0}\mathcal{SQ}_g$. Suppose that the $\mathbb{C}$-subspace ($\mathbb{C}$-line) generated by $v$ is contained in $T^{\mathbb{R}}_{q_0}\mathcal{S}\mathcal{Q}_g$ and the differential of $\tb$ is $\mathbb{C}$-linear on the $\mathbb{C}$-subspace. Then, $D\tbb_{j}^{10}[v]$ in \eqref{eq:C-linear-10} corresponding to $v$ satisfies $D\tbb_{j}^{10}[iv]=i\,D\tbb_{j}^{10}[v]$, which is equivalent to the condition
$$
\left.
\dfrac{\partial \tbb_{j}(q_t)}{\partial \overline{t}}
\right|_{t=0}=0
$$
for $j=1$, $\cdots$, $3g-3$,
where $\{q_t\}_{|t|<\epsilon_0}$ is a path in $\mathcal{Q}_g$ tangent to $v$ at $t=0$. From the above calculation with \Cref{prop:derivative_pairing_Teichmuller_Beltrami}, this is also equivalent to
$$
\iint_{M_0}
\dfrac{\alpha_j(x_0)}{|q_0|}
\overline{
\left(
\varphi_i
-L_{\xi_i}(q_0)
\right)
}=0
$$
for $j=1$, $\cdots$, $3g-3$,
where $\xi\in \mathcal{C}^0(\mathcal{U},\mathcal{A}^{0,0}(\Theta_{M_0}))$ with $\delta\xi =X$. Take $Q_X\in H^0(\mathcal{U},\mathcal{A}^{0,0}(\Omega_{M_0}^{\otimes 2}))$ for $\{L_{\xi_i}(q_0)\}_{i\in I}$ with the uniqueness condition \eqref{eq:uniquness_horizontal_lift-2}. Then, $\varphi_i-L_{\xi_i}(q_0)-Q_X$ is a holomorphic quadratic differenital on $M_0$ which satisfies
$$
\iint_{M_0}
\dfrac{\alpha_j(x_0)}{|q_0|}
\overline{
\left(
\varphi_i
-L_{\xi_i}(q_0)+Q_X
\right)
}=0
$$
which means that $\varphi_i=L_{\xi_i}(q_0)-Q_X$ since $\{\alpha_j(x_0)\}_{j=1}^{3g-3}$ is a basis of $\mathcal{Q}_{x_0}$.
From the discussion in \S\ref{subsec:horizontal_lifts_SQg}, $v=[X,\varphi]_{q_0}=[X,L_{\xi_i}(q_0)-Q_X]_{q_0}=V^H([X])\in T^H_{q_0}\mathcal{Q}_g$. This means that $T^H_{q_0}\mathcal{Q}_g$ is maximal in the sense of the statement of this proposition.
\end{proof}

\begin{corollary}[Characterization of the horizontal lifts]
For a generic $q_0\in \mathcal{SQ}_g$,
the horizontal lift defined as above is presented intrinsically by
\begin{align*}
T^H_{q_0}\mathcal{Q}_g
&=
\left
\{V\in T^{\mathbb{R}}_{q_0}\mathcal{SQ}_g\mid
D\!\tb^{10}|_{q_0}[iV]=i\,D\!\tb^{10}|_{q_0}[V]
\right\} \\
&=
\left\{
V\in T_{q_0}\mathcal{Q}_g\mid
D\!\tb^{10}|_{q_0}[iV]=i\,D\!\tb^{10}|_{q_0}[V],
\ D\LOneNorm|_{q_0}[V]=0
\right\}.
\end{align*}
In particular,
the total space $T^H\mathcal{SQ}^0_g:=\cup_{q_0\in \mathcal{SQ}_g^0}T^H_{q_0}\mathcal{Q}_g$
is a complex subbundle of the pull-back bundle $T\mathcal{Q}_g|_{\mathcal{SQ}^0_g}\to \mathcal{SQ}_g^0$ via the inclusion $\mathcal{SQ}_g\hookrightarrow \mathcal{Q}_g$, where $\mathcal{SQ}_g^0$ is a subset of $\mathcal{SQ}_g$ consisting of generic differentials.
\end{corollary}

\begin{proof}
The characterization follows from the following discussion in the linear algebra:
Let $V\subset \mathbb{C}^m$ be an $\mathbb{R}$-subspace and $L\colon V\to \mathbb{C}^n$ be an $\mathbb{R}$-linear map.
Then,
$$
\{v\in V\mid L(iv)=iL(v)\}
$$
is maximal among $\mathbb{C}$-subspaces in $\mathbb{C}^m$ contained in $V$ on which $L$ is $\mathbb{C}$-linear.
\end{proof}

\begin{remark}
Since $T^H_{q_0}\mathcal{Q}_g$ is the subspace of $T^\mathbb{R}_{q_0}\mathcal{SQ}_g$, the differential of $\tb_0$ coincides with that of $\tb$ on $T^H_{q_0}\mathcal{Q}_g$.
Hence, from \Cref{prop:C-linear-tb}, the differential of $\tb_0$ is also $\mathbb{C}$-linear on $T^H_{q_0}\mathcal{Q}_g$. 
\end{remark}

\begin{remark}
In \Cref{prop:C-linear-tb}, we assume $q_0\in \mathcal{SQ}_{g}$.
The proof works for all generic differential $q_0\in \mathcal{Q}_g$
such that
$$
T_{q_0}^H\mathcal{Q}_g=\left\{
V\in T_{q_0}\mathcal{Q}_g\mid
D\!\tb^{10}|_{q_0}[iV]=i\,D\!\tb^{10}|_{q_0}[V],
\ D\LOneNorm|_{q_0}[V]=0
\right\}
$$
is the maximal $\mathbb{C}$-subspace of $T_{q_0}\mathcal{Q}_g$ contained in the kernel ${\rm Ker}({\rm Re}(D\LOneNorm|_{q_0}))$ on which the differential of the Teichm\"uller Beltrami map $\tb$ is $\mathbb{C}$-linear.
\end{remark}

From \Cref{prop:C-linear-tb}, we expect that the Teichm\"uller Beltrami map is a kind of nice on the horizontal lifts from the complex analysis point of view. For instance, relating the observation in \S\ref{subsec:CR-tori}
and the conjecture given in \S\ref{subsec:conjecure_Levi}, we pose the following problem.

\begin{problem}
\label{prob:2}
Study the subbundle $T^H\mathcal{SQ}^0_g=\cup_{q_0\in \mathcal{SQ}_g^0}T^H_{q_0}\mathcal{Q}_g$
in $T\mathcal{Q}_g|_{\mathcal{SQ}_g}$ of horizontal lifts and its images in $T\teich_g$ via the differential of the Teichm\"uller Beltrami map.
For instance, does
$T^H\mathcal{SQ}^0_g$ define an (abstruct) CR structure on $\mathcal{SQ}_g^0$ of CR-codimension $6g-7$\,? If so, is the Teichm\"uller Beltrami map a CR isomorphism\,?
\end{problem}
For CR structures, see \cite{MR1211412} for instance.

\chapter{Appendices}
\label{chap:appendices}

\section{Appendix 1 : Kodaira-Spencer theory and Teichm\"uller theory}
\label{subsec:appendix_1}
In this section, I summarize the relation between the Kodaira-Spencer theory and the (quasiconformal) Teichm\"uller theory.
Namely, we discuss the correspondence \eqref{eq:identify_H1_T_g} in view of the homolomorphic families.

Let $B$ be a neighborhood of the origin in $\mathbb{C}$.
Let $(\mathcal{M},\pi,B)$ be a holomorphic family of compact Riemann surfaces of genus $g$. For $t\in B$, set $M_t=\pi^{-1}(t)$. By taking $B$ sufficiently small, we identify $\mathcal{M}$ with the product $M_{0}\times B$ via the local trivialization as $C^\infty$-manifolds (cf. \cite[Theorem 2.4]{MR815922}). By the implicit mapping theorem, when $B$ is taken sufficiently small if necessary again, there is a covering $\{U_i\}_{i\in I}$ of $M_{0}$ and an injective holomorphic map $Z_i\colon U_i\times B_0\to \mathbb{C}\times B_0$ such that $z_i^t=Z_i\mid_{U_i\times \{t\}}\colon U_i\times \{t\}\to \mathbb{C}$ ($i\in I$) make an analytic chart of $M_t$ for all $t\in B$. We identify $U_i\times \{t\}$ with $U_i$ and define $z_{ij}^t=z^{t}_i\circ (z_j^{t})^{-1}$ on $U_i\cap U_j$. We set
\begin{align*}
X_{ij}(z) & = \left.\dfrac{\partial z_{ij}^t}{\partial t}\right\mid_{t=0}\circ (z_{ij}^0)^{-1}(z)\quad (z\in z_i^0(U_i\cap U_j)).
\end{align*}
Then, $\{X_{ij}\}_{i\in I}\in Z^1(\{U_i\}_i, \Theta_{M_0})$.
The cohomology class $[\{X_{ij}\}]\in H^1(\{U_i\}_i,\Theta_{M_0})$ defines the infinitesimal deformation at $t=0$ for the holomorphic family $(\mathcal{M},\pi,B)$ in the Kodaira-Spencer theory (cf. \cite[\S4.2]{MR815922}).


Let
\begin{align*}
X_{i}(z) & = -\left.\dfrac{\partial z_{i}^t}{\partial t}\right\mid_{t=0}\circ (z_{i}^0)^{-1}(z)\quad (z\in z_i^0(U_i)).
\end{align*}
Since $z_i^t\circ (z_j^0)^{-1}=z_{ij}^t\circ z_j^t\circ (z_j^0)^{-1}$ on $z_j(U_i\cap U_j)$, $\{X_i\}_{i\in I}\in C^0(\{U_i\}_i, \mathcal{A}^{0,0}(\Theta_{M_0}))$ satisfies
$$
X_j-X_i=X_{ij}
$$
on $U_i\cap U_j$. Namely, $\xi=\{X_i\}_{i\in I}$ satisfies $\delta\xi=\{X_{ij}\}_{i,j}$.
By definition,
$$
z^t_i\circ (z_i^0)^{-1}(z)=z-tX_i(z)+o(t)
$$
for $z\in z_i(U_i)$ and $i\in I$. Therefore,
$$
\mu = -(X_i)_{\overline{z}}
$$
on $U_i$ is recognized as the infinitesimal Beltrami differential corresponding to the cohomology class $[\{X_{ij}\}]\in H^1(\{U_i\}_i,\Theta_{M_0})$. (Compare \cite[\S7.2.4]{MR1215481}).

\section{Appendix 2 : Regularity of the operators}
\label{appendix:regularity}
In this section, we discuss the regularities of $T(\omega)$ and $H(\omega)$ defined at \eqref{eq:operator_T} and \eqref{eq:operator_H}. The following  might be well-known. However the author cannot find any suitable reference, and we shall give a proof for completeness. The author thanks Professor Hiroshi Yanagihara for kindly allowing the author to give his argument for the smoothness here.

In the following, we use the following facts:
\begin{itemize}
\item[(1)]
for $\omega\in C^\infty_0(\mathbb{C})$, $T(\omega)$ and $H(\omega)$ are of class $C^\infty$ on $\mathbb{C}$ and satisfies
$$
T(\omega)_z=H(\omega),\
T(\omega)_{\overline{z}}=\omega,\
\mbox{and}\
H(\omega)_{\overline{z}}=\omega_z
$$
on $\mathbb{C}$; and
\item[(2)] $T$ and $H$ are extended to for $\omega\in L^p(\mathbb{C})$ with compact support for some $p\ge 2$, $T(\omega)$ and $H(\omega)$ are holomorphic outside the support of $\omega$.
\end{itemize}
See \cite[III \S7]{MR0344463} for instance.

\begin{proposition}
Let $D\subset \mathbb{C}$ be a bounded domain with smooth boundary.
Let $\omega\in C^\infty(D)\cap L^p(D)$ for some $p\ge 2$.
We extend $\omega$ as a function on $\mathbb{C}$ by setting $\omega\equiv 0$ on the complement of $D$. 
Then, $T(\omega)$ and $H(\omega)$ are smooth functions on $D$ which satisfy
$$
T(\omega)_z=H(\omega),\
T(\omega)_{\overline{z}}=\omega,\
\mbox{and}\
H(\omega)_{\overline{z}}=\omega_z
$$
on $D$.
\end{proposition}

\begin{proof}
The following argument is due to Professor Hiroshi Yanagihara.

Let $V\subset D$ be a domain with $\overline{V}\subset D$.
Take $\chi\in C^\infty_0(\mathbb{C})$ with $\chi\equiv 1$ on $V$, $\chi\equiv 0$ on $\mathbb{C}\setminus D$ and $0\le \chi\le 1$. Then
$$
T(\omega)=T(\chi \omega)+T(1-\chi)\omega)
$$
Since $T((1-\chi)\omega)$ is holomorphic on $V$ and $\chi \omega\in C^\infty_0(\mathbb{C})$, $T(\omega)$ is smooth on $V$. Moreover,
\begin{align*}
T(\omega)_z &=H(\chi \omega)+H((1-\chi)\omega)=H(\omega) \\
T(\omega)_{\overline{z}} &=\chi \omega+(1-\chi)\omega=\omega 
\end{align*}
on $V$. Though $(1-\chi)\omega$ is not smooth on $\mathbb{C}$,
we can easily check that the argument \cite[III \S7.2]{MR0344463} is available in our case for proving the above two differential formulae. Since $V$ is taken arbitrary, the above two equations hold on $D$. By the same argument, it is shown that $H(\omega)$ is also smooth on $D$, and satisfies
\begin{align*}
H(\omega)_{\overline{z}}
&=H(\chi \omega)_{\overline{z}}+H((1-\chi)\omega)_{\overline{z}}=\omega_z
\end{align*}
on $V$, since $\chi\equiv 1$ on $V$ and $H((1-\chi)\omega)$ is holomorphic on $V$. Hence, this equation also hold on $D$.
%
%
\end{proof}

\section{Appendix 3 : Royden's regularity criterion on the dual metric}
\label{sec:Royden_s_criterion}
In this section, we shall give a proof of the Royden's criterion on the regularity of the dual Finsler metric given in \Cref{prop:Royden_criterion}.
The following argument is due to \cite[\S9.3, Lemma 3]{MR903027}. 

%
\subsection{Proof of \Cref{prop:Royden_criterion}}
Let $\mathbb{S}_x=\{\eta\in \mathbb{C}^n\mid G(x,\eta)=1\}$ for $x\in U$. From the positivity of $G$, $\mathbb{S}_x$ is compact for each $x\in U$.
From the linearity of the pairing, and the complex homogeneity of $G$, we can easily the positivity and the complex homogeneity of $F$. For the continuity of $F$, let $(x_0,\xi_0)\in U\times \mathbb{C}^n-\{0\}$. Take $\eta_0\in \mathbb{S}_{x_0}$ with $F(x_0,\xi_0)={\rm Re}(\eta_0(\xi_0))$.
For $(x,\xi)\in U\times \mathbb{C}^n$,
\begin{align*}
F(x_0,\xi_0)
&={\rm Re}(\eta_0(\xi_0))={\rm Re}(\eta_0(\xi))+{\rm Re}(\eta_0(\xi_0))-{\rm Re}(\eta_0(\xi))\\
&\le F(x,\xi)+\sup_{\eta\in \mathbb{S}_{x_0}}|{\rm Re}(\eta(\xi_0))-{\rm Re}(\eta(\xi))|.
\end{align*}
From the continuity of $G$, there is a neighborhood $U_1$ of $x_0$ and a relatively compact neighborhood $V_1$ of $\mathbb{S}_{x_0}$ in $\mathbb{C}^n-\{0\}$ such that $\mathbb{S}_x\subset V_1$ for $x\in U_1$. From the above discussion, we obtain
$$
|F(x,\xi)-F(x_0,\xi_0)|\le \sup_{\eta\in V_1}|{\rm Re}(\eta(\xi_0-\xi))|
$$
for $(x,\xi)\in U\times \mathbb{C}^n$ with $x\in U_1$. Therefore, $F$ is continuous.
For the convexity of $F$, let $\xi_1$, $\xi_2\in \mathbb{C}^n$ and $\alpha_1$, $\alpha_2\in \mathbb{C}$.
Take $\eta_0\in \mathbb{S}_x$ and $F(x,\alpha_1\xi_1+\alpha_2\xi_2)={\rm Re}(\eta_0(\alpha_1\xi_1+\alpha_2\xi_2))$. Then
\begin{align*}
F(x,\alpha_1\xi_1+\alpha_2\xi_2)
&={\rm Re}(\eta_0(\alpha_1\xi_1+\alpha_2\xi_2))\\
&\le 
|\alpha_1|{\rm Re}(\eta_0(\xi_1))+|\alpha_2|{\rm Re}(\eta_0(\xi_2)) \\
&\le |\alpha_1|F(x,\xi_1)+|\alpha_2|F(x,\xi_2).
\end{align*}

We claim

\begin{claim}
\label{claim:claim1_appendix3}
Let $x\in U$.
When $G(x,\cdot)$ is of class $C^1$ in the variable $\eta$ with $\eta\ne 0$, 
$F(x,\xi)$ is strictly convex in $\xi\in \mathbb{C}^n$.
\end{claim}

\begin{proof}[Proof of \Cref{claim:claim1_appendix3}]
Indeed, if $F(x,\xi)$ is not strictly convex in $\xi$, we find  $\xi_1$, $\xi_2\in \mathbb{C}^n$ such that $\xi_1\ne \xi_2$ and $F(x,\xi_1)=F(x,\xi_2)=F(x,(\xi_1+\xi_2)/2)=1$. In particular $\xi_1+\xi_2\ne 0$.
Take $\eta_0\in \mathbb{S}_{x}$ such that ${\rm Re}(\eta_0(\xi_1+\xi_2))=2$. 
Since ${\rm Re}(\eta_0(\xi_j))\le F(x,\xi_j)=1$, ${\rm Re}(\eta_0(\xi_j))=1$  for $j=1,2$. Since $G(x,\cdot)$ is of class $C^1$ and $G(x,t\eta)=tG(x,\eta)$ for $t>0$,  the gradient of $G(x,\cdot)$ at $\eta$ does not vanish. Therefore, $\mathbb{S}_x$ is a $C^1$-submanifold and has a unique tangent plane of (real) dimension $2n-1$ at $\eta_0$. Since ${\rm Re}(\xi_1(\eta))\le F(x,\xi_1)=1$ for all $\eta\in \mathbb{S}_x$ and $\xi_1(\eta_0)=1$, the tangent plane of $\mathbb{S}_x$ at $\eta_0$ is described as
$$
\{\eta\in \mathbb{C}^n\mid {\rm Re}(\eta(\xi_1))=1\}.
$$
Namely, $\xi_1$ is the normal vector to the tangent plane.
By applying the same discussion to $\xi_2$, we get $\xi_2=t\xi_1$ for some $t\in \mathbb{R}$ with $t\ne 0$ from the uniqueness of the tangent plane. Since $F(x,\xi_1)=F(x,\xi_2)=1$ and $\xi_1+\xi_2\ne 0$, $\xi_1=\xi_2$. This is a contradiction.
\end{proof}

We next claim the following.

\begin{claim}
\label{claim:claim2_appendix3}
Suppose $G(x,\eta)$ is strictly convex in $\eta\in \mathbb{C}^n$ for all $x\in U$. Then,
\begin{itemize}
\item[(1)]
$G(x,\eta)$ is uniformly convex in $\eta\in \mathbb{C}^n$ in the sense that for any $\epsilon>0$ and a compact set $K\subset U$, there is $\delta>0$ such that for $x\in K$ and $\eta_1$, $\eta_2\in \mathbb{S}_x$, $G(x,\eta_1+\eta_2)\ge 2-\delta$ implies $G(x,\eta_1-\eta_2)<\epsilon$;
\item[(2)]
$F(x,\xi)$ is smooth in $\xi\in \mathbb{C}^n$ in the sense that for $\epsilon>0$ and a compact set $K\subset U$, there is $\delta>0$ such that if $F(x,\xi_1-\xi_2)<\delta$,
\begin{equation}
\label{eq:Gardiner-1}
F(x,\xi_1+\xi_2)\ge F(x,\xi_1)+F(x,\xi_2)-\epsilon F(x,\xi_1-\xi_2)
\end{equation}
for $x\in K$ and $\xi_1,\xi_2\in \mathbb{C}^n$;
\item[(3)]
For any $\xi\in \mathbb{C}^n-\{0\}$ and $x\in U$, there is a unique $\eta\in \mathbb{S}_x$ with  $F(x,\xi)={\rm Re}(\eta(\xi))$; and
\item[(4)]
For $\xi,h\in \mathbb{C}^n$ with $F(x,\xi)=1$, $f(t)=F(x,\xi+th)$ is differentiable at $t=0$. Furthermore, $f'(0)=\eta_0(h)$, where $\eta_0\in \mathbb{S}_x$ with $F(x,\xi)={\rm Re}(\eta_0(\xi))$.
\end{itemize}
\end{claim}

\begin{proof}[Proof of \Cref{claim:claim2_appendix3}]
(1)\quad 
Otherwise, there are $\epsilon_0>0$ and a compact set $K_0\subset U$ such that for any $n\in \mathbb{N}$, there is $x_n\in K_0$ and $\eta_{1,n}$, $\eta_{2,n}\in \mathbb{S}_{x_n}$ such that $G(x_n,\eta_{1,n}+\eta_{2,n})>2-(1/n)$ and $G(x_n,\eta_{1,n}-\eta_{2,n})\ge \epsilon_0$.
By the compactness of $\cup_{x\in K_0}\mathbb{S}_x$, we may assume that $x_n\to x_0\in K_0$ and $\eta_{j,n}\to \eta_j\in \mathbb{S}_{x_0}$ as $n\to \infty$. From the continuity of $G$, we get
$G(x_0,\eta_1+\eta_2)=2$ and $G(x_0,\eta_1-\eta_2)\ge \epsilon_0$. This contradicts to the strictly convexity of $G$.

\medskip
\noindent
(2)\quad
For $\epsilon>0$ and a compact set $K\subset U$,
we take $\delta>0$ for $2\epsilon$ and $K$ as discussed in the uniform convexity of $G$.

Let $x\in K$, and $\xi_1$, $\xi_2\in \mathbb{C}^n$.
To prove \eqref{eq:Gardiner-1}, we may assume that $F(x,\xi_1)=1$, $F(x,\xi_2)\le 1$ and $\xi_1\ne 0$, $\xi_2\ne 0$ since the both sides of \eqref{eq:Gardiner-1} is homogeneous by multiplying positive constant.

Let $\eta_1$, $\eta_2\in \mathbb{S}_x$ with $F(x,\xi_j)={\rm Re}(\eta_j(\xi_j))$ for $j=1$, $2$. 
Suppose $F(x,\xi_1-\xi_2)<\delta/2$. Then ${\rm Re}(\eta_1(\xi_1))=F(x,\xi_1)=1$ and
$$
{\rm Re}(\eta_2(\xi_2))=F(x,\xi_2)>F(x,\xi_1)-\delta/2=1-\delta/2.
$$
Since $|{\rm Re}(\eta_j(\xi_1-\xi_2))|\le F(x,\xi_1-\xi_2)<\delta/2$ for $j=1,2$ and $F(x,\xi_1+\xi_2)\le F(x,\xi_1)+F(x,\xi_2)\le 2$,
\begin{align*}
G(x,\eta_1+\eta_2)
&\ge G(x,\eta_1+\eta_2)F\left(x,\dfrac{\xi_1+\xi_2}{2}\right)
\ge {\rm Re}\left((\eta_1+\eta_2)\left(\dfrac{\xi_1+\xi_2}{2}\right)\right) \\
&={\rm Re}(\eta_1(\xi_1))+{\rm Re}(\eta_2(\xi_2))+
\dfrac{1}{2}{\rm Re}(\eta_1(\xi_2-\xi_1))+\dfrac{1}{2}{\rm Re}(\eta_2(\xi_2-\xi_1))
\\
&>2-\delta
\end{align*}
From the uniform convexity of $G$, we have $G(x,\eta_1-\eta_2)<2\epsilon$.
Thus, we obtain
\begin{align*}
&F(x,\xi_1)+F(x,\xi_2)-\epsilon F(x,\xi_1-\xi_2)
\\
&F(x,\xi_1)+F(x,\xi_2)-\dfrac{1}{2}G(x,\eta_1-\eta_2)F(x,\xi_1-\xi_2)
\\
&\le {\rm Re}(\eta_1(\xi_1))+{\rm Re}(\eta_2(\xi_2))-\dfrac{1}{2}{\rm Re}((\eta_1-\eta_2)(\xi_1-\xi_2))
\\
&=
\dfrac{1}{2}{\rm Re}((\eta_1+\eta_2)(\xi_1+\xi_2))\le F(x,\xi_1+\xi_2)
\end{align*}
since $G(x,\eta_1+\eta_2)/2\le (G(x,\eta_1)+G(x,\eta_2))/2=1$.

\medskip
\noindent
(3)\quad
We may assume that $F(x,\xi)=1$. The existence is clear. Let $\eta_1$, $\eta_2\in \mathbb{S}_x$ with $1=F(x,\xi)={\rm Re}(\eta_1(\xi))={\rm Re}(\eta_2(\xi))$. 
Then
\begin{align*}
1\ge G\left(x,\dfrac{\eta_1+\eta_2}{2}\right)
&=G\left(x,\dfrac{\eta_1+\eta_2}{2}\right)F(x,\xi)
\ge {\rm Re}\left(\dfrac{\eta_1+\eta_2}{2}(\xi)\right)=1.
\end{align*}
Since $G(x,\cdot)$ is strictly convex, $\eta_1=\eta_2$.

\medskip
\noindent
(4)\quad
First we check that $f(t)$ is differentiable at $t=0$. 
For $|t_2|>|t_1|>0$ with $t_1/t_2>0$, 
$$
F(x,\xi+t_1h)=F\left(x,\dfrac{t_1}{t_2}(\xi+t_2h)+\left(1-\dfrac{t_1}{t_2}\right)\xi\right)
\le \dfrac{t_1}{t_2}
F(x,\xi+t_2h)+\left(1-\dfrac{t_1}{t_2}\right)F(x,\xi).
$$
Hence $f(t)-f(0)/t$ is increasing function in $t$.
Therefore, $f(t)$ has left and right derivatives at $t=0$.
By applying the smoothness discussed in (2) for $\xi_1=\xi+th$ and $\xi_2=\xi-th$, for any $\epsilon>0$, there is $\delta>0$ such that when $|t|<\delta/2F(x,h)$,
\begin{align*}
\dfrac{f(t)-f(0)}{t}-\dfrac{f(-t)-f(0)}{-t}
&=\dfrac{F(x,\xi+th)+F(x,\xi-th)-2F(x,\xi)}{t}
\\
&\le \epsilon \dfrac{F(x,(\xi+th)-(\xi-th))}{t}=2\epsilon F(t,h).
\end{align*}
Therefore, the left and right derivatives at $t=0$ agree.
Let $\eta_t\in \mathbb{S}_x$ with $F(x,\xi+th)={\rm Re}(\eta_t(\xi+th))$. Then,
$$
\dfrac{f(t)-f(0)}{t}={\rm Re}\left(
\dfrac{\eta_t(\xi+th)-\eta_0(\xi)}{t}
\right)
={\rm Re}(\eta_t(h))+{\rm Re}\dfrac{\eta_t(\xi)-\eta_0(\xi)}{t}
$$
As $t\to 0$, the limits of $(f(t)-f(0))/t$ and ${\rm Re}(\eta_t(h))$ exist because of the differentiability of $f(t)$ and the uniqueness observed in (3). Therefore, 
$$
\lim_{t\to 0}\dfrac{{\rm Re}(\eta_t(\xi))-{\rm Re}(\eta_0(\xi))}{t}
$$
also exists. Since ${\rm Re}(\eta_t(\xi))\le F(x,\xi)={\rm Re}(\eta_0(\xi))$,
the numerator is non-positive, and so is the limit when $t\to +0$. However as $t\to -0$, the limit is non-negative. Hence the limit should be zero. Thus we obtain
$$
f'(0)=\lim_{t\to 0}{\rm Re}(\eta_t(h))={\rm Re}(\eta_0(h)),
$$
which is what we wanted.
\end{proof}

Finally, we claim

\begin{claim}
\label{claim:claim3_appendix3}
If $G(x,\cdot)$ is strictly convex for $x\in U$, and $G$ is totally differentiable in the variable $x\in U$ and the $x$-derivative is continuous on $U\times (\mathbb{C}^n-\{0\})$, $F$ is of class $C^1$ on $U\times (\mathbb{C}^n-\{0\})$. 
\end{claim}

\begin{proof}[Proof of \Cref{claim:claim3_appendix3}]
From the assumption,
\begin{align}
G(x+\Delta x,\eta)
&=G(x,\eta) 
+A_{x,\eta}(\Delta x)+\overline{A_{x,\eta}(\Delta x)}
+o(|\Delta x|),
\label{eq:claim_3_Royden1}
\end{align}
as $|\Delta x|\to 0$,
where $A_{x,\eta}$ is a $\mathbb{C}$-linear functional on $\mathbb{C}$ which varies continuously in $x$ and $\eta$, and $|\cdot|$ means here the Euclidean norm on $\mathbb{C}^n$.
By the mean value theorem,
$$
G(x+\Delta x,\eta)-G(x,\eta)=
A_{x+t\Delta x,\eta}(\Delta x)+\overline{A_{x+t\Delta x,\eta}(\Delta x)}
$$
for some $t=t(x,y,\Delta x)$ with $0<t(x,y,\Delta x)<1$.
Therefore,
from the continuity of the $x$-derivative $A_{x,\eta}$ of $G$.
\begin{align*}
&\dfrac{G(x+\Delta x,\eta)-(G(x,\eta)+A_{x,\eta}(\Delta x)+\overline{A_{x,\eta}(\Delta x)})}{|\Delta x|} \\
&=
2{\rm Re}\left(
(A_{x+t(x,y,\Delta x)\Delta x,\eta}-A_{x,\eta})
\left(\dfrac{\Delta x}{|\Delta x|}\right)
\right)
\to 0
\end{align*}
as $|\Delta x|\to 0$ whenever $(x,\eta)$ stays in a compact set in $U\times (\mathbb{C}^n-\{0\})$.
Therefore, the term $o(|\Delta x|)$ in \eqref{eq:claim_3_Royden1} is uniformly small on any compact set of $U\times( \mathbb{C}^n-\{0\})$ in the sense that for any $\epsilon>0$ and compact sets $K_1\subset U$ and $K_2\subset \mathbb{C}^n-\{0\}$,
there are $\delta>0$ and $C_0>0$ such that
\begin{equation}
\label{eq:claim_3_Royden1-improve}
\left|
G(x+\Delta x,\eta)
-
\left(
G(x,\eta) 
+A_{x,\eta}(\Delta x)+\overline{A_{x,\eta}(\Delta x)}
\right)
\right|\le
C_0\epsilon |\Delta x|
\end{equation}
for $(x,\eta)\in K_1\times K_2$ and $|\Delta x|<\delta$ with $x+\Delta x\in K_1$.

Let $(x,\xi)\in U\times \mathbb{C}^n$ with $\xi\ne 0$. Let $\eta_{x,\xi}\in \mathbb{S}_x$ with
$F(x,\xi)={\rm Re}(\eta_{x,\xi}(\xi))$. From (3) of \Cref{claim:claim2_appendix3}, $\eta_{x,\xi}$ uniquely exists. From the uniqueness discussed in (3), we can check that $\eta_{x,\xi}$ depends continuously in $(x,\xi)\in U\times (\mathbb{C}^n-\{0\})$.
From the identify
$$
\dfrac{1}{1+a}=1-a+\dfrac{a^2}{1+a^2}
$$
for $|a|<1$, there is $\delta'>0$ such that
\begin{align*}
&F(x+\Delta x,\xi)={\rm Re}(\eta_{x+\Delta x,\xi}(\xi))
=\dfrac{{\rm Re}(\eta_{x+\Delta x,\xi}(\xi))}{G(x+\Delta x,\eta_{x+\Delta x,\xi})} \\
&=\dfrac{{\rm Re}(\eta_{x+\Delta x,\xi}(\xi))}{G(x,\eta_{x+\Delta x,\xi})} 
-\dfrac{{\rm Re}(\eta_{x+\Delta x,\xi}(\xi))}{G(x,\eta_{x+\Delta x,\xi})^2}(A_{x,\eta_{x+\Delta x,\xi}}(\Delta x)+\overline{A_{x,\eta_{x+\Delta x,\xi}}(\Delta x)})
+C_0'\epsilon |\Delta x|
\\
&\le F(x,\xi)
-\dfrac{{\rm Re}(\eta_{x+\Delta x,\xi}(\xi))}{G(x,\eta_{x+\Delta x,\xi})^2}(A_{x,\eta_{x+\Delta x,\xi}}(\Delta x)+\overline{A_{x,\eta_{x+\Delta x,\xi}}(\Delta x)})
+C_0'\epsilon |\Delta x|
\end{align*}
for some $C_0'>0$ when $(x,\xi)\in K_1\times K_2$ and $|\Delta x|<\delta'$.
Since $\eta_{x+\Delta x,\xi}\to \eta_{x,\xi}$ and $A_{x,\eta_{x+\Delta x,\xi}}\to A_{x,\eta_{x,\xi}}$ as $|\Delta x|\to 0$, we get
$$
\left|\dfrac{{\rm Re}(\eta_{x+\Delta x,\xi}(\xi))}{G(x,\eta_{x+\Delta x,\xi})^2}A_{x,\eta_{x+\Delta x,\xi}}(\Delta x)
-\dfrac{{\rm Re}(\eta_{x,\xi}(\xi))}{G(x,\eta_{x,\xi})^2}A_{x,\eta_{x,\xi}}(\Delta x)
\right|
=C_0''\epsilon |\Delta x|
$$
for some $C_0''>0$,
when $(x,\xi)\in K_1\times K_2$ and $|\Delta x|<\delta'$ for some $\delta''>0$ with $x+\Delta x\in K_1$.
Therefore, we obtain
\begin{align*}
F(x+\Delta x,\xi)
&\le F(x,\xi)
-\dfrac{{\rm Re}(\eta_{x,\xi}(\xi))}{G(x,\eta_{x,\xi})^2}(A_{x,\eta_{x,\xi}}(\Delta x)+\overline{A_{x,\eta_{x,\xi}}(\Delta x)}) \\
&\qquad+(C_0'+C_0'')\epsilon |\Delta x|)
\end{align*}
when $(x,\xi)\in K_1\times K_2$ and $|\Delta x|<\min(\delta,\delta',\delta'')$ with $x+\Delta x\in K_1$.

Since ${\rm Re}(\eta_{x,\xi}(\xi))=F(x,\xi)$ and $G(x,\eta_{x,\xi})=1$,
by interchanging $x$ and $x+\Delta x$ in the above inequality, we deduce
\begin{align}
F(x+\Delta x,\xi)= F(x,\xi) 
-F(x,\xi)(A_{x,\eta_{x,\xi}}(\Delta x)+\overline{A_{x,\eta_{x,\xi}}(\Delta x)})+o(|\Delta x|)
\label{eq:claim_3_Royden2}
\end{align}
as $|\Delta x|\to 0$.
Since the term $o(|\Delta x|)$ in \eqref{eq:claim_3_Royden1} is uniformly small on any compact sets in $U\times (\mathbb{C}^n-\{0\})$,
so is the term $o(|\Delta x|)$ in \eqref{eq:claim_3_Royden2}.

Since $G$ is strictly convex, from (4) of \Cref{claim:claim2_appendix3}, 
$$
F(x,\xi+\Delta \xi)-F(x,\xi)={\rm Re}(\eta_{x,\xi}(\Delta \xi))+o(|\Delta \xi|)
$$
as $|\Delta \xi|\to 0$. Therefore,
\begin{align*}
&F(x+\Delta x,\xi+\Delta \xi)-F(x,\xi) \\
&F(x+\Delta x,\xi+\Delta \xi)-F(x,\xi+\Delta \xi)
+F(x,\xi+\Delta \xi)-F(x,\xi) \\
&=-F(x,\xi+\Delta \xi)(A_{x,\eta_{x,\xi+\Delta \xi}}(\Delta x)+\overline{A_{x,\eta_{x,\xi+\Delta \xi}}(\Delta x)})+o(|\Delta x|)\\
&\quad +\eta_{x,\xi}(\Delta \xi)+o(|\Delta \xi|) \\
&=-F(x,\xi)(A_{x,\eta_{x,\xi}}(\Delta x)+\overline{A_{x,\eta_{x,\xi}}(\Delta x)})+{\rm Re}(\eta_{x,\xi}(\Delta \xi))+o(|\Delta x|+|\Delta \xi|) 
\end{align*}
as $|\Delta x|\to 0$ and $|\Delta \xi|\to 0$. Since each partial derivatives $A_{x,\eta_{x,\xi}}$ and $\eta_{x,\xi}$ are continuous in $(x,\xi)\in U\times (\mathbb{C}^n-\{0\})$, we conclude that $F$ is of class $C^1$ on $U\times (\mathbb{C}^n-\{0\})$.
\end{proof}

\subsection{Remark on the reflexive duality}
\label{subsec:reflexive_dual}
Let $E\to M$ be a complex vector bundle and $E^*\to M$ be the dual bundle.
Let $G$ be a continuous function on $E^*$ which is positive convex and complex homogeneous on the fibers on $E^*$.
Royden's criterion implies that when $G=G(x,\eta)$ is
\begin{itemize}
\item strictly convex;
\item of class $C^1$ in each fiber; and
\item totally differentiable in the $x$-direction and the $x$-partial derivatives are continuous on the total space,
\end{itemize}
then the dual metric 
$$
F(x,\xi)=\sup\{{\rm Re}(\eta(\xi))\mid \eta\in E^*_x, G(x,\eta)=1\}
$$
on $E$ is strictly convex and of class $C^1$. It is remarkable that the regularity of the dual metric $F$ is (seemed to be) stronger than that of the initial metric $G$. Namely, when $G$, in addition, satisfies the reflexive duality in the sense that
$$
G(x,\eta)=\sup\{{\rm Re}(\eta(\xi))\mid \xi\in E_x, F(x,\xi)=1\},
$$
then $G$ gets a better regularity, that is, $G$ was of class $C^1$, which is apparently stronger than the initial asumption of $G$. As noted in \S\ref{chap:introduction}, the $L^1$-norm function and the Teichm\"uller metric are in this situation. Thus, the reflexive duality is possibly an important condition for the Finsler metrics.

\bibliographystyle{plain}
\bibliography{References-1}

\def\cprime{$'$} \def\cprime{$'$}
\begin{thebibliography}{10}

\bibitem{MR1323428}
Marco Abate and Giorgio Patrizio.
\newblock {\em Finsler metrics---a global approach}, volume 1591 of {\em Lecture Notes in Mathematics}.
\newblock Springer-Verlag, Berlin, 1994.
\newblock With applications to geometric function theory.

\bibitem{MR1635773}
Marco Abate and Giorgio Patrizio.
\newblock Isometries of the {T}eichm\"{u}ller metric.
\newblock {\em Ann. Scuola Norm. Sup. Pisa Cl. Sci. (4)}, 26(3):437--452, 1998.

\bibitem{MR590044}
William Abikoff.
\newblock {\em The real analytic theory of {T}eichm\"uller space}, volume 820 of {\em Lecture Notes in Mathematics}.
\newblock Springer, Berlin, 1980.

\bibitem{MR0204641}
Lars~V. Ahlfors.
\newblock Some remarks on {T}eichm\"{u}ller's space of {R}iemann surfaces.
\newblock {\em Ann. of Math. (2)}, 74:171--191, 1961.

\bibitem{MR0136730}
Lars~V. Ahlfors.
\newblock Curvature properties of {T}eichm\"{u}ller's space.
\newblock {\em J. Analyse Math.}, 9:161--176, 1961/62.

\bibitem{MR1172107}
Tadashi Aikou.
\newblock On complex {F}insler manifolds.
\newblock {\em Rep. Fac. Sci. Kagoshima Univ. Math. Phys. Chem.}, (24):9--25, 1991.

\bibitem{MR4321185}
Vincent Alberge and Athanase Papadopoulos.
\newblock On five papers by {H}erbert {G}r\"{o}tzsch.
\newblock In {\em Handbook of {T}eichm\"{u}ller theory. {V}ol. {VII}}, volume~30 of {\em IRMA Lect. Math. Theor. Phys.}, pages 393--415. Eur. Math. Soc., Z\"{u}rich, [2020] \copyright 2020.

\bibitem{MR2913101}
Jayadev Athreya, Alexander Bufetov, Alex Eskin, and Maryam Mirzakhani.
\newblock Lattice point asymptotics and volume growth on {T}eichm{\"u}ller space.
\newblock {\em Duke Math. J.}, 161(6):1055--1111, 2012.

\bibitem{MR3071503}
Artur Avila and S{\'e}bastien Gou{\"e}zel.
\newblock Small eigenvalues of the {L}aplacian for algebraic measures in moduli space, and mixing properties of the {T}eichm{\"u}ller flow.
\newblock {\em Ann. of Math. (2)}, 178(2):385--442, 2013.

\bibitem{MR2264836}
Artur Avila, S{\'e}bastien Gou{\"e}zel, and Jean-Christophe Yoccoz.
\newblock Exponential mixing for the {T}eichm{\"u}ller flow.
\newblock {\em Publ. Math. Inst. Hautes \'Etudes Sci.}, (104):143--211, 2006.

\bibitem{MR2316268}
Artur Avila and Marcelo Viana.
\newblock Simplicity of {L}yapunov spectra: proof of the {Z}orich-{K}ontsevich conjecture.
\newblock {\em Acta Math.}, 198(1):1--56, 2007.

\bibitem{MR0430318}
Lipman Bers.
\newblock Fiber spaces over {T}eichm{\"u}ller spaces.
\newblock {\em Acta. Math.}, 130:89--126, 1973.

\bibitem{MR1211412}
Albert Boggess.
\newblock {\em C{R} manifolds and the tangential {C}auchy-{R}iemann complex}.
\newblock Studies in Advanced Mathematics. CRC Press, Boca Raton, FL, 1991.

\bibitem{MR1853077}
Ana Cannas~da Silva.
\newblock {\em Lectures on symplectic geometry}, volume 1764 of {\em Lecture Notes in Mathematics}.
\newblock Springer-Verlag, Berlin, 2001.

\bibitem{MR3413977}
David Dumas.
\newblock Skinning maps are finite-to-one.
\newblock {\em Acta Math.}, 215(1):55--126, 2015.

\bibitem{MR1009162}
Nelson Dunford and Jacob~T. Schwartz.
\newblock {\em Linear operators. {P}art {I}}.
\newblock Wiley Classics Library. John Wiley \& Sons, Inc., New York, 1988.
\newblock General theory, With the assistance of William G. Bade and Robert G. Bartle, Reprint of the 1958 original, A Wiley-Interscience Publication.

\bibitem{MR2010740}
Alex Eskin, Howard Masur, and Anton Zorich.
\newblock Moduli spaces of abelian differentials: the principal boundary, counting problems, and the {S}iegel-{V}eech constants.
\newblock {\em Publ. Math. Inst. Hautes \'Etudes Sci.}, (97):61--179, 2003.

\bibitem{MR1724021}
Robert~J. Fisher and H.~Turner Laquer.
\newblock Second order tangent vectors in {R}iemannian geometry.
\newblock {\em J. Korean Math. Soc.}, 36(5):959--1008, 1999.

\bibitem{MR648106}
Otto Forster.
\newblock {\em Lectures on {R}iemann surfaces}, volume~81 of {\em Graduate Texts in Mathematics}.
\newblock Springer-Verlag, New York-Berlin, 1981.
\newblock Translated from the German by Bruce Gilligan.

\bibitem{MR1031389}
Masaki Fukui.
\newblock Complex {F}insler manifolds.
\newblock {\em J. Math. Kyoto Univ.}, 29(4):609--624, 1989.

\bibitem{MR903027}
Frederick~P. Gardiner.
\newblock {\em Teichm{\"u}ller theory and quadratic differentials}.
\newblock Pure and Applied Mathematics (New York). John Wiley \& Sons, Inc., New York, 1987.
\newblock A Wiley-Interscience Publication.

\bibitem{MR1730906}
Frederick~P. Gardiner and Nikola Lakic.
\newblock {\em Quasiconformal {T}eichm\"{u}ller theory}, volume~76 of {\em Mathematical Surveys and Monographs}.
\newblock American Mathematical Society, Providence, RI, 2000.

\bibitem{MR1326623}
H.~Grauert.
\newblock Theory of {$q$}-convexity and {$q$}-concavity.
\newblock In {\em Several complex variables, {VII}}, volume~74 of {\em Encyclopaedia Math. Sci.}, pages 259--284. Springer, Berlin, 1994.

\bibitem{MR0281134}
Joseph Grifone.
\newblock Sur les connexions d'une vari\'{e}t\'{e} finsl\'{e}rienne et d'un syst\`eme m\'{e}canique.
\newblock {\em C. R. Acad. Sci. Paris S\'{e}r. A-B}, 272:A1510--A1513, 1971.

\bibitem{MR0336636}
Joseph Grifone.
\newblock Structure presque-tangente et connexions. {I}.
\newblock {\em Ann. Inst. Fourier (Grenoble)}, 22(1):287--334, 1972.

\bibitem{MR0341361}
Joseph Grifone.
\newblock Structure presque-tangente et connexions. {II}.
\newblock {\em Ann. Inst. Fourier (Grenoble)}, 22(3):291--338. (loose errata), 1972.

\bibitem{MR0245787}
Richard~S. Hamilton.
\newblock Extremal quasiconformal mappings with prescribed boundary values.
\newblock {\em Trans. Amer. Math. Soc.}, 138:399--406, 1969.

\bibitem{MR523212}
John Hubbard and Howard Masur.
\newblock Quadratic differentials and foliations.
\newblock {\em Acta Math.}, 142(3-4):221--274, 1979.

\bibitem{MR2245223}
John~Hamal Hubbard.
\newblock {\em {T}eichm{\"u}ller theory and applications to geometry, topology, and dynamics. {V}ol. 1}.
\newblock Matrix Editions, Ithaca, NY, 2006.
\newblock {T}eichm{\"u}ller theory, With contributions by Adrien Douady, William Dunbar, Roland Roeder, Sylvain Bonnot, David Brown, Allen Hatcher, Chris Hruska and Sudeb Mitra, With forewords by William Thurston and Clifford Earle.

\bibitem{MR0624820}
Y\^{o}ichi Imayoshi.
\newblock Holomorphic families of {R}iemann surfaces and {T}eichm\"{u}ller spaces.
\newblock In {\em Riemann surfaces and related topics: {P}roceedings of the 1978 {S}tony {B}rook {C}onference ({S}tate {U}niv. {N}ew {Y}ork, {S}tony {B}rook, {N}.{Y}., 1978)}, volume No. 97 of {\em Ann. of Math. Stud.}, pages 277--300. Princeton Univ. Press, Princeton, NJ, 1981.

\bibitem{MR0955842}
Y\^{o}ichi Imayoshi and Hiroshige Shiga.
\newblock A finiteness theorem for holomorphic families of {R}iemann surfaces.
\newblock In {\em Holomorphic functions and moduli, {V}ol. {II} ({B}erkeley, {CA}, 1986)}, volume~11 of {\em Math. Sci. Res. Inst. Publ.}, pages 207--219. Springer, New York, 1988.

\bibitem{MR1215481}
Yoichi Imayoshi and Masahiko Taniguchi.
\newblock {\em An introduction to {T}eichm{\"u}ller spaces}.
\newblock Springer-Verlag, Tokyo, 1992.

\bibitem{MR842190}
Birger Iversen.
\newblock {\em Cohomology of sheaves}.
\newblock Universitext. Springer-Verlag, Berlin, 1986.

\bibitem{MR1386110}
Shingo Kawai.
\newblock The symplectic nature of the space of projective connections on {R}iemann surfaces.
\newblock {\em Math. Ann.}, 305(1):161--182, 1996.

\bibitem{MR499279}
Yujiro Kawamata.
\newblock On deformations of compactifiable complex manifolds.
\newblock {\em Math. Ann.}, 235(3):247--265, 1978.

\bibitem{MR0247599}
Joseph Klein and Andr\'{e} Voutier.
\newblock Formes ext\'{e}rieures g\'{e}n\'{e}ratrices de sprays.
\newblock {\em Ann. Inst. Fourier (Grenoble)}, 18:241--260, 1968.

\bibitem{MR96276}
Shoshichi Kobayashi.
\newblock Theory of connections.
\newblock {\em Ann. Mat. Pura Appl. (4)}, 43:119--194, 1957.

\bibitem{MR0377126}
Shoshichi Kobayashi.
\newblock Negative vector bundles and complex {F}insler structures.
\newblock {\em Nagoya Math. J.}, 57:153--166, 1975.

\bibitem{MR1403586}
Shoshichi Kobayashi.
\newblock Complex {F}insler vector bundles.
\newblock In {\em Finsler geometry ({S}eattle, {WA}, 1995)}, volume 196 of {\em Contemp. Math.}, pages 145--153. Amer. Math. Soc., Providence, RI, 1996.

\bibitem{MR2194466}
Shoshichi Kobayashi.
\newblock {\em Hyperbolic manifolds and holomorphic mappings}.
\newblock World Scientific Publishing Co. Pte. Ltd., Hackensack, NJ, second edition, 2005.
\newblock An introduction.

\bibitem{MR1393940}
Shoshichi Kobayashi and Katsumi Nomizu.
\newblock {\em Foundations of differential geometry. {V}ol. {I}}.
\newblock Wiley Classics Library. John Wiley \& Sons, Inc., New York, 1996.
\newblock Reprint of the 1963 original, A Wiley-Interscience Publication.

\bibitem{MR1393941}
Shoshichi Kobayashi and Katsumi Nomizu.
\newblock {\em Foundations of differential geometry. {V}ol. {II}}.
\newblock Wiley Classics Library. John Wiley \& Sons, Inc., New York, 1996.
\newblock Reprint of the 1969 original, A Wiley-Interscience Publication.

\bibitem{MR815922}
Kunihiko Kodaira.
\newblock {\em Complex manifolds and deformation of complex structures}, volume 283 of {\em Grundlehren der Mathematischen Wissenschaften [Fundamental Principles of Mathematical Sciences]}.
\newblock Springer-Verlag, New York, 1986.
\newblock Translated from the Japanese by Kazuo Akao, With an appendix by Daisuke Fujiwara.

\bibitem{MR1684522}
Katarzyna Konieczna and Pawe\l Urba\'{n}ski.
\newblock Double vector bundles and duality.
\newblock {\em Arch. Math. (Brno)}, 35(1):59--95, 1999.

\bibitem{MR2000471}
Maxim Kontsevich and Anton Zorich.
\newblock Connected components of the moduli spaces of {A}belian differentials with prescribed singularities.
\newblock {\em Invent. Math.}, 153(3):631--678, 2003.

\bibitem{MR1142683}
Samuel~L. Krushkal.
\newblock The {G}reen function of {T}eichm{\"u}ller spaces with applications.
\newblock {\em Bull. Amer. Math. Soc. (N.S.)}, 27(1):143--147, 1992.

\bibitem{MR0344463}
O.~Lehto and K.~I. Virtanen.
\newblock {\em Quasiconformal mappings in the plane}.
\newblock Springer-Verlag, New York-Heidelberg, second edition, 1973.
\newblock Translated from the German by K. W. Lucas, Die Grundlehren der mathematischen Wissenschaften, Band 126.

\bibitem{MR867407}
Olli Lehto.
\newblock {\em Univalent functions and {T}eichm\"{u}ller spaces}, volume 109 of {\em Graduate Texts in Mathematics}.
\newblock Springer-Verlag, New York, 1987.

\bibitem{MR1816050}
L\'aszl\'o{} Lempert and R\'obert Sz\H~oke.
\newblock The tangent bundle of an almost complex manifold.
\newblock {\em Canad. Math. Bull.}, 44(1):70--79, 2001.

\bibitem{MR0417456}
Howard Masur.
\newblock Extension of the {W}eil-{P}etersson metric to the boundary of {T}eichmuller space.
\newblock {\em Duke Math. J.}, 43(3):623--635, 1976.

\bibitem{MR644018}
Howard Masur.
\newblock Interval exchange transformations and measured foliations.
\newblock {\em Ann. of Math. (2)}, 115(1):169--200, 1982.

\bibitem{MR1283559}
Howard Masur.
\newblock The {T}eichm\"{u}ller flow is {H}amiltonian.
\newblock {\em Proc. Amer. Math. Soc.}, 123(12):3739--3747, 1995.

\bibitem{MR2655331}
Howard Masur.
\newblock Geometry of {T}eichm\"{u}ller space with the {T}eichm\"{u}ller metric.
\newblock In {\em Surveys in differential geometry. {V}ol. {XIV}. {G}eometry of {R}iemann surfaces and their moduli spaces}, volume~14 of {\em Surv. Differ. Geom.}, pages 295--313. Int. Press, Somerville, MA, 2009.

\bibitem{MR1214233}
Howard Masur and John Smillie.
\newblock Quadratic differentials with prescribed singularities and pseudo-{A}nosov diffeomorphisms.
\newblock {\em Comment. Math. Helv.}, 68(2):289--307, 1993.

\bibitem{MR3822740}
Carlos Matheus Silva~Santos.
\newblock {\em Dynamical aspects of {T}eichm\"{u}ller theory}, volume~7 of {\em Atlantis Studies in Dynamical Systems}.
\newblock Atlantis Press, [Paris]; Springer, Cham, 2018.
\newblock $SL(2,\Bbb R)$-action on moduli spaces of flat surfaces.

\bibitem{MR3033515}
Vladimir~S. Matveev and Marc Troyanov.
\newblock The {B}inet-{L}egendre metric in {F}insler geometry.
\newblock {\em Geom. Topol.}, 16(4):2135--2170, 2012.

\bibitem{MR1031909}
C.~McMullen.
\newblock Iteration on {T}eichm\"{u}ller space.
\newblock {\em Invent. Math.}, 99(2):425--454, 1990.

\bibitem{MR2525030}
Curtis~T. McMullen.
\newblock Rigidity of {T}eichm\"{u}ller curves.
\newblock {\em Math. Res. Lett.}, 16(4):647--649, 2009.

\bibitem{MR2428390}
Peter~W. Michor.
\newblock {\em Topics in differential geometry}, volume~93 of {\em Graduate Studies in Mathematics}.
\newblock American Mathematical Society, Providence, RI, 2008.

\bibitem{MR4028456}
Hideki Miyachi.
\newblock Pluripotential theory on {T}eichm\"{u}ller space {I}: {P}luricomplex {G}reen function.
\newblock {\em Conform. Geom. Dyn.}, 23:221--250, 2019.

\bibitem{MR4633651}
Hideki Miyachi.
\newblock Pluripotential theory on {T}eichm\"{u}ller space {II}---{P}oisson integral formula.
\newblock {\em Adv. Math.}, 432:Paper No. 109265, 64, 2023.

\bibitem{https://doi.org/10.48550/arxiv.2211.16132}
Hideki Miyachi, Ken'Ichi Ohshika, and Athanase Papadopoulos.
\newblock The {T}eichm{{\"u}}ller-randers metric, To appear in Annales de l'Institut Fourier.

\bibitem{MR2102340}
Gheorghe Munteanu.
\newblock {\em Complex spaces in {F}insler, {L}agrange and {H}amilton geometries}, volume 141 of {\em Fundamental Theories of Physics}.
\newblock Kluwer Academic Publishers, Dordrecht, 2004.

\bibitem{MR927291}
Subhashis Nag.
\newblock {\em The complex analytic theory of {T}eichm\"{u}ller spaces}.
\newblock Canadian Mathematical Society Series of Monographs and Advanced Texts. John Wiley \& Sons, Inc., New York, 1988.
\newblock A Wiley-Interscience Publication.

\bibitem{MR0388432}
Jean Pradines.
\newblock Repr\'{e}sentation des jets non holonomes par des morphismes vectoriels doubles soud\'{e}s.
\newblock {\em C. R. Acad. Sci. Paris S\'{e}r. A}, 278:1523--1526, 1974.

\bibitem{MR0211370}
Giovanni~Battista Rizza.
\newblock {$F$}-forme quadratiche ed hermitiane.
\newblock {\em Rend. Mat. e Appl. (5)}, 23:221--249, 1964.

\bibitem{MR0833808}
H.~L. Royden.
\newblock Complex {F}insler metrics.
\newblock In {\em Complex differential geometry and nonlinear differential equations ({B}runswick, {M}aine, 1984)}, volume~49 of {\em Contemp. Math.}, pages 119--124. Amer. Math. Soc., Providence, RI, 1986.

\bibitem{MR0288254}
Halsey~L. Royden.
\newblock Automorphisms and isometries of {T}eichm{\"u}ller space.
\newblock In {\em Advances in the {T}heory of {R}iemann {S}urfaces ({P}roc. {C}onf., {S}tony {B}rook, {N}.{Y}., 1969)}, pages 369--383. Ann. of Math. Studies, No. 66. Princeton Univ. Press, Princeton, N.J., 1971.

\bibitem{MR0216446}
Hanno Rund.
\newblock Generalized metrics on complex manifolds.
\newblock {\em Math. Nachr.}, 34:55--77, 1967.

\bibitem{MR1390760}
Takashi Sakai.
\newblock {\em Riemannian geometry}, volume 149 of {\em Translations of Mathematical Monographs}.
\newblock American Mathematical Society, Providence, RI, 1996.
\newblock Translated from the 1992 Japanese original by the author.

\bibitem{MR0393477}
Kurt Strebel.
\newblock On the trajectory structure of quadratic differentials.
\newblock In {\em Discontinuous groups and {R}iemann surfaces ({P}roc. {C}onf., {U}niv. {M}aryland, {C}ollege {P}ark, {M}d., 1973)}, volume No. 79 of {\em Ann. of Math. Stud.}, pages 419--438. Princeton Univ. Press, Princeton, NJ, 1974.

\bibitem{MR0505549}
Kurt Strebel.
\newblock On quasiconformal mappings of open {R}iemann surfaces.
\newblock {\em Comment. Math. Helv.}, 53(3):301--321, 1978.

\bibitem{MR2023456}
Shigeru Takamura.
\newblock Towards the classification of atoms of degenerations. {I}. {S}plitting criteria via configurations of singular fibers.
\newblock {\em J. Math. Soc. Japan}, 56(1):115--145, 2004.

\bibitem{MR2254876}
Shigeru Takamura.
\newblock {\em Splitting deformations of degenerations of complex curves}, volume 1886 of {\em Lecture Notes in Mathematics}.
\newblock Springer-Verlag, Berlin, 2006.
\newblock Towards the classification of atoms of degenerations. III.

\bibitem{MR0003242}
Oswald Teichm{\"u}ller.
\newblock {E}xtremale quasikonforme {A}bbildungen und quadratische {D}ifferentiale.
\newblock {\em Abh. Preuss. Akad. Wiss. Math.-Nat. Kl.}, 1939(22):197, 1940.

\bibitem{MR1164870}
Anthony~J. Tromba.
\newblock {\em {T}eichm{\"u}ller theory in {R}iemannian geometry}.
\newblock Lectures in Mathematics ETH Z\"urich. Birkh\"auser Verlag, Basel, 1992.
\newblock Lecture notes prepared by Jochen Denzler.

\bibitem{MR866707}
William~A. Veech.
\newblock The {T}eichm{\"u}ller geodesic flow.
\newblock {\em Ann. of Math. (2)}, 124(3):441--530, 1986.

\bibitem{MR1094714}
William~A. Veech.
\newblock Moduli spaces of quadratic differentials.
\newblock {\em J. Analyse Math.}, 55:117--171, 1990.

\bibitem{MR0982185}
Michael Wolf.
\newblock The {T}eichm\"{u}ller theory of harmonic maps.
\newblock {\em J. Differential Geom.}, 29(2):449--479, 1989.

\bibitem{MR0842050}
Scott~A. Wolpert.
\newblock Chern forms and the {R}iemann tensor for the moduli space of curves.
\newblock {\em Invent. Math.}, 85(1):119--145, 1986.

\bibitem{MR2641916}
Scott~A. Wolpert.
\newblock {\em Families of {R}iemann surfaces and {W}eil-{P}etersson geometry}, volume 113 of {\em CBMS Regional Conference Series in Mathematics}.
\newblock Published for the Conference Board of the Mathematical Sciences, Washington, DC; by the American Mathematical Society, Providence, RI, 2010.

\bibitem{MR1123804}
Hiroshi Yanagihara.
\newblock Variational formula of inverse of quasiconformal mappings.
\newblock {\em Complex Variables Theory Appl.}, 17(1-2):73--78, 1991.

\bibitem{MR0350650}
Kentaro Yano and Shigeru Ishihara.
\newblock {\em Tangent and cotangent bundles: differential geometry}, volume No. 16 of {\em Pure and Applied Mathematics}.
\newblock Marcel Dekker, Inc., New York, 1973.

\end{thebibliography}

\end{document}